\documentclass{memo-l}

\usepackage[all,cmtip,curve]{xy}
\usepackage{amssymb,amsmath,mathrsfs,amsthm,amscd,booktabs,
hyperref,amsfonts,graphicx}

\usepackage{tikz}
\usetikzlibrary{shapes,arrows}

\newtheorem{theorem}{Theorem}[chapter]
\newtheorem{lemma}[theorem]{Lemma}
\newtheorem{cor}[theorem]{Corollary}

\newtheorem{prop}[theorem]{Proposition}
\newtheorem{proposition}[theorem]{Proposition}

\theoremstyle{definition}

\theoremstyle{remark}
\newtheorem{remark}[theorem]{Remark}
\newtheorem{remarks}[theorem]{Remarks}

\numberwithin{section}{chapter}
\numberwithin{equation}{chapter}

\usepackage{color}

\DeclareMathOperator{\dvsr}{div}
\newcommand{\zero}{{\bf 0}}
\newcommand{\one}{{\bf 1}}
\newcommand{\Div}{\operatorname{Div}}
\newcommand{\divi}{\operatorname{div}}
\newcommand{\sep}{\operatorname{sep}}
\newcommand{\XminusT}{(x-T)}
\newcommand{\xminusT}{(x-T)'}
\newcommand{\GG}{\mathcal{G}}
\newcommand{\iso}{\stackrel{\sim}{\rightarrow}}
\newcommand{\val}{\operatorname{val}}
\newcommand{\pr}{\operatorname{pr}}
\newcommand{\rk}{\operatorname{rank}}
\newcommand{\vd}{\mathbf{d}}
\newcommand{\ve}{\mathbf{e}}

\DeclareMathOperator{\car}{Car}

\def\Hom{\mathrm{Hom}}
\def\tor{\mathrm{tor}}
\def\tr{\mathrm{tr}}
\DeclareMathOperator{\End}{End}
\def\Pic{\textnormal{Pic}}

\def\Aut{\textrm{\upshape Aut}}
\def\Gal{\textrm{\upshape Gal}}
\def\ord{\textrm{\upshape ord}}

\def\p{\mathbb{P}} 
\def\P{\mathbb{P}}

\def\F{\mathbb{F}}
\def\G{\mathbb{G}}
\def\Q{\mathbb{Q}}

\def\Z{\mathbb{Z}}

\def\II{\mathcal{I}}
\def\JJ{\mathcal{J}}
\def\OO{\mathcal{O}}

\newcommand{\n}{\mathfrak{n}}

\usepackage[OT2,T1]{fontenc}
\DeclareSymbolFont{cyrletters}{OT2}{wncyr}{m}{n}
\DeclareMathSymbol{\sha}{\mathalpha}{cyrletters}{"58}

\def\<{\langle}
\def\>{\rangle}
\def\into{\hookrightarrow}
\def\onto{\twoheadrightarrow}
\def\isoto{\tilde{\to}}
\def\tensor{\otimes}
\def\compose{\circ}
\def\sdp{{\rtimes}}
\def\nodiv{\nmid}%{\not|}
\def\PSL{\mathrm{PSL}}
\def\SL{\mathrm{SL}}
\def\GL{\mathrm{GL}}

\def\im{\text{im}}

\DeclareMathOperator{\res}{Res}
\DeclareMathOperator{\spec}{Spec}

\newcommand{\AC}{\mathcal{A}}
\newcommand{\CC}{\mathcal{C}}
\newcommand{\DD}{\mathcal{D}}
\renewcommand{\SS}{\mathcal{S}}
\newcommand{\TT}{\mathcal{T}}
\newcommand{\UU}{\mathcal{U}}
\newcommand{\VV}{\mathcal{V}}
\newcommand{\XX}{\mathcal{X}}
\newcommand{\YY}{\mathcal{Y}}
\newcommand{\ZZ}{\mathcal{Z}}
\newcommand{\EE}{\mathcal{E}}
\newcommand{\LL}{\mathcal{L}}

\newcommand{\Fp}{{\mathbb{F}_p}}
\newcommand{\Fq}{{\mathbb{F}_q}}
\newcommand{\Fqn}{{\mathbb{F}_{q^n}}}
\newcommand{\Fpbar}{{\overline{\mathbb{F}}_p}}
\newcommand{\Qbar}{{\overline{\mathbb{Q}}}}
\newcommand{\kbar}{{\overline{k}}}
\newcommand{\Kbar}{{\overline{K}}}
\newcommand{\ratto}{{\dashrightarrow}}
\newcommand{\Ql}{{\mathbb{Q}_\ell}}
\newcommand{\Zl}{{\mathbb{Z}_\ell}}
\newcommand{\Zp}{{\mathbb{Z}_p}}
\newcommand{\A}{\mathbb{A}}
\newcommand{\Ksep}{K^{sep}}

\newcommand\Fpt{\Fp(t)}
\newcommand\Fqu{\Fq(u)}
\newcommand\Cr{C_r}
\newcommand\Jr{J_r}
\newcommand\C{C}
\newcommand\J{J}
\newcommand\Jnew{{J_r^{\mathrm{new}}}}
\newcommand\rp{{r'}}
\newcommand\Crp{C_\rp}
\newcommand\Jrp{J_\rp}

\newcommand{\K}{K}

\newcommand{\psmat}[1]{\bigl(\begin{smallmatrix}#1\end{smallmatrix}\bigr)}
\newcommand{\pmat}[1]{\begin{pmatrix}#1\end{pmatrix}}

\DeclareMathOperator{\NS}{NS}
\DeclareMathOperator{\proj}{Proj}
\DeclareMathOperator{\ch}{Char}
\DeclareMathOperator{\cha}{char}
\newcommand{\Fl}{{\F_\ell}}
\newcommand{\Flbar}{{\overline{\F}_\ell}}
\DeclareMathOperator{\disc}{Disc}
\DeclareMathOperator{\MW}{MW}

\DeclareMathOperator{\gal}{Gal}
\DeclareMathOperator{\Fr}{Fr}
\DeclareMathOperator{\lcm}{lcm}
\DeclareMathOperator{\stab}{Stab}

\newcommand\Cbar{\overline{C}}

\newcommand\Pt{\P^1_t}
\newcommand\Pu{\P^1_u}

\newcommand\BB{\mathcal{B}}
\newcommand\FlG{\Fl[G]}
\newcommand\MWJ{\MW(J)}
\newcommand\NSS{\NS(S)}

\newcommand\Lone{L^1\NSS}
\newcommand\Ltwo{L^2\NSS}
\newcommand\Lonebar{\bar{L}^1\NSS}
\newcommand\Ltwobar{\bar{L}^2\NSS}

\newcommand\KSJ{\mathrm{KS}_\JJ}
\newcommand\KSY{\mathrm{KS}_\YY}
\newcommand\Lie{\underline{\mathrm{Lie}}}
\newcommand\czech{\v Cech}%{\v Cech}

\newcommand\dOne{{b_1}}
\newcommand\dTwo{{b_2}}
\newcommand\eOne{{a_1}}
\newcommand\eTwo{{a_2}}
\newcommand\eThree{{a_3}}

\newcommand\Neron{N\'eron}

\newcommand\RoneOmegaYtensorOmegaU{
  {\Omega^1_U \tensor_{\OO_U} R^1\pi_*\Omega^1_{\YY/U} }
}
\newcommand\RoneOYtensorOmegaU{
  {\Omega^1_U \tensor_{\OO_U} R^1\pi_*\OO_\YY}
}
\newcommand\RoneOJtensorOmegaU{{\Omega^1_U\tensor_{\OO_U}R^1\sigma_*\OO_\JJ}}

\makeindex

\begin{document}

\frontmatter

\title{Explicit arithmetic of Jacobians of generalized Legendre curves over global function fields}

\author[]{Lisa\ Berger}
\address{Department of Mathematics, 
Stony Brook University, 
Stony Brook, NY 11794, USA}
\email{lbrgr@math.sunysb.edu}

\author[]{Chris\ Hall}
\address{Department of Mathematics, 
Western University,
London, Ontario, N6A 5B7, Canada}
\email{chall69@uwo.ca}

\author[]{Ren\'{e}\ Pannekoek}
\address{
Department of Mathematics,
Imperial College London
SW7 2AZ, United Kingdom}
\email{pannekoek@gmail.com}

\author[]{Jennifer\ Park}
\address{Department of Mathematics, 
University of Michigan, 
Ann Arbor, MI 48109, USA }
\email{jmypark@umich.edu}

\author[]{Rachel\ Pries}
\address{Department of Mathematics, 
Colorado State University, 
Fort Collins, CO 80523, USA}
\email{pries@math.colostate.edu}

\author[]{Shahed\ Sharif}
\address{Department of Mathematics, 
CSU San Marcos, 
San Marcos, CA 92096, USA}
\email{ssharif@csusm.edu}

\author[]{Alice\ Silverberg}
\address{Department of Mathematics, 
UC Irvine, 
Irvine, CA 92697, USA}
\email{asilverb@math.uci.edu}

\author[]{Douglas\ Ulmer}
\address{School of Mathematics, 
Georgia Institute of Technology, 
Atlanta, GA 30332, USA}
\email{ulmer@math.gatech.edu}

% kludge to fix title page spacing:
\author[]{\vbox to 3in{}}

\thanks{ This project was initiated at the workshop on Cohomological
  methods in abelian varieties at the American Institute of
  Mathematics, March 26--30, 2012.  We thank AIM and the workshop
  organizers for making this paper possible.  The first seven authors
  thank the last for initiating this project at the AIM workshop and
  for his leadership of the project.  Author Hall was partially
  supported by Simons Foundation award 245619 and IAS NSF grant
  DMS-1128155.  Author Park was partially supported by NSF grant
  DMS-10-69236 and NSERC PGS-D and PDF grants.  Author Pries was
  partially supported by NSF grants DMS-11-01712 and DMS-15-02227.
  Author Silverberg was partially supported by NSF grant CNS-0831004.
  Author Ulmer was partially supported by Simons Foundation award
  359573.  Any opinions, findings, and conclusions or recommendations
  expressed in this material are those of the authors and do not
  necessarily reflect the views of the supporting agencies.  We also
  thank Karl Rubin and Yuri G.\ Zarhin for helpful conversations and
  an anonymous referee for a careful reading and productive comments.  }

%    \date is required; it is the date received by the editor.
\date{}

%\subjclass[2010]{Primary }
%    Recognition of the 2010 edition of the Mathematics Subject
%    Classification requires a version of amsbook.cls from July 2009
%    or later.  If "2010" is not recognized, please upgrade.

%\keywords{}

%\dedicatory{Dedication text (use \\[2pt] for line break if necessary)}

\maketitle

\tableofcontents

\begin{abstract}
  We study the Jacobian $J$ of the smooth projective curve $C$ of
  genus $r-1$ with affine model $y^r = x^{r-1}(x + 1)(x + t)$ over the
  function field $\F_p(t)$, when $p$ is prime and $r\ge 2$ is an
  integer prime to $p$.  When $q$ is a power of $p$ and $d$ is a
  positive integer, we compute the $L$-function of $J$ over
  $\F_q(t^{1/d})$ and show that the Birch and Swinnerton-Dyer
  conjecture holds for $J$ over $\F_q(t^{1/d})$.  When $d$ is
  divisible by $r$ and of the form $p^\nu +1$, and $K_d :=
  \F_p(\mu_d,t^{1/d})$, we write down explicit points in $J(K_d)$,
  show that they generate a subgroup $V$ of rank $(r-1)(d-2)$ whose
  index in $J(K_d)$ is finite and a power of $p$, and show that the
  order of the Tate-Shafarevich group of $J$ over $K_d$ is
  $[J(K_d):V]^2$.  When $r>2$, we prove that the ``new''
  part of $J$ is isogenous over $\overline{\Fp(t)}$ to the square of a
  simple abelian variety of dimension $\phi(r)/2$ with endomorphism
  algebra $\Z[\mu_r]^+$.  For a prime $\ell$ with
  $\ell \nmid pr$, we prove that $J[\ell](L)=\{0\}$ for any abelian
  extension $L$ of $\overline{\F}_p(t)$.
\end{abstract}

%    Include unnumbered chapters (preface, acknowledgments, etc.) here.
%\include{}

\mainmatter
%    Include main chapters here.

% !TEX root = EHR.tex
%This is the introduction

\chapter*{Introduction}

It is known that for every prime $p$ and every genus $g>0$, there
exist Jacobians $J$ of dimension $g$ over the rational function field
$K=\Fpt$ such that the rank of $J(K)$ is arbitrarily large
\cite{Ulmer07}.  One of the main goals in this work is to make this
phenomenon more explicit.  Specifically, for any prime number $p$ and
infinitely many positive integers $g$, we exhibit a curve of
genus $g$ over $K$ and explicit divisors on that curve 
that generate a subgroup $V$ of large rank in the Mordell-Weil group
of the Jacobian of the curve.  We also prove precise results on the
conjecture of Birch and Swinnerton-Dyer for these Jacobians, giving
information about the index of the subgroup $V$ in the Mordell-Weil group,
and about the Tate-Shafarevich group of the Jacobian.
 
All of this work generalizes previous results in the case $g=1$ from
\cite{Legendre, Legendre2, Legendre3}.  In those papers, the authors
analyze the arithmetic of the Legendre curve $y^2 = x(x+1)(x+t)$, an
elliptic curve defined over $K$.  For each field $K_d$ appearing in a
tower of field extensions of $K$, they prove that the
Legendre curve over $K_d$ satisfies the conjecture of Birch and
Swinnerton-Dyer.  Furthermore, for infinitely many $d$, they find
explicit divisors on the Legendre curve that generate a subgroup $V$ of
large rank in the Mordell-Weil group.  They bound the
index of the subgroup $V$ in the Mordell-Weil group and give results
about the Tate-Shafarevich group. 

The statements of the main results in this paper are quite parallel to
those for the Legendre elliptic curve.  However, since we
work in higher genus---where the curve and its Jacobian are distinct
objects---the proofs are more complicated and require more advanced
algebraic geometry.  For example, we have to construct the regular
minimal model of our curve from first principles (rather than relying
on Tate's algorithm), the relations among the points we write down are
less evident, and the analysis of torsion in the Jacobian requires
more work.  Moreover, our results cast new light on those of
\cite{Legendre} insofar as we determine the structure of the group of
points under consideration as a module over a suitable group ring.  

As part of our analysis, we prove several results in more generality
than needed here, and these results may be of use in analyzing the
arithmetic of other curves over function fields.  These include a
proof that the N\'eron-Severi group of a general class of surfaces is
torsion-free (Propositions~\ref{prop:geometric-method} and
\ref{prop:cohom-method}) and an integrality result for heights on
Jacobians (Proposition~\ref{prop:integrality}).  We also note that the
monodromy questions answered in the last chapter inspired a related
work \cite{Hall} in which a new method to compute monodromy groups of
superelliptic curves is developed.

\section*{Historical background}

Let $g$ be a positive integer.  Over a fixed number field, it is not
known whether there exist Jacobian varieties of dimension $g$ whose
Mordell-Weil groups have arbitrarily large rank.  In contrast, there
are several results of this type over a fixed function field, some of
which we describe below.

In \cite{TateShaf67}, Shafarevich and Tate construct elliptic curves
with arbitrarily large rank over $\Fq(t)$.  The curves in their
construction are isotrivial, i.e., each is isomorphic, after a finite
extension, to a curve defined over $\Fq$.

In \cite{Shioda86}, Shioda studies the elliptic curve over $k(t)$
defined by $y^2 = x^3 + at^nx + bt^m$ where $k$ is an arbitrary field
and $a,b\in k$ satisfy $ab(4a^3t^{3n}+27b^2t^{2m})\neq 0$.  When
$\text{char}(k)=0$, he proves the rank of the Mordell-Weil group has a
uniform upper bound of $56$, and he gives necessary and sufficient
conditions on $m$ and on $n$ for meeting this bound.  When
$\text{char}(k) = p \equiv -1 \pmod 4$ and when $d=(p^{\nu}+1)/2$ as
$\nu$ varies over positive odd integers, he proves that the elliptic
curves over $\kbar(t)$ defined by $y^2 = x^3 + x +t^{d}$ achieve
arbitrarily large rank.
%, which depends on the value of $p \pmod 3$, as $\nu$ varies over positive odd values.
These curves are given as examples of the main result of \cite{Shioda86},
in which Shioda computes the Picard number for Delsarte surfaces.
Fundamental to this work is the realization  of any Delsarte surface as a
quotient of a Fermat surface.  

Motivated by this work of Shioda, in \cite{UlmerDPCT} Ulmer proves
that the non-isotrivial elliptic curve $y^2 + xy = x^3 - t$ over
$\Fpt$ obtains arbitrarily large rank over the fields $\Fp(t^{1/d})$,
where $d$ ranges over divisors of $p^n +1$.  He realizes the
corresponding elliptic surface as a quotient of a Fermat surface; from
earlier work of Shioda and Katsura \cite{ShiodaKatsura}, this Fermat
surface admits a dominant rational map from a product of Fermat
curves.  It follows that this elliptic curve satisfies the conjecture
of Birch and Swinnerton-Dyer.  Furthermore, the zeta function of the
elliptic surface can be determined from that of the Fermat surface.
Using Jacobi sums, lower bounds are found for the rank of the elliptic
curve over towers of function fields.

The geometric construction in \cite{ShiodaKatsura} is later
generalized in the work of Berger \cite{Berger}, where towers of
surfaces dominated by products of curves are constructed as suitable
blow-ups of products of smooth curves.  In \cite{UlmerDPCT}, Ulmer
elaborates on the geometry and arithmetic of this construction,
proving a formula for the ranks of the Jacobians of the curves
constructed in \cite{Berger}.

In \cite{UlmerZarhin}, Ulmer and Zarhin combine this rank formula with
work on endomorphisms of abelian varieties. For $k$ a field of
characteristic zero, they construct absolutely simple Jacobians over
$k(t)$ with bounded ranks in certain towers of extensions of $k(t)$.
As one example, they prove that the Mordell-Weil group of the Jacobian
of the genus $g$ curve defined by $ty^2 = x^{2g+1} - x + t-1$ has rank
$2g$ over the field $\Qbar(t^{1/p^r})$ for any prime power $p^r$.  In
\cite{PriesUlmerAS}, Pries and Ulmer introduce an analogous
construction of surfaces that are dominated by a product of curves at
each layer in a tower of Artin-Schreier extensions.  They prove a
formula for the ranks of the Jacobians of their curves, and produce
examples of Jacobians with bounded and with unbounded ranks.

Another example from \cite{UlmerDPCT} is the curve over $k(t)$ defined by
$y^2 + xy + ty = x^3 + tx^2$.  For $k$ an algebraically closed field
of characteristic zero and $d$ a positive integer, the curve has rank zero
over the fields $k(t^{1/d})$.  For $k=\Fpbar$ and $d = p^n + 1$, the curve
has rank $d-2$ over $k(t^{1/d})$, and explicit generators are found.
Later, in  \cite{Legendre}, a 2-isogeny to the Legendre
curve $y^2 = x(x+1)(x+t)$ is obtained, 
and this construction motivates our work.

\section*{The main results}

Let $p$ be an odd prime, let $r \ge 2$ be an integer not divisible
by $p$, and let $K=\Fpt$.  Generalizing the results in \cite{Legendre},
\cite{Legendre2}, and \cite{Legendre3}, we consider the smooth projective
curve $\C=\Cr$ of genus $g=r-1$ over $K$ with affine model
$$
%\Cr : 
  y^r=x^{r-1}(x+1)(x+t).
$$
The Jacobian $\Jr$ of $\Cr$ is a principally polarized
abelian variety over $K$ of dimension $g$. 

We study the arithmetic of $J=\Jr$ over extensions of $K$ of the form
$\Fqu$ where $u\in\Kbar$ satisfies $u^d=t$, e.g.,~the extension
$K_d=\Fp(\mu_d,u)$.  Some of our results hold for general data
$p,q,r,d$, while others hold under specific constraints.  We first
state a result in a specific case:

\newtheorem{intro-thm}{Theorem}

\begin{intro-thm} \label{Tintrotheorem1} 
  Let $p$ be a prime number, let $d=p^\nu+1$ for some integer $\nu
  >0$, and let $r$ be a divisor of $d$.  Then there is an explicit
  group of divisors generating a subgroup $V\subset \J(K_d)$ with the
  following properties:
  \begin{enumerate}
  \item The $\Z$-rank of $V$ is $(r-1)(d-2)$ and the torsion of $V$ has order $r^3$.
  \item The index of $V$ in $\J(K_d)$ is finite and a power of $p$.
  \item The Tate-Shafarevich group $\sha(\J/K_d)$ of $\J/K_d$ is finite of order
  $$
    |\sha(\J/K_d)| = [\J(K_d):V]^2.
  $$
  \end{enumerate}
\end{intro-thm}

\noindent
We prove even more about $V$, describing it completely as a module
over a certain group ring and as a lattice with respect to the
canonical height pairing on $\J$.

In the general case, we compute the $L$-function and prove the BSD
conjecture:

\begin{intro-thm}
  Let $p$ be a prime number, let $q$ be a power of $p$, and let $r$
  and $d$ be positive integers not divisible by $p$.  Then:
  \begin{enumerate}
  \item The conjecture of Birch and Swinnerton-Dyer holds for $\J$ over $\Fqu$.
  \item The $L$-function of $\J/\Fqu$ can be expressed explicitly in terms of
    Jacobi sums.  (See Theorem~\ref{thm:L} below for the precise statement.)
  \item For sufficiently large $q$, the order of vanishing of $L(\J/\Fqu,s)$ at
    $s=1$ can be expressed in terms of the action on the set
    $(\Z/d\Z)\times(\Z/r\Z)$ of the subgroup of $(\Z/\lcm(d,r)\Z)^\times$
    generated by $p$.  (See Proposition~\ref{prop:rank} below for the precise
    statement.)
  \end{enumerate}
\end{intro-thm}

\noindent
The rank calculation in this result of course agrees with that given
by the explicit points in the case $d=p^\nu+1$, $r\mid d$, and
$\Fq=\Fp(\mu_d)$.  We expect that there are many other values of $q$,
$r$ and $d$ yielding large ranks, as in \cite{Legendre2}.

Finally, we prove very precise results about the decomposition of $\J$
up to isogeny into simple abelian varieties and about torsion in
abelian extensions.  To state them, note that if $\rp\mid r$, then
there is a surjective map of curves $\Cr\to\Crp$ and a corresponding
homomorphism of Jacobians $\Jr\to\Jrp$ induced by push-forward of
divisors.  We define $\Jnew$ to be the identity component of the
intersection of the kernels of these homomorphisms over all divisors
$\rp$ of $r$ with $\rp<r$.

\begin{intro-thm}\ 
  \begin{enumerate}
  \item If $r=2$, then $\Jnew$ equals $\Jr$, which is an abelian
    variety of dimension 1, and thus $\Jnew$ is absolutely simple.
  \item If $r>2$, then $\Jnew$ is simple over $\Fpbar(t)$, while over
    $\overline{\Fpt}$ it is isogenous to the square of a simple
    abelian variety of dimension $\phi(r)/2$ whose endomorphism algebra is
    the real cyclotomic field $\Q(\mu_r)^+$.
  \item If $L$ is an abelian extension of $\Fpbar(t)$ and if $\ell$ is
    prime with $\ell \nmid r$, then $\Jr(L)$ contains no non-trivial
    elements of order $\ell$.
  \end{enumerate}
\end{intro-thm}

\section*{Overview of the paper}

Our study involves more than one approach to the key result of part
(1) of Theorem \ref{Tintrotheorem1} (the lower bound on the rank of
$\J$ over $K_d$ when $d=p^\nu+1$).  Some of the arguments are more
elementary or less elementary than others, with correspondingly weaker
or stronger results.  We include these multiple approaches so that the
reader may see many techniques in action, and may choose the
approaches that suit his or her temperament and background.

In Chapter 1, we give basic information about the curve $\C$ and Jacobian
$\J$ we are studying.  We write down explicit divisors in the case
$d=p^\nu+1$, and we find relations satisfied by the classes of these
divisors in $\J$.  These relations turn out to be the only ones, but
that is not proved in general until much later in the paper.

In Chapter 2, we assume that $r$ is prime 
and use descent arguments to bound the rank of $\J$ from below in the
case when $d=p^\nu+1$.  The reader who is willing to assume $r$ is prime
need only read these first two chapters to obtain one of the main
results of the paper.

In Chapter 3, we construct the minimal, regular, proper model
$\XX\to\P^1$ of $\C/\Fqu$ for any values of $d$ and $r$.  In
particular, we compute the singular fibers of $\XX\to\P^1$.  This 
% allows us to compute the conductor of $\J$ at each place of $\Fqu$,
% and therefore the degree of the $L$-function of $\J/\Fqu$.  It also
allows us to compute the component groups of the N\'eron model of $\J$.
We also give a precise connection between the model $\XX$ and a
product of curves.

In Chapter 4, we consider the case where $d=p^\nu+1$ and $r\mid d$, and we
compute the heights of the explicit divisors introduced in Chapter 1.  This
allows us to compute the rank of the explicit subgroup $V$ and its structure
over the group ring $\Z[\mu_r\times \mu_d]$.

In Chapter 5, we give an elementary calculation of the $L$-function of
$\J$ over $\Fqu$ (for any $d$ and $r$) in terms of Jacobi sums.  We
also show that the BSD conjecture holds for $\J$, and we give an
elementary calculation of the rank of $\J(\Fqu)$ for any $d$ and $r$
and all sufficiently large $q$.

In the fairly technical Chapters 6 and 7, we prove several results
about the surface $\XX$ that allow us to deduce that the index of $V$
in $\J(\K_d)$ is a power of $p$ when $d=p^\nu+1$ and $r$ divides $d$.
We also use the BSD formula to relate this index to the order of the
Tate-Shafarevich group.

In the equally technical Chapter 8, we prove strong results on the
monodromy of the $\ell$-torsion of $\J$ for $\ell$ prime to $pr$.
This gives precise statements about torsion points on $\J$ over
abelian or solvable extensions of $\Fpbar(t)$ and about the
decomposition of $\J$ up to isogeny into simple abelian varieties.

The methods of this paper can be used to study other curves as well.
We give an explicit family of curves in Section~\ref{s:other-curve}
and point out how some of the results of this paper extend to the
Jacobians of these curves.  At the request of the referee, we include
Appendix~\ref{Sappendix}, in which we give more details on these
examples.  Specifically, we prove a lower bound on the rank of the
Jacobian of the hyperelliptic curve $X$ over $\Fq(t)$ defined by
$ y^2 = x \prod_{i = 1}^g (x+a_i)(t+a_ix) $ with $g$ odd and with
distinct nonzero $a_i \in \Fq$ over fields of the form
$K_d = \Fq(\mu_d, t^{1/d})$ with $d=q^\nu+1$.  We also show that the
BSD conjecture holds for the Jacobian of $X$ over $\F_{q^n}(t^{1/d})$
for all $n$ and all $d$ prime to $p$, and we give an upper bound on
the rank using the $L$-function.

Recently, Ulmer and Voloch~\cite{UlmerVoloch} introduced a family of
curves generalizing those treated in this paper.  
They study curves defined by the equation 
$$y^r = h(x)h(t/x),$$
where $h$ is a polynomial that is not of the form $f^m$, for any
$m \neq 1$, $m \mid r$. (The generalized Legendre curves studied in
this paper and the hyperelliptic curves discussed in the appendix are
all examples from this family.)  They prove that the number of points
in an arithmetic family of such curves is unbounded, and they also
show that the surface over $k$ defined by this equation is dominated
by a product of curves.  The emphasis in \cite{UlmerVoloch} is on
rational points on the curves, whereas techniques from this paper may
be useful for proving interesting results on the Jacobians of these
curves.

\section*{Guide}
The leitfaden below indicates dependencies among the chapters of the
paper.  We also record here the chapters or sections needed to prove
various parts of the main results.

A proof of lower bounds as in Theorem~1(1) (i.e., that the rank of
$\J(K_d)$ is at least $(r-1)(d-2)$ and the torsion has order $r^3$) in
the case where $r$ is prime and divides $d=p^\nu+1$ is contained in
Chapters~1 and 2, and more specifically follows from
Proposition~\ref{prop:torsion} and Theorem~\ref{subgoal}.

The lower bounds of Theorem~1(1) in the case of general $r$ dividing
$d=p^\nu+1$ are proved in Section~\ref{ss:visible} using results from
Chapter 1, Section~\ref{sec:height-pairing}, and earlier parts of
Chapter~\ref{ch:heights}.  Theorem~1(1) is established in full
generality in Corollary~\ref{cor:visible}.

Parts (2) and (3) of Theorem~1 are proved in Chapter 7, specifically
in Theorem~\ref{thm:V-vs-MW} and Theorem~\ref{thm:sha} respectively,
using results from Chapters~1, 3, 4, 5, and 6.

Theorem~2 is proved in Chapter~\ref{ch:L} using results from
Chapters~1 and 3.

Finally, Theorem~3 is proved in Chapter~\ref{ch:monodromy} using
definitions from Chapter~1 and precise information on the N\'eron
model of $\J$ deduced from properties of the regular proper model
$\XX$ of Chapter~\ref{ch:models}.  (To be precise, the claim about
$p$-torsion is not treated in Chapter~\ref{ch:monodromy}, but a
stronger result is proved in Section~\ref{s:KS}.)

\tikzstyle{block} = [rectangle, draw,
    text width=8em, text centered, minimum height=4em]
\tikzstyle{line} = [draw, -latex']
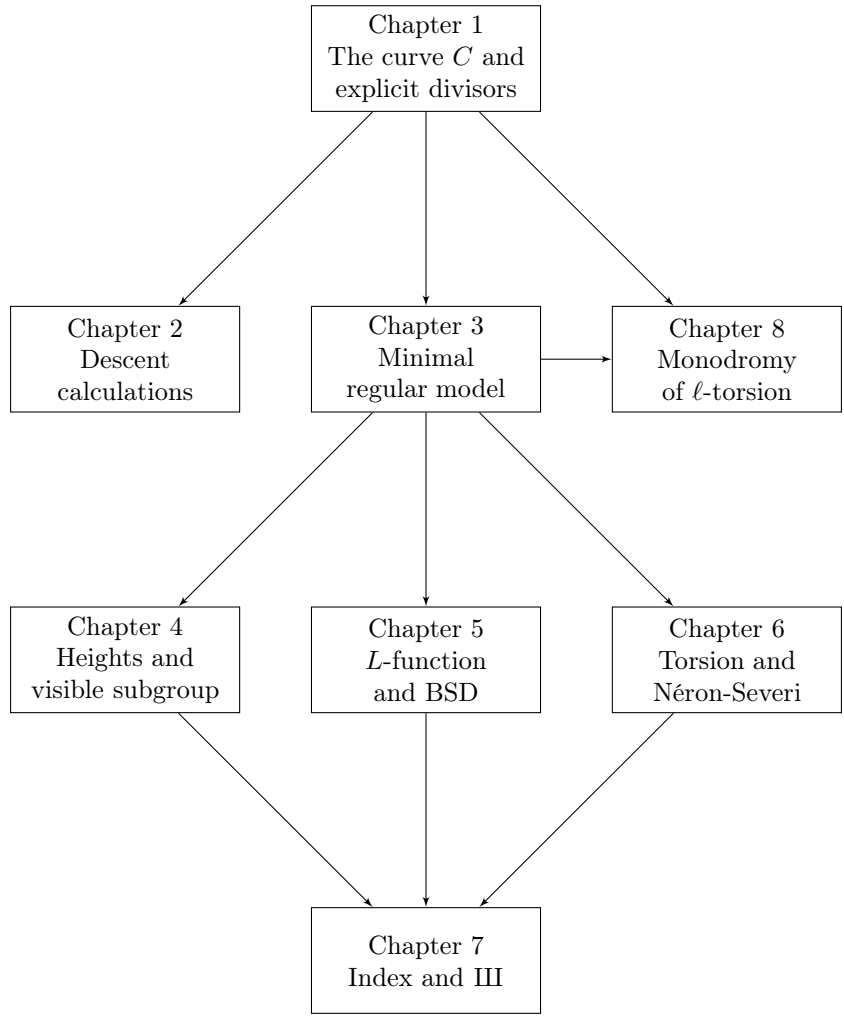
\begin{figure}
\begin{center}
\begin{tikzpicture}[node distance = 4cm]
    % Place nodes
    \node [block] (Ch1) {Chapter~1 \\ The curve $C$ and \\ explicit divisors};
    \node [block, below of=Ch1] (Ch3) {Chapter~3 \\ Minimal \\ regular model};
    \node [block, left of=Ch3] (Ch2) {Chapter~2 \\ Descent \\ calculations};
    \node [block, right of=Ch3] (Ch8) {Chapter~8 \\ Monodromy \\ of $\ell$-torsion};
    \node [block, below of=Ch2] (Ch4) {Chapter~4 \\
      Heights and \\ visible subgroup};
    \node [block, below of=Ch3] (Ch5) {Chapter~5 \\ $L$-function
      and BSD };
   \node [block, below of=Ch8] (Ch6) {Chapter~6 \\ Torsion and N\'eron-Severi};
    \node [block, below of=Ch5] (Ch7) {Chapter~7 \\  Index and $\sha$};
    % Draw edges
    \path [line] (Ch1) -- (Ch2);
    \path [line] (Ch1) -- (Ch3);
       \path [line] (Ch1) -- (Ch8);
     \path [line] (Ch3) -- (Ch8);
    \path [line] (Ch3) -- (Ch4);
    \path [line] (Ch3) -- (Ch5);
\path [line] (Ch3) -- (Ch6);
   \path [line] (Ch4) -- (Ch7);
   \path [line] (Ch5) -- (Ch7);
    \path [line] (Ch6) -- (Ch7);
\end{tikzpicture}
\caption{Leitfaden}    
\end{center}
\end{figure}

\newpage
\section*{Notation}

Throughout, $k$ is a field of characteristic $p\ge0$ and $K$ is the
rational function field $k(t)$.  We write $\Fq$ to denote a finite
ground field of cardinality $q$, with $q$ being a power of $p$.  If
$n$ is positive and not divisible by the characteristic $p$ of $k$, we
write $\mu_n$ for the group of $n$-th roots of unity in an algebraic
closure of $k$.  For a prime $p$ and a positive integer $d$ not
divisible by $p$, we write $K_d$ for the extension $\Fp(\mu_d,u)$ of
$\Fpt$ with $u^d=t$.  We view $k(u)$ as the function field of $\P^1_u$
where the subscript $u$ reminds us that the coordinate is $u$.

%%%%%%%%%%%%%%%%%%%%%%%%%%%%%%%%%%%%%%%%%%%%%%%%%%%%%%%%%%%%%%%%

%%% Local Variables: 
%%% mode: latex
%%% TeX-master: "EHR"
%%% End: 

%This is chapter I

\chapter{The curve, explicit divisors, and relations}\label{ch:C}

In this chapter, we define a curve $C$ over $K=k(t)$ whose Jacobian
$J$ is the main object of study.  When $k=\Fp$, there is a rich supply
of explicit points on $C$ defined over certain extensions of $K$, and
the divisors supported on these points turn out to generate a subgroup
of $J$ of large rank.

More precisely, we study the arithmetic of $C$ and $J$ over extensions
of $\Fp(t)$ of the form $\Fq(t^{1/d})$ for $q$ a power of $p$ and $d
\in {\mathbb N}$ relatively prime to $p$.  Let $K=\Fp(t)$ and
$K_d=\Fp(\mu_d,u)$ where $\mu_d$ denotes the $d$-th roots of unity and
$u=t^{1/d}$.  These fields are the most important fields in the paper,
especially when $d$ has the form $d=p^\nu+1$ for an integer $\nu>0$,
although we consider more general extensions of the form
$\Fq(t^{1/d})$ as well.

\section{A generalization of the Legendre curve}\label{s:curve}
Choose a positive integer $r$ not divisible by $p=\cha(k)$.  We
consider the smooth, absolutely irreducible, projective curve $C$ over
$k(t)$ associated to the affine curve
\begin{equation}\label{equationOfC}
y^r = x^{r-1}(x+1)(x+t).
\end{equation}

Note that when $r=2$, this is an elliptic curve called the Legendre
curve which was studied in
\cite{Legendre}.

\subsection{Constructing a smooth model}\label{ss:C}
We explicitly construct the smooth projective model of $C$.  First,
consider the projective curve in $\p^2$ over $\Fp(t)$ given by
$$C':\qquad Y^rZ = X^{r-1}(X+Z)(X+tZ).$$

A straightforward calculation using the Jacobian criterion shows that
$C'$ is smooth when $r=2$, in which case we take $C=C'$.  If $r>2$,
then the Jacobian criterion reveals that $C'$ is singular at the point
$[0,0,1]$ and is smooth elsewhere.  We produce a smooth projective
curve by blowing up this point.

Let $V$ be the complement of $[0,0,1]$ in $C'$.  Let $U$ be the
affine curve with equation
$$v=u^{r-1}(uv+1)(uv+t).$$
Another Jacobian criterion calculation shows that $U$ is smooth.  The
map 
$$X=uv,\qquad Y=v,\qquad Z=1$$
gives an isomorphism $\pi$ between $U\setminus\{u=v=0\}$ and
$V\setminus\{[-1,0,1],[-t,0,1],[0,1,0]\}$.  Gluing $U$ and $V$ along
this map yields a smooth projective curve which we denote $C$.

We claim that $\pi:C\to C'$ is the normalization of $C'$.  Indeed,
$\pi$ factors through the normalization of $C'$ since $C$ is smooth
and thus normal.  Moreover, $\pi$ is visibly finite and birational.
Since a finite birational morphism to a normal scheme is an
isomorphism, $\pi:C\to C'$ is indeed the normalization of $C$.

Note that $\pi$ is a bijection as well.  In fact, it is a universal
homeomorphism\footnote{Indeed, $\pi$ is projective, so universally
  closed and surjective, and it is injective and induces isomorphisms
  on the residue fields, so is universally injective by 
\cite[3.5.8]{GrothendieckEGA1}.},
so for every field extension $L$ of $\Fp(t)$, there is a bijection of
rational points 
$$C(L)\ \isoto\ C'(L).$$  
It is convenient to specify
points of $C$ by giving the corresponding points of $C'$ using the
affine coordinates $x=X/Z$ and $y=Y/Z$.

The reader who prefers to avoid the abstraction in the last two
paragraphs is invited to work directly with the smooth curve $C$.
This adds no significant inconvenience to what follows.

\subsection{First points}
Let $Q_\infty$ be the point of $C$ corresponding to the point at
infinity on $C'$, namely $[0,1,0]$.  Let $Q_0$, $Q_1$, and $Q_t$ be
the points of $C$ given by $(x,y)=(0,0)$, $(-1,0)$, and $(-t,0)$
respectively.  (Here we use the convention mentioned at the end of the
preceding subsection, namely we define points of $C$ via $C'$.)

\subsection{Genus calculation}
\begin{lemma}\label{cjh:l40}
The curve $C$ has genus $g=r-1$.
\end{lemma}

\begin{proof}
Consider the covering
$$f : C \rightarrow \p^1$$
induced by the function $x$. The ramification points of $f$ are
$Q_0,Q_1,Q_t$ and $Q_{\infty}$, each with ramification index $r$.
The Riemman-Hurwitz formula implies
$$2g - 2 = -2r + 4(r-1),$$
thus $g=r-1$.
\end{proof}

\subsection{Immersion in $J$}\label{ss:immersion}
Let $J$ be the Jacobian of $C$; it is a principally polarized abelian
variety of dimension $g=r-1$.  We imbed $C$ in $J$
via the Abel-Jacobi map using $Q_\infty$ as a base point:
\begin{align*}
C&\to J\\
P&\mapsto [P-Q_\infty]
\end{align*}
where $[P-Q_\infty]$ is the class of $P-Q_\infty$ in $\Pic^0(C)=J$.

\subsection{Automorphisms}\label{ss:autos}
Note that if $k$ contains $\mu_r$, the $r$-th roots of unity, then
every element of $\mu_r$ gives an automorphism of $C$.  More
precisely, we have automorphisms
$$(x,y) \mapsto (x,\zeta_r^{j} y)$$
where $\zeta_r$ is a primitive $r$-th root of unity and $0\le j<r$.
These automorphisms fix $Q_\infty$, so the induced automorphisms of
$J$ are compatible with the embedding $C\into J$.

As we verify below, these are not all of the automorphisms of
$C$, but they are the only ones that play an important role in
this paper.

\subsection{Complement:  Hyperelliptic model and $2$-torsion}
We remark that the curve $C$ is hyperelliptic.  More precisely, making
the substitution $(x,y)\to(x,xy)$ in the equation
$y^r=x^{r-1}(x+1)(x+t)$, we see that $C$ is birational to the curve
given by
$$x^2+(t+1-y^r)x+t=0$$
and projection on the $y$-coordinate makes this a (separable) 2-to-1
cover of the projective line.  If $p\neq2$, we may complete the square
and make the appropriate change of coordinates (a translation of
$x$) to arrive at the equation
\begin{align*}
x^2&=y^{2r}-2(t+1)y^r+(t-1)^2\\
&=\left(y^r-(\sqrt{t}+1)^2\right)\left(y^r-(\sqrt{t}-1)^2\right). 
\end{align*}

If, in addition, $d$ is even and $r$ is odd, then $\sqrt{t}\in k(u)$ and
the two factors on the right hand side are irreducible in $k(u)[y]$.
It follows from \cite[Lemma~12.9]{ps} that $J$ has no 2-torsion over
$k(u)$.

This is a first hint towards later results.  For example,
$J(K_d)_{tor}$ has order $r^3$ when $r$ divides $d=p^\nu+1$
(Theorem~\ref{thm:V-vs-MW}). More generally (Corollary~\ref{cor:ptors}
and Theorem~\ref{cjh:p1}), $J$ has no torsion of order prime to $r$
over any abelian extension of $\kbar(t)$.

\section{Explicit points and the visible subgroup}\label{s:points}
Next, we write down several points on $C$ defined over the extensions
$K_d$, and we consider the subgroup they generate in the Jacobian.

\subsection{Special extensions}
For the rest of Chapter \ref{ch:C}, we assume that $d=p^\nu+1$ for
some integer $\nu>0$, and we assume that $r$ divides $d$.  In this
situation, it turns out that $C$ has a plentiful supply of points
defined over $K_d$, and the divisors supported on these points
generate a subgroup of $J(K_d)$ of large rank.

The extension $K_d/K$ is Galois with Galois group the semidirect
product of $\Gal(\Fp(\mu_d,t)/K)\cong\Gal(\Fp(\mu_d)/\Fp)$ (a cyclic
group of order $2\nu$ generated by the $p$-power Frobenius) by
$\Gal(K_d/\Fp(\mu_d,t))$ (a cyclic group of order $d$ generated by a
primitive $d$-th root of unity).

\begin{remark}
  There are many triples $p,r,d$ satisfying our hypotheses.  Indeed,
  for a fixed prime $p$, there are infinitely many integers $r>1$ such
  that $r$ divides $p^\mu+1$ for some $\mu$.  (The number of such $r$
  less that $X$ is asymptotic to $X/(\log X)^{2/3}$; see
  \cite[Theorem~4.2]{PomeranceUlmer13}.)  For any such $p$ and $r$,
  there are infinitely many $\nu$ such that $r$ divides $p^\nu+1$.
  Indeed, $p^\mu+1$ divides $p^\nu+1$ whenever $\nu=m\mu$ with $m$ odd.

Alternatively, for a fixed $r$, there are infinitely many primes $p$
such that $r$ divides $p^\mu+1$ for some $\mu$.  These $p$ are
determined by congruence conditions modulo $r$, namely by the
requirement that $-1$ be in the subgroup of $(\Z/r\Z)^\times$
generated by $p$.
\end{remark}

\subsection{Explicit points}
We continue to assume that $d=p^\nu+1$ and $r|d$.  Under these
hypotheses, we note that
$$P(u):=\left(u,u(u+1)^{d/r}\right)$$
is a point on $C$ defined over $K_d$.  Indeed,
\begin{align*}
u^{r-1}(u+1)(u+t)&=u^r(u+1)(1+u^{p^\nu})\\
&=u^r(u+1)^{p^\nu+1}\\
&=\left(u(u+1)^{d/r}\right)^r.
\end{align*}
We find other points by applying the automorphisms $\zeta_r^j$
discussed in Section~\ref{ss:autos} above and the action of the
elements of the Galois group of $K_d/K$.  In all, this yields $rd$
distinct points.

Although it is arguably unnatural, for typographical convenience we
fix a primitive $d$-th root of unity $\zeta_d\in K_d$ and we set
$\zeta_r=\zeta_d^{d/r}$.  Then the points just constructed can be
enumerated as
$$P_{i,j}=\left( \zeta_d^i u, \zeta_r^j\zeta_d^iu ( \zeta_d^i u+1 )^{d/r}  \right)$$
where $i\in\Z/d\Z$ and $j\in\Z/r\Z$.

Identifying $C$ with its image in $J$ via the map in
Section~\ref{ss:immersion} produces divisor classes in $J(K_d)$ that
we also denote by $P_{i,j}$. The subgroup generated by these points is
one of the main objects of study in this paper.

\subsection{$R$-module structure}\label{ss:R}
Next we introduce a certain group ring acting on $J(K_d)$.  We noted
above that there is an action of $\mu_r\subset\Aut(C)$ on $C$ and on
$J$.  There are also actions of
$\mu_d\cong\Gal(K_d/\Fp(\mu_d,t))\subset\Gal(K_d/K)$ on $C(K_d)$ and
on $J(K_d)$, and these actions are compatible with the inclusion
$C\into J$.

Let $R$ be the integral group ring of $\mu_d\times\mu_r$, i.e., let
$$R=\frac{\Z[\sigma,\tau]}{(\sigma^d-1,\tau^r-1)}.$$

The natural action of $R$ on the points $P_{i,j}$ is:
$$\sigma^i\tau^j(P_{a,b})=P_{a+i,b+j}.$$
(Here and below we read the indices $i$ modulo $d$ and $j$ modulo
$r$.)

\subsection{The ``visible'' subgroup} \label{SdefV}
We define $V=V_{r,d}$ to be the subgroup of $J(K_d)$ generated by the
$P_{i,j}$.  It is evident that $V$ is also the cyclic $R$-submodule of
$J(K_d)$ generated by $P_{0,0}$. In other words, there is a surjective
homomorphism of $R$-modules
\begin{align*}
R&\to V\\
\sum_{ij}a_{ij}\sigma^i\tau^j&\mapsto
\sum_{ij}a_{ij}\sigma^i\tau^j(P_{0,0})
=\sum_{ij}a_{ij} P_{i,j}.
\end{align*}

One of the main results of the paper is a complete determination
of the ``visible'' subgroup $V$.  Here we use visible in the
straightforward sense that these are divisors we can easily see.  As
far as we know, there is no connection with the Mazur-Stein theory of
visible elements in the Tate-Shafarevich group.

\section{Relations}\label{s:relations}
As above, let $V=V_{r,d}$ be the $R$-submodule of $J(K_d)$ generated
by $P_{0,0}$.  The goal of this section is to work toward computing
the structure of $V$ as a group and as an $R$-module.  Explicitly,
we show that $V$ is a quotient of $R/I$ for a certain ideal $I$.
Ultimately, in Chapter~\ref{ch:heights}, we verify that $V$ is
isomorphic to $R/I$ as an $R$-module and compute the structure of
$R/I$ as a group.

Throughout, we identify $C(K_d)$ with its image in $J(K_d)$ via the
immersion $P\mapsto [P-Q_\infty]$.

%\subsection{Relations}
Considering the divisors of $x$, $x+1$, and $x+t$, one finds that the
classes of $Q_0$, $Q_1$, and $Q_t$ are $r$-torsion.  Considering the
divisor of $y$, one finds that $Q_t\sim Q_0-Q_1$, so $Q_t$ is
in the subgroup generated by $Q_0$ and $Q_1$.

Now consider the functions $x-\zeta_d^iu$, 
$$\Delta_j:=\zeta_d^{-jd/r}y-x(x+1)^{d/r},$$
and 
$$\Gamma_j:=\zeta_d^{-jd/r}yx^{d/r-1}-u^{d/r}(x+1)^{d/r}.$$
Calculating as in \cite[Proposition~3.2]{Legendre}, we find that
$$\dvsr(x-\zeta_d^iu)=\sum_{j=0}^{r-1} P_{i,j}-rQ_\infty,$$
$$\dvsr(\Delta_j)=\sum_{i=0}^{d-1}P_{i,j}+(r-1)Q_0+Q_1-(r+d)Q_\infty,$$
and
$$\dvsr(\Gamma_j)=\sum_{i=0}^{d-1}P_{i,-i+j}+Q_1-(d+1)Q_\infty.$$

Considering the divisor of $\Gamma_j$ for any $j$ shows that $Q_1$ is
in $V$, and then considering the divisor of $\Delta_j$ for any $j$
shows that $Q_0$ is also in $V$.  (Here we use the fact that $Q_0$ is
$r$-torsion.)  Thus $V$ contains the classes of $Q_0$, $Q_1$ and
$Q_t$.

Now for $1\le j\le r-1$ we set
$$D_j:=\dvsr(\Delta_j/\Delta_{j-1})=\sum_i (P_{i,j}-P_{i,j-1}),$$
and
$$E_j:=\dvsr(\Gamma_j/\Gamma_{j-1})=\sum_i (P_{i,j-i}-P_{i,j-1-i}),$$
and for $0\le i\le d-1$ we set
$$F_i:=\dvsr(x-\zeta_d^iu)=\sum_j P_{i,j}-rQ_\infty.$$
These divisors are zero in the Jacobian $J_r(K_d)$. 

Restating this in terms of the module homomorphism $R\to V$, we see that 
for $1\le j\le r-1$ the elements
$$d_j:=\sum_i (\sigma^i\tau^j-\sigma^i\tau^{j-1})
=(\tau^j-\tau^{j-1})\sum_i \sigma^i,$$
$$e_j:=\sum_i (\sigma^i\tau^{j+d-i}-\sigma^i\tau^{j-1+d-i})=
( \tau^{j}-\tau^{j-1})\sum_i\sigma^i\tau^{d-i},$$
and for $0\le i\le d-1$ the elements
$$f_i:=\sum_j \sigma^i\tau^j=\sigma^i\sum_j\tau^j$$
map to zero in $V$.

Let $I$ be the ideal of $R$ generated by
$$(\tau-1)\sum_i \sigma^i,\qquad
(\tau-1)\sum_i\sigma^i\tau^{d-i}, \quad\text{and}\quad \sum_j\tau^j.$$
Then it is easy to see that $d_j$, $e_j$, and $f_i$ all lie in $I$,
that they generate it as an ideal, and in fact that they form a basis
of $I$ as a $\Z$-module.

Thus there is a surjection of $R$-modules $R/I\to V$.  We will
eventually show that this surjection is in fact an isomorphism; see
Theorem \ref{Firstmaintheorem}.

Note that $R$ has rank $rd$ as a $\Z$-module, so the rank of $R/I$ as
a $\Z$-module is $rd-d-2(r-1)=(r-1)(d-2)$.  Thus the rank of $V$ is at
most $(r-1)(d-2)$.

\section{Torsion}\label{s:torsion}
In this section, we show that certain torsion divisors are not zero;
more precisely that the order of the torsion subgroup of $V$ is
divisible by $r^3$.  The main result is Proposition~\ref{prop:torsion}
below.

\begin{lemma}\label{lemma:Q0-Q1}
  The classes of $Q_0$ and $Q_1$ each have order $r$ and generate a
  subgroup of $V$ isomorphic to $\Z/r\Z \times
  \Z/r\Z$.
\end{lemma}

\begin{proof} 
  It suffices to prove the claim over $\Fpbar K_d=\Fpbar(u)$.  We have
  already seen that $Q_0$ and $Q_1$ have order dividing $r$.  Suppose
  that $aQ_0 + bQ_1 = 0$ in $V$ for integers $a$,
  $b \in \{0, 1, \cdots r-1\}$, not both equal to zero.  Then there is
  a function $h$ in the function field of the curve $C$ with
  $\divi(h) = (a/r) \divi(x) + (b/r) \divi(x+1)$.  Since we are
  working over $\Fpbar$, we may choose $h$ such that
  $h^r = x^a(x+1)^b$.  Let $Y$ denote the curve with function field
  $K_d(x,h)$.  Consider the inclusions
  $K_d(x) \hookrightarrow K_d(x,h) \hookrightarrow K_d(C)$ and the
  corresponding surjections $C \rightarrow Y \rightarrow \p^1$. The
  map $Y \rightarrow \p^1$ is of degree greater than one, and
  $C \rightarrow \p^1$ is fully ramified over $x=-t$. This is a
  contradiction, since $Y \rightarrow \p^1$ is unramified over
  $x=-t$. Hence, $aQ_0 + bQ_1 = 0$ only when $r$ divides both $a$
  and $b$, and $Q_0$ and $Q_1$ generate independent cyclic
  subgroups of order $r$.
\end{proof}

Next we introduce elements of $V\subset J(K_d)$ as follows:
$$Q_2:=\sum_{j=0}^{r-1}\sum_{k=0}^{r-1-j}
\sum_{i\equiv k\bmod r}P_{i,j}$$
and
$$Q_3:=Q_0-2Q_2.$$

\begin{lemma}\label{lemma:more-relations}
\mbox{}  \begin{enumerate}
  \item $(1-\zeta_r)Q_2=Q_0$.
  \item If $r$ is odd, then $rQ_2=0$.
  \item If $r$ is even, then $2rQ_2=0$ and $(r/2)Q_3=0$.
  \end{enumerate}
\end{lemma}

\begin{proof}
(1) We have
\begin{align*}
  (1-\zeta_r)Q_2&=(1-\zeta_r)\sum_{j=0}^{r-1}\sum_{k=0}^{r-1-j}
\sum_{i\equiv k\bmod r}P_{i,j}\\
&=\sum_{j=0}^{r-1}\sum_{k=0}^{r-1-j}
 \sum_{i\equiv k\bmod r}\left(P_{i,j}-P_{i,j+1}\right)\\
&=\sum_{k=0}^{r-1}\sum_{i\equiv k\bmod r}
 \sum_{j=0}^{r-1-k}\left(P_{i,j}-P_{i,j+1}\right)\\
&=\sum_{k=0}^{r-1}\sum_{i\equiv k\bmod r}\left(P_{i,0}-P_{i,r-k}\right)\\
&=\sum_{i=0}^{d-1}\left(P_{i,0}-P_{i,-i}\right).
\end{align*}
Considering the divisor of $\Delta_0/\Gamma_0$ shows that the last
quantity is equal to $Q_0$ in $J$.

(2) Assume that $r$ is odd, which implies that $j(j-r)/2$ is an
integer for all integers $j$.  Consider the element of $R$ given by
$$\rho_{odd}:=\sum_{j=0}^{r-1}\left(\frac{j(j-r)}{2}\left(d_j-e_j\right)
+(r-j)\sum_{i\equiv  j\bmod r}f_i\right),$$
where $d_j$, $e_j$, and $f_i$ are as in the previous
subsection.  We compute that
$$\rho_{odd}=r\left(\sum_{j=0}^{r-1}\sum_{k=0}^{r-1-j}
\sum_{i\equiv k\bmod r}\sigma^i\tau^j\right).$$
Applying both sides of this equality to $P_{0,0}$ proves that $rQ_2=0$
in $J$. 

(3) Now assume that $r$ is even and consider
$$\rho_{even}:=\sum_{j=0}^{r-1}\left({j(j-r)}\left(d_j-e_j\right)
+2(r-j)\sum_{i\equiv  j\bmod r}f_i\right).$$
A calculation similar to the one above shows that
$$\rho_{even}=2r\left(\sum_{j=0}^{r-1}\sum_{k=0}^{r-1-j}
\sum_{i\equiv k\bmod r}\sigma^i\tau^j\right),$$
and applying both sides of this equality to $P_{0,0}$ proves that
$2rQ_2=0$ in $J$. 

Finally, we note that when $r$ is even, then $(1-j)(j-r)/2$ is an
integer for all integers $j$.  Consider
$$\rho'_{even}:=\sum_{j=1}^{r-1}\left(\frac{(1-j)(j-r)}{2}\left(d_j-e_j\right)\right)
-\sum_{j=0}^{r-1}\sum_{i\equiv  j\bmod r}(r-j)f_i.$$
We compute that
$$\rho'_{even}=(r/2)\left(\sum_{i=0}^{d-1}\left(\sigma^i-\sigma^i \tau^{-i}\right)
-2\sum_{j=0}^{r-1}\sum_{k=0}^{r-1-j}
\sum_{i\equiv k\bmod r}\sigma^i\tau^j\right).$$
Applying both sides of this equality to $P_{0,0}$ and noting as above
that $Q_0=\sum_iP_{i,0}-P_{i,-i}$ shows that $(r/2)(Q_0-2Q_2)=0$ in
$J$. 

This completes the proof of the lemma.
\end{proof}

The reader who wonders where $Q_2$ and $Q_3$ come from should consult
the proof of Proposition~\ref{prop:R/I-as-group}.

We write $\langle Q_0,Q_1,Q_2\rangle$ for the subgroup of $J(K_d)$
generated by $Q_0$, $Q_1$, and $Q_2$. Note that $\langle
Q_1,Q_2,Q_3\rangle=\langle Q_0,Q_1,Q_2\rangle$. 

\begin{prop}\label{prop:torsion}
Let $T$ be the subgroup $\langle Q_0,Q_1,Q_2\rangle$ of $J(K_d)$.
Then the order of $T$ is $r^3$.  More precisely:
\begin{enumerate}
\item If $r$ is odd, then the map $(a,b,c)\mapsto aQ_0+bQ_1+cQ_2$
  induces an isomorphism $(\Z/r\Z)^3\cong T$.
\item If $r$ is even, then the map $(a,b,c)\mapsto aQ_1+bQ_2+cQ_3$
  induces an isomorphism
  $(\Z/r\Z)\times(\Z/2r\Z)\times(\Z/(r/2)\Z)\cong T$.
\end{enumerate}
\end{prop}

\begin{proof}
  (1) Assume that $r$ is odd.  Lemmas~\ref{lemma:Q0-Q1} and
  \ref{lemma:more-relations}(2) show that the map under consideration
  is well-defined.  It is surjective by the definition of $T$.  To see
  that it is injective, suppose that $aQ_0+bQ_1+cQ_2=0$.  Applying
  $(1-\zeta_r)$ and using Lemma~\ref{lemma:more-relations}(1) shows
  that $cQ_0=0$.  By Lemma~\ref{lemma:Q0-Q1}, $c=0$ in $\Z/r\Z$, and
  applying Lemma~\ref{lemma:Q0-Q1} again shows that $a=b=0$ in
  $\Z/r\Z$.  This shows the map is injective, thus an isomorphism.

  (2) Now assume that $r$ is even.  Lemmas~\ref{lemma:Q0-Q1} and
  \ref{lemma:more-relations}(3) show that the map under consideration
  is well-defined.  It is again surjective by the definition of $T$.
  To see that it is injective, suppose that $aQ_1+bQ_2+cQ_3=0$.
  Applying $(1-\zeta_r)$ and using Lemma~\ref{lemma:more-relations}(1)
  and Lemma~\ref{lemma:Q0-Q1} shows that $b-2c=0$ in $\Z/r\Z$ and, in
  particular, that $b$ is even.  Using that $2Q_2=Q_0-Q_3$, we compute
  \begin{align*}
    0&=aQ_1+bQ_2+cQ_3\\
&=cQ_0+aQ_1+(b-2c)Q_2\\
&=cQ_0+aQ_1+\frac{b-2c}2(Q_0-Q_3)\\
&=(b/2)Q_0+aQ_1.
  \end{align*}
  By Lemma~\ref{lemma:Q0-Q1}, $b/2=a=0$ in $\Z/r\Z$ and therefore
  $b=0$ in $\Z/2r\Z$.  Since $b-2c=0$ in $\Z/r\Z$, we see that $c=0$
  in $\Z/(r/2)\Z$, and this shows the map is injective, thus an
  isomorphism.

This completes the proof of the Proposition.
\end{proof}

\section{First main theorem}
We can now state the main ``explicit points'' theorem of this paper.

Recall that the group ring $R=\Z[\sigma,\tau]/(\sigma^d-1,\tau^r-1)$
acts on $J(K_d)$ and that $V$ is the cyclic submodule of $J(K_d)$
generated by $P_{0,0}$.  Recall also that $I\subset R$ is the ideal
generated by 
$$(\tau-1)\sum_i \sigma^i,\qquad
(\tau-1)\sum_i\sigma^i\tau^{d-i}, \quad\text{and}\quad \sum_j\tau^j.$$

\break

\begin{theorem}\label{Firstmaintheorem}
\mbox{}
\begin{enumerate}
\item The map
\begin{align*}
R&\to V\\
\sum_{ij}a_{ij}\sigma^i\tau^j&\mapsto\sum_{ij}a_{ij} P_{i,j}
\end{align*}
induces an isomorphism $R/I\cong V$ of $R$-modules.
\item As a $\Z$-module, $V$ has rank $(r-2)(d-2)$, and its torsion
  subgroup has order $r^3$ and is equal to the group $T$ defined in
  Proposition~\ref{prop:torsion}.
\end{enumerate}
\end{theorem}

We prove Theorem \ref{Firstmaintheorem} in Chapter~\ref{ch:heights} by
computing the canonical height pairing on $V$; see
Theorem~\ref{thm:visible}.  In the case when $r$ is an odd prime, we
give a more elementary proof of part (2) in Chapter~\ref{ch:descent}
using a descent calculation; see Theorem~\ref{subgoal}.

\section{Complement:  Other curves}\label{s:other-curve}
The basic ``trick'' allowing one to write down points on $C$ extends
to many other curves.  In this section, we briefly discuss one class
of examples. A more detailed analysis is provided in
Appendix~\ref{Sappendix}.

Let $p$ be an odd prime and $k$ a field of characteristic $p$ and
cardinality $q$.  Fix an odd integer $g>1$ and a polynomial
$h(x)\in k[x]$ of degree $g$.  Assume $h$ has distinct, nonzero roots.

Let $X$ be the smooth, projective curve over $K=k(t)$ defined by
$$y^2=x h(x) x^g h(t/x).$$
Since the right hand side has degree $2g+1$ in $x$, the genus of $X$
is $g$.  Let $\infty$ be the ($K$-rational) point at infinity on $X$.

Let $J$ be the Jacobian of $X$.  We embed $X$ in $J$ using $\infty$ as the
base point.

If $d=q^\nu+1$ and $K_d=k(\mu_d,u)$ with $u^d=t$,  then  $X$ has a
$K_d$-rational point, namely
$$P(u):  (x,y)=\left(u,u^{(g+1)/2}h(u)^{d/2}\right).$$
Letting the Galois group of $K_d$ over $K$ act on $P(u)$ yields points
$P_j=P(\zeta_d^j u)$ where $\zeta_d$ is a primitive $d$-th root of unity and
$j=0,\dots,d-1$. We consider the subgroup $V$ of $J(K_d)$ generated by
the images of the $d$ points $P_j$, and the images of the points where
$y = 0$.

By writing down the divisors of certain functions, as in
Section~\ref{s:relations}, we show that the rank of the subgroup $V$
is at most $d$.

It is natural to bound the rank of $V$ from below by computing a
coboundary map related to 2-descent.  More precisely, extending $k$ if
necessary we may assume that the roots of
$h$ lie in $k$.  Then the Weierstrass points of $X$ 
are defined over $K$ and the divisors of degree zero supported on them
generate the full 2-torsion subgroup of $J$.  In particular,
$J[2]\cong\mu_2^{2g}\cong(\Z/2\Z)^{2g}$ over $K$.  We obtain a
coboundary map
$$J(K_d)/2J(K_d)\into
H^1(K_d,J[2])\cong\left(K_d^\times/K_d^{\times2}\right)^{2g}.$$
Analyzing the image of $V$ under this map (along the lines of
\cite[Section~4]{Legendre}) allows one to show that the rank of $V$ is
at least $d-2$ when $d$ is of the form $q^\nu+1$ (and at least $d-1$
when $g>1$).  This work is also closely related to the calculations in
Chapter~2 for the curve $C$ that is our main object of study.

In the appendix, we also consider a certain surface $\XX_d$ equipped
with a morphism $\XX_d\to\P^1$ whose generic fiber is $X/K_d$, and we
prove that this surface is dominated by a product of curves.  This
shows that the BSD conjecture holds for $J$ over $\Fq(u)$ where
$u^d=t$, $q$ is any power of $p$, and $d$ is any positive integer
prime to $p$.  All this is closely related to our work in Chapters~3
and 5 on $C$.

In the last part of the appendix, we obtain an upper bound on the
order of vanishing of the $L$-function of $J/K_d$ at $s=1$, thereby
bounding the rank of $J(K_d)$ from above.  This is closely related to
our work in Chapter~5 on the $L$-function of $J_C$.

We note that some of the finer analysis of this paper is unlikely to
go through without much additional work.  For example, the upper and
lower bounds on the rank of $J$ over $K_d$ differ significantly, and
we have not determined the exact rank.  Indeed, the degree of the
$L$-function of $J$ over $K_d$ is asymptotic to $g^2d$ as
$d\to\infty$, whereas the rank of $V$ is less than $d$.  This suggests
that the leading coefficient of $L(J/K_d,s)$ at $s=1$ is likely to be
of arithmetic nature, and that the connection between the index of $V$
in $J(K_d)$ and the order of $\sha(J/K_d)$ may not be as simple as it
is for the curve $C$ studied in the rest of this paper.  We would be
delighted if readers of this paper took up these questions.

%%% Local Variables: 
%%% mode: latex
%%% TeX-master: "EHR"
%%% End: 

%Chapter 2
\chapter{Descent calculations}\label{ch:descent}

Throughout this chapter, $r$ is an odd, positive, prime number
dividing $d$, and $d=p^\nu+1$ for some integer $\nu >0$.  Let
$K_d=\Fp(\mu_d,u)$ where $u^d=t$.  In this context, there is a fairly
short and elementary proof that the visible subgroup of $J(K_d)$ has
large rank.

More precisely, let $C$ be the curve studied in Chapter~\ref{ch:C},
let $J$ be its Jacobian, and let $V$ be the ``visible'' subgroup of
$J(K_d)$ defined in Section~\ref{s:points}, so that $V$ is generated
by the image of the point $P=(u,u(u+1)^{d/r})$ under the Abel-Jacobi
mapping $C \into J$ and its Galois conjugates.  Recall that the
choices made in Chapter~\ref{ch:C} allow us to index these points as
$P_{i,j}$ with $i\in\Z/d\Z$ and $j\in\Z/r\Z$.

Using the theory of descent, as developed in \cite{bps}, we prove the
following theorem.

\begin{theorem}\label{subgoal}
The subgroup $V$ of $J(K_d)$ has rank $(r-1)(d-2)$. Moreover,  
$$J(K_d)[r^{\infty}] \cong V[r^{\infty}] \cong (\Z/r\Z)^3.$$
\end{theorem}

The proof is given in Section \ref{Chap2PfSect}.  

\section{The isogeny $\phi$}
Recall that there is an action of the $r$-th roots of unity $\mu_r$ on
$C$ and an induced action on $J$.  Recall also that $\zeta_r \ =
\zeta_d^{d/r}\in K_d$ is a fixed $r$-th root of unity.  If $D$ is a
divisor of degree 0 on $C/K_d$ then the divisor
$$(1+\zeta_r+\cdots+\zeta_r^{r-1})^*(D)$$
is easily seen to be the pullback of a divisor of degree 0 on $\P^1$
under the map $C\to\P^1$ that is the projection on the $x$
coordinate.  Since the Jacobian of $\P^1$ is trivial, the endomorphism
$(1+\zeta_r+\cdots+\zeta_r^{r-1})$ acts trivially on $J$.

We want to restate this in terms of the endomorphism ring of $J$.  To
avoid notational confusion, write $H$ for the cyclic group of order
$r$ and let $\Z[H]$ be the group ring of $H$.  Somewhat abusively, we
use $\zeta_r$ also to denote an $r$-th root of unity in characteristic
zero.  Then, as usual, $\Z[\zeta_r]$ will denote the ring of integers
in the cyclotomic field $\Q(\zeta_r)$.  The action of $\mu_r$ on $J$
induces a homomorphism $\Z[H]\to\End(J)$ where $\End(J)$ denotes the
endomorphism ring of $J$ over $K_d$.

There is a surjective ring homomorphism $\Z[H]\to\Z[\zeta_r]$ sending
the elements of $H$ to the powers of $\zeta_r$.  The kernel is
generated by $\sum_{h\in H}h$.  The discussion above shows that the
homomorphism $\Z[H]\to\End(J)$ factors through $\Z[\zeta_r]$.  The
induced map $\Z[\zeta_r]\to\End(J)$ is an embedding, since $\End(J)$
is torsion-free.

Let $\phi : J\to J$ be the endomorphism $1-\zeta_r$. 

\begin{proposition}\label{propertiesOfPhi}
  The endomorphism $\phi=1-\zeta_r$ is a separable isogeny of degree
  $r^2$. 
\end{proposition}

\begin{proof}
  In $\Z[\zeta_r]$, there is an equality of ideals
  $(1-\zeta_r)^{r-1}=(r)$, i.e., the ratio of $(1-\zeta_r)^{r-1}$ and
  $r$ is a unit.  It follows that $(1-\zeta_r)^{r-1}$ and the
  separable isogeny $r:J\rightarrow J$ factor through each other.
  Therefore $1-\zeta_r$ is an isogeny, and
$$\deg(1-\zeta_r)^{r-1} = \deg {r} =  r^{2g} = r^{2(r-1)}.$$
Since $\phi=1-\zeta_r$, this proves that $\deg(\phi)=r^2$.
Since $r$ is prime to $p$, it follows that $\phi$ is separable.
\end{proof}

We write $J(K_d)[\phi]$ and $V[\phi]$ for the kernel of $\phi$ on
$J(K_d)$ and $V$ respectively.

\begin{cor}\label{cor:phi-torsion}
  $J(K_d)[\phi]$ is a two-dimensional vector space over $\F_r$ with
  basis $Q_0$ and $Q_1$.  Moreover, $V[\phi]=J(K_d)[\phi]$.
\end{cor}

\begin{proof}
For the first assertion, we verify that the divisor classes $Q_0$ and
$Q_1$ are contained in the kernel of $\phi$, and they generate a
subgroup of $J(K_d)$ of order $r^2$ by Lemma~\ref{lemma:Q0-Q1}. Since
$\phi$ has degree $r^2$, it follows that $Q_0$ and $Q_1$ generate the
kernel.  Since $Q_0$ and $Q_1$ lie in $V[\phi]\subset J(K_d)[\phi]$ we
also have $V[\phi]=J(K_d)[\phi]$.
\end{proof}

For any element $\epsilon \in \operatorname{End}(J)$, let
$\epsilon^\vee$ denote its Rosati dual, that is, its image under the
Rosati involution on $\operatorname{End}(J)$. If $\epsilon$ is an
automorphism of $J$ coming from an automorphism of $C$, then one has
$\epsilon^\vee = \epsilon^{-1}$. It follows that
$\phi^\vee=1-\zeta_r^{-1}$.

\begin{lemma}
\label{rosati}
We have $J[\phi] = J[\phi^{\vee}]$, as group subschemes of $J$.
\end{lemma}

\begin{proof}
  Since $(1-\zeta_r)/(1-\zeta_r^{-1})$ is a unit in $\Z[\zeta_r]$, it
  is clear that the endomorphisms $\phi = 1-\zeta_r$ and $\phi^\vee
  = 1-\zeta_r^{-1}$ factor through each other and thus have the same
  kernel.
\end{proof}

\section{The homomorphism $\XminusT$}
Let $\Delta=\{Q_0,Q_1,Q_t\}$, the set of affine ramification points of
the morphism $C\to\P^1$, $(x,y)\mapsto x$, which lie over $x=0$,
$x=-1$, and $x=-t$ respectively.  We write $\Div(C_{K_d})$ for the
$K_d$-rational divisors on $C$ and $\Div^0(C_{K_d})$ for those
of degree 0.  There is a canonical
surjective homomorphism $\Div^0(C_{K_d})\to J(K_d)$.

Following ideas from \cite{bps}, we define a homomorphism
$$\XminusT : \Div(C_{K_d}) \rightarrow 
\prod_{Q \in \Delta} K_d^{\times}/K_d^{\times r}$$
that plays a crucial role in the proof of Theorem~\ref{subgoal}.
Its properties are described in Proposition~\ref{propertiesOfCC}.
For an element $v \in \prod_{Q \in \Delta} K_d^{\times}/K_d^{\times r}$,
we write $v = (v_0,v_1,v_t)$, where $v_i$ is the coordinate
corresponding to $Q_i$.

Let $C^\circ \subset C$ be the complement of $\Delta \cup \{ Q_\infty
\}$. We define the homomorphism
$$\xminusT : \Div(C_{K_d}^\circ) \rightarrow 
\prod_{Q \in \Delta} K_d^{\times}/K_d^{\times r}$$
by setting
$$P \mapsto \left( x(P) - x(Q) \right)_{Q \in \Delta},$$
and defining $(x-T)'$ on divisors by multiplicativity.  (The
individual points $P$ in a divisor $D$ need not be $K_d$-rational, but
if $D$ is $K_d$-rational, then $\xminusT$ takes values in
$\prod K_d^{\times}/K_d^{\times r}$.)

We now define the homomorphism 
$$\XminusT : \Div(C_{K_d}) \rightarrow 
\prod_{Q \in \Delta} K_d^{\times}/K_d^{\times r}$$
as follows: let $D \in \Div(C_{K_d})$ be a divisor on $C_{K_d}$ and
choose $D' \in \Div(C_{K_d}^\circ) \subset \Div(C_{K_d})$ such that
$D'$ is linearly equivalent to $D$. Then set
$$\XminusT(D) := \xminusT(D').$$
For a proof that $\XminusT$ is
well-defined, see \cite[6.2.2]{bps}.

Fix a separable closure $K_d^{\sep}$ of $K_d$, and let $\GG$ be
$\Gal(K_d^{\sep}/K_d)$. For any $\GG$-module $M$ and integer $i \geq
0$, we abbreviate the usual notation $H^i(\GG,M)$ for the $i$-th
Galois cohomology group of $M$ to $H^i(M)$. For a finite $\GG$-module
$M$ of cardinality not divisible by $p$, we denote by $M^{\vee}$ the
dual $\GG$-module $\Hom(M,K_d^{\sep \times})$.

\begin{proposition}
\label{propertiesOfCC}
There is a homomorphism $\alpha : H^1(J[\phi]) \rightarrow \prod_{Q
  \in \Delta} K_d^{\times}/K_d^{\times r}$ such that:
\begin{enumerate}
\item there is a short exact sequence of $\GG$-modules
\[0 \rightarrow H^1(J[\phi]) \stackrel{\alpha}{\rightarrow}  
\prod_{Q \in \Delta} K_d^{\times}/K_d^{\times r} \stackrel{N}{\rightarrow} 
K_d^{\times}/K_d^{\times r} \rightarrow 0,\]
where $N$ is the map sending $(v_0,v_1,v_t)$ to $v_1v_t/v_0$; and
\item the homomorphism $\XminusT$ restricted to $\Div^0(C_{K_d})$ is
  the composition
\[\Div^0(C_{K_d})
\onto J(K_d)/\phi J(K_d) \stackrel{\partial}{\hookrightarrow}  
H^1(J[\phi]) \stackrel{\alpha}{\rightarrow}  
\prod_{Q \in \Delta} K_d^{\times}/K_d^{\times r},\]
where $\partial$ is induced by the Galois cohomology coboundary map
for $\phi$.
\end{enumerate}
\end{proposition}

\begin{proof}
  The proof is an application of the general theory of descent as
  developed in \cite{bps}.

  Let $E$ be $(\Z/r \Z )^{ \Delta }$, the $\GG$-module of
  $\Z/r\Z$-valued functions on $\Delta$. Note that the $\GG$-action on
  $E$ is trivial. We define a $\GG$-module map $\alpha^{\vee} : E
  \rightarrow J[\phi]$ defined by $h \mapsto \sum_{Q \in \Delta} h(Q)
  \cdot [Q]$. Note that this is well-defined since $J[\phi]$ is
  annihilated by $r$. Proposition~\ref{propertiesOfPhi} shows that
  $\alpha^\vee$ is surjective. Its kernel $R_0$ is the
  $\Z/r\Z$-submodule of $E$ generated by the map $\rho$ that sends
  $Q_0 \mapsto -1, Q_1 \mapsto 1, Q_t \mapsto 1$.  The resulting short
  exact sequence of $\GG$-modules
\begin{equation}
\label{isacomplex}
0 \rightarrow R_0 \rightarrow  E \stackrel{\alpha^\vee}{\rightarrow}
 J[\phi] \rightarrow 0
\end{equation}
is split-exact, since it consists of modules that are free as
$\Z/r\Z$-modules and have trivial $\GG$-action. Dualizing
\eqref{isacomplex} and taking Galois cohomology, we
obtain\begin{equation}
\label{embedH1UsingE}
0 \rightarrow H^1(J[\phi]^\vee) \rightarrow H^1(E^{\vee}) 
\rightarrow H^1(R_0^\vee) \rightarrow 0,
\end{equation}
which is again split-exact by functoriality.  Then
$$H^1(J[\phi]^\vee) = H^1(J[\phi^\vee]) = H^1(J[\phi]),$$ 
where the last step follows from Lemma~\ref{rosati}. Next, we compute
that
\[H^1(E^{\vee}) = H^1(\mu_r^\Delta) = \prod_{Q \in \Delta}
K_d^{\times}/K_d^{\times r},\] the last step being a consequence of
Hilbert's Theorem 90. Choosing the isomorphism $\Z/r\Z \iso R_0$ given
by $1 \mapsto \rho$, we identify $H^1(R_0^\vee)$ with $H^1(\mu_r) =
K_d^{\times}/K_d^{\times r}$, where the last step again follows from
Hilbert's Theorem 90. With these identifications, the short exact
sequence \eqref{embedH1UsingE} becomes the short exact sequence in the
statement of part (1).
%The exactness of its bottom row follows from the exactness of \eqref{embedH1UsingE}. 
Part (2) follows from Proposition 6.4 in \cite{bps}.
\end{proof}

It follows from Proposition~\ref{propertiesOfCC} that $\XminusT$
induces a map
\[
J(K_d) \longrightarrow \prod_{Q \in \Delta} K_d^{\times}/K_d^{\times r}.
\]
We denote this map also by $\XminusT$. The map $\XminusT$ can be
seen as a computation-friendly substitute for the coboundary map
$\delta : J(K_d) \rightarrow H^1(J[\phi])$, since $\XminusT =
\alpha \circ \delta$, where $\alpha$ is an injection.

The rest of this section is devoted to the computation of
$\XminusT(Q)$ for $Q\in\Delta$, $Q=Q_\infty$, and
$Q=P_{i,j}$.

The following lemma states that $\XminusT$ can be ``evaluated on the
coordinates for which it makes sense to do so.''

\begin{lemma}
\label{evaluateWhereItMakesSense}
Suppose $Q\in\Delta$ and let $D \in \Div(C_{K_d})$ be a divisor with
support outside of $\{Q,Q_\infty\}$.  Then
$$\XminusT(D)_Q = \prod_P (x(P)-x(Q))^{\operatorname{ord}_P(D)},$$
where the product is taken over all points $P$ in the support of $D$.
\end{lemma}

\begin{proof}
  Choose $D' \in \Div(C_{K_d}^\circ)$ linearly equivalent to $D$ and
  $g \in K_d(C)^{\times}$ such that $D' = D+\divi(g)$. Observe that
  $\divi(g)$ is supported outside $Q$ and $Q_\infty$. Then
\begin{align*}
\XminusT(D)_Q =\xminusT(D')_Q
& = \prod_P (x(P)-x(Q))^{\operatorname{ord}_P(D')}\\
& =\prod_P (x(P)-x(Q))^{\operatorname{ord}_P(D+\divi(g))} \\
& = \prod_P (x(P)-x(Q))^{\operatorname{ord}_P(D)} 
\prod_P (x(P)-x(Q))^{\operatorname{ord}_P(g)}.
\end{align*}
In the last expression however, the contribution of the second product
is trivial:
$$\prod_P (x(P)-x(Q))^{\operatorname{ord}_P(g)} = 
\prod_{P} g(P)^{\operatorname{ord}_P(x-x(Q))} = g(Q)^r g(Q_\infty)^{-r} = 1,$$
where the first equality is due to Weil reciprocity and the second one
rests on the calculation that $\divi(x-x(Q)) = r \cdot Q - r \cdot
Q_\infty$ for $Q \in \Delta$.
\end{proof}

We end this section by applying Proposition~\ref{propertiesOfCC} and
Lemma~\ref{evaluateWhereItMakesSense} to compute the images under
$\XminusT$ of various divisors.

\begin{proposition}\label{computation}
We have:
\begin{align*}
\XminusT(Q_0) &= (t,1,t),\\
\XminusT(Q_1) &= (-1,1/(1-t),t-1),\\
\XminusT(Q_t) &= (-t,1-t,t/(t-1)),\\
\XminusT(Q_\infty) &= (1,1,1),\\
\noalign{\text{and}}
\XminusT(P_{i,j}) &= (\zeta_d^i u,\zeta_d^i u + 1,\zeta_d^i u + t).
\end{align*}
\end{proposition}

\begin{proof}
Let $\XminusT(Q_0)=(v_0,v_1,v_t)$.  By
Lemma~\ref{evaluateWhereItMakesSense}, $v_1=1$ and $v_t=t$.  By
Proposition~\ref{propertiesOfCC}, $v_1v_t/v_0=1$ in
$K_d^\times/K_d^{\times r}$, so $v_0=t$.  This shows that 
$\XminusT(Q_0) = (t,1,t)$.  The calculations for $Q_1$ and $Q_t$ are
similar and will be left as an exercise for the reader.  Using the
linear equivalence $(r+1)Q_\infty\sim(r-1)Q_0+Q_1+Q_t$ yields that
$\XminusT(Q_\infty) = (1,1,1)$.  Finally, the assertions for
$P_{i,j}$ follow immediately from the definition of $\XminusT$.
\end{proof}

\section{The image of $\XminusT$}

Recall that $V \subset J(K_d)$ is the subgroup generated by the
classes of $P_{i,j}$, where $i\in\Z/d\Z$ and $j\in\Z/r\Z$ and where we
identify $C$ with its image in $J$ by $P\mapsto [P-Q_\infty]$.
Observe that the known torsion elements $Q_0$, $Q_1$, $Q_t$ and $Q_2$
(with $Q_2$ defined as in Section~\ref{s:torsion}) are all contained
in $V$.  

\begin{proposition}
\label{dimensiond}
The dimension of $\XminusT(V)$ is 
$$\dim_{\F_r} (\XminusT(V)) = d.$$
\end{proposition}
\begin{proof}
   First, $\dim_{\F_r} \XminusT(V) \leq d$ since 
   \[\XminusT(P_{i,j}-Q_\infty) 
%= \XminusT(\zeta_r^j(P_{i,0}-Q_\infty)) 
= \XminusT(P_{i,0} - Q_\infty).\]
  To show that the dimension is precisely
  $d$, we project from $\prod_{Q \in \Delta} K_d^{\times} /
  K_d^{\times r}$ to a finite-dimensional quotient space of dimension
  $d$, and conclude by showing that the projection is surjective.

  For an irreducible polynomial $\pi$ in $k[u]$, the valuation it
  induces on $K_d^\times$ is denoted $\val_\pi : K_d^\times \rightarrow
  \Z$. We define the following map:
\begin{align*}
\pr : \prod_{Q \in \Delta} K_d^{\times} / K_d^{\times r} & \rightarrow \F_r^{d} \\\
(v_0,v_1,v_t)& \mapsto \left(\val_{u+\zeta_d^{-1}}(v_1),
\val_{u+\zeta_d^{-2}}(v_1),\ldots,\val_{u+\zeta_d}(v_1),\val_{u+1}(v_1)\right)
\end{align*}

By Proposition~\ref{computation}, $\XminusT(P_{i,j}-Q_\infty)
= (\zeta_d^i u,\zeta_d^i u + 1,\zeta_d^i u + t)$. We see that $\pr$
maps the image of $P_{i,j} - Q_\infty$ to the $i$-th basis
vector. Hence $\pr$ maps $\XminusT(V)$ surjectively onto
$\F_r^d$. This establishes the proposition.
\end{proof}

\begin{lemma}
\label{dimensiontorsion}
The images under $\XminusT$ of $Q_1$ and $Q_2$ are linearly
independent.
\end{lemma}
\begin{proof}
  Since $\XminusT(P_{i,j}-Q_\infty) = \XminusT(P_{i,0} - Q_\infty)$,
  as noted in the proof of Proposition~\ref{dimensiond},
  the image of $Q_2=\sum_{j=0}^{r-1}\sum_{k=0}^{r-1-j} \sum_{i\equiv
    k\mod r} P_{i,j}$ is the
  same as that of $\sum_{i=0}^{d-1} (d-i) (P_{i,0} - Q_\infty)$. Using
  the notation of the proof of Proposition~\ref{dimensiond}, we
  find $\pr(\XminusT(Q_2)) = (-1,-2,\ldots,-d+1,0) \in \F_r^d$.

  Proposition~\ref{computation} gives
  $\XminusT(Q_1-Q_\infty)=(-1,1/(1-t),t-1)$.
  The factorization $1-t = \prod_{i=0}^{d-1} (1-\zeta_d^i u)$ in $K_d$ yields that
  $\pr(\XminusT(Q_1-Q_\infty)) = (-1,-1,-1,\ldots,-1)$. The lemma now
  follows.
\end{proof}

\section{Proof of the main theorem}
\label{Chap2PfSect}
We now use properties of $V$ as a module over $S=\Z[\zeta_r]$ to
prove Theorem~\ref{subgoal}.  Write $\phi=1-\zeta_r$ both for the
element of $S$ and for the corresponding isogeny of $J$.  Note that
$S/\phi S\cong\F_r$.

\begin{proposition}\label{Prop:V-dims}
$\dim_{\F_r}(V/\phi V)=d$ and $\dim_{\F_r}(V[\phi])=2$.
\end{proposition}

\begin{proof}
  $V$ is generated over $S$ by the $P_{i,0}$ with $i\in\Z/d\Z$, so
  $\dim_{\F_r}V/\phi V$ is at most $d$.  On the other hand, the map
$$\XminusT:V\to\prod_{Q\in\Delta}K_d^\times/K_d^{\times r}$$
factors through $V/\phi V$ (Proposition~\ref{propertiesOfCC} or
Proposition~\ref{computation}) and its image has $\F_r$-dimension $d$
(Proposition~\ref{dimensiond}), so $\dim_{\F_r}V/\phi V\ge d$ and
therefore $=d$.  That $\dim_{\F_r}V[\phi]=2$ was proven in
Corollary~\ref{cor:phi-torsion}.
\end{proof}

\begin{proof}[Proof of Theorem~\ref{subgoal}]
Let $V_{tor}$ be the torsion submodule of $V$.  Then $V/V_{tor}$ is
torsion-free, so projective over $S$, so locally free of some rank
$\rho$.  We use the preceding proposition to compute that $\rho=d-2.$

Let $S_{(\phi)}$ and $V_{(\phi)}$ denote the localizations of $S$ and
$V$ respectively at the (prime) ideal generated by $\phi$.  By the
structure theorem for modules over a PID, we have
$$V_{(\phi)}\cong S_{(\phi)}^\rho\oplus\bigoplus_{i=1}^tS/(\phi^{e_i})$$
for some integers $t$ and $e_1,\dots,e_t$ with the $e_i>0$.  Also,
$$\rho+t=\dim_{\F_r}V_{(\phi)}/\phi V_{(\phi)} = 
\dim_{\F_r}V/\phi V = d$$
and
$$t=\dim_{\F_r}V_{(\phi)}[\phi]=\dim_{\F_r}V[\phi]=2.$$
It follows that $\rho=d-2$.  Therefore
$$\rk_\Z V=(\rk_\Z S)(\rk_S V)=(r-1)(d-2).$$

For the torsion assertions, we note from Lemma~\ref{dimensiontorsion}
that $Q_1$ and $Q_2$ are linearly independent in $V/\phi V$, since
their images in $\prod_{Q\in\Delta}K_d^\times/K_d^{\times r}$ are
linearly independent. Thus they form a basis of
$V_{tor}/\phi V_{tor}$.  Moreover, since $\phi Q_1=0$, $\phi Q_2=Q_0$
(Lemma~\ref{lemma:more-relations}), and $\phi Q_0=0$, we have that
$$V[r^\infty]=V[\phi^\infty]\cong S/(\phi)\oplus S/(\phi^2)$$
as $S$-modules and 
$$V[r^\infty]\cong \F_r^3=(\Z/r\Z)^3$$
as $\Z$-modules.  Finally, since $J(K_d)[\phi]=V[\phi]$
(Corollary~\ref{dimensiontorsion}), we have that 
$$J(K_d)[r^\infty]=V[r^\infty]\cong (\Z/r\Z)^3.$$
This completes the proof of the theorem.
\end{proof}

\begin{remark}
  With very small changes, the proof of Theorem~\ref{subgoal} can be
  modified to handle the case where $r$ is an odd prime power. On
  other hand, these methods do not suffice to treat the general case,
  because if $r$ is divisible by two distinct odd primes, then
  $1-\zeta_r$ is a unit in $\Z[\zeta_r]$.
\end{remark}

%%% Local Variables: 
%%% mode: latex
%%% TeX-master: "EHR"
%%% End: 

%This is Chapter 3

\chapter[Minimal regular model]{Minimal regular model, local
  invariants, and domination by a product of curves}
\label{ch:models}

In Section~\ref{s:models} of this chapter, we construct two
useful models for the curve $C/K_d$ over $\P^1_u$, i.e., surfaces $\XX$
and $\YY$ equipped with projective morphisms to $\P^1_u$ with generic
fibers $C/K_d$.  The model $\XX$ is a smooth surface, whereas the
model $\YY$ is normal with mild singularities.  We also work
out the configuration of the components of the singular fibers of
$\XX\to\P^1_u$ (i.e., their genera, intersections, and
self-intersections).  The explicit model $\XX$ and the analysis of the
fibers play a key role in the height calculations of
Chapter~\ref{ch:heights} and in the monodromy calculations
of Chapter \ref{ch:monodromy}.

The analysis of the fibers of $\XX\to\P^1_u$ is used in
Section~\ref{s:local-invs} to obtain important local invariants of the
N\'eron model of $J$ including its component groups and the connected
component of the identity.  The local invariants of the N\'eron model
are used in our analysis of the $L$-function of $J$ in
Chapter~\ref{ch:L}.

Finally, in Section~\ref{s:DPC} we discuss a precise connection
between the model $\YY$ and a certain product of curves.  The fact
that $\XX$ and $\YY$ are birationally dominated by a product of
curves, as shown in Section~\ref{ss:DPC}, allows us to prove the BSD
conjecture for $J$.  The finer analysis of the geometry of the
dominating map, which occupies the rest of Section~\ref{s:DPC}, may be
of use in further study of explicit points on $C$, but it is not
crucial for the rest of the current paper and may be omitted by
readers not interested in the details.

\section{Models}\label{s:models}
In this section, $k$ is an arbitrary field.  We fix positive
integers $r$ and $d$ both prime to the characteristic of $k$, and we
let $C$ be the curve over $k(u)$ defined as in Section~\ref{s:curve}
where $u^d=t$.  In the applications later in the paper, $k$ is a
finite extension of $\Fp(\mu_d)$ for some prime $p$ not dividing $rd$.

For convenience, in the first part of this section, we assume that $d$
is a multiple of $r$.  The general case is treated in
Section~\ref{ss:general-d}.

The model $\YY$ we construct is a suitable compactification of a
blow-up of the irreducible surface in affine 3-space over $k$ defined
by $y^r=x^{r-1}(x+1)(x+u^d)$.  The model $\XX$ we construct is
obtained by resolving isolated singularities of $\YY$.

\subsection{Construction of $\YY$}\label{ss:YY}
Let $R=k[u]$, $U=\spec R$, $R'=k[u']$, and $U'=\spec R'$.  We glue $U$
and $U'$ via $u'=u^{-1}$ to obtain $\P^1_u$ over $k$. 

On $\P^1$ over $k$ define 
$$\EE=\OO_{\P^1}(d)\oplus\OO_{\P^1}(d+d/r)\oplus\OO_{\P^1}$$
so that $\EE$ is a locally free sheaf of rank 3 on $\P^1$.  Its
projectivization $\P(\EE)$ is a $\P^2$ bundle over $\P^1$.  We
introduce homogeneous coordinates $X,Y,Z$ on the part of $\P(\EE)$
over $U$ and homogeneous coordinates $X',Y',Z'$ on the part over $U'$.
Then $\P(\EE)$ is the result of gluing $\proj(R[X,Y,Z])$ and
$\proj(R'[X',Y',Z'])$ via the identifications $u=u^{\prime-1}$,
$X=u^dX'$, $Y=u^{d+d/r}Y'$, and $Z=Z'$.

Now define $\ZZ\subset\P(\EE)$ as the closed subset where
$$Y^rZ=X^{r-1}(X+Z)(X+u^dZ)$$
in $\proj(R[X,Y,Z])$ and
$$Y^{\prime r}Z'=X^{\prime r-1}(X'+u^{\prime d}Z')(X'+Z')$$
in $\proj(R'[X',Y',Z'])$.  Then $\ZZ$ is an irreducible, projective
surface equipped with a morphism to $\P^1_u$.  The generic fiber is the
curve denoted $C'$ in Section~\ref{s:curve}, which is singular at $[0,0,1]$.

We write $\ZZ_U$ and $\ZZ_{U'}$ for the parts of $\ZZ$ over $U$ and
$U'$ respectively.  Then $\ZZ_{U'}$ is isomorphic to $\ZZ_U$; indeed,
up to adding primes to coordinates, they are defined by the same
equation.  (This is why it is convenient to assume that $r$ divides
$d$.)  We thus focus our attention on $\ZZ_U$, i.e., on
$$\proj\left(R[X,Y,Z]/(Y^rZ-X^{r-1}(X+Z)(X+u^dZ))\right).$$

We next consider the standard cover of $\ZZ_U$ by affine opens
where $X$, $Y$, or $Z$ are non-vanishing.  These opens are
\begin{align*}
\ZZ_1&:=\spec\left(R[x_1,y_1]/(y_1^r-x_1^{r-1}(x_1+1)(x_1+u^d))\right),\\
\ZZ_2&:=\spec\left(R[x_2,z_2]/(z_2-x_2^{r-1}(x_2+z_2)(x_2+u^dz_2))\right),\\
\ZZ_3&:=\spec\left(R[y_3,z_3]/(y_3^rz_3-(1+z_3)(1+u^dz_3))\right).
\end{align*}

The surface $\ZZ_1$ is singular along the curve $x_1=y_1=0$, so we
blow up along that curve.  (Strictly speaking, $\ZZ_1$ is singular
along this curve only if $r>2$.  Nevertheless, we proceed as follows even
if $r=2$.)  More precisely, we define
\begin{align*}
\ZZ_{11}&:=\spec\left(R[x_{11},y_{11}]/
   (y_{11}-x_{11}^{r-1}(x_{11}y_{11}+1)(x_{11}y_{11}+u^d))\right),\\
\ZZ_{12}&:=\spec\left(R[x_{12},y_{12}]/
   (x_{12}y_{12}^r-(x_{12}+1)(x_{12}+u^d))\right),
\end{align*}
and let $\tilde\ZZ_1$ be the glueing of $\ZZ_{11}$ and $\ZZ_{12}$
given by $(x_{11},y_{11})=(1/y_{12},x_{12}y_{12})$.  The morphism
$\tilde\ZZ_1\to\ZZ_1$ defined by
$(x_1,y_1)=(x_{11}y_{11},y_{11})=(x_{12},x_{12}y_{12})$ is projective,
surjective, and an isomorphism away from $x_1=y_1=0$.

We define $\YY_U$ to be the glueing of $\ZZ_2$, $\ZZ_3$, and
$\tilde\ZZ_1$ by the identifications
\[(x_2,z_2)=(1/y_3,z_3/y_3)\quad\text{and}\quad
(y_3,z_3)=(1/x_{11},1/(x_{11}y_{11})).\] 
Define $\YY_{U'}$ similarly (by glueing opens $\ZZ_2'$, $\ZZ_3'$,
$\ZZ_{11}'$, and $\ZZ_{12}'$), and let $\YY$ be the glueing of $\YY_U$
and $\YY_{U'}$ along their open sets lying over $\spec k[u,u^{-1}]$.
The result of this glueing is a projective surface with a morphism to
$\P^1_u$ whose generic fiber is the curve $C/k(u)$.  Note that,
directly from its definition, $\YY$ is a local complete intersection.

It is easy to see that $\YY$ is already covered by the affine opens
$\ZZ_{11}$, $\ZZ_2$, and $\ZZ_3$, and we use this cover in some
calculations later in Section~\ref{s:KS}.  On the other hand, the
coordinates of $\ZZ_{12}$ are also convenient, which is why they are
included in the discussion.

A straightforward calculation with the Jacobian criterion
%easiest to start with \ZZ_3 and then check the remaining points
shows that $\YY_U\to U$ and $\YY_U\to\spec k$ are smooth except at the
points
$$u=y_{11}=0,\qquad x_{11}^r=1,$$ 
and, when $d>1$, at the points
$$u^d=1,\qquad y_3=0,\qquad z_3=-1.$$  
The fibers of $\YY_U\to U$ are irreducible except over $u=0$, where
the fiber has irreducible components $y_{11}=0$ and
$x_{11}^r(x_{11}y_{11}+1)=1$, both of which are smooth rational
curves.  The points of intersection of these irreducible components
are the singular points in the fiber over $u=0$.  Similar results hold
for $\YY_{U'}$.

To finish our analysis of $\YY$, we note that it satisfies Serre's
conditions $S_n$ for all $n\ge0$ since it is a local complete
intersection.  Moreover, it has isolated singularities, so it
satisfies condition $R_1$ (regularity in codimension 1).  It follows
from Serre's criterion that $\YY$ is normal.

Summarizing, the discussion above proves the following result.

\begin{prop}\label{prop:YY-props}
  The surface $\YY$ and morphism $\YY\to\P^1_u$ have the following
  properties: 
  \begin{enumerate}
  \item $\YY$ is irreducible, projective, and normal.
  \item The morphism $\YY\to\P^1_u$ is projective and generically
    smooth.
  \item The singularities of $\YY\to\P^1_u$ consist of $r$ points in
    the fiber over $u=0$, one point in each fiber over points $u \in \mu_d$ 
and $r$ points in the fiber over $u=\infty$.  When $d>1$,
    these are also the singularities of $\YY$, whereas if $d=1$, only
    the singularities of $\YY\to\P^1_u$ over points $u \in \mu_d$
    are singularities of $\YY$.
  \item The fibers of $\YY\to\P^1_u$ are irreducible except over
    $u=0,\infty$ where they are unions of two smooth rational curves
    meeting transversally in $r$ points.
  \item The generic fiber of $\YY\to\P^1_u$ is a smooth projective
    model of the curve defined by $y^r=x^{r-1}(x+1)(x+u^d)$ over
    $k(u)$.
  \end{enumerate}
 \end{prop}

\begin{remark}
  It is tempting to guess that $\YY$ is the normalization of $\ZZ$,
  but this is not correct.  Indeed, the morphism $\YY\to\ZZ$ contracts
  the curve $u=y_{11}=0$ in $\ZZ_{11}$, so is not finite.  It is not hard to
  check that the normalization of $\ZZ$ is in fact the surface
  obtained from $\YY$ by contracting this curve and the analogous
  curve over $u=\infty$.
\end{remark}

\subsection{Singularities of $\YY$}
We now show that $\YY$ has mild singularities.  Recall that
rational double points on surfaces are classified by Dynkin diagrams
of type $ADE$. (See for example \cite[3.31--32]{Badescu}.)  In
particular, to say that a point $y\in \YY$ is a rational double point of
type $A_n$ is to say that $y$ is a double point and that there is a
resolution $\XX\to\YY$ such that the intersection matrix of the fiber
over $y$ is of type $A_n$.

\begin{prop} \label{PsingularA} The singularities of
  $\overline{\YY}:=\YY\times_k\kbar$ are all rational double points.
  More precisely, the singularities in the fibers over $u=0$ and
  $u=\infty$ are analytically equivalent to the singularity
  $\alpha\beta=\gamma^d$ and are thus double points of type
  $A_{d-1}$.\footnote{Of course, a ``double point'' of type $A_0$ is
    in fact a smooth point.} The singularities over the points $u \in
  \mu_d$ are analytically equivalent to the singularity
  $\alpha\beta=\gamma^r$ and are thus double points of type $A_{r-1}$.
\end{prop}

\begin{proof}
For notational simplicity, we assume that $k$ is algebraically closed,
so that $\overline{\YY}=\YY$.

First consider the fiber over $u=0$.  We use the coordinates of the
open $\ZZ_{11}$, that is, the hypersurface in $\A^3$ defined by
$y-x^{r-1}(xy+1)(xy+u^d)=0$.  (We drop the subscripts to lighten
notation.)  The singularities are at the points with $u=y=0$, and
$x^r=1$, and we work in the completed local ring of $\A^3$ at each one
of these points.  Choose an $r$-th root of unity $\zeta$ and change
coordinates $x=x'+\zeta$ so that $x'$, $y$, and $u$ form a system of
parameters at one of the points of interest.  In the completed local
ring, the element
$$x^{r-1}(xy+1)=(x'+\zeta)^{r-1}((x'+\zeta)y+1)$$ 
is a unit, and since $d$ is prime to the characteristic of $k$, it is
also a $d$-th power.  Defining $\gamma$ by
$u=\gamma(x^{r-1}(xy+1))^{-1/d}$, then $x'$, $y$, and $\gamma$ are a
system of parameters, and in these parameters, the defining equation
becomes
\[y\left(1-(x'+\zeta)^r((x'+\zeta)y+1)\right)-\gamma^{d}=0.\]
Finally, note that
\[(1-(x'+\zeta)^r((x'+\zeta)y+1)=-r\zeta^{r-1}x'-\zeta y+(\deg\ge2)\]
where ``$\deg\ge2$'' stands for terms of degree at least two in $x'$
and $y$.  Since $r$ is prime to the characteristic of $k$, the
coefficient of $x'$ is not zero so we may set
\[\alpha=y,\qquad \beta=(1-(x'+\zeta)^r((x'+\zeta)y+1),\]
and have $\alpha,\beta,\gamma$ as a system of parameters.  In these
coordinates, the defining equation becomes $\alpha\beta=\gamma^{d}$.
This proves that the singularities of $\YY$ over $u=0$ are
analytically equivalent to $\alpha\beta=\gamma^{d}$.

The argument for the points over $u=\infty$ is identical to the above.

Now consider the fiber over a point $u \in \mu_d$, using the
coordinates of the open $\ZZ_3$, that is, the hypersurface in $\A^3$
defined by $y^rz-(1+z)(1+u^dz)=0$.  (Again we omit subscripts to
lighten notation.)  Choose a $d$-th root of unity $\zeta$ and let
$u=u'+\zeta$.  The singular point over $u=\zeta$ has coordinates
$u'=y=0$, $z=-1$. Setting $z=\alpha-1$, the defining equation becomes
\[y^r(\alpha-1)-\alpha\left(1+(u'+\zeta)^d(\alpha-1)\right)= 0.\] 
As before, $\alpha-1$ is an $r$-th power in the completed local
ring, and we set $y=\gamma(\alpha-1)^{-1/r}$.  Moreover,
\[\left(1+(u'+\zeta)^d(\alpha-1)\right)=-d\zeta^{d-1}u'+\alpha+(\deg\ge2)\]
so we may set $\beta=\left(1+(u'+\zeta)^d(\alpha-1)\right)$ and have
$\alpha,\beta,\gamma$ as a set of parameters.  In these parameters,
the defining equation becomes $\gamma^r=\alpha\beta$.

To finish, it remains to observe that the singularity at the origin
defined by $\alpha\beta=\gamma^n$ is a rational double point of type
$A_{n-1}$.  This is classical and due to Jung over the complex
numbers.  That it continues to hold in any characteristic not dividing
$n$ is
stated in many references (for example \cite[Page~15]{Artin77}), but
we do not know of a reference for a detailed proof of this
calculation.\footnote{In connection with a related proof, Artin writes
  ``Following tradition, we omit the rather tedious verification of
  these results.''}
It is, however, a straightforward calculation, and we leave it as an
exercise for the reader.
\end{proof}

\begin{remark}  
  This paper contains two other proofs that the singularities of $\YY$
  are rational double points.  The first comes from resolving the
  singularity with an explicit sequence of blow-ups; see
  Section~\ref{ss:bad-fibers}.  Doing this reveals that the
  configurations of exceptional curves are those of rational double
  points of type $A_n$ with $n=d-1$ or $r-1$.  (It also reveals that
  the singularities are ``absolutely isolated double points,'' i.e.,
  double points such that at every blow up the only singularities are
  isolated double points.  This is one of the many characterizations
  of rational double points.)
%rational=>abs isolated is Badescu 3.26 following Lipman
%the converse is stated in Ried and elsewhere.
The calculation in Section~\ref{ss:bad-fibers} is
independent of Proposition \ref{PsingularA}, so there is no circularity.

The second alternative proof (given in Section~\ref{s:DPC} below) uses
the fact that $\YY$ is a quotient of a smooth surface by a finite
group acting with isolated fixed points and cyclic stabilizers.  This
shows that the singularities of $\YY$ are cyclic quotient
singularities, therefore rational singularities, and it is clear from
the equations that they are double points.  (The action is explicit,
and we may also apply \cite[Exercise~3.4]{Badescu}.)
\end{remark}

\subsection{Construction of $\XX$}\label{ss:XX}
With $\YY$ in hand, $\XX$ is very simple to describe: We define
$\XX\to\YY$ to be the minimal desingulariztion of $\YY$.

Let us recall how to desingularize a rational double point of type
$A_n$.  The resolution has an exceptional divisor consisting of a
string of $n-1$ smooth, rational curves, each meeting its neighbors
transversally and each with self-intersection $-2$.  If $n$ is odd, we
blow up $(n-1)/2$ times, each time introducing 2 rational curves.  If
$n$ is even, the first $(n-2)/2$ blow-ups each introduce 2 rational
curves, and the last introduces a single rational curve.

\subsection{Fibers of $\XX\to\P^1_u$}\label{ss:bad-fibers}
In this subsection, we record the structure of the bad fibers of
$\XX\to\P^1_u$.  More specifically, we work out the configuration of
irreducible components in the fibers: their genera, intersection
numbers, and multiplicities in the fiber.

First consider the fiber of $\YY\to\P^1_u$ over $u=0$.  Using the
coordinates of the chart $\ZZ_{11}$ above, this fiber is the union of
two smooth rational curves $y=0$ and $1-x^r(xy+1)=0$ meeting at the
$r$ points $y=0$, $x^r=1$.  These crossing points are singularities of
$\YY$ of type $A_{d-1}$.  In the resolution $\XX\to\YY$, each of them
is replaced with a string of $d-1$ rational curves.  It is not hard to
check (by inspecting the first blow-up) that the components $y=0$ and
$1-x^r(xy+1)=0$ meet the end components of these strings transversely
and do not meet the other components.  We label the components so that
those in the range $j(d-1)+\ell$ with $1\le\ell\le d-1$ come from the
point with $x=\zeta_r^j$.

Resolving the singularities thus yields the configuration of curves
displayed in Figure~\ref{fig:u-equals-zero} below.  (This picture is
for $d>1$.  If $d=1$, then $\YY$ does not have singularities in the
fibers over $u=0$, and the fiber consists of a pair of smooth rational
curves meeting tranversally at $r$ points.)  In the figure, $C_0$ is
the strict transform of $1-x^r(xy+1)=0$, $C_{r(d-1)+1}$ is the strict
transform of $y=0$, and the other curves are the components introduced
in the blow-ups.  \setlength{\unitlength}{.3cm}
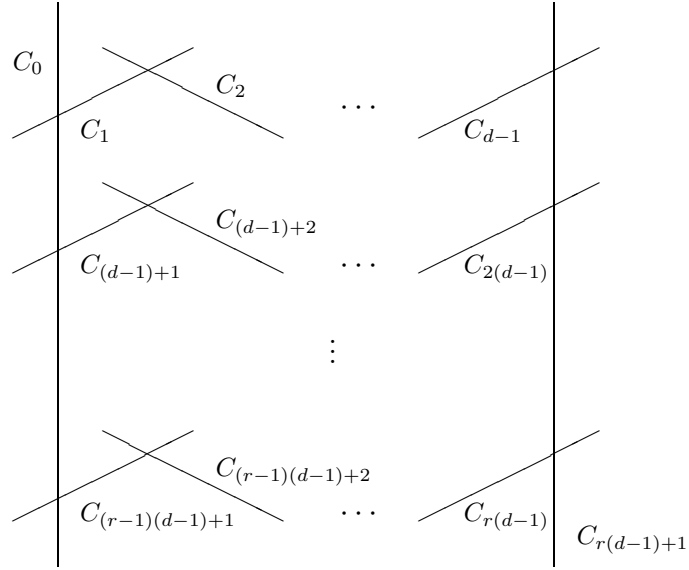
\begin{figure}[h]\centering
  \begin{picture}(26,25)
    \put(2,0){\line(0,1){25}} %C_0
    \put(0,22){$C_0$}
    \put(24,0){\line(0,1){25}} %C last
    \put(25,1){$C_{r(d-1)+1}$}
    \put(0,19){\line(2,1){8}}
    \put(3,19){$C_1$}
    \put(4,23){\line(2,-1){8}}
    \put(9,21){$C_2$}
    \put(14.5,20){\Large$\cdots$}
    \put(18,19){\line(2,1){8}}
    \put(20,19){$C_{d-1}$}
    \put(0,13){\line(2,1){8}}
    \put(3,13){$C_{(d-1)+1}$}
    \put(4,17){\line(2,-1){8}}
    \put(9,15){$C_{(d-1)+2}$}
    \put(14.5,13){\Large$\cdots$}
    \put(18,13){\line(2,1){8}}
    \put(20,13){$C_{2(d-1)}$}
    \put(0,2){\line(2,1){8}}
    \put(3,2){$C_{(r-1)(d-1)+1}$}
    \put(4,6){\line(2,-1){8}}
    \put(9,4){$C_{(r-1)(d-1)+2}$}
    \put(14.5,2){\Large$\cdots$}
    \put(18,2){\line(2,1){8}}
    \put(20,2){$C_{r(d-1)}$}
    \put(14,9){\Large$\vdots$}
  \end{picture}
  \caption{Special fiber at $u=0$ for $d>1$.  We have $g(C_i)=0$ for
    all $i$, $C_0^2=C_{r(d-1)+1}^2=-r$, and $C_i^2=-2$ for $1\le i\le
    r(d-1)$.  All components are reduced in the fiber.}
\label{fig:u-equals-zero}
\end{figure}

Each component is a smooth rational curve, and all intersections are
transverse. The components introduced in the blow-up have
self-intersection $-2$. Since the intersection number of any component
of the fiber with the total fiber is 0, the self-intersections of the
strict transforms of $C_0$ and $C_{r(d-1)+1}$ are both $-r$. Those
components are reduced in the fiber of $\YY\to\P^1_u$, so they must
also be reduced in the fiber of $\XX\to\P^1_u$. It follows that all
components of the fiber of $\XX\to\P^1_u$ are reduced. We note that
the fiber at 0 is thus semi-stable.

As already noted, a neighborhood of $u=0$ in $\YY$ is isomorphic to a
neighborhoood of $u=\infty$ in $\YY$, so the the fiber at $u=\infty$
of $\XX\to\P^1_u$ is isomorphic to that at $u=0$.  (Note that $r$
divides $d$ in the construction of $\YY$ in Section~\ref{ss:YY}.  We
see in Section~\ref{ss:general-d} that the fibers over $u=0$ and
$u=\infty$ are not isomorphic for general $d$.)

We now turn to the fiber over a point $u \in \mu_d$.  Since
$\P^1_u\to\P^1_t$ is unramified over $t=1$, the fibers of
$\XX\to\P^1_u$ over the points with $u^d=1$ are independent of $d$,
and we may thus assume that $d=1$.  We work in the chart $\ZZ_3$ with
equation $y^rz-(1+z)(1+tz)=0$ where the singularity has coordinates
$t=1$, $y=0$, $z=-1$.  Replacing $t$ with $t+1$ and $z$ with $z-1$,
the equation becomes $y^r(z-1)-z(z-t+tz)=0$, and the singularity is at
the origin and is of type $A_{r-1}$.  The fiber is the curve
$y^r(z-1)=z^2$, which has geometric genus $(r-2)/2$ or $(r-1)/2$ as $r$
is even or odd, with a double point at $y=z=0$.

We know that the singular point blows up into a chain of $r-1$
rational curves, and our task now is to see how the proper transform
of the fiber intersects these curves.  Since the case $r=2$ already
appears in \cite{Legendre}, we assume $r>2$ for convenience.  It is
also convenient to separate the cases where $r$ is odd and where $r$
is even.

First consider the case where $r$ is odd.  After the first blow-up,
the relevant piece of the strict transform of $\YY$ has equation
$y^{r-2}-y^{r-1}z+z(z-t+tyz)=0$; the exceptional
divisor is $z(z-t)=0$, the union of two reduced lines meeting at the
origin; and the proper transform of the original fiber meets the
exceptional divisor at the origin.  The next blow up introduces two
lines meeting transversally at one point, and they have multiplicity 2
in the fiber.  The strict transform of the original fiber passes
through the intersection point and meets the components transversally.
This picture continues throughout each of the blowups, and after
$(r-1)/2$ steps the strict transform of the original fiber meets the
chain of $r-1$ rational curves at the intersection point of the middle
two curves, and it meets each of these curves transversally.  

The picture for $r$ odd is as in
Figure~\ref{fig:superelliptic-dual-graph-odd} below.  The curves $D_i$
and $E_i$ appear at the $i$-th blow up; their multiplicities in the
fiber are $i$; and the self intersections of each $D_i$ or $E_i$ is
$-2$, The curve $F$ is the proper transform of the original fiber, the
smooth projective curve of genus $(r-1)/2$ associated to
$y^r=x^{r-1}(x+1)^2$.  Also $F$ is reduced in the fiber and its
self-intersection is $-(r-1)$.  The intersections of distinct adjacent
components are transversal.

\setlength{\unitlength}{.4cm}
\begin{figure}[h]\centering
  \begin{picture}(26,12)
    \put(2,0){\line(0,1){12}}
    \put(.5,10){$F$}
    \put(0,4){\line(1,1){7.5}}
    \put(3.3,9){$D_s$}
    \put(6,11){\line(4,-1){6}}
    \put(8,9){$D_{s-1}$}
    \put(13.5,10){\Large$\dots$}
    \put(16,9.5){\line(4,1){6}}
    \put(18,9){$D_2$}
    \put(20,11){\line(4,-1){6}}
    \put(23,9){$D_1$}
    \put(0,8){\line(1,-1){7.5}}
    \put(3.3,2.8){$E_s$}
    \put(6,1){\line(4,1){6}}
    \put(8,2.8){$E_{s-1}$}
    \put(13.5,2){\Large$\dots$}
    \put(16,2.5){\line(4,-1){6}}
    \put(18,2.8){$E_2$}
    \put(20,1){\line(4,1){6}}
    \put(23,2.8){$E_1$}
  \end{picture}
  \caption{Special fiber when $u^d = 1$ and $r=2s+1$.
  % r odd.  
    Here $g(D_i)=g(E_i)=0$, $g(F)=(r-1)/2$, $D_i^2=E_i^2=-2$, and
    $F^2=1-r$. Multiplicities in the fiber are $m(D_i)=m(E_i)=i$ and
    $m(F)=1$.}
\label{fig:superelliptic-dual-graph-odd}
\end{figure}
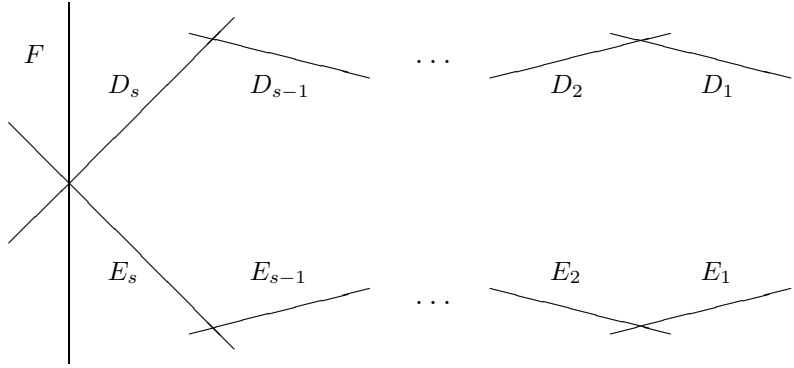

The case where $r$ is even is similar until the last stage. After
$(r-2)/2$ blow-ups, there is a chain of $r-2$ rational curves and the
strict transform of the original fiber passes through the intersection
point of the middle two curves.  The equation at this point is
$y^2-y^{(r+2)/2}z+z(z-t+ty^{(r-2)/2}z)=0$.  The tangent cone is
$y^2+z^2-tz=0$, a smooth irreducible conic, so the last blow up
introduces one smooth rational curve.  After the last blow-up, the
equation becomes $1-y^{r/2}z+z(z-t+ty^{r/2}z)=0$, and the strict
transform of the original fiber meets the last exceptional divisor in
two points namely $t=y=0$, $z=\pm1$.  (Note that $r$ even implies
$p\neq2$, so there really are two points of intersection.)  

The picture for $r$ even is given in
Figure~\ref{fig:superelliptic-dual-graph-even} below.  Again $D_i$ and
$E_i$ have multiplicity $i$ in the fiber and self-intersection $-2$.
The curve $G$ has multiplicity $s=r/2$ in the fiber and
self-intersection $-2$.  The curve $F$ is the strict transform of the
original fiber and is the smooth projective curve associated to
$y^r=x^{r-1}(x+1)^2$.  It is reduced in the fiber, has genus
$(r-2)/2$, and has self-intersection $-r$.

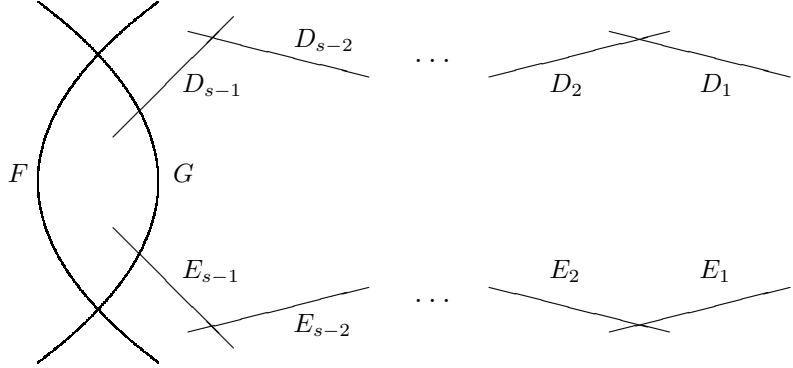
\begin{figure}[h]\centering
  \centering
  \begin{picture}(28,12)
    \qbezier(1,0)(9,6)(1,12)
    \put(5.5,6){$G$}
    \qbezier(5,0)(-3,6)(5,12)
    \put(0,6){$F$}
    \put(3.5,7.5){\line(1,1){4}}
    \put(5.8,9){$D_{s-1}$}
    \put(6,11){\line(4,-1){6}}
    \put(9.5,10.5){$D_{s-2}$}
    \put(13.5,10){\Large$\dots$}
    \put(16,9.5){\line(4,1){6}}
    \put(18,9){$D_2$}
    \put(20,11){\line(4,-1){6}}
    \put(23,9){$D_1$}
    \put(3.5,4.5){\line(1,-1){4}}
    \put(5.8,2.8){$E_{s-1}$}
    \put(6,1){\line(4,1){6}}
    \put(9.5,1){$E_{s-2}$}
    \put(13.5,2){\Large$\dots$}
    \put(16,2.5){\line(4,-1){6}}
    \put(18,2.8){$E_2$}
    \put(20,1){\line(4,1){6}}
    \put(23,2.8){$E_1$}
    % \put(5.5,1.5){$E_{s-1}$}
    % \put(8,2){\line(6,1){6}}
    % \put(15,2.5){\Large$\dots$}
    % \put(17,2){\line(6,1){6}}
    % \put(22,3){\line(6,-1){6}}
    % \put(3,9){\line(6,1){6}}
    % \put(5.5,10){$D_{s-1}$}
  \end{picture}
  \caption{Special fiber when $u^d=1$ and $r = 2s$.  We have
    $g(D_i)=g(E_i)=g(G)=0$, $g(F)=(r-2)/2=s-1$, $D_i^2=E_i^2=G^2=-2$,
    and $F^2=-r$.  Multiplicities in the fiber are $m(D_i)=m(E_i)=i$,
    $m(G)=s$, and $m(F)=1$.}
\label{fig:superelliptic-dual-graph-even}
\end{figure}

Note that the fibers of $\XX\to\P^1_u$ over points with $u^d=1$ are
not semi-stable.   However, it follows from \cite[Theorem 3.11]{Saito} 
that $C/K_d$ acquires semi-stable reduction at these places
after a tamely ramified extension.  All other fibers of $\XX\to\P^1_u$
are semi-stable.  This yields the second part of proposition below.

Summarizing this subsection:

\begin{prop}\label{prop:reduction}
  The configurations of components in the singular fibers of
  $\XX\to\P^1_u$ (genera, intersection numbers, and multiplicity in
  the fiber) are as described above and pictured in
  Figures~\ref{fig:u-equals-zero},
  \ref{fig:superelliptic-dual-graph-odd}, and
  \ref{fig:superelliptic-dual-graph-even}.  The action of
  $\gal(\Ksep/K)$ on $H^1(C\times_K\Kbar,\Ql)$ is at worst tamely
  ramified at every place of $K$.
\end{prop}

\subsection{General $d$}\label{ss:general-d}
Until now in this section, we have worked under the hypothesis that
$r$ divides $d$.  In this subsection, we briefly sketch the
construction of a regular minimal model $\XX\to\P^1_u$ for general
$d$.

In fact, the only issue is near $u=\infty$: The charts $\ZZ_2$,
$\ZZ_3$, and $\ZZ_{11}$ are well defined without assuming that $r$
divides $d$, and they glue as above to give an irreducible, normal
surface $\YY^o$ with a projective morphism $\YY^o\to\A^1_u$ that is a
model of $C$ over $k(u)$.  Over $u=0$ and $u \in \mu_d$, the same
steps as before lead to a regular, minimal model $\XX^o\to\A^1_u$.
This model is semi-stable at $u=0$ with reduction exactly as pictured
in Figure~\ref{fig:u-equals-zero}, and the reduction at points $u \in
\mu_d$ is as pictured in
Figures~\ref{fig:superelliptic-dual-graph-odd} and
\ref{fig:superelliptic-dual-graph-even}.

The situation over $u=\infty$ is more complicated, and the most
efficient way to proceed is to first ``go up'' to level $d'=\lcm(d,r)$
and then take the quotient by the roots of unity of order
$d'/d=r/\gcd(d,r)$. Let $\mathcal{H}=\mu_{d'/d}\subset\mu_{d'}$.

They key point to note is that in constructing the model
$\XX_{d'}\to\P^1_u$ where $u^{d'}=t$, we started with a completion of
the affine model $y=x^{r-1}(x+1)(x+u^{d'})$, made a change of
coordinates $u=u^{\prime-1}$, $x=u^{d'}x'$, $y=u^{d'+d'/r}y'$, and
then performed a blow-up by substituting $x'\to x'y'$, $y'\to y'$.
This yields the chart with equation $y'-x^{\prime r-1}(x'y'+u^{\prime
  d})(x'y'+1)=0$.  The action of $\mathcal{H}$ on these last
coordinates is thus $\zeta(u',x',y')=(\zeta^{-1}
u',\zeta^{d'/r}x',\zeta^{-d'/r}y')$.  Further blowing up yields the
regular minimal model $\XX_{d'}$ whose fiber over $u'=0$ is as
described in Figure~\ref{fig:u-equals-zero}.  The action of
$\mathcal{H}$ lifts canonically to the model $\XX_{d'}$.

We now consider the action of $\mathcal{H}$ on the special fiber over
$u'=0$.  This action preserves the end components and permutes the
horizontal chains with $\gcd(d,r)$ orbits.  The action has 4 isolated
fixed points, which are roughly speaking at the points where $x'$ or
$y'$ are $0$ or $\infty$.  (Specifying them exactly requires
considering other charts, and we omit the details since they are not
important for what follows.)  The exponents on the action on the
tangent space are $(1,1)$ or $(1,-1)$ with one of each type on each
component.  Resolving these quotient singularities leads to chains of
rational curves of length $1$ and $d'/d-1=r/\lcm(d,r)-1$ respectively.
(The configuration of the component is computed using the
``Hirzebruch-Jung continued fraction'' as in \cite[III.5-6]{BHPV}.)

The picture is given in Figure~\ref{fig:fiber-infinity-general-d}
below. In that picture, all components are smooth rational curves.
The components labeled $R_{i,j}$ with $1\le i\le\gcd(r,d)$ and $1\le
j\le d-1$ are the images of the components $C_\ell$ with $1\le\ell\le
r(d'-1)$ in Figure~\ref{fig:u-equals-zero}.  The components labeled
$D_{d'/d}$ and $E_{d'/d}$ are the images of $C_0$ and $C_{r(d'-1)+1}$
respectively.  The components $C_1$ and $C_2$ in
Figure~\ref{fig:fiber-infinity-general-d} come from resolving the
singularity with local exponents $(1,1)$, and the components $D_i$ and
$E_i$ with $1\le i\le d'/d-1$ come from resolving the singularities
with local exponents $(1,-1)$.

\setlength{\unitlength}{.4cm}
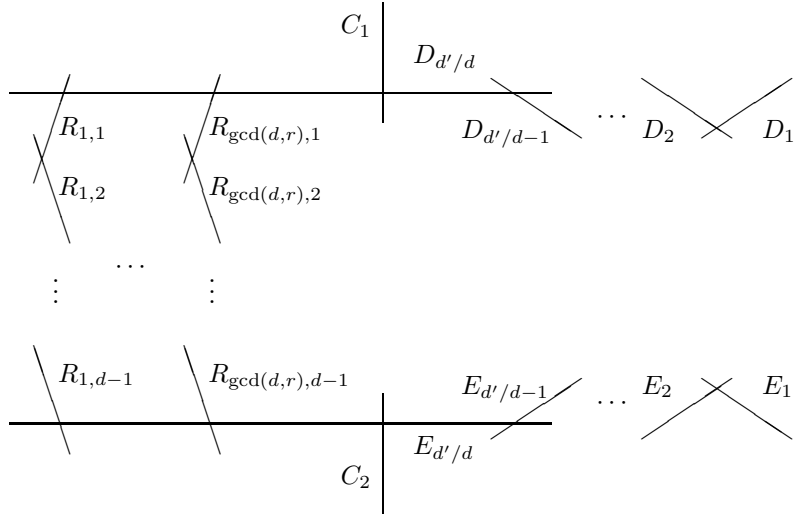
\begin{figure}[h]
  \centering
  \begin{picture}(27,18)
    \put(0,5){\line(1,0){18}}
    \put(13.4,4){$E_{d'/d}$}
    \put(0,16){\line(1,0){18}}
    \put(13.4,17){$D_{d'/d}$}
    \put(12.4,2){\line(0,1){4}}
    \put(11,3){$C_2$}
    \put(12.4,19){\line(0,-1){4}}
    \put(11,18){$C_1$}
    % First R-chain
    \put(2,4){\line(-1,3){1.2}}
    \put(1.6,6.4){$R_{1,{d-1}}$}
    \put(1.4,9){$\vdots$}
    \put(2,11){\line(-1,3){1.2}}
    \put(1.6,12.6){$R_{1,{2}}$}
    \put(0.8,13){\line(1,3){1.2}}
    \put(1.6,14.6){$R_{1,{1}}$}
    \put(3.5,10){$\cdots$}
    % Second R-chain
    \put(7,4){\line(-1,3){1.2}}
    \put(6.6,6.4){$R_{{\gcd(d,r)},{d-1}}$}
    \put(6.6,9){$\vdots$}
    \put(7,11){\line(-1,3){1.2}}
    \put(6.6,12.6){$R_{{\gcd(d,r)},{2}}$}
    \put(5.8,13){\line(1,3){1.2}}
    \put(6.6,14.6){$R_{{\gcd(d,r)},{1}}$}
    % E-chain
    \put(16,4.5){\line(3,2){3}}
    \put(15,6){$E_{d'/d-1}$}
    \put(19.5,5.5){$\cdots$}
    \put(21,4.5){\line(3,2){3}}
    \put(21,6){$E_2$}
    \put(23,6.5){\line(3,-2){3}}
    \put(25,6){$E_1$}
    % D-chain
    \put(16,16.5){\line(3,-2){3}}
    \put(15,14.5){$D_{d'/d-1}$}
    \put(19.5,15){$\cdots$}
    \put(21,16.5){\line(3,-2){3}}
    \put(21,14.5){$D_2$}
    \put(23,14.5){\line(3,2){3}}
    \put(25,14.5){$D_1$}
  \end{picture}
  \caption{Special fiber at $u = \infty$, where $u^{d} = t$, $d$ not
    divisible by $r$. All components are smooth rational curves.  We
    have $R_{i,j}^2=-2$, $C_i^2=-d'/d$, $D_i^2=E_i^2=-2$ for $1\le
    i\le d'/d-1$ and $D_{d'/d}^2=E_{d'/d}^2=-\gcd(r,d)-1$. 
    Multiplicities in the fiber are $m(E_{i,j})=d'/d$, $m(C_i)=1$, and
    $m(D_i)=m(E_i)=i$ for $1\le i\le d'/d$.}
  \label{fig:fiber-infinity-general-d}
\end{figure}

Since the components of the fiber pictured in
Figure~\ref{fig:u-equals-zero} are reduced and the quotient map is
\'etale away from the isolated fixed points, the multiplicites in the
fiber of the components $R_{i,j}$, $D_{d'/d}$, and $E_{d'/d}$ are
$d'/d$ and the self-intersections of the $R_{i,j}$ are all $-2$.  The
components $C_1$ and $C_2$ are reduced in the fiber and have
self-intersection $-d'/d$.  The components $D_i$ and $D_i$ with $1\le
i\le d'/d-1$ have self-intersection $-2$ and multiplicity $i$ in the
fiber.  The components $D_{d'/d}$ and $E_{d'/d}$ have self
intersection $-\gcd(d,r)-1$ and multiplicity $d'/d$ in the fiber.

\begin{remark}\label{rem:rationality}
  The strings of rational curves in the fiber over $u=0$ correspond to
  $r$-th roots of unity, and the components in the string
  corresponding to $\zeta \in \mu_r$ are defined over $\Fp(\zeta)$.
  Similarly, the strings of curves $R_{ij}$ correspond to roots of
  unity of order $\gcd(d,r)$.  On the other hand, over places $u$
  corresponding to a $d$-th root of unity $\zeta' \in \mu_d$,
  components in the fibers are all rational over the field
  $\Fp(\zeta')$.  Also, when $r$ is even, the two points of
  intersection of the curves $F$ and $G$ over a place with $u =\zeta'$
  are defined over $\Fp(\zeta')$.
\end{remark}

\section{Local invariants of the N\'eron model}\label{s:local-invs}
In this section we record the local invariants of the N\'eron model of
$J$, i.e., its component group and connected component of the identity.
%conductor, and the local factor of the $L$-function. 

\subsection{Component groups}
The results of \cite{blr} Chapter 9, Section 6 allow us to read off
the group of components of the special fiber of the N\'eron model of
$J_C$ from our knowledge of the fibers of $\XX\to\P^1_u$.

\begin{prop}\label{prop:component-groups}
  Suppose that $r$ divides $d$ and consider the group of connected
  components of the N\'eron model of $J$ at various places of
  $\Fpbar(u)$.
\begin{enumerate}
\item At $u=0$ and $u=\infty$, the group of connected components is
  isomorphic to $(\Z/rd\Z)\times(\Z/d\Z)^{r-2}$.
\item At places where $u^d=1$, the group of connnected components is
  isomorphic to $\Z/r\Z$.
\end{enumerate} 
\end{prop}

\begin{proof}
  Part (1) is exactly the situation treated as an example in
  \cite{blr}; see 9.6 Corollary 11.  Part (2) is an exercise using
  \cite[9.6, Theorem 1]{blr} and the well-known fact that the
  determinant of the matrix of a root system of type $A_{r-1}$ is $r$.
\end{proof}

\begin{remark}\label{rem:rationality-of-components}
All components of all fibers of $\XX\to\p^1_u$ are
rational over $\Fp(\mu_d)$.  It follows that the group of connected
components of $J_C$ at each place of $\Fp(\mu_d,u)$ is split, i.e.,
$\Gal(\Fpbar/\Fp(\mu_d))$ acts trivially on it.
\end{remark}

\subsection{Connected components}
Recall that the connected component of a smooth, commutative algebraic
group over a perfect field has a filtration whose subquotients are a
unipotent group (itself a repeated extension of copies of the additive
group $\G_a$), a torus, and an abelian variety.  For a place $v$ of
$K_d$, let $a_v$, $m_v$, and $g_v$ be the dimensions of the unipotent
(additive), toral (multiplicative), and abelian variety subquotients
of the connected component of the N\'eron model of $J_C$ at $v$.
Since $C$ has genus $r-1$, there is an equality $a_v+m_v+g_v=r-1$.  At
places of good reduction, $g_v=r-1$.

\begin{prop}\label{prop:connected-component}
Let $K_d = \Fp(\mu_d,u)$.
\begin{enumerate}
\item
If $v$ is the place of $K_d$ over $u=0$, then
$a_v=g_v=0$ and $m_v=r-1$.  

\item
If $v$ is a place of $K_d$ over $u \in \mu_d$
and $r$ is even, then $a_v=(r-2)/2$, $m_v=1$, and $g_v=(r-2)/2$.

\item 
If $v$ is a place of $K_d$ over $u \in \mu_d$
and $r$ is odd, then $a_v=(r-1)/2$, $m_v=0$, and $g_v=(r-1)/2$.
\item
If $v$ is the place of $K_d$ over $u=\infty$, then $a_v=r-\gcd(r,d)$,
$m_v=\gcd(r,d)$, and $g_v=0$.
\end{enumerate}
\end{prop}

\begin{proof}
  It suffices to compute $m_v$ and $g_v$.  We note that
  \cite[Section 9.2]{blr} gives $g_v$ and $m_v$ in terms of the
  special fiber at $v$ of a minimal regular model of $C$, i.e., in
  terms of $\XX$.  Over $u=0$, where $\XX\to\P^1_u$ has semi-stable
  reduction, \cite[9.2, Example 8]{blr} shows that $g_v=0$ and
  $m_v=r-1$, proving part (1).  
  
  In general, both $m_v$ and $g_v$ only depend on the reduced curve
  underlying the fiber \cite[9.2, Proposition 5]{blr}. By \cite[9.2,
  Proposition 10]{blr}, $g_v$ is the sum of the genera of the
  irreducible components of the reduced special fiber.  When $r$ is
  even, the reduced fiber is semi-stable, and again \cite[9.2, Example
  8]{blr} shows that $m_v=1$, proving part (2).  When $r$ is odd,
  applying \cite[9.2, Proposition 10]{blr} (with $C'$ and $C$ the
  reduced special fiber, which is tree-like) shows that $m_v=0$,
  proving part (3).  Over $u=\infty$, all components are rational
  curves, so $g_v=0$.  The reduced fiber is semistable of arithmetic
  genus $\gcd(d,r)-1$, so $m_v=\gcd(d,r)-1$, which completes the proof
  of part (4).
\end{proof}

\section{Domination by a product of curves}\label{s:DPC}

In Section \ref{ss:DPC}, we show that the surface $\YY$ constructed in
Section \ref{ss:YY} is dominated by a product of curves.  In the
following subsections, we upgrade this to a precise isomorphism
between $\YY$ and a quotient of a product of curves by a finite group
in the style of Berger's construction \cite{Berger} and of
\cite{UlmerDPCT}.  This casts some light on the singularities of
$\YY$, and it may prove useful later for constructing explicit points
on $C$ over $K_d$ for values of $d$ other than divisors of $p^f+1$ as
in \cite[Section 10]{Legendre2}.

Throughout, $k$ is a field of characteristic $p\ge0$, and $r$ and $d$
are positive integers prime to $p$ such that $r$ divides $d$.  We
assume also that $k$ contains the $d$-th roots of unity.

\subsection{Domination of $\YY$ by a product of curves}
\label{ss:DPC}
The surface $\YY$ is birational to the affine surface over $k$ given
by $y^r=x^{r-1}(x+1)(x+u^d)$.  Consider the smooth projective curves
over $k$ given by
$$\CC=\CC_{r,d}: z^d=x^r-1\quad\text{and}\quad
\DD=\DD_{r,d}: w^d=y^r-1.$$
Then a simple calculation shows that the assignment
\begin{align}
\phi^*(u)&=zw,\notag\\
\phi^*(x)&=z^d,\label{eq:domination}\\
\phi^*(y)&=xyz^d\notag
\end{align}
defines a dominant rational map $\phi:\CC\times_k\DD\ratto\YY$.  

In the rest of this section, we analyze the geometry of
this map more carefully.

\subsection{Constructing $\CC$ with its $G$ action}
First, we construct a convenient model of the curve $\CC$ over $k$ with
equation $z^d=x^r-1$.  Namely, we glue the smooth $k$-schemes
$$\UU_1=\spec k[x_1,z_1]/\left(z_1^d-x_1^r+1\right)$$
and
$$\UU_2=\spec k[x_2,z_2]/\left(x_2^r(z_2^d+1)-1\right)$$
via the identifications $x_1=x_2^{-1}z_2^{-d/r}$ and $z_1=z_2^{-1}$.
The result is a smooth projective curve that we call $\CC$. 

%Now let $G=\mu_r\times\mu_d$
% where $\mu_n$ stands for the group of $n$-th roots of unity in $k$.  
There is an action of $G=\mu_r\times\mu_d$ on $\CC$
defined by
$$(\zeta_r,\zeta_d)(x_1,z_1)=(\zeta_rx_1,\zeta_dz_1)\quad\text{and}\quad
(\zeta_r,\zeta_d)(x_2,z_2)=(\zeta_r^{-1}\zeta_d^{-d/r}x_2,\zeta_d^{-1}z_2).$$

There are three collections of points on $\CC$ with non-trivial
stabilizers: the $r$ points where $z_1=0$ and $x_1^r=1$, which each
have stabilizer $1\times\mu_d$; the $d$ points where $x_1=0$ and
$z_1^d=-1$, which each have stabilizer $\mu_r\times1$; and the $r$
points where $z_2=0$ and $x_2^r=1$, which each have stabilizer
$$H:=\left\{\left.(\zeta_d^{-di/r},\zeta_d^i)\right|
0 \leq i \leq d-1\right\}.$$ 
We call these fixed points of type $\mu_d$, $\mu_r$, and $H$
respectively.

\subsection{$\CC\times_k\DD$ and its fixed points}
We let $\DD$ be the curve defined just as $\CC$ was, but with
opens $\VV_1$ and $\VV_2$ defined with coordinates $y_1$, $y_2$,
$w_1$, $w_2$ in place of $x_1$,\dots,$z_2$.  We let $G$ act
``anti-diagonally'' on the product surface $\CC\times_k\DD$, i.e.,
by the action on $\CC$ defined above, and by the inverse action on
$\DD$ (so that
$(\zeta_r,\zeta_d)(y_1,w_1)=(\zeta_r^{-1}y_1,\zeta_d^{-1}w_1)$).

If $(c,d)\in\CC\times_k\DD$, then the stabilizer of $(c,d)$ is the
intersection of the stabilizers $c$ and $d$ in $G$.  This yields the
following list of points $(c,d)$ of $\CC\times_k\DD$ with non-trivial
stabilizers:\\
\indent (i) if both $c$ and $d$ are fixed points of
type $\mu_d$, then $\stab(c,d)=\mu_d$;\\
\indent(ii) if both $c$ and $d$ are fixed points of
type $\mu_r$, then $\stab(c,d)=\mu_r$;\\
\indent(iii) if both $c$ and $d$ are fixed points of
type $H$, then $\stab(c,d)=H$;\\
\indent(iv) if $c$ is of type $\mu_d$ and $d$ is of type $H$,
then $\stab(c,d)=(1\times\mu_d)\cap H$, a cyclic group of order $d/r$;\\
\indent(v) if $c$ is of type $H$ and $d$ is of type $\mu_d$,
then $\stab(c,d)=(1\times\mu_d)\cap H$.

We call the fixed points of types (i)-(iii) ``unmixed'' and the fixed
points of types (iv) and (v) ``mixed.''  Note that at an unmixed fixed
point, the action on the tangent space in suitable coordinates is of
the form $(\zeta,\zeta^{-1})$, while at a mixed fixed point, the
action is by scalars $\zeta$.

\subsection{$\widetilde{\CC\times_k\DD}$ with its $G$-action}
We define $\widetilde{\CC\times_k\DD}$ to be the blow up of
$\CC_D\times_k\DD$ at each of its $2r^2$ mixed fixed points.  The
action of $G$ on $\CC_D\times_k\DD$ lifts uniquely to
$\widetilde{\CC\times_k\DD}$.  By the remark above about the action
of $G$ on the tangent space at the mixed fixed points, $G$ fixes the
exceptional divisor of
$\widetilde{\CC\times_k\DD}\to\CC\times_k\DD$ pointwise.  These
are ``divisorial'' fixed points.  The other fixed points of $G$ acting
on $\widetilde{\CC\times_k\DD}$ are the inverse images of the
unmixed fixed points of $\CC\times_k\DD$.

Now consider the quotient $\widetilde{\CC\times_k\DD}/G$.  It is
smooth away from the images of the unmixed fixed points.  Those of
type (i) fall into $r$ orbits and their images in the quotient are
rational double points of type $A_{d-1}$. Those of type (ii) fall into
$d$ orbits and their images in the quotient are rational double points
of type $A_{r-1}$.  Those of type (iii) fall into $r$ orbits and their
images in the quotient are rational double points of type $A_{d-1}$.

\subsection{An isomorphism}
The main goal of this section is the following isomorphism.  Recall
$\YY$, the model of $C/K_d$ defined in Section~\ref{ss:YY}.
\begin{prop}\label{prop:DPCiso}
  There is a unique isomorphism 
$$\rho:\left(\widetilde{\CC\times_k\DD}\right)/G\to\YY$$
such that the composition
$$\CC\times_k\DD\ratto\widetilde{\CC\times_k\DD}\to
\widetilde{\CC\times_k\DD}/G\to\YY$$ is the rational map
$\phi:\CC\times_k\DD\ratto\YY$ of equation~\eqref{eq:domination} in
Section \ref{ss:DPC}.
\end{prop}

Uniqueness is clear.  The key point in the proof of existence is the
following lemma.

\begin{lemma}\label{lemma:q-finite}
  There exists a $G$-equivariant, quasi-finite morphism
  $\psi:\widetilde{\CC\times_k\DD}\to\YY$ \textup{(}with $G$ acting
  trivially on $\YY$\textup{)} inducing the rational map
  $\phi:\CC\times_k\DD\ratto\YY$.
\end{lemma}

\begin{proof}
  The rational map $\phi$ induces a rational map
  $\widetilde{\CC\times_k\DD}\ratto\YY$, and what has to be shown is
  that there is a quasi-finite morphism $\psi$ representing this rational
  map.  To do so, we cover $\widetilde{\CC\times_k\DD}$ with affine
  opens and check that each is mapped by a quasi-finite morphism (the
  unique one compatible with $\phi$) into $\YY$.  The details are
  tedious but straightforward calculations with coordinates.

  It is helpful to have a standard representation of elements of
  the function field $\YY$.  Using the coordinates of $\ZZ_1$, 
  the field $k(\YY)$ is generated by $x$, $y$, and $u$ with relation
  $y^r=x^{r-1}(x+1)(x+u^d)$.  (We drop the subscripts $1$ to avoid
  confusion with coordinates on $\widetilde{\CC\times_k\DD}$ below.)
  Inclusion of the opens $\ZZ_2$, $\ZZ_3$, $\ZZ_{11}$, $\ZZ_{12}$,
  $\ZZ_2'$, $\ZZ_3'$, $\ZZ_{11}'$, and $\ZZ_{12}'$ into $\YY$ induces
  isomorphisms between the function fields of the opens and that of
  $\YY$.  This leads to the following equalities in $k(\YY)$:
\begin{align*}
x_2&=x/y, &y_3&=y/x, &x_{11}&=x/y, &x_{12}&=x,\\
z_2&=1/y, &z_3&=1/x, &y_{11}&=y, &y_{12}&=y/x,\\
u'&=u^{-1},\\
x_2'&=u^{d/r}x/y, &y_3'&=u^{-d/r}y/x, &x_{11}'&=u^{d/r}x/y, &x_{12}'&=u^{-d}x,\\
z_2'&=u^{d+d/r}/y, &z_3'&=u^d/x, &y_{11}'&=u^{-d-d/r}y, &y_{12}'&=u^{-d/r}y/x.
\end{align*}

Similarly, the function field of $\widetilde{\CC\times_k\DD}$ is
generated by $x_1,z_1,y_1,w_1$ with relations $z_1^d=x_1^r-1$ and
$w_1^d=y_1^r-1$.  Inclusion of the opens $\UU_i\times\VV_j$ leads to
the equalities:
\begin{align*}
z_2&=z_1^{-1}, &w_2&=w_1^{-1},\\
x_2&=z_1^{d/r}/x_1, &y_2&=w_1^{d/r}/y_1.
\end{align*}
The blowing up of points in $\UU_1\times\VV_2$ and $\UU_2\times\VV_1$
made to pass from $\CC\times_k\DD$ to $\widetilde{\CC\times_k\DD}$
leads to additional equalities stated below.

For the last key piece of data, we recall the field inclusion $\phi^*:k(\YY)\into
k(\CC\times_k\DD)$.  It yields equalities
$$\phi^*(u)=z_1w_1, \qquad\phi^*(x)=z_1^d, \qquad\phi^*(y)=x_1y_1z_1^d.$$

We now cover $\widetilde{\CC\times_k\DD}$ with affine opens (many
of them, unfortunately) and for each of them check that there is a
quasi-finite morphism from the open to $\YY$ that induces the field
inclusion $\phi^*$.  The compatibility with $\phi^*$ shows that these
morphisms agree on the overlaps, so this yields a global quasi-finite
morphism $\psi:\widetilde{\CC\times_k\DD}\to\YY$.  Since the image of
$\phi^*$ is generated by $z_1w_1$, $z_1^d$, and $x_1y_1$, it lies
inside the $G$-invariant subfield of $k(\widetilde{\CC\times_k\DD})$, and
this shows that $\psi$ collapses the orbits of $G$; this is the
claimed equivariance.

We now make the necessary coordinate calculations, starting with the
open $\UU_1\times\VV_1$.  The formulae above show that
$$\phi^*(u)=z_1w_1, \qquad \phi^*(x_{12})=z_1^d,
 \qquad\phi^*(y_{12})=x_1y_1.$$
This shows that there is a morphism
$\psi_{11}:\UU_1\times\VV_1\to\ZZ_{12}\into\YY$ inducing $\phi$.  To
see that $\psi_{11}$ is quasi-finite, we note that fixing the value of
$x_{12}$ implies at most $d$ choices for $z_1$, which in turn allows
for at most $r$ choices of $x_1$.  Fixing $y_{12}$ then determines
$y_1$ and fixing $u$ determines $w_1$.  This shows that $\psi_{11}$
has fibers of cardinality at most $rd$ (and generically equal to
$rd$).

The rest of the proof proceeds similarly with other affine opens of
$\widetilde{\CC\times_k\DD}$.  Considering $\UU_2\times\VV_2$, 
we note that
$$\phi^*(u')=z_2w_2, \qquad\phi^*(x'_{12})=w_2^d, \qquad
\phi^*(y'_{12})=x_2^{-1}y_2^{-1}.$$
Since $x_2$ and $y_2$ are units on $\UU_2\times\VV_2$, these formulae
define a morphism $\psi_{22}:\UU_2\times\VV_2\to\ZZ_{12}'\into\YY$
that is compatible with $\phi$.  The reader may check that $\psi_{22}$
and the other morphisms defined below are quasi-finite.

These formulae define $\psi$ away from the blow ups of the mixed fixed
points.

Now we focus our attention near a particular mixed fixed point in
$\UU_1\times_k\VV_2$, say $P_{ij}$ given by $x_1=\zeta_r^i$, $z_1=0$,
$y_2=\zeta_r^j$, and $w_2=0$.  Let
$$f=\frac{(x_1^r-1)}{(x_1-\zeta_r^i)}\frac{(y_2^r-1)}{(y_2-\zeta_r^j)}.$$
Inverting $f$ gives an affine open subset of $\UU_1\times_k\VV_2$ on
which the only solution of $z_1=w_2=0$ is $P_{ij}$.  We may cover the
blow up at $P_{ij}$ of this open with two affine opens:
$$T_{ij}^{1}=\spec \frac{k[x_1,s,y_2,w_2][1/f]}{(...)}$$
and
$$T_{ij}^{2}=\spec \frac{k[x_1,z_1,y_2,t][1/f]}{(...)}$$
where $z_1=sw_2$ on $T_{ij}^1$ and $w_2=tz_1$ on $T_{ij}^2$.

Noting that
$$\phi^*(u)=s, \qquad\phi^*(x_{11})=y_2w_2^{d/r}x_1^{-1}, \qquad
\phi^*(y_{11})=x_1s^dw_2^{d-d/r}y_2^{-1},$$ we define a morphism
$\psi_{121ija}$ from the open of $T_{ij}^1$ where $x_1\neq0$ to
$\ZZ_{11}$.  Noting that
$$\phi^*(u)=s, \qquad\phi^*(x_{12})=s^dw_2^d, \qquad
\phi^*(y_{12})=x_1y_2^{-1}w_2^{-d/r},$$ 
we define a morphism $\psi_{121ijb}$ from the open of $T_{ij}^1$ where
$w_2\neq0$ to $\ZZ_{12}$.  Since $w_2$ and $x_1$ do not vanish
simultaneously on $T_{ij}^1$, this defines a morphism
$\psi_{121ij}:T_{ij}^1\to\YY$.

Similarly, noting that 
$$\phi^*(u')=t, \qquad\phi^*(x_{12}')=t^dz_1^d, \qquad
\phi^*(y_{12}')=x_1y_2^{-1}z_1^{-d/r},$$ we define a morphism
$\psi_{122ija}$ from the open of $T_{ij}^2$ where $z_1\neq0$ to
$\ZZ_{12}'$.  Noting that
$$\phi^*(u')=t, \qquad\phi^*(x_{11}')=y_2z_1^{d/r}x_1^{-1}, \qquad
\phi^*(y_{11}')=t^dx_1z_1^{d-d/r}y_2^{-1},$$ 
we define a morphism $\psi_{122ijb}$ from the open of $T_{ij}^2$ where
$x_1\neq0$ to $\ZZ_{11}'$.  Since $z_1$ and $x_1$ do not vanish
simultaneously on $T_{ij}^2$, this defines a morphism
$\psi_{122ij}:T_{ij}^2\to\YY$.

The morphisms $\psi_{121ij}$ and $\psi_{122ij}$ for varying $ij$ patch
together to give a morphism $\psi_{12}$ from the part of
$\widetilde{\CC\times_k\DD}$ lying over $\UU_1\times\VV_2$ to $\YY$.

It remains to consider neighborhoods of the blow ups of the mixed
fixed points in $\UU_2\times_k\VV_1$.  Let $Q_{ij}$ be the point where
$x_2=\zeta_r^i$, $z_2=0$, $y_1=\zeta_r^j$, and $w_1=0$.  Let
$$g=\frac{(x_2^r-1)}{(x_2-\zeta_r^i)}\frac{(y_1^r-1)}{(y_1-\zeta_r^j)}.$$
Inverting $g$ gives an affine open subset of $\UU_2\times_k\VV_1$ on
which the only solution of $z_2=w_1=0$ is $Q_{ij}$.  We may cover the
blow up at $Q_{ij}$ of this open with two affine opens:
$$T_{ij}^{3}=\spec \frac{k[x_2,s,y_1,w_1][1/g]}{(...)}$$
and
$$T_{ij}^{4}=\spec \frac{k[x_2,z_2,y_1,t][1/g]}{(...)}$$
where $z_2=sw_1$ on $T^3_{ij}$ and $w_1=tz_2$ on $T^4_{ij}$.

Noting that 
$$\phi^*(u')=s, \qquad\phi^*(y_3')=y_1x_2^{-1}w_1^{-d/r}, \qquad
\phi^*(z_3')=w_1^{d},$$
we define a morphism $\psi_{213ija}$ from the open of $T_{ij}^3$ where
$w_1\neq0$ to $\ZZ_3'$.  Noting that 
$$\phi^*(u')=s, \qquad\phi^*(x_2')=x_2w_1^{d/r}y_1^{-1}, \qquad
\phi^*(z_2')=x_2w_1^{d+d/r}y_1^{-1},$$
we define a morphism $\psi_{213ijb}$ from the open of $T_{ij}^3$ where
$y_1\neq0$ to $\ZZ_2'$.  Since $w_1$ and $y_1$ do not vanish
simultaneously on $T_{ij}^3$, this defines a morphism
$\psi_{213ij}:T_{ij}^3\to\YY$.

Noting that
$$\phi^*(u)=t, \qquad\phi^*(x_2)=x_2z_2^{d/r}y_1^{-1}, \qquad
\phi^*(z_2)=x_2z_2^{d+d/r}y_1^{-1},$$
we define a morphism $\psi_{214ija}$ from the open of $T_{ij}^4$ where
$y_1\neq0$ to $\ZZ_2$.  Noting that 
$$\phi^*(u)=t, \qquad\phi^*(y_3)=y_1x_2^{-1}z_2^{-d/r}, \qquad
\phi^*(z_3)=z_2^d,$$
we define a morphism $\psi_{214ijb}$ from the open of $T_{ij}^4$ where
$z_2\neq0$ to $\ZZ_3$.  
Since $y_1$ and $z_2$ do not vanish
simultaneously on $T_{ij}^4$, this defines a morphism
$\psi_{214ij}:T_{ij}^4\to\YY$.

The morphisms $\psi_{213ij}$ and $\psi_{214ij}$ for varying $ij$ patch
together to give a morphism $\psi_{21}$ from the
part of $\widetilde{\CC\times_k\DD}$ lying over $\UU_2\times\VV_1$ to
$\YY$. 

Finally, the morphisms $\psi_{11}$, $\psi_{22}$, $\psi_{12}$, and
$\psi_{21}$ patch together to give a quasi-finite morphism
$\psi:\widetilde{\CC\rtimes\DD}\to\YY$ that collapses the orbits of
$G$ and induces $\phi$.  This completes the proof of the lemma.
\end{proof}

\begin{proof}[Proof of Proposition~\ref{prop:DPCiso}]
  By Lemma~\ref{lemma:q-finite}, there is a quasi-finite morphism
  $\psi:\widetilde{\CC\times_k\DD}\to\YY$ of generic degree $rd$.  By
  $G$-equivariance, this factors through the quotient to give a
  quasi-finite morphism $\rho:\widetilde{\CC\times_k\DD}/G\to\YY$.
  Considering degrees shows that $\rho$ is birational.  On the other
  hand, $\rho$ is proper (because $\CC\times_k\DD$ is projective) and
  quasi-finite, so finite.  But $\YY$ is normal and a birational,
  finite morphism to a normal scheme is an isomorphism.  This
  establishes that $\rho$ gives the desired isomorphism.
\end{proof}

\begin{remark}
Examining the morphism above shows that the fixed points of types (i)
and (iii) map to the singular points of $\YY$ in the fibers over $u=0$
and $\infty$.  The fixed points of type (ii) map to the singular
points in the fibers over points $u \in \mu_d$.
%with $u^d=1$. 
 This gives another proof
that the singularities of $\YY$ are rational double points of type
$A_{d-1}$ and $A_{r-1}$.
\end{remark}

%%% Local Variables: 
%%% mode: latex
%%% TeX-master: "EHR"
%%% End: 

%This is chapter 4

\chapter{Heights and the visible subgroup}\label{ch:heights}

In this chapter, we work over $K_d=\Fp(\mu_d,u)$ with $u^d=t$, and we
assume that $d=p^\nu+1$ and $r$ divides $d$.  We have explicit points
$P_{ij}$ defined in Chapter~\ref{ch:C} and the subgroup $V$ of
$J(K_d)$ generated by their classes.  Our first main task is to
compute the N\'eron-Tate canonical height pairing on $V$.  We then
compare this with a group-theoretic pairing defined on $R/I$ where $R$
is the group ring $\Z[\mu_d\times\mu_r]$ and $I$ is the ideal defined
in Section~\ref{s:relations}.  This allows us to show that there is an
$R$-module isomorphism $V\cong R/I$.  We also compute the discriminant
of the height pairing on $V$.

\section{Height pairing}
\label{sec:height-pairing}

In this section, we compute the height pairing on various points of
$J(K_d)$. Recall that we identify $C$ with its image in $J$ by
$P\mapsto [P-Q_\infty]$.  We consider the N\'eron-Tate canonical
height pairing divided by $\log |\Fp(\mu_d)|$, as discussed for
example in \cite[Section 4.3]{CRM}.  This is a $\Q$-valued, non-degenerate,
bilinear pairing that is defined at the beginning of the next subsection.

We compute $\langle P_{ij}, P_{00}\rangle$ for $0\le i\le d-1$
and $0\le j\le r-1$. This determines the pairing, since its
compatibility with the action of $\mu_d\times\mu_r$ implies that
$\langle P_{ij},P_{i'j'}\rangle=\langle P_{i-i',j-j'}, P_{00}\rangle$.

\begin{theorem}\label{thm:height-pairing}
  The height pairing $\langle P_{ij}, P_{00}\rangle$ is given by
  \[
  \langle P_{ij}, P_{00}\rangle = -\frac{d-1}{rd} \cdot
  \begin{cases}
    -(r-1)(d-2) & \text{if } (i,j) = (0,0), \\
    r-2 & \text{if } i\not\equiv 0 \bmod{r}, j = 0, \\
    2r-2 & \text{if } i\neq 0, i \equiv 0 \bmod{r}, j = 0, \\
    d - 2 & \text{if } i = 0, j \neq 0, \\
    r - 2 & \text{if } i \neq 0, j \neq 0, i+j \equiv 0 \bmod{r}, \\
    -2 & \text{if } i \neq 0, j \neq 0, i+j \not\equiv 0 \bmod{r}.
  \end{cases}.
  \]
\end{theorem}

This was already proved in \cite[Section 8]{Legendre} in the case
$r=2$, so to avoid distracting special cases, we assume $r>2$ for the
rest of this section.

\subsection{Basic theory}
\label{sec:basic-theory}

Let $P$ and $P'$ be two points on $C(K_d)$ identified as usual with a
subset of $J(K_d)$ using $Q_\infty$ as a base point; we later set
$P = P_{00}$ and $P' = P_{ij}$. Then the height pairing is defined by
\begin{align*}
\langle P, P'\rangle &= -(P-Q_\infty-D_P)\cdot (P'-Q_\infty)\\
&=-P \cdot P' + P \cdot Q_\infty + P' \cdot Q_\infty 
- Q_\infty^2 - D_P \cdot P', 
\end{align*}
with notation as follows: we identify a point of $C$ with the
corresponding section of the regular proper model $\pi:\XX\to\P^1_u$
and the dot indicates the intersection pairing on $\XX$.  The divisor
$D_P$ is a divisor with $\Q$-coefficients that is supported on
components of fibers of $\pi$ and satisfies
$(P- Q_\infty+D_P)\cdot Z=0$ for every component $Z$ of every fiber of
$\pi$.  We may also insist that $D_P\cdot Q_\infty=0$ in which case
$D_P$ is uniquely determined.  The ``correction term'' $D_P \cdot P'$
is a sum of local terms that depend only on the components of the
fiber that $P$ and $P'$ meet.  The other intersection pairings can be
computed as sums of local terms, except the self-intersection
$Q_\infty\cdot Q_\infty$ and (when $P'=P$) $P\cdot P$.  The latter two
are computable in terms of the degree of a conormal bundle.

\subsection{Auxiliary results}
\label{sec:auxiliary-results}

The following results are useful for computing the various intersection
numbers.

In the first result, we focus attention on the special fiber at a
place $v$ with components $C_0, C_1, \dots, C_n$ and let $A_{ij}$
(with indices $0\le i,j\le n$) be the intersection matrix:
$A_{ij}=C_i\cdot C_j$.  We number the components so that $Q_\infty$
meets $C_0$.  We write $(D_P\cdot P')_v$ for the part of the
intersection multiplicity coming from intersections in the fiber over
$v$.  With these conventions, it is easy to see that if $P$ or $P'$
meets $C_0$, then $(D_P\cdot P')_v=0$.

\begin{lemma}\label{lemma:dp-dot-p-cofactor}
  Suppose that $P$ intersects $C_k$, and $P'$ intersects $C_\ell$, with
  $k,\ell> 0$. Let $B$ be the matrix obtained by deleting the $0$-th
  row and column from $A$. Let $B'$ be the submatrix obtained by
  deleting the $k$-th row and $\ell$-th column from $B$. Finally, let
  $D_P$ denote the fibral divisor satisfying the conditions described
  above.  Then
  \[
  D_P \cdot P' = (-1)^{k+\ell+1} \frac{\det (B')}{\det(B)}
   = (-1)^{k+\ell} \frac{\det (-B')}{\det(-B)}.
  \]
\end{lemma}

\begin{proof}
  Write $D_P = \sum_{h=0}^{n} d_h C_h$ with $d_h\in\Q$.  The
  conditions on $D_P$ imply $D_P \cdot C_h = (Q_\infty - P) \cdot C_h$
  for all $h$. Also $d_0 = 0$ because $D_P \cdot Q_\infty =0$.  The
  intersection number $(D_P\cdot P')_v$ is just $d_\ell$.

  Writing $\vd = (d_1, \dots, d_n)^t$, the conditions on $D_P$ are
  equivalent to
$$  B\vd = -\ve_k,$$
where $\ve_k$ is the $k$-th standard basis vector.  Since $B$ is
non-singular, the unique solution $\vd$ is given by Cramer's rule, and
thus
$$(D_P\cdot P')_v=d_\ell=(-1)^{k+\ell+1}\frac{\det (B')}{\det(B)},$$
as desired.
\end{proof}

\begin{lemma}\label{lemma:am-defn-det}
  Let $A_m$ be the $m \times m$ root matrix of type $A$, in other
  words, the matrix whose entries are given by
  \[
  a_{ij} = \begin{cases}
    -2 & \text{if } i = j, \\
    1 & \text{if } |i - j| = 1,\\
    0 & \text{otherwise.}
  \end{cases}
  \]
  Then $\det (-A_m) =(m+1)$.
\end{lemma}

\begin{proof}
  This is a standard exercise using induction on $m$.  See
  \cite[Page~63]{HumphreysILAART}.
\end{proof}

\begin{lemma}\label{lemma:block-Am-determinant}
  Let $d_1, \dots, d_r$ be positive integers, and let $B = B(d_1,
  \dots, d_r)$ be the block matrix
  \[
  \left[\begin{array}{ccccc}
    A_{d_{1}-1} & & & & e_{d_{1}-1} \\
    & A_{d_{2}-1} & & & e_{d_{2}-1} \\
    & & \ddots & & \vdots \\
    & & & A_{d_{r}-1} & e_{d_{r}-1} \\
    e_{d_{1}-1}^T & e_{d_{2}-1}^T & \cdots & e_{d_{r}-1}^T & -r
  \end{array}\right],
  \]
  where $A_{m}$ is the $m\times m$ root matrix discussed in
  Lemma~\ref{lemma:am-defn-det}, and $e_{m}$ is the column vector of
  length $m$ with a $1$ in the last spot and $0$ everywhere else.
  \textup{(}If $m=0$, then $A_{m}$ and $e_{m}$ are by convention empty
  blocks.\textup{)} Then
  \[
  \det (-B) = \bigg(\prod d_i\bigg)\bigg(\sum \frac{1}{d_i}\bigg).
  \]
\end{lemma}

\begin{proof}
  We compute the determinant of $-B$ by applying Laplace (cofactor)
  expansion, first across the bottom row, and then down the rightmost
  column.  Using Lemma~\ref{lemma:am-defn-det}, the cofactor
  corresponding to the entry $r$ is $r\prod d_i$.  Another application
  of Lemma~\ref{lemma:am-defn-det} shows that the cofactor
  corresponding to removing the bottom row, the row containing the 1
  in $e_{d_i-1}$, the rightmost column, and the column containing the
  1 of $e^T_{d_j-1}$ is $-(d_i-1)\prod_{j\neq i}d_j$ if $i=j$ and zero
  otherwise.  This shows that the determinant of $-B$ is
$$r\prod_{i=1}^rd_i-\sum_{i=1}^r(d_i-1)\prod_{j\neq i}d_j,$$
which is equal to $(\prod d_i)(\sum 1/d_i)$ as desired.
\end{proof}

\begin{lemma}\label{lemma:B'-det}
  Let $d\ge2$ and $r\ge2$ be integers and let $B=B(d,d,\dots,d)$ \textup{(}with
  $d$ repeated $r$ times, using the notation of
  Lemma~\ref{lemma:block-Am-determinant}\textup{)}.
\begin{enumerate}
\item Let $B'$ be the matrix obtained from
  $B$ by deleting the first row and the $d$-th column.  Then
  $\det(-B')=(-1)^{d-1}d^{r-2}$.
\item Let $B''$ be the matrix obtained from
  $B$ by deleting row $d-1$ and column $2(d-1)$.  Then
  $\det(-B'')=(-1)^{d-1}(d-1)^2d^{r-2}$.
\end{enumerate}
\end{lemma}

\begin{proof}
(1)  For $1\le n\le d-2$, let $R_n$ be the matrix obtained from $A_{d-1}$ by
  deleting the first $n$ rows and the first $n-1$ columns.  Similarly,
  let $S_n$ be the matrix obtained from $A_{d-1}$ by
  deleting the first $n$ columns and the first $n-1$ rows.  The matrix
  $B'$ under discussion is thus
$$B'=B'_1=\left[\begin{array}{cccccc}
R_1 & & & & & e_{d-2} \\
& S_1 & & & & e_{d-1} \\
& & A_{d-1} & & & e_{d-1} \\
& & & \ddots & & \vdots \\
& & & & A_{d-1} & e_{d-1} \\
e_{d-1}^T & e_{d-2}^T & e_{d-1}^T & \cdots & e_{d-1}^T & -r
\end{array}\right]$$
where there are $r-2$ blocks of $A_{d-1}$.  Note that if $d \geq 3$,
the upper left entry of $B'_1$ is 1 and the rest of the first column
is zero.  Expanding in cofactors down the first column shows that
$\det(B'_1)=\det(B'_2)$ where
$$B'_2=\left[\begin{array}{cccccc}
R_2 & & & & & e_{d-3} \\
& S_1 & & & & e_{d-1} \\
& & A_{d-1} & & & e_{d-1} \\
& & & \ddots & & \vdots \\
& & & & A_{d-1} & e_{d-1} \\
e_{d-2}^T & e_{d-2}^T & e_{d-1}^T & \cdots & e_{d-1}^T & -r
\end{array}\right].$$
Continuing in similar fashion for another $d-3$ steps shows that
$\det(B_1')=\det(B_{d-1}')$ where
$$B'_{d-1}=\left[\begin{array}{cccccc}
0 & S_1 & & & & e_{d-1} \\
& & A_{d-1} & & & e_{d-1} \\
& & & \ddots & & \vdots \\
& & & & A_{d-1} & e_{d-1} \\
1 & e_{d-2}^T & e_{d-1}^T & \cdots & e_{d-1}^T & -r
\end{array}\right].$$
Now we expand across rows of the $S_1$, finding that
$\det(B_1')=-\det(B_{d}')$ where 
$$B'_{d}=\left[\begin{array}{cccccc}
0 & S_2 & & & & e_{d-2} \\
& & A_{d-1} & & & e_{d-1} \\
& & & \ddots & & \vdots \\
& & & & A_{d-1} & e_{d-1} \\
1 & e_{d-3}^T & e_{d-1}^T & \cdots & e_{d-1}^T & -r
\end{array}\right].$$
Continuing in similar fashion for another $d-3$ steps shows that
$\det(B_1')=(-1)^d\det(B_{2d-3}')$ where
$$B'_{2d-3}=
\left[\begin{array}{cccccc}
    0 & & & & 1 \\
    & A_{d-1} & & & e_{d-1} \\
    & & \ddots & & \vdots \\
    & & & A_{d-1} & e_{d-1} \\
    1 & e_{d-1}^T & \cdots & e_{d-1}^T & -r
\end{array}\right].$$
Expanding in cofactors across the top row and then the leftmost column
shows that $\det(B'_{2d-3})=(-1)^{r(d-1)+1}d^{r-2}$.  Thus
$$\det(-B'_1)=(-1)^{r(d-1)}\det(B'_{1})=
(-1)^{r(d-1)+d}\det(B_{2d-3}')=
(-1)^{d-1}d^{r-2},$$ 
as desired.

(2) The matrix $B''$ has the form
$$\left[\begin{array}{ccccccc}
  A_{d-2} & e_{d-2} & & & & e_{d-2} \\
  & & A_{d-2} & & & & \\
  & & e_{d-2}^T & & & 1 \\
  & & & A_{d-1} & & e_{d-1} \\
  & & & & \ddots & \vdots \\
  & 1 & e_{d-2}^T & e_{d-1}^T & \cdots & -r
\end{array}\right].$$
To compute the determinant, we expand in cofactors along
$(d-1)^{\textup{st}}$ column.  There are two cofactors; the first one,
corresponding the the last entry of $e_{d-2}$, has the matrix $R'_{2}$
in the upper left corner, where $R'_n$ is defined like $R_n$ except we
delete the \emph{last} $n$ rows and the last $n-1$ columns. This
cofactor is zero since the corresponding matrix visibly has
less-than-maximal rank. The other term comes from the 1 in the last
row, yielding
$$\det (-B'') = (-1)^{r(d-1)+1}\det \left(-\left[
  \begin{array}{cccccc}
  A_{d-2} & & & & & \\
  & A_{d-2} & & & & \\
  & e_{d-2}^T & & & 1 \\
  & & A_{d-1} & & e_{d-1} \\
  & & & \ddots & \vdots \\
  & & & A_{d-1} & e_{d-1}
  \end{array}\right]\right).$$
Expanding in cofactors along $(2d-3)^{\textup{rd}}$ row and using similar
reasoning leads to
$$\det (-B'') = (-1)^{(d-1)}\det \left(-\left[
  \begin{array}{cccccc}
  A_{d-2} & & & & & \\
  & A_{d-2} & & & & \\
  & & A_{d-1} & & \\
  & & & \ddots &  \\
  & & & & A_{d-1}
  \end{array}\right]\right).$$
The claim now follows by Lemma~\ref{lemma:am-defn-det}.
\end{proof}

The next result is useful for finding the component that a section
meets at a bad fiber.  To set it up, let $k$ be a field, let
$R=k[u]_{(u)}$ (localization of the polynomial ring at $u=0$), and let
$\ZZ=\spec R[\alpha,\beta]/(\alpha\beta-u^n)$ where $n$ is prime to
the characteristic of $k$. Suppose that $\YY\to\spec R$ is a proper
relative curve and that $P$ is a point in the special fiber of $\YY$
near which $\YY$ is \'etale locally isomorphic to $\ZZ$.  More
precisely, we assume that there is a Zariski open neighborhood $U$ of
$P$ in $\YY$ and an \'etale $R$-morphism $\phi:U\to\ZZ$ sending $P$ to
the origin ($u=\alpha=\beta=0$) in $\ZZ$.  Let $f=\phi^*(\alpha)$ and
$g=\phi^*(\beta)$.  Let $\pi:\XX\to\YY$ be the minimal regular model
of $\YY$ and suppose that $s:\spec R\to\XX$ is a section such that
$\pi\compose s$ passes through $P$. Let $Q$ be the closed point of
$\spec R$.

\begin{lemma}\label{lemma:fg-equals-u-blow-up}
With the notation above:
  \begin{enumerate}
      \item The fiber of $\pi$ over $P$ consists of a chain of $n-1$
        rational curves $Z_1,\dots,Z_{n-1}$ that can be numbered so that
        $Z_i$ meets $Z_j$ if and only if $|i-j|=1$ and so that $E_1$
        meets the strict transform of $f=0$ in $\XX$.
      \item $s(Q)$ meets $E_i$ if and only if $g\compose s\in R$ has
        $\ord_u(g\compose s)=i$.
      \item $g/u^i$ restricted to $E_i$ induces an isomorphism
        $E_i\cong\P^1$.  In particular, two sections $s$ and $s'$
        meeting $E_i$ intersect there if and only if $g\compose
        s\equiv g\compose s'\mod{u^{i+1}}$.
  \end{enumerate}
\end{lemma}

\begin{proof}
  Blowing up $\ZZ$ at the origin $\lfloor n/2\rfloor$ times yields a
  minimal resolution $\tilde\ZZ\to\ZZ$ with exceptional divisor a
  chain of rational curves $Z_1\cup\cdots\cup Z_{n-1}$ as in the
  statement, with $E_1$ meeting the strict transform of
  $\alpha=0$. The fiber product of $\tilde\ZZ\to\ZZ$ with $U\to\ZZ$ is
  isomorphic to a neighborhood of the inverse image of $P$ in $\XX$,
  and this gives the first claim.  Moreover, the other two claims are
  reduced to the analogous statements on $\ZZ$, and these are easily
  checked by considering the explicit blow-ups used to pass from $\ZZ$
  to $\tilde\ZZ$.
\end{proof}

\subsection{First global intersection numbers}\label{ss:first-int-numbers}
We abuse notation somewhat and use $P_{ij}$ and $Q_{\infty}$ to
denote the sections of $\XX\to\P^1$ or $\YY\to\P^1$ induced by the
$K_d$-rational points with those names, but we try to make clear
the context in each such case.

The section $P_{ij}$ of $\YY\to\P^1$ lies in the union of the opens
$\ZZ_{12}$ and $\ZZ'_{12}$ discussed in Section~\ref{ss:YY}, and it
has coordinates:
$$(x_{12},y_{12})=\left(\zeta_d^iu,\zeta_r^j(\zeta_d^iu+1)^{d/r}\right)
\quad\text{and}\quad
(x'_{12},y'_{12})=\left(\zeta_d^iu^{\prime d-1},\zeta_r^j(\zeta_d^i+u')^{d/r}\right).$$
Since the section $Q_\infty$ does not meet these opens, it follows
that the global intersection number $P_{ij}\cdot Q_\infty=0$ for all
$i$ and $j$.

Examining the coordinates above, it is clear that if $i\neq0$, then
$P_{ij}$ and $P_{00}$ do not meet in $\YY$ (and \emph{a fortiori} in
$\XX$) except possibly over $u=0$ or $u=\infty$.  Also, if $j\neq0$,
then $P_{0j}$ and $P_{00}$ visibly do not meet except possibly over
$u=-1$.  Thus to finish the height computation it suffices to compute
local intersection numbers for $u\in\{0,\mu_d,\infty\}$, the
``correction factors'' $D_{P_{00}}\cdot P_{ij}$ at those same places,
and the self-intersections $P_{00}^2$ and $Q_\infty^2$.

\subsection{Pairings at $u = 0$}\label{sec:pairing-at-u}
We now consider the configuration of $Q_\infty$ and the $P_{ij}$ with
respect to the components of the special fiber of $\XX\to\P^1$ over
$u=0$, which is pictured in Figure~\ref{fig:u-equals-zero} in
Chapter~\ref{ch:models}. 

First, we note that the component labeled $C_{0}$ is the strict
transform of the component $u=y_{12}^r-x_{12}-1=0$ in the chart
$\ZZ_{12}$ and also of the component $u=z_2-x_2^r(x_2+z_2)=0$ in the
chart $\ZZ_2$.  The point $Q_\infty$ extends to the section
$x_2=z_2=0$ in the chart $\ZZ_2$, so it lies on the component $C_0$.

Next, we note that the section $P_{ij}$ of $\YY\to\P^1$ specializes to
the point $x_{12}=0$, $y_{12}=\zeta_r^i$, so the corresponding section
of $\XX\to\P^1$ must meet one of the components $C_{j(d-1)+k}$ with
$1\le k\le d-1$.

To find the component that $P_{ij}$ meets, we use
Lemma~\ref{lemma:fg-equals-u-blow-up}.  To that end, let
$f=y_{12}^r-x_{12}-1$ and $g=x_{12}/(x_{12}+1)$.  In a neighborhood of
the points $u=x_{12}=y_{12}^r-1=0$, the equation defining $\ZZ_{12}$
is $fg=u^d$.  We claim that near each of these points, $f$ and $g$
define an \'etale morphism to the scheme $\ZZ$ defined just before
Lemma~\ref{lemma:fg-equals-u-blow-up}.  (Here by ``near'' we mean in a
Zariski open neighborhood $U$ of the point of interest in the fiber
product of $\YY\to\P^1$ and $\spec R\to\P^1$.)  The claim follows
easily from the Jacobian criterion, as discussed for example in
\cite[Definition~3, Page~36]{blr}.  Indeed, we define
$$\phi:U\to\A^2_{\ZZ}=\spec
R[\alpha,\beta,\gamma,\delta]/(\alpha\beta-u^d)$$ by
$\phi^*(\alpha)=y_{12}^r-x_{12}-1$, $\phi^*(\beta)=x_{12}/(x_{12}+1)$,
$\phi^*(\gamma)=x_{12}$, and $\phi^*(\delta)=y_{12}$.  Then in the
notation of \cite{blr}, the image of $\phi$ is cut out by
$g_1=\alpha-\delta^r-\gamma-1$ and $g_2=(\gamma+1)\beta-\gamma$, and
they have independent relative differentials of $\A^2_{\ZZ}/\ZZ$ wherever
$\beta\neq1$ and $\delta\neq0$, which is satisfied in a neighborhood
of the points of interest.

The upshot is that the hypotheses of
Lemma~\ref{lemma:fg-equals-u-blow-up} are satisfied.  Since
$g=x_{12}/(x_{12}+1)=\zeta_d^iu/(\zeta_d^i+1)$ has $\ord_u(g)=1$, 
it follows that $P_{ij}$ lands on component $C_{j(d-1)+1}$.  Note also that
the value of $g/u$ on $P_{ij}$ at $u=0$ is $\zeta_d^i$, so the
$P_{ij}$ all land on distinct points.  In other words, their local
intersection multiplicity is zero.

To finish the analysis, we need to compute the local correction factor
$(D_{P_{00}}\cdot P_{ij})_{u=0}$.
Recall that the matrix $B$ constructed in
Lemma~\ref{lemma:dp-dot-p-cofactor} is obtained by deleting the first
row and column from the intersection matrix for the special
fiber. Using the ordering given above for the components, then $B =
B(d, d, \dots, d)$ as in Lemma~\ref{lemma:block-Am-determinant}. There
are $r$ copies of $A_{d-1}$ in $B$, so that $B$ is an $m \times m$
matrix where $m = r(d-1) + 1$.

First suppose that $j = 0$. Let $B'$ be the matrix obtained by
deleting the first row and column from $B$; a straightforward
calculation shows that $B' = B(d-1, d, \dots,d)$ as in
Lemma~\ref{lemma:block-Am-determinant}. Therefore $\det (-B')$ is
equal to
$$(d-1)d^{r-1}\left(\frac{1}{d-1} + \frac{r-1}{d}\right) = d^{r-2}(rd-r+1).$$
Since $\det(-B)=rd^{r-1}$, applying
Lemma~\ref{lemma:dp-dot-p-cofactor} yields that
$$\left(D_{P_{00}} \cdot P_{ij}\right)_{u=0} =\det(-B')/\det(-B)
= \frac{d-1}{d} + \frac{1}{rd}.$$

Next we consider the case $j \neq 0$.  By symmetry, it suffices to
treat the case $j=1$.  Letting $B'$ be the matrix obtained by deleting
the first row and the $d$-th column of $B$,
Lemma~\ref{lemma:B'-det}(1) implies that
$\det(-B')=(-1)^{d-1}d^{r-2}$.  Applying
Lemma~\ref{lemma:dp-dot-p-cofactor} yields that
$$\left(D_{P_{00}} \cdot P_{ij}\right)_{u=0} =(-1)^{1+d}\det(-B')/\det(-B)
=\frac{1}{rd}.$$

Summarizing this subsection:

\begin{prop}\label{prop:dp-u-zero}
  The local intersection numbers $(P_{00}\cdot P_{ij})_{u=0}$ are zero
  for all $(i,j)\neq(0,0)$.  The local correction factor at $u=0$ is
  given by:
  $$\left(D_{P_{00}} \cdot P_{ij}\right)_{u=0} =
  \begin{cases}
    \frac{d-1}{d} + \frac{1}{rd} & \text{if } j = 0, \\
    \frac{1}{rd} & \text{if } j \neq 0.
  \end{cases}$$
\end{prop}

\subsection{Pairings at $u=\infty$}
\label{sec:pairing-at-u=infty}
The argument here is very similar to that at $u=0$.  In particular,
the configuration of components is again given by
Figure~\ref{fig:u-equals-zero} in Chapter~\ref{ch:models} and the
section of $\XX\to\P^1$ corresponding to $Q_\infty$ meets the
component $C_0$.  The section $P_{ij}$ of $\YY\to\P^1$ specializes to
the point $x'_{12}=0$, $y'_{12}=\zeta_r^{i+j}$ so the corresponding
section of $\XX\to\P^1$ meets component $C_{(i+j)(d-1)+k}$ for some
$k$ with $1\le k\le d-1$.  (Here and below, we read $i+j$ modulo $r$
and take a representative in $\{0,\dots,r-1\}$.)

Applying Lemma~\ref{lemma:fg-equals-u-blow-up} with $f=y_{12}^{\prime
  r}-x'_{12}-1$ and $g=x'_{12}/(x'_{12}+1)$, we find that $P_{ij}$
meets component $C_{(i+j+1)(d-1)}$ and there are no intersections
among the distinct $P_{ij}$.

It remains to compute the correction factor $D_{P_{00}}\cdot
P_{ij}$ using the lemmas in
Section~\ref{sec:auxiliary-results}.  If $i+j\equiv0\mod r$, then
the matrix obtained by deleting row and column $d-1$ from $B$ has the
form:
$$B'=\left[\begin{array}{cc}
             A_{d-2} & 0 \\
             0 & B(0,d,\dots,d)
           \end{array}\right].$$ 

Applying Lemmas~\ref{lemma:am-defn-det} and
\ref{lemma:block-Am-determinant} shows that $\det(-B')$ is
$$(d-1) d^{r-1}\left(1 + \frac{1}{d} + \cdots + \frac{1}{d}\right)
  = (d-1)d^{r-1}\frac{d + r - 1}{d}.$$
Thus the local correction factor in this case is
$$\left(D_{P_{00}}\cdot P_{ij}\right)_{u=\infty}=
\det(-B')/\det(-B)=\frac{(d-1)(r+d-1)}{rd}.$$

If $i+j\not\equiv0\mod r$, by symmetry we may assume that
$i+j\equiv1\mod r$.  In this case, the matrix obtained by deleting row
$d-1$ and column $2(d-1)$ is the matrix $B''$ of
Lemma~\ref{lemma:B'-det}(2), which has
$\det(-B'')=(-1)^{d-1}(d-1)^2d^{r-2}$.  Thus
$$\left(D_{P_{00}}\cdot P_{ij}\right)_{u=\infty}=
(-1)^{d-1}\det(-B'')/\det(-B)=\frac{(d-1)^2}{rd}.$$

Summarizing this section:

\begin{prop}\label{prop:dp-u-infty}
  The local intersection numbers $(P_{00}\cdot P_{ij})_{u=\infty}$ are
  zero for all $(i,j)\neq(0,0)$.  The local correction factor at
  $u=\infty$ is given by
$$\left(D_{P_{00}} \cdot P_{ij}\right)_{u=\infty} =
   \begin{cases}
   \frac{(d-1)(r+d-1)}{rd} & \text{if } i+j \equiv 0 \bmod{r}, \\
   \frac{(d-1)^2}{rd} & \text{if } i+j \not\equiv 0 \bmod{r}.
\end{cases}$$
\end{prop}

\subsection{Pairings at $u = \zeta_d^k$}
\label{sec:pairing-at-u-1}
We now focus attention on the fiber of $\XX\to\P^1$ over
$u=\zeta_d^k$.  The configuration of components is given in
Figures~\ref{fig:superelliptic-dual-graph-odd} ($r$ odd) and
\ref{fig:superelliptic-dual-graph-even} ($r$ even) of
Chapter~\ref{ch:models}.  The component $F$ there is the strict
transform of the fiber of $\YY\to\P^1$ at $u=\zeta_d^k$, and the
section $Q_\infty$ of $\XX\to\P^1$ meets this component.

In the coordinates of the chart $\ZZ_{12}$, where
$P_{ij}=(\zeta_d^iu,\zeta_r^j(\zeta_d^iu+1)^{d/r})$, the section of
$\YY\to\P^1$ corresponding to $P_{ij}$ passes through the singular
point in the fiber if and only if $\zeta_d^{i+k}=-1$, or equivalently,
if and only if $d$ is even and $i+k\equiv d/2\mod d$.  In this case,
the section of $\XX\to\P^1$ corresponding to $P_{ij}$ meets one of the
components $D_i$, $E_i$, or $G$.  Since $P_{ij}$ is a section, it has
to meet a component of multiplicity one in the fiber, i.e., either
$D_1$ or $E_1$.  Which one it meets is a matter of labeling
conventions, but we need to show that all $P_{ij}$ with $\zeta_d^{i+k}
= -1$ land on \emph{the same} component, so we must work out a few
more details.

Dropping subscripts, consider the chart $\ZZ=\ZZ_{12}$ defined by the
equation $xy^r=(x+1)(x+u^d)$.  Changing coordinates $x=x'-1$ and
$u=u'+\zeta_d^k$, the equation is $(x'-1)y^r=x'(x'+u'v)$ where $v$ is
a unit in the local ring at $x'=y=u'=0$.  The section $P_{ij}$ of
$\YY\to\P^1$ has coordinates
$$\left(x'(P),y(P)\right)
=\left(\zeta_d^{i+k}+1
  +\zeta_d^iu',\zeta_r^j(\zeta_d^{i+k}+1+\zeta_d^iu')^{d/r}\right)
=\left(\zeta_d^iu',\zeta_r^{i+j}u^{\prime d/r}\right)$$ 
where the second equality uses that $\zeta_d^{i+k}=-1$.

Now we blow up the origin in $x',y,u'$ space and consider the chart
with coordinates $x'',y',u'$ where $x'=u'x''$ and $y=u'y' $.  The
strict transform of $\ZZ$ is defined by
\begin{equation}
  (u'x''-1)u^{\prime r-2}y^{\prime r}=x''(x''+v')\label{eq:blowup-zz}
\end{equation}
where $v'$ is a unit near the origin.  It is possible to check that
$v'$ reduces to $d\zeta_d^{k(d-1)}=\zeta_d^{-k}$ modulo the maximal
ideal.  The exceptional divisor is $u'=x''(x''+\zeta_d^{-k})=0$, with
two components that we label $D_1$ ($x''=0$) and $E_1$
($x''+\zeta_d^{-k}=0$).  Note that the original fiber of $\YY\to\P^1$
does not meet the chart under consideration.  The section $P_{ij}$ has
coordinates $x''(P_{ij})=\zeta_d^i$,
$y'(P_{ij})=\zeta_r^{i+j}u^{\prime d/r-1}$ and thus meets the
component $E_1$. Moreover, when $\zeta_d^k=-1$, then $P_{00}$ and
$P_{0j}$ intersect on $E_1$ with multiplicity $d/r-1$.

Recapping the geometry, the section $P_{ij}$ of $\XX\to\P^1$ meets the
component $E_1$ over $u=\zeta_d^k$ if and only if $\zeta_d^{i+k}=-1$,
otherwise it meets $F$.  The sections $P_{00}$ and $P_{ij}$
($(i,j)\neq(0,0)$) meet over $u=\zeta_d^k$ if and only if
$\zeta_d^k=-1$, $i=0$, and $d/r>1$, in which case their intersection
multiplicity is $d/r-1$.

It remains to compute the correction factor $D_{P_{00}}\cdot P_{ij}$.
It is zero except when $\zeta_d^k=-1$ and $i=0$, in which case both
$P_{00}$ and $P_{ij}$ meet component $E_1$.  The intersection matrix
of the fiber omitting the component $F$ is $B=A_{r-1}$, and $B'$ the
matrix obtained by deleting the last row and column of $B$ is
$A_{r-2}$.  Lemma~\ref{lemma:dp-dot-p-cofactor} implies that
$$D_{P_{00}}\cdot P_{0j}=\det(-B')/\det(-B)=(r-1)/r.$$

Summarizing this section:

\begin{prop}\label{prop:dp-u-zeta-k}
  The local intersection numbers at $u=\zeta_d^k$ are given by
  $$\left(P_{00} \cdot P_{ij}\right)_{u=\zeta_d^k}=
  \begin{cases}
    d/r-1&\text{if $i=0$, $j\neq0$, and $\zeta_d^k=-1$,}\\
    0&\text{if $i\neq0$ or $\zeta_d^k\neq-1$.}
  \end{cases}$$ 
  The local correction factor at $u=\zeta_d^k$ is given by
$$\left(D_{P_{00}} \cdot P_{ij}\right)_{u=\zeta_d^k} =
\begin{cases}
  (r-1)/r& \text{if $\zeta_d^k=-1$ and $i=0$,}\\
  0& \text{if $\zeta_d^k\neq-1$ or $i\neq0$.}
\end{cases}$$
\end{prop}

\begin{remark}
  Recall that $d=p^\nu+1$.  If $p$ is odd, then $d$ is even, and there
  is exactly one value of $k$ modulo $d$, namely $d/2$, such that
  $\zeta_d^k=-1$.  If $p=2$, then $-1=1$ and $\zeta_d^k=-1$ again for
  exactly one value of $k$ modulo $d$, namely $k=0$.  Thus for a fixed
  $P_{ij}$ with $i=0$, there is exactly one value of $k$ such that the
  intersection number is non-zero and the correction factor is
  non-zero at $u=\zeta_d^k$.
\end{remark}

\subsection{Self-intersections}
\label{sec:p_002-q_infty2}

We now compute the self-intersections of $P_{00}$ and $Q_\infty$,
proceeding as follows. For a point $P\in C(K_d)$, we continue to
identify $P$ with the corresponding section of $\XX\to\P^1$.  Let
$\II$ be the ideal sheaf of $P$, considered as a divisor on
$\XX$. Recall that the conormal sheaf to $P$ is the sheaf $\II/\II^2$
on $P$. By~\cite[V, 1.4.1]{hartshorne}, $P^2 = - \deg \II/\II^2$. Thus
the method is to compute the divisor of a global section of this
sheaf.

It is convenient to rephrase this in terms of differentials.  Because
$P$ is both a smooth subvariety of $\XX$ and a section of
$\XX\to\P^1$, the exact sequence
$$0\to\II/\II^2\to\left(\Omega^1_{\XX}\right)_{|P}\to\Omega^1_{P}\to0$$
splits canonically, and we obtain an identification
$\II/\II^2\cong(\Omega^1_{\XX/\P^1})_{|P}$.  In other words,
$\II/\II^2$ is identified with the sheaf of relative differentials
restricted to $P$.  For typographical convenience, we write $\omega_P$
for $(\Omega^1_{\XX/\P^1})_{|P}$.

Consider $Q_\infty$.  As a section of $\YY\to\P^1$, it is given by
$x_2$ in the chart $\ZZ_2$ and $x_2'$ in $\ZZ_2'$; these are related
by $x_2'=u^{d/r}x_2$ on the overlap.  It follows that $dx_2$ defines a
global section of $\omega_{Q_\infty}$ that generates it away from
$u=\infty$ and has a zero of order $d/r$ there.  We conclude
that $Q_\infty^2=-d/r$ in $\YY$.  Since $\XX\to\YY$ is an isomorphism
in a neighborhood of $Q_\infty$, the same equality holds in $\XX$.

Now consider $P_{00}$.  In the chart $\ZZ_{12}$, which is defined
(dropping subscripts) by $xy^r-(x+1)(x+u^d)$, there is an equality
$$0=(y^r-2x-u^d-1)dx+rxy^{r-1}dy$$
in $\Omega^1_{\ZZ_2/\P^1}$. It follows that $dx$ generates
$\Omega^1_{\ZZ_2/\P^1}$ wherever $xy\neq0$.  In particular, restricted
to $P_{00}$, it generates $\omega_{P_{00}}$ away from $u=0$, $u=-1$,
and $u=\infty$.  We extend $dx$ to a global section $s$ of
$\omega_{P_{00}}$ and compute its divisor.

Near $u=0$, passing from $\YY$ to $\XX$ requires several blow ups.  We
have already seen that after the first blow up, the strict transform
of $P_{00}$ lies in the smooth locus, so the rest of the blow ups are
irrelevant for the current calculation.  We make the first blow up
more explicit.  First, let $y=y'+1$, so the equation defining
$\ZZ_{12}$ is
$$x((y'+1)^r-1)-x^2-u^d-xu^d=0.$$
Blowing up the origin, the equation becomes
$$(x'(ry''+\cdots+u^{r-1}y''^{ r})-x^{\prime
  2}-u^{d-2}-x'u^{d-1}$$ 
and the section $P_{00}$ becomes $x'=1$, $y'=((u+1)^{d/r}-1)/u$.
Differentiating the equation, one checks that $dx'$ generates $\omega$
near $u=0$, and since $x=ux'$, it follows that $dx$ extends to a
section with a simple zero at $u=0$.

Near $u=-1$, several blow ups are required to pass from $\YY$ to
$\XX$.  After the first blow up, $P_{00}$ lies in the smooth locus and
the later blow ups are irrelevant for the current calculation.  The
relevant chart after the first blow up was given
in~\eqref{eq:blowup-zz}; for reference, we copy it here:
$$(u'x''-1)u^{\prime r-2}y^{\prime r}=x''(x''+v').$$
Differentiating this relation, one finds that the coefficient of
$dx''$ is non-zero near $u'=0,x''=1$, and this shows that shows that
$dy'$ generates $\omega_{P_{00}}$ there.  Considering the valuation of
the coefficient of $dy'$ shows that $dx''$ vanishes to order
$d-d/r-1$. Since $dx=dx'=u'dx''$, it follows that $dx$ vanishes to
order $d-d/r$.

Finally, near $u=\infty$, a calculation very similar to that near
$u=0$ shows that $dx$ has a simple pole there.  In all,
the divisor of $dx$ has degree $d-d/r$ and so $P_{00}^2=d/r-d$.

Summarizing this subsection:

\begin{prop}\label{prop:Q-infty-self-intersection}
  The self-intersections of $P_{00}$ and $Q_\infty$ are
$$P_{00}^2 = -d + \frac{d}{r} 
\quad\text{and}\quad 
Q_\infty^2 = -\frac{d}{r}.$$
\end{prop}

\subsection{Proof of Theorem~\ref{thm:height-pairing}}
\label{sec:proof-theor-refthm:h}

We now put all the calculations together. The local contributions to
$D_{P_{00}} \cdot P_{ij}$ were computed in
Propositions~\ref{prop:dp-u-zero}, \ref{prop:dp-u-infty}, and
\ref{prop:dp-u-zeta-k}; the results of these propositions are
summarized in Table~\ref{tab:local-contrib-dp-dot-p}. In that table,
all congruences are mod $r$.  In the third column, we sum all local
contributions over the places $u=\zeta_d^k$ with $k=0,\dots,d-1$.

\renewcommand\arraystretch{2}
\begin{table}[htb]
  \centering
  \[
  \begin{array}[c]{lccc}
    (i,j) & u = 0 & u = \infty & u^d = 1 \\ \toprule
    (0,0) & \dfrac{rd-r+1}{rd} & \dfrac{(d-1)(r+d-1)}{rd} & \dfrac{r-1}{r} \\
    i \not\equiv 0, j = 0 & \dfrac{rd - r + 1}{rd} & \dfrac{(d-1)^2}{rd} & 0 \\
    i \neq 0, i \equiv 0, j = 0 & \dfrac{rd - r + 1}{rd} & \dfrac{(d-1)(r+d-1)}{rd} & 0 \\
    i = 0, j \neq 0 & \dfrac{1}{rd} & \dfrac{(d-1)^2}{rd} & \dfrac{r-1}{r} \\
    i \neq 0, j \neq 0, i + j \equiv 0 & \dfrac{1}{rd} & \dfrac{(d-1)(r+d-1)}{rd} & 0 \\
    i \neq 0, j \neq 0, i + j \not\equiv 0 & \dfrac{1}{rd} & \dfrac{(d-1)^2}{rd} & 0 \\
    \bottomrule
  \end{array}
  \]
  \caption{Local contributions to $D_{P_{00}} \cdot P_{ij}$}
  \label{tab:local-contrib-dp-dot-p}
\end{table}
\renewcommand\arraystretch{1}

By summing the local contributions to the intersection numbers
$P_{00}\cdot P_{ij}$ given in Propositions~\ref{prop:dp-u-zero},
\ref{prop:dp-u-infty}, and \ref{prop:dp-u-zeta-k}, noting that
$P_{ij}\cdot Q_\infty=0$ for all $i$ and $j$ as in
Section~\ref{ss:first-int-numbers}, and recalling the
self-intersection numbers in the preceding subsection, we deduce that:
\begin{align*}
P_{ij}\cdot P_{00}&=
\begin{cases}
-d+\frac dr&\text{if $(i,j)=(0,0)$,}\\
\frac dr-1&\text{$i=0$, $j\neq0$,}\\
0&\text{if $i\neq0$,}
\end{cases}\\
P_{ij}\cdot Q_{\infty}&=0,\\
Q_{\infty}^2&=-\frac dr.
\end{align*}

Finally, recalling that
$$\langle {P_{00}}, P_{ij}\rangle 
= -P_{00}\cdot P_{ij} + P_{00}\cdot Q_\infty + P_{ij}\cdot Q_\infty 
- Q_\infty^2 - D_{P_{00}} \cdot P_{ij}$$
and summing the contributions above yields the theorem.
\qed

\begin{remark}
  At this point, it would be possible to deduce from
  Theorem~\ref{thm:height-pairing} and an elaborate exercise in row
  reduction that the rank of $V$ is equal to $(r-1)(d-2)$.  We take a
  slightly more indirect approach in the next two sections that
  yields more information about $V$, ultimately allowing us to
  determine $V$ precisely as a module over the group ring
  $R=\Z[\mu_d\times\mu_r]$.
\end{remark}

\section{A group-theoretic pairing}
Recall the group ring 
$$R=\Z[\mu_d\times\mu_r]\cong
\frac{\Z[\sigma,\tau]}{(\sigma^d-1,\tau^r-1)}$$ introduced in
Section~\ref{ss:R} of Chapter~\ref{ch:C} and the ideal $I\subset R$
introduced in Section~\ref{s:relations}.  In this section, we
define a positive definite bilinear form on $R/I$ and compare it with
the height pairing on $V$ via the map $R/I\to V$.  This comparison
plays a key role in showing that the map $R/I\to V$ is an
isomorphism and thus that $J(K_d)$ has large rank.

\subsection{A rational splitting}
\label{sec:rational-splitting}

For notational simplicity, in this and the following subsection we
write $G$ for $\mu_d\times\mu_r$.  Let $R^0=R\tensor\Q=\Q[G]$ be the
rational group ring.  Because $G$ is abelian, the regular
representation of $R^0$ on itself breaks up into $\Q$-irreducibles
each appearing with multiplicity one.
% the irreducible are all cyclotomic fields indexed by divisors $s$ of $r$
  % and $e$ of $d$ and by elements of $(\Z/\gcd(s,e)\Z)^\times$.
  % If it is of use in computing invariants, we could add a 
  % remark here about $R^0/I^0$ in these terms
  As a result of the multiplicity condition, if $I^0$ is any ideal of $R^0$ and
  $\pi:R^0\to R^0/I^0$ is the projection, then there is a unique
  $G$-equivariant splitting $\rho:R^0/I^0\to R^0$.

  We work this out explicitly in the case where $I$ is as in
  Section~1.3  and $I^0=I\tensor\Q$.
  We write 
  $$s_j=\sum_{i\equiv j\bmod r}\sigma^i,$$ 
  so that
  $$\sum_{i=0}^{d-1}\sigma^i\tau^{d-i}=\sum_{j=0}^{r-1}s_j\tau^{r-j}.$$
  Recall that $I$ is the ideal of $R$ generated by 
  $$(\tau-1)\sum_{i=0}^{d-1}\sigma^i,\qquad(\tau-1)\sum_{j=0}^{r-1}s_j\tau^{r-j},
  \qquad\text{and}\qquad\sum_{j=0}^{r-1}\tau^j.$$

  \begin{lemma}\label{lemma:splitting}
    The unique $G$-equivariant splitting
    $\rho:R^0/I^0\to R^0$ is determined by
    $$\rho(\sigma^a\tau^b)=\sigma^a\tau^b
    \left(1+\frac2{rd}\sum_{i=0}^{d-1}\sigma^i\sum_{j=0}^{r-1}\tau^j
      -\frac1d \sum_{i=0}^{d-1}\sigma^i-\frac1d\sum_{j=0}^{r-1}s_j\tau^{r-j}
      -\frac1r\sum_{j=0}^{r-1}\tau^j\right).$$
  \end{lemma}

  \begin{proof}
    The formula defines a $G$-equivariant map $R\to R$.  We have to check
    that it kills the ideal $I^0$, so that it descends to $\rho:R^0/I^0\to R^0$,
    and that it is a splitting.  

    The fact that $\rho$ kills $I^0$ follows from the following easily checked
    identities in $R$:
    \begin{align*}
      \left(\sum_{i=0}^{d-1}\sigma^i\right)^2&=d \sum_{i=0}^{d-1}\sigma^i,\\
      \left(\sum_{i=0}^{d-1}\sigma^i\right)\left(\sum_{j=0}^{r-1}s_j\tau^{r-j}\right)
   &=\frac dr \left(\sum_{i=0}^{d-1}\sigma^i\right)\left( \sum_{j=0}^{r-1}\tau^j\right),\\
      \left(\sum_{j=0}^{r-1}s_j\tau^{r-j}\right)^2&=d \sum_{j=0}^{r-1}s_j\tau^{r-j},\\
      \left(\sum_{j=0}^{r-1}\tau^j\right)^2&=r \sum_{j=0}^{r-1}\tau^j,\\
      \left(\sum_{j=0}^{r-1}\tau^j\right)\left(\sum_{j=0}^{r-1}s_j\tau^{r-j}\right)
   &=\left(\sum_{i=0}^{d-1}\sigma^i\right)\left( \sum_{j=0}^{r-1}\tau^j\right).
    \end{align*}
Using these, it is a straightforward computation to check that
$\rho(I^0)=0$.

To see that $\rho$ is a splitting, it suffices to check that the
expression in parentheses on the right hand side of
Lemma~\ref{lemma:splitting} has the form $1+\iota$ where $\iota\in
I^0$.  But
$$r\sum\sigma^i
=(1+\tau+\cdots+\tau^{r-1})(1-\tau)(\sum\sigma^i)\in I$$
and 
$$r\sum s_j  \tau ^{r-j}
=(1+\tau+\cdots+\tau^{r-1})(1-\tau)(\sum s_j \tau ^{r-j})\in I,$$ so
$\sum\sigma^i$ and $\sum s_j \tau ^{r-j}$ lie in $I^0$.  Since
$\sum\tau^j$ also lies in $I^0$, it follows that $\rho$ has the form
$\rho(r)=r(1+\iota)$ with $\iota\in I^0$, so $\rho:R^0/I^0\to R^0$ is
a splitting.
\end{proof}

\subsection{A pairing}
Now we introduce an inner product on $R^0$ by declaring that
$$\left\langle\sum_g a_gg,\sum_gb_gg\right\rangle_{R^0}=\sum_g a_gb_g.$$
In other words $\langle g,h\rangle_{R^0}=\delta_{gh}$.  Crucially,
this inner product is \emph{positive definite}.

The splitting $\rho$ produces an inner product on $R^0/I^0$ that is
also positive definite.  Namely, we set
  $$\left\langle a,b\right\rangle_{R^0/I^0}:=
  \left\langle \rho(a),\rho(b)\right\rangle_{R^0}.$$ 
The values of this pairing are determined by the following
proposition and $G$-equivariance.

\begin{prop}\label{prop:G-pairing}
  With notation as above,
 \begin{multline}
   \left\langle \sigma^i\tau^j,1\right\rangle_{R^0/I^0}=\\
   \frac1{rd}\begin{cases}
     (r-1)(d-2)&\text{if $i=j=0$,}\\
     2-r&\text{if $i\not\equiv0\bmod r$, $j=0$,}\\
     2-2r&\text{if $i\equiv0\bmod r$, $i\not\equiv0\bmod d$, $j=0$,}\\
     2-d&\text{if $i=0$, $j\not\equiv0\bmod r$,}\\
     2-r&\text{if $i\not\equiv0\bmod r$, $i+j\equiv0\bmod r$,}\\
     2&\text{if $i\not\equiv0\bmod d$, $j\not\equiv0\bmod r$,
       $i+j\not\equiv0\bmod r$.}
      \end{cases}
\end{multline}
\end{prop}

We leave the proof as an exercise for the reader.  It is convenient
for the calculation to note that if $a$ and $b$ are in $R^0/I^0$ and
if $\tilde b$ is any lift of $b$ to $R^0$, then
$$\langle\rho(a),\rho(b)\rangle_{R^0}=\langle\rho(a),\tilde b\rangle_{R^0}.$$ 
This follows from the fact that the pairing is $G$-equivariant, plus
the fact that the irreducible subrepresentations of $R$ appear with
multiplicity one. Using this observation and $G$-equivariance shows
that computing the pairing on $R^0/I^0$ amounts to reading off the
coefficients of $\rho(1)$.

\subsection{Comparison of pairings}
We now compare the group-theoretic pairing of the preceding
subsections to the height pairing.

More precisely, there is a well-defined map $R^0/I^0\to
J(K_d)\tensor\Q$ given by $r\mapsto r(P_{00})$ whose image is by
definition $V\tensor\Q$.  There is a pairing on $R^0/I^0$ obtained by
using the map to $V\tensor\Q$ and the height pairing on
$J(K_d)\tensor\Q$.

Comparing the height pairing (computed in
Theorem~\ref{thm:height-pairing}) with the group-theoretic pairing
(computed in Proposition~\ref{prop:G-pairing}) shows that they are the
same up to a scalar: the height pairing is $(d-1)$ times the group
theoretic pairing.  More formally, we have shown the following.

\begin{prop}\label{prop:pairings}
For all $a,b\in R$, there is an equality
$$\left\langle a(P_{00}),b(P_{00})\right\rangle=
(d-1)\left\langle a,b\right\rangle_{R^0/I^0}.$$ 
Here, the left hand pairing is the height pairing on $J_r(K_d)$.
\end{prop}

\begin{cor}\label{cor:rank-lower-bound}
  The map $(R/I)/tor\to V/tor$ is injective and therefore an
  isomorphism.  The rank of $V$ is thus $(r-1)(d-2)$.
\end{cor}

\begin{proof}
  Proposition~\ref{prop:pairings} shows that the pairing on
  $(R/I)/tor$ induced by the homomorphism $(R/I)/tor\into R^0/I^0\to
  V\tensor\Q$ is positive definite.  It follows immediately that the
  homomorphism $(R/I)/tor\to V/tor$ is injective, and it is surjective
  by the definition of $V$, so it is an isomorphism.
\end{proof}

\section{Structure of the visible subgroup}
In this section, we complete our analysis of $V$ by showing that it is
isomorphic to $R/I$ as an $R$-module and by analyzing the torsion in
$R/I$ as an abelian group.

\subsection{$R/I$ as a group}\label{ss:row-reduction}
We noted in Section~\ref{s:relations} that $I$ is a free $\Z$-module
of rank $d+2(r-1)$, so $R/I$ has rank $(r-1)(d-2)$.  With more work we
can compute the torsion subgroup of $R/I$.

\begin{prop}\label{prop:R/I-as-group}
  There is an isomorphism of $\Z$-modules
  $$R/I\cong \Z^{(r-1)(d-2)}\oplus T$$ 
  where
  $$T=\begin{cases}
    (\Z/r\Z)^3&\text{if $r$ is odd,}\\
    \Z/(r/2)\Z\oplus\Z/r\Z\oplus\Z/(2r)\Z&\text{if $r$ is even.}
  \end{cases}$$
  Thus the torsion subgroup of $R/I$ has order $r^3$.
\end{prop}

\begin{proof}
  The plan for the proof is to choose bases of $R$ and $I$ as
  $\Z$-modules, use them to write down the matrix of the inclusion of
  $\Z$-modules $I\to R$, and use row operations to compute the
  invariant factors of this matrix.

  Here is some useful notation.  Let $\phi:\Z^r\to\Z^d$ be the
  homomorphism
    $$\phi(a_1,\dots,a_r)=(a_1,\dots,a_r,a_1,\dots,a_r,\dots,a_1,\dots,a_r).$$
    In words, $\phi$ simply repeats its argument $d/r$ times.  Let
    $\psi:\Z^r\to\Z^{dr}$ be the homomorphism
    $$\psi(a_1,\dots,a_r)=
    (\phi(a_1,\dots,a_r),\phi(a_2,a_3,\dots,a_r,a_1),\dots,
    \phi(a_r,a_1,\dots,a_{r-1})).$$ 
    In words, $\psi$ rotates its argument $r$ times and repeats each
    result $d/r$ times. It is convenient to apply $\psi$ to an
    $s\times r$ matrix, by applying it to each row, thus obtaining a
    map from $s\times r$ matrices to $s\times dr$ matrices. Let $I_d$
    denote the $d\times d$ identity matrix; let $\zero_r$ denote the
    zero vector in $\Z^r$, and let $\one_r$ denote the vector
    $(1,1,\dots,1)\in\Z^r$.

    As an ordered basis of $R$ we choose
    $$1,\sigma,\dots,\sigma^{d-1},\tau,\sigma\tau,\dots,\sigma^{d-1}\tau,
    \tau^2,\dots,\sigma^{d-1}\tau^{r-1}.$$
    As an ordered basis of $I$ we choose 
    $$f_0,f_1,\dots,f_{d-1},d_1,\dots,d_{r-1},e_1,\dots,e_{r-1},$$
    defined as in Section~\ref{s:relations}.

    With respect to these bases, the first $d+r-1$ rows of the matrix of the
    inclusion $I\to R$ have the form\\
    \[
    \begin{array}{cccccc}
      I_d&I_d&I_d&\cdots&I_d&I_d\\
      \phi(-\one_r)&\phi(\one_r)&\phi(\zero_r)&\cdots&\phi(\zero_r)&\phi(\zero_r)\\
      \phi(\zero_r)&\phi(-\one_r)&\phi(\one_r)&\cdots&\phi(\zero_r)&\phi(\zero_r)\\
      \vdots&\vdots&\vdots&\ddots&\vdots&\vdots\\
      \phi(\zero_r)&\phi(\zero_r)&\phi(\zero_r)&\cdots&\phi(-\one_r)&\phi(\one_r)\\
    \end{array}.
    \]\\

    The last $r-1$ rows are $\psi$ applied to the $(r-1)\times r$ matrix:\\
    \[
    \left(
      \begin{array}{rrrrr}
        -1&1&0&\cdots&0\\
        0&-1&1&\cdots&0\\
        \vdots&\vdots&\vdots&\ddots&\vdots\\
        0&\cdots&0&-1&1\\
      \end{array}
    \right).
    \]\\

    We refer to the rows by the names of the corresponding generators of
    $I$.  Thus $f_i$ for $i=0,\dots,d-1$ refers to the first $d$
    rows, $d_j$ for $j=1,\dots,r-1$ refers to the next $r-1$ rows, and
    $e_j$ for $j=1,\dots,r-1$ refers to the last $r-1$ rows.

    We now perform row operations on this matrix as follows.  First, we
    replace row $d_1$ with 
    $$\sum_{j=1}^{r-1}jd_j+\sum_i f_i,$$
    which has the effect of replacing row $d_1$ with
    \[
    \begin{array}{cccccc}
      \phi(\zero_r)&\phi(\zero_r)&\phi(\zero_r)&\cdots&\phi(\zero_r)&\phi(r\one_r).\\
    \end{array}
    \]
    Next, we replace row $e_1$ with 
    $$\sum_{j=1}^{r-1}je_j,$$
    which has the effect of replacing row $e_1$ with
    $$\psi(-1,-1,\dots,-1,r-1).$$

    Now we replace row $e_2$ with
    $$\sum_{j=2}^{r-1}\binom j2 e_j,$$
    which has the effect of replacing row $e_2$ with
    $$\psi\left(0,-1,-2,\dots,-(r-2),\binom{r-1}2\right).$$

    Now we subtract a suitable combination of the $f_i$ rows from the last
    $r-1$ rows so as to make the lower left $(r-1)\times d$ block
    identically zero. The last $r-1$ rows $e_1, \dots, e_{r-1}$ then take the
    form
    {\small
      \[
      \begin{array}{lllll}
        \phi(0)&\multicolumn{2}{l}{\phi(0,\dots,0,r,-r)}& \, & {\phi(0,\dots,0,r,0,-r)} \\
               & \quad & \dots \quad & \multicolumn{2}{l}{\phi(r,0,\dots,0,-r)} \\
        \noalign{\medskip}
        \phi(0)& \multicolumn{2}{l}{\phi\big(-1,-1,\dots,\binom r2-1,
                 \frac{-2-r(r-3)}{2}\big)}&\, &{\phi\big(-2,\dots,\binom r2-2,r-2,
                                                \frac{-4-r(r-3)}{2}\big)} \\
               & \quad & \dots \quad & \multicolumn{2}{l}{\phi\big(1-r+\binom
                                       r2,1,\dots, 1,\frac{2(1-r)-r(r-3)}{2}\big)} \\
        \noalign{\medskip}
        \phi(0)&\phi(0,-1,2,-1,0\dots,0)& \, & \, \dots&\\
        \noalign{\medskip}
        \phi(0)&\phi(0,0,-1,2,-1,0\dots,0)& \, & \, \dots&\\
               & \qquad \vdots&&&\\
        \noalign{\medskip}
        \phi(0)&\phi(0,\dots,-1,2,-1)& \, & \, \dots.&
      \end{array}
      \]} \\

    Now we replace row $e_2$ with  $e_2-\sum_{j=2}^{r-1}\binom j2 d_j$,
    which yields

    {\small
      \[
      \begin{array}{llllll}
        \phi(0)&\multicolumn{2}{l}{\phi(0,\dots,\binom r2,\frac{r(3-r)}{2})}
        &\multicolumn{2}{l}{\phi(0,\dots,\binom r2,r,\frac{r(3-r)}{2})}&\dots \\
               & \multicolumn{2}{l}{\phi(0,\binom r2,r,\dots,r,\frac{r(3-r)}{2})}
        &\phi(0,\frac{r(3-r)}{2},\dots,\frac{r(3-r)}{2},r(2-r)).
      \end{array}
      \]
    }\\

    We now divide into two cases according to the parity of $r$.  If $r$
    is odd, we replace $e_2$ with 
    $$e_2-\frac{r-1}2e_1+\frac{r-1}2d_1,$$
    which yields
    \[
    \begin{array}{ccccc}
      \phi(0)&\phi(0,\dots,0,r)&\phi(0,\dots,r,r)&
                                                   \dots&\phi(0,r,\dots,r).
    \end{array}
    \]
    Note that every entry in this vector is divisible by $r$.
    Arranging the rows in the order
    $$f_0,\dots,f_{d-1},d_2,e_3,\dots,e_{r-1},e_1,e_2,d_3,\dots,d_{r-1}, d_1$$
    yields a matrix in row-echelon form and with the property that the
    leading entry of each row divides every entry to the right.  Looking
    at the leading terms then reveals that the invariant factors are $1$
    repeated $d+2r-5$ times and $r$ repeated 3 times.  

    Now we turn to the case when $r$ is even.  Replacing row $e_2$ with 
    $$e_2-\frac r2e_1$$
    yields
    \[
    \begin{array}{llll}
      \phi(0) & \phi(0,\dots,0,-\frac r2,\frac{3r}2) & 
                    \multicolumn{2}{l}{\phi(0,\dots,-\frac r2,r,\frac{3r}2)} \\
              & & \, \dots & \phi\big(-\frac{r^2}2,-\frac{r(r-3)}2,\dots,
                             -\frac{r(r-3)}2,-\frac{r(r-4)}2\big).
    \end{array}
    \]
    Note that every entry in this vector is divisible by $r/2$.  

    Now we replace $e_1$ with 
    $$e_1+2e_2+(r-1)d_1,$$
    which yields
    {\small
      $$\xymatrix{
        \phi(0) &\phi(0,\dots,0,2r)&\phi(0,\dots,0,2r,2r)&
        \dots&\phi(0,2r,2r,\dots,2r).
      }$$}
    Note that every entry in this vector is divisible by $2r$.  
    Arranging the rows in the order
    $$f_0,\dots,f_{d-1},d_2,e_3,\dots,e_{r-1},e_2,e_1,d_3,\dots,d_{r-1},d_1$$
    yields a matrix in row-echelon form and with the property that the
    leading entry of each row divides every entry to the right.
    Looking at the leading terms then reveals that the invariant
    factors are $1$ repeated $d+2r-5$ times and $r/2$, $r$, and $2r$
    each appearing once.

    This completes the proof of the theorem.
  \end{proof}

We record the torsion classes provided by the proof.  They are not
used later in the paper, but they help explain the definition of the
elements $Q_2$ and $Q_3\in J(K_d)$ introduced in
Section~\ref{s:torsion}.

 \begin{prop}\label{prop:torsion-in-R/I}
   If $r$ is odd, the classes of 
   $$\sum_i\sigma^i,\qquad\sum_i\sigma^i\tau^{d-i},\quad\text{and}\quad
   \sum_{j=0}^{r-1}\sum_{k=0}^{r-1-j}\sum_{i\equiv k\bmod r}\sigma^i\tau^j$$
   in $R/I$ are torsion of order $r$ and generate a group of order $r^3$.
    If $r$ is even, the classes of
    \begin{multline*}
      \sum_i\sigma^i,\qquad
      \sum_{j=0}^{r-1}\sum_{k=0}^{r-1-j}\sum_{i\equiv k\bmod r}\sigma^i\tau^j,
      \quad\text{and}\quad\\
      -\sum_{j=0}^{r-2}\sum_{i\equiv r-1-j\bmod r}\sigma^i\tau^j
      +\sum_{i\not\equiv0\bmod r}\sigma^i\tau^{r-1}
      +2\sum_{j=1}^{r-1}\sum_{k=r-j}^{r-1}\sum_{i\equiv k\bmod
        r}\sigma^i\tau^j
    \end{multline*}
    in $R/I$ are torsion of orders $r$, $2r$, and $r/2$ respectively,
    and they generate a group of order $r^3$.
  \end{prop}

  \begin{proof}
    Considering the row $d_1$, after row reduction as above, we find
    that $\sum_i\sigma^i\tau^{r-1}$ is $r$ torsion, and this element
    is equivalent in $R/I$ to $\sum_i\sigma^i$.

    Assume $r$ is odd. Considering the row $e_1$, we see that 
    $$\sum_{j=1}^{r-1}\left(\sum_{i\equiv r-1-j\bmod r}\sigma^i-
      \sum_{i\equiv r-1\bmod r}\sigma^i\right)\tau^j$$
    is $r$-torsion.  Adding $\sum_{i\equiv r-1\bmod r}f_i$, one checks that this
    is equivalent in $R/I$ to
    $$\sum_{j=0}^{r-1}\sum_{i\equiv r-1-j\bmod r}\sigma^i\tau^j,$$
    which in turn is equivalent to $\sum_i\sigma^i\tau^{d-i}$.
    Also, from the row $e_2$, we can see that
    $$\sum_{j=1}^{r-1}\sum_{k=r-j}^{r-1}\sum_{i\equiv k\bmod r}\sigma^i\tau^j$$
    is $r$-torsion.  The negative of this element is equivalent in $R/I$ to
    $$\sum_{j=0}^{r-1}\sum_{k=0}^{r-1-j}\sum_{i\equiv k\bmod r}\sigma^i\tau^j.$$

    Since the three $r$-torsion elements just exhibited are associated to
    distinct rows of a matrix in row-echelon form, they are independent,
    i.e., they generate a subgroup of order $r^3$.  This completes the
    proof in the case that $r$ is odd.

    When $r$ is even, the proof for the first class is as in the case for
    $r$ odd.  From the relation from row $e_1$, we can conclude that 

    $$\sum _{j=1}^{r-1} \sum _{i=r-j}^{r-1} \sum _{i \equiv k\bmod r} \sigma^i\tau^j$$
    is $2r$-torsion.
    Since $\sum _j \tau ^j = 0$ in $R/I$, the negative of this is
    equivalent to  

    $$\sum_{j=0}^{r-1} \sum _{k=0}^{r-1-j} \sum_{i \equiv k\bmod r} \sigma^i\tau^j.$$
    Combining the relation from row $e_2$ with the fact that $\sum _i
    \sigma ^i$ is $r$-torsion, we can check that  
    \begin{equation*}
      2 \sum_{\substack{1 \leq j \leq r-1 \\ r-j \leq k \leq r-1 \\ i
          \equiv k \bmod r}} \sigma ^i \tau ^j + \sum_{\substack{i \equiv
          r-1 \bmod r \\ 0 \leq j \leq r-1}} \sigma ^i \tau ^j \\ + \sum
      _{i \not \equiv 0\bmod r} \sigma ^i \tau 
      ^{r-1} - \sum_{\substack{0 \leq j \leq {r-2} \\ i \equiv r - 1 - j
          \bmod r}} \sigma ^i \tau ^j 
    \end{equation*}
    is $\frac{r}{2}$-torsion.   Since $\sum _j \tau ^j = 0$ in $R/I$, the
    second term is zero, and the result follows as above.  
  \end{proof}

  \subsection{$R/I$ and $V$}\label{ss:visible}
  We can now finish the proof that $V$ is isomorphic as an $R$-module to
  $R/I$.

  \begin{theorem}\label{thm:visible}
    The projection $R/I\to V$ defined by $r\mapsto r(P_{00})$ is an
    isomorphism.
  \end{theorem}

  \begin{proof}
    Write $W$ for $R/I$.  We have a commutative diagram with exact rows:
    $$\xymatrix{
      0\ar[r]&W_{tor}\ar[r]\ar[d]&W\ar[r]\ar[d]&W/W_{tor}\ar[r]\ar[d]&0\\
      0\ar[r]&V_{tor}\ar[r]&V\ar[r]&V/V_{tor}\ar[r]&0.}$$ By definition,
    the middle vertical arrow is surjective, thus so is the right vertical
    arrow.  By Corollary~\ref{cor:rank-lower-bound}, the right vertical
    arrow is injective, so it is an isomorphism.  The snake lemma then
    shows that the left vertical arrow is surjective.  But
    Proposition~\ref{prop:R/I-as-group} shows that $W_{tor}$ has order
    $r^3$, whereas Proposition~\ref{prop:torsion} shows that $V_{tor}$ has
    order at least $r^3$.  It follows that the left vertical arrow is also
    an isomorphism.  Now another application of the snake lemma shows the
    middle vertical arrow is an isomorphism as well, and this is our
    claim.
  \end{proof}

  \begin{cor}\label{cor:visible}
    The subgroup $V$ of $J(K_d)$, generated by $P_{00}$ and its conjugates
    under $\gal(K_d/K)$, is isomorphic as a $\Z$-module to
    $$\Z^{(r-1)(d-2)}\oplus\begin{cases}
      (\Z/r\Z)^3&\text{if $r$ is odd,}\\
      \Z/(r/2)\Z\oplus\Z/r\Z\oplus\Z/(2r)\Z&\text{if $r$ is even.}
    \end{cases}$$
  \end{cor}

  \begin{remark}
    It would be possible at this point to give lower bounds on the rank of
    $J$ over various subfields of $\Fpbar(t^{1/d})$, along the lines of
    \cite[Corollary~4.4]{Legendre}.  However, we delay the discussion of
    ranks until the end of the following chapter, where it is possible to
    give exact values for the rank.
  \end{remark}

\section{Discriminants}
In this section we work out the discriminant of the height pairing on
$V/tor$.  This is used in Chapter~\ref{ch:Sha} to obtain information
on the index of $V$ in $J(K_d)$ and on the Tate-Shafarevich group of
$J/K_d$.

With notation as in the previous subsection, let $W=R/I$.  Recall that
there is a canonical $G$-equivariant splitting $\rho:W\to R^0$ and a
pairing on $W$ given by $\langle a,b\rangle_{R^0/I^0}=\langle
\rho(a),\rho(b)\rangle_{R^0}$ where the second pairing is the
Euclidean pairing on $R^0$.  Recall that, up to a scalar ($d-1$), the
pairing on $W$ is the canonical height pairing on $V$.

Write $\det(W/tor)$ for the discriminant of this pairing on $W$ modulo
torsion and $\det(I)$ for the discriminant of the pairing on $I$
induced by that on $R$.

We would like to relate these discriminants to each other. To that
end, we consider a slightly more general situation: let $H$ be an
arbitrary ideal of $R$, and $U = R^0/H^0$. One still has a
$G$-equivariant splitting $\varrho: U \to R^0$ and an induced pairing
on $U$.

\begin{prop} With notation as above,
    $$\det(H)=\frac{|U_{tor}|^2}{\det(U/tor)}.$$
\end{prop}

\begin{proof}
  First suppose that $H$ is saturated, i.e., that $U$ is torsion-free.
  Let $e_1,\dots,e_k$ be a $\Z$-basis of $H$ and extend it to a
  $\Z$-basis $e_1,\dots,e_n$ of $R$.  Write $\overline e_i$ for the
  image of $e_i$ in $U$, so that $\overline e_{k+1},\dots,\overline
  e_n$ is a $\Z$-basis of $U$.  Because the pairing on $R$ is the
  Euclidean pairing, the discriminant
    $$\left|\det\left(\langle e_i,e_j\rangle\right)\right|=1.$$
    Now let
    $$f_i=\begin{cases}
      e_i&\text{if $i\le k$,}\\
      \varrho(\overline e_i)&\text{if $i>k$.}
    \end{cases}
    $$
  This is a $\Q$-basis of $R^0$.  The change of basis matrix is
  upper triangular with 1's on the diagonal, so it has determinant 1
  and
    $$\left|\det\left(\langle f_i,f_j\rangle\right)\right|=
    \left|\det\left(\langle e_i,e_j\rangle\right)\right|=1.$$ Now
    $\varrho(U)$ is orthogonal to $H$, so the new Gram matrix
    $\left(\langle f_i,f_j\rangle\right)$ is block diagonal.  Its upper left
    $k\times k$ block is just $\left(\langle e_i,e_j\rangle\right)$ and
    the determinant of this block is $\pm\det(H)$.  The lower right
    $(n-k)\times(n-k)$ block is just $\left(\langle \varrho(\overline
      e_i),\varrho(\overline e_j)\rangle\right)$ and the determinant of
    this block is $\pm\det(U)=\pm\det(U/tor)$.  Thus these two
    discriminants are reciprocal and this proves the claim in the case when
    $H$ is saturated.

    For general $H$, let $H'$ be the saturation, so that $|H'/H|=|U_{tor}|$
    and $R/H'=U/tor$.  Then 
    $$\det(H)=|H'/H|^2\det(H')=|U_{tor}|^2\det(H')=\frac{|U_{tor}|^2}{\det(U/tor)},$$
    as desired.
  \end{proof}

  \begin{prop} We have 
    $$\det(I)=r^{d+2}d^{2r-2}.$$
  \end{prop}

  \begin{proof}
    It is not hard to check that the following is a $\Z$-basis for $I$:
    \begin{align*}
      \alpha_i&=\sigma^i\sum \tau^j\qquad i=0,\dots,d-1,\\
      \beta_j&=(\tau^j-1)\sum \sigma^i\qquad j=1,\dots,r-1,\\
      \gamma_j&=(\tau^j-1)\sum \sigma^i\tau^{d-i}\qquad j=1,\dots,r-1.
    \end{align*}

    The values of the pairing are
    \begin{align*}
      \langle\alpha_i,\alpha_{i'}\rangle&=r\delta_{ii'},\\ 
      \langle\alpha_i,\beta_j\rangle&=0,\\ 
      \langle\alpha_i,\gamma_j\rangle&=0,\\ 
      \langle\beta_j,\beta_{j'}\rangle&=d(\delta_{jj'}+1),\\ 
      \langle\beta_j,\gamma_{j'}\rangle&=0,\\
      \langle\gamma_j,\gamma_{j'}\rangle&=d(\delta_{jj'}+1),
    \end{align*}
    so the Gram matrix for this basis of $I$ is block diagonal.  An
    inductive argument shows that if $A$ is the sum of an identity matrix
    of size $a\times a$ and a matrix of the same size with all entries 1,
    then $\det(A)=a+1$.  Thus $\det(I)=r^{d+2}d^{2r-2}$
    as desired.
  \end{proof}

\begin{cor} \label{cor:det(W/tor)} \label{cor:Vdet}
If $W=R/I$, then $$\det(W/tor)=r^{4-d}d^{2-2r}.$$  
Also $$\det(V/tor)=(d-1)^{(r-1)(d-2)}r^{4-d}d^{2-2r}.$$
\end{cor}

\begin{proof}
  The first claim follows from Proposition~\ref{prop:R/I-as-group}.
  The second follows from Theorem~\ref{thm:visible} and
  Corollary~\ref{cor:det(W/tor)}, keeping in mind the scalar $(d-1)$
  relating the group-theoretic and height pairings as in
  Proposition~\ref{prop:pairings}.
\end{proof}

    %%% Local Variables: 
    %%% mode: latex
    %%% TeX-master: "EHR"
    %%% End: 

%This is Chapter 5

\chapter{The $L$-function and the BSD conjecture}\label{ch:L}

In this chapter, we compute the Hasse-Weil $L$-function of the
Jacobian $J$ of $C$ over certain extensions of $\Fp(t)$ and prove the
conjecture of Birch and Swinnerton-Dyer for $J$.  This leads to a
combinatorial calculation of the rank of $J$.  We use the refined BSD
conjecture in Chapter~\ref{ch:Sha} to relate the Tate-Shafarevich
group of $J$ to the visible subgroup $V$ defined in
Section~\ref{SdefV}.

We work in the context of general $r$ and $d$ in this chapter; namely,
$k=\Fq$ is any finite field of characteristic $p$, $d$ is any integer
prime to $p$, $K=k(u)$ with $u=t^{1/d}$, $r$ is any integer prime to
$p$, $C$ is the curve of genus $r-1$ over $K$ defined as in
Section~\ref{s:curve} of Chapter 1, and $J$ is the Jacobian of $C$.
Unless stated otherwise we do not assume that $r$ divides $d$ nor that
$d$ divides $q-1$.

\section{The $L$-function}

\subsection{Definition and first properties}
We fix a prime $\ell\neq p$ and consider 
$$H^1(C\times\Kbar,\Ql)\cong H^1(J\times\Kbar,\Ql)$$ 
as a representation of $\gal(\Ksep/K)$ where $K=\Fq(u)$.

The corresponding $L$-function $L(J/K,s)=L(C/K,s)$ is defined by the
Euler product
$$L(J/K,s)=\prod_v
\det\left(1-\Fr_v\,q_v^{-s}\left|H^1(J\times\Kbar,\Ql)^{I_v}\right.\right)^{-1}.$$
Here $v$ runs through the places of $K$, $\Fr_v$ is the (geometric) Frobenius
element at $v$, $q_v$ is the cardinality of the residue field at $v$,
$I_v$ is the inertia group at $v$, and $H^1(J\times\Kbar,\Ql)^{I_v}$
is the subspace of $H^1(J\times\Kbar,\Ql)$ invariant under $I_v$.

It is known that $L(J/K,s)$ is a rational function in $q^{-s}$ (where
$q = \# k = \# \F_q$).  Proposition~\ref{prop:K/k-trace} in the next
chapter shows that the $K/k$-trace of $J$ vanishes. This implies that
$L(J/K,s)$ is in fact a polynomial in $q^{-s}$.

The Grothendieck-Ogg-Shafarevich formula gives the degree of
$L(J/K,s)$ as a rational function in $q^{-s}$ (and therefore as a
polynomial in our case) in terms of the conductor of the
representation $H^1(J\times\Kbar,\Ql)$.  We review this in
Section~\ref{ss:GOS} below.  

See \cite[Section~6.2]{CRM} for more details and references about the
preceding two paragraphs.  We do not need to go into details about
these assertions here, because we give an elementary calculation of
$L(J/K,s)$ from its definition in Section~\ref{s:L-calc} below that
shows that it is a polynomial of known degree.

\subsection{Analysis of local factors}\label{ss:local-Ls}
In this subsection, we make the local factor 
$$L_v:=\det\left(1-\Fr_v\,q_v^{-s}\left|H^1(J\times\Kbar,\Ql)^{I_v}\right.\right)$$
more explicit using the regular proper model $\XX$ constructed in
Section~\ref{s:models}.  Roughly speaking, the familiar fact that we
may calculate the local $L$-factor at places of good reduction by
counting points continues to hold at all places.  Some care is
required because the genus is greater than $1$ and the result
ultimately depends on delicate properties of the N\'eron model.

\begin{prop}\label{prop:invariants}
  For a place $v$ of $K=k(u)$, let $D_v$ be a decomposition group at
  $v$ and let $I_v \subset D_v$ be the corresponding inertia group.
  Let $\XX_v$ be the fiber of $\XX\to\P^1_u$ over the corresponding
  point of $\P^1_u$.  Then there is a canonical isomorphism
$$H^1(J\times\Kbar,\Ql)^{I_v}\cong H^1(\XX_v\times\kbar,\Ql)$$
that is compatible with the actions of $D_v/I_v\cong\gal(\kbar/k)$. 
\end{prop}

\begin{proof}
  This seems to be well-known to experts, but it is hard to find an
  early reference.  A recent preprint of Bouw and Wewers
  \cite{BouwWewers} has a nice exposition that we include here for
  the convenience of the reader.\footnote{The main point of
    \cite{BouwWewers} is that local $L$-factors can be computed
    efficiently from semi-stable models rather than regular models,
    especially for superelliptic curves.  This is relevant for our
    work, but we need the regular proper model $\XX$ for other
    reasons, e.g., computing heights, so the approach of
    \cite{BouwWewers} would not in the end save us anything.}

  First, we have the standard fact that $H^1$ is closely connected to
  the Picard group: Writing $V_\ell$ for the Tate module, then
$$H^1(J\times\Kbar,\Ql)\cong V_\ell \Pic^0(C)\quad\text{and}\quad
H^1(\XX_v\times\kbar,\Ql)\cong V_\ell\Pic^0(\XX_v).$$ 
Second, let $\JJ\to\P^1_u$ be the N\'eron model of the Jacobian $J$,
and let $\JJ_v^0$ be the connected component of the identity of the
fiber at $v$.  Then by \cite[Lemma 2]{SerreTate68},
$$\left(V_\ell \Pic^0(C)\right)^{I_v}\cong V_\ell\JJ_v^0.$$
Finally, and this is the delicate point, the hypotheses of \cite[9.5,
Theorem~4b]{blr} are satisfied and this implies that
$$\JJ_v^0\cong\Pic^0(\XX_v).$$
(Roughly speaking, this result says that the N\'eron model represents
the relative Picard functor.  In order to apply it, we need to know
that $\XX$ is a regular proper model and that the $\gcd$ of the
multiplicites of the components of $\XX_v$ is one.  This last point
was shown directly in Section~\ref{s:models}, and it also follows from
the fact that $C/K$ has a rational point so $\XX\to\P^1$ has a
section.)

Combining the displayed isomorphisms completes the proof.
\end{proof}

Next we make the connection with point counting.  Write $k_v$ for the
residue field at $v$ and $k_{v,n}$ for the extension of $k_v$ of
degree $n$.  Then the Grothendieck-Lefschetz trace formula applied to
$\XX_v$ says:
\begin{equation}\label{eq:Lefschetz}
|\XX_v(k_{v,n})|=\sum_{i=0}^2(-1)^i
\tr\left(\Fr_v^n|H^i(\XX_v\times\kbar,\Ql)\right).
\end{equation}

The fibers $\XX_v$ are connected, so $H^0(\XX_v\times\kbar,\Ql)=\Ql$
with trivial Frobenius action.  On the other hand,
$H^2(\XX_v\times\kbar,\Ql)$ has dimension equal to the number of
irreducible components of $\XX_v\times\kbar$ and is isomorphic to
$\Ql(-1)$ tensored with a permutation representation keeping track of
the action of Frobenius on the set of irreducible components.  In
particular, the trace of $\Fr_v^n$ on $H^2(\XX_v\times\kbar,\Ql)$ is
equal to $c_{v,n}|k_{v,n}|$ where $c_{v,n}$ is the number of
irreducible components of $\XX_v\times\kbar$ that are rational over
$k_{v,n}$ and $|k_{v,n}|$ is the cardinality of $k_{v,n}$.  Thus,
computing $H^1(\XX_v\times\kbar,\Ql)$ with its Frobenius action is
reduced to counting points.  The end result is recorded in
Proposition~\ref{prop:Lv} below once we establish the necessary
notation for characters.

\subsection{Conductors and degree of the $L$-function}
\label{ss:GOS}
We write $c_v$ for the exponent of the conductor of the representation
$H^1(J\times\Kbar,\Ql)$ at $v$.  Since the latter is tamely ramified
(by Proposition~\ref{prop:reduction}), the conductor at $v$ is simply the
codimension of $H^1(J\times\Kbar,\Ql)^{I_v}$ in
$H^1(J\times\Kbar,\Ql)$.  Using Proposition~\ref{prop:invariants},
$c_v$ is the difference between the $\Ql$-dimension of the Tate module
of the generic fiber and that of the special fiber.  In terms of the
notation in Proposition~\ref{prop:connected-component}, the dimension
of the Tate module of the special fiber is $2g_v+m_v$.  It follows
that $c_v=2(r-1)-2g_v-m_v$, and using
Proposition~\ref{prop:connected-component}, we find that
\begin{equation}\label{eq:conductor}
c_v=\begin{cases}
r-1&\text{if $v$ lies over $t=0$ or $t=1$,}\\
2r-\gcd(d,r)-1&\text{if $v$ lies over $t=\infty$,}\\
0&\text{otherwise.}
\end{cases}
\end{equation}

Assuming Proposition~\ref{prop:K/k-trace} below, we know that the
$L$-function is a polynomial in $T=q^{-s}$.  In this case, the
Grothendieck-Ogg-Shafarevich formula gives its degree as
\begin{equation}\label{eq:L-degree}
\deg L(J/K,T)=-4(r-1)+\sum_v c_v=(d-1)(r-1)-(\gcd(d,r)-1).
\end{equation}

We confirm this below with a more elementary proof that avoids
the forward reference to Proposition~\ref{prop:K/k-trace}.

\section{The conjecture of Birch and Swinnerton-Dyer for $J$}

In this section we continue studying the arithmetic of $J$ in the case
of general $r$ and $d$, so $K=k(u)$ with $u^d=t$, and $k$ is finite of
characteristic $p$ not dividing $rd$.  As above, let $L(J/K,s)$ be the
Hasse-Weil $L$-function of $J$.  We write $L^*(J/K,1)$ for the leading
coefficient in the Taylor expansion of $L(J/K,s)$ near $s=1$.  (This
is defined because we know that $L(J/K,s)$ is a rational function
that is regular in a neighborhood of $s=1$.)

We let $\sha(J/K)$ be the Tate-Shafarevich group of $J$.  This is not
yet known \emph{a priori} to be finite, but we show that it is
finite in our case.  We let $R$ be the determinant of the canonical height
pairing on $J(K)$ modulo torsion.  (This is $(\log q)^{\rk
  J(K)}$ times the determinant of the $\Q$-valued pairing discussed in
Chapter~\ref{ch:heights}.)  Finally, we let $\tau=\tau(J/K)$ be the
Tamagawa number associated to $J$.  This is defined precisely and
computed explicitly in Section~\ref{s:tamagawa}.

\begin{theorem}\label{thm:BSD}
  The conjecture of Birch and Swinnerton-Dyer holds for $J$ over
  $K=\Fq(t^{1/d})$.  More precisely, we have
$$\ord_{s=1}L(J/K,s)=\rk J(K),$$
and $\sha(J/K)$ is finite, and
$$L^*(J/K,1)=\frac{|\sha(J/K)|\,R\,\tau}{|J(K)_{tor}|^2}.$$
\end{theorem}

\begin{proof}
  We saw in Section~\ref{s:DPC} that the surface $\XX$ is dominated by
  a product of curves.  This implies the Tate conjecture for $\XX$ and
  therefore the BSD conjecture for $J$.  See \cite[Sections~8.2 and
  6.3]{CRM} for more details on these implications.
\end{proof}

\begin{remark}
  The most complete reference for the leading term part of the BSD
  conjecture (i.e., the second displayed equation in the Theorem) is
  \cite{KT}.  The formulation in \cite{KT} differs slightly from that
  above.  We compare the two formulations and show they are equivalent
  in Section~\ref{ss:BSDs} below.
\end{remark}

\section{Elementary calculation of the $L$-function}
\label{s:L-calc}

In this section we calculate the Hasse-Weil $L$-function of $J$ in
terms of Jacobi sums.  The arguments here are quite parallel to those
in Section~3 of \cite{Legendre2}, so we use some of the
definitions and notations of that paper, and we omit some of the
details.

\subsection{Characters and Jacobi sums}
Let $\Qbar$ be an algebraic closure of $\Q$, and let $\OO_{\Qbar}$ be
the ring of integers of $\Qbar$.  Choose a prime
$\mathfrak{p}\subset\OO_{\Qbar}$ over $p$ and define
$\Fpbar:=\OO_{\Qbar}/\mathfrak{p}$, so that $\Fpbar$ is an algebraic
closure of $\Fp$.  All finite fields in this section are considered as
subfields of $\Fpbar$. Reduction modulo $p$ defines an isomorphism
between the roots of unity with order prime to $p$ in
$\OO_{\Qbar}^\times$ and $\overline{\F}_p^\times$. The Teichm\"uller
character $\tau:\overline{\F}_{p}^\times\to\OO_{\Qbar}^\times$ is the
unique homomorphism that gives a right inverse to the reduction map.

Consider a multiplicative character $\chi: k^\times\to\Qbar^\times$
for the finite field $k$.  We employ the usual convention
that $\chi(0)=0$ if $\chi$ is non-trivial, and $\chi_{triv}(0)=1$.

If $\chi_1$ and $\chi_2$ are multiplicative
characters $k^\times\to\Qbar^\times$,
we define a Jacobi sum
$$J(\chi_1,\chi_2):=\sum_{u+v+1=0}\chi_1(u)\chi_2(v)$$
where the sum is over $u,v\in k$. 
If we need to emphasize the underlying field, we write
$J_k(\chi_1,\chi_2)$.

\subsection{Orbits and Jacobi sums}
\label{sec:orbits-jacobi-sums}

We write $\<a\>$ for the fractional part of a rational number $a$, so
that $\<a\>\in[0,1)$ and $a-\<a\>\in\Z$.  If $i\in\Z/n\Z$, and if
$\tilde \imath$ and $\tilde \imath'\in\Z$ are representatives of $i$, then
$\<\tilde \imath/n\>=\<\tilde \imath'/n\>$, so we may unambiguously define
$\<i/n\>$ as $\<\tilde \imath/n\>$.

Define
\begin{equation} \label{DdefS}
S=\left\{(i,j)\in\Z/d\Z\times\Z/r\Z\,\left|\, i\neq0, j\neq0,
    \left\<\frac id\right\>+\left\<\frac jr\right\>\not\in\Z\right.\right\}.
    \end{equation}
Then $(\Z/\lcm(d,r)\Z)^\times$ acts on $S$ diagonally by 
$$t \cdot (i,j)=(ti,jj)\quad\text{for}\quad t \in (\Z/\lcm(d,r)\Z)^\times.$$
We write $O$ for the set of orbits of $S$ under 
the diagonal action of the cyclic subgroup of
$(\Z/\lcm(d,r)\Z)^\times$ generated by $q$.  

If $o\in O$ is an orbit, we write $|o|$ for the cardinality of $o$.
Define a Jacobi sum by
\begin{equation} \label{DdefJ}
J_o=J(\chi_i,\rho_j),
\end{equation}
where $(i,j)\in o$, where the sum is over $\F_{q^{|o|}}$, and where
\[\chi_i=\tau^{i(q^{|o|}-1)/d}, \ \rho_j=\tau^{j(q^{|o|}-1)/r}.\]
Well-known properties of Jacobi sums show that $J_o$ is independent of
the choice of $(i,j)$ and that it is a Weil integer of size $q^{1/2}$.

\subsection{The $L$-function in terms of Jacobi sums}

% Recall that the Hasse-Weil $L$-function of $J$ (and of $C$) is defined
  % as $L(J/K,s)=L(J/K,T)$ where $T=q^{-s}$, and
  % $$L(J/K,T)=\prod_{v}\det\left(1-T^{\deg(v)}\Fr_v\left|
  %     H^1(\overline{C},\Ql)^{I_v}\right.\right)^{-1}$$ 
  % where the product is over the places of $K$, $\deg(v)$ is the degree of
  % the place $v$, $\Fr_v$ is the (geometric) Frobenius at $v$, $I_v$ is
  % an inertia group at $v$, and $\overline{C}=C\times_KK^{sep}$.  

  \begin{theorem}\label{thm:L}
    With notations as above, the Hasse-Weil $L$-function of $J/K$ is
    $$L(J/K,s)=\prod_{o\in O}\left(1-J_o^2q^{-|o|s}\right).$$
  \end{theorem}

  The proof of Theorem~\ref{thm:L} is given in
  Section~\ref{ss:Lproof} after some preliminaries in the next
  subsection. 

  \begin{remark}
    Note that the degree of $L(J/K,s)$ as a polynomial in $q^{-s}$ is the
    cardinality of $S$, namely $(d-1)(r-1)-(\gcd(d,r)-1)$.  This
    confirms the calculation of the degree in Section~\ref{ss:GOS}.
  \end{remark}

  \subsection{Explicit local $L$-factors}\label{ss:localL}
  We now turn to some preliminaries toward the proof of
  Theorem~\ref{thm:L}. 

  If $\beta$ is an $\Fqn$-rational point of $\P^1_u$ and $v$ is the
  place of $k=\Fq(u)$ under $\beta$, we write $a_{\beta,q^n}$ for the
  trace of the $q^n$-power Frobenius on $H^1(J\times\Kbar,\Ql)^{I_v}$,
  or equivalently (by Proposition~\ref{prop:invariants}) on
  $H^1(\XX_v\times\kbar,\Ql)$.  We may compute this trace using
  Equation~\eqref{eq:Lefschetz} and the remarks in the paragraph
  following it.

  \begin{prop}\label{prop:Lv}
    Let $s=\gcd(r,q^n-1)$ and $\phi=\tau^{(q^n-1)/s}$.  For all
    $\beta\in\Fqn$, we have 
    $$a_{\beta,q^n}=-\sum_{j=1}^{s-1}\sum_{\gamma\in\Fqn}
    \phi^j\left(\gamma^{r-1}(\gamma+1)(\gamma+\alpha)\right)$$ 
    where $\alpha=\beta^d$.  If $\beta=\infty$, then
    $$a_{\beta,q^n}=\gcd(d,s)-1=\gcd(d,r,q^n-1)-1.$$
  \end{prop}

  \begin{proof}
    If $\beta\not\in\{0,\mu_d,\infty\}$, then the fiber $\XX_v$ is the
    smooth projective model of the affine curve
    $y^r=x^{r-1}(x+1)(x+\beta^d)$ with one point at infinity.  A standard
    exercise gives the number of points as an exponential sum:
    \begin{align*}
      |\XX_v(\Fqn)|&=1+\sum_{j=0}^{s-1}\sum_{\gamma\in\Fqn}
                     \phi^j\left(\gamma^{r-1}(\gamma+1)(\gamma+\alpha)\right)\\
                   &=q^n+1+\sum_{j=1}^{s-1}\sum_{\gamma\in\Fqn}
                     \phi^j\left(\gamma^{r-1}(\gamma+1)(\gamma+\alpha)\right).
    \end{align*}
    Since $\XX_v\times\kbar$ is irreducible, using 
    Equation~\eqref{eq:Lefschetz} and the remarks in the paragraph
    following it shows that
    $$a_{\beta,q^n}=-\sum_{j=1}^{s-1}\sum_{\gamma\in\Fqn}
    \phi^j\left(\gamma^{r-1}(\gamma+1)(\gamma+\alpha)\right),$$ 
    as claimed.

    If $\beta=0$, then the calculations in Section~\ref{ss:bad-fibers}
    (see Figure~\ref{fig:u-equals-zero}) show that 
    $$|\XX_v(\Fqn)|=\left(s(d-1)+2\right)q^n+2-s.$$
    On the other hand, the number $c_{v,n}$ of rational components is
    $s(d-1)+2$, so the trace is $s-1$.  The displayed formula in the
    Proposition simplifies:
    \begin{align*}
      -\sum_{j=1}^{s-1}\sum_{\gamma\in\Fqn}
      \phi^j\left(\gamma^{r-1}(\gamma+1)(\gamma+\alpha)\right)
      =&  -\sum_{j=1}^{s-1}\sum_{\gamma\in(\Fqn)^\times}
         \phi^j(\gamma+1)\\
      =&s-1,
    \end{align*}
    so the exponential sum is the trace, as desired.

    If $\beta^d=1$, we consider the cases $r$ odd
    (\S{}\ref{ss:bad-fibers},
    Figure~\ref{fig:superelliptic-dual-graph-odd}) and $r$ even
    (\S{}\ref{ss:bad-fibers},
    Figure~\ref{fig:superelliptic-dual-graph-even}) separately.  Let
    $F$ be the smooth projective model of the curve
    $y^r=x^{r-1}(x+1)^2$.  In both cases, the number $c_{v,n}$ of
    irreducible components is $r$.  When $r$ is odd, the number of
    $\Fqn$-rational points is $(r-1)q^n+|F(\Fqn)|$, and the curve $F$
    has unibranch singularities at $(0,0)$ and $(-1,0)$ and one point
    at infinity.  We see that
    $$q^n+1-a_{\beta,q^n}=|F(\Fqn)|=q^n+1+
    \sum_{j=1}^{s-1}\sum_{\gamma\in\Fqn}
    \phi^j\left(\gamma^{r-1}(\gamma+1)^2\right),$$
    and this gives the desired result.  
    If $r$ is even, we have
    $$|\XX_v(\Fqn)|=(r-1)q^n-1+|F(\Fqn)|,$$
    and the curve $F$ has a unibranch singularity at
    $(0,0)$, a singularity with two branches at $(-1,0)$, and one point at
    infinity.  We see that
    $$q^n+1-a_{\beta,q^n}=|F(\Fqn)|-1=q^n+1+
    \sum_{j=1}^{s-1}\sum_{\gamma\in\Fqn}
    \phi^j\left(\gamma^{r-1}(\gamma+1)^2\right)$$
    and this gives the desired result.  

    Finally, at $\beta=\infty$
    (\S\ref{ss:bad-fibers}, Figure~\ref{fig:fiber-infinity-general-d}), we have that
    $c_{v,n}=2d'/d+2+\gcd(d,r,q^n-1)$ and
    $|\XX_v(\Fqn)|=c_{v,n}q^n+2-\gcd(d,r,q^n-1)$, so we find that
    $a_{\beta,q^n}=\gcd(d,r,q^n-1)-1$, as desired.

    This completes the proof of the Proposition.
  \end{proof}

  \subsection{Proof of Theorem~\ref{thm:L}}\label{ss:Lproof}

  The proof is very similar to that of \cite[Theorem~3.2.1]{Legendre2}, so
  we omit many details.  We keep the notation of
  Proposition~\ref{prop:Lv}. 

  By a standard unwinding, we have
  \begin{equation}\label{eq:L1}
    \log L(J/K,T)=
    \sum_{n\ge1}\frac{T^n}n\sum_{\beta\in\P^1(\Fqn)}a_{\beta,q^n}  
  \end{equation}
  where, as in the previous subsection, $a_{\beta,q^n}$ is the trace of
  the $q^n$-power Frobenius on $H^1(\overline{C},\Ql)^{I_v}$ with $v$
  the place of $K=\Fq(u)$ under $\beta$.

  Now let $e=\gcd(d,q^n-1)$ and $\psi=\tau^{(q^n-1)/e}$.  Grouping
  points $\beta\in\P^1(\Fqn)$ by their images under
  $\beta\mapsto\alpha=\beta^d$ and using Proposition~\ref{prop:Lv}, we
  have
  $$\sum_{\beta\in\P^1(\Fqn)}a_{\beta,q^n}=
  a_{\infty,q^n}-\sum_{\alpha\in\Fqn}\sum_{i=0}^{e-1}\psi^i(\alpha)
  \sum_{j=1}^{s-1}\sum_{\gamma\in\Fqn}
  \phi^j\left(\gamma^{r-1}(\gamma+1)(\gamma+\alpha)\right).$$
  Changing the order of summation and replacing $\alpha$ with
  $\alpha\gamma$, the last displayed quantity is equal to
  $$a_{\infty,q^n}-\sum_{i=0}^{e-1}\sum_{j=1}^{s-1}J_{\Fqn}(\psi^i,\phi^j)^2.$$

  Note that $a_{\infty,q^n}=\gcd(e,s)-1$; $J(\psi^0,\phi^j)=0$ for
  $0<j<s$; $J(\psi^i,\phi^j)=\pm1$ when $0<i<e$, $0<j<s$; and
  $\<i/e\>+\<j/s\>\in\Z$.  We find that
  \begin{equation}\label{eq:L2}
    \sum_{\beta\in\P^1(\Fqn)}a_{\beta,q^n}=-\sum_{\substack{0<i<e \\ 0<j<s\\
        \<i/e\>+\<j/s\>\not\in\Z}}
    J_{\Fqn}(\psi^i,\phi^j)^2.  
  \end{equation}

  On the other hand,
  \begin{equation}\label{eq:L3}
    \log\prod_{o\in O}\left(1-J_o^2T^{|o|}\right)=
    -\sum_{n\ge1}\frac{T^n}n\sum_{\substack{o\text{ such that }\\|o|\text{ divides }n}}
    J_o^{2n/|o|}|o|.  
  \end{equation}

  The coefficient of $T^n/n$ can be rewritten as
  $$\sum_{\substack{(i,j)\in S\\ (q^n-1)(i,j)=(0,0)}}
  J_{\F_{q^{|o|}}}\left(\tau^{i(q^{|o|}-1)/d},\tau^{j(q^{|o|}-1)/r}\right)^{2n/|o|}.$$ 
  Using the Hasse-Davenport relation, we have
  \begin{multline*}
    \sum_{\substack{(i,j)\in S\\ (q^n-1)(i,j)=(0,0)}}
    J_{\F_{q^n}}\left(\tau^{i(q^n-1)/d},\tau^{j(q^n-1)/r}\right)^{2}
    \\=\sum_{\substack{i\in(0,e),\ j\in(0,s)\\ \<i/e\>+\<j/s\>\not\in\Z}}
    J_{\F_{q^n}}\left(\tau^{i(q^n-1)/e},\tau^{j(q^n-1)/s}\right)^{2}.
  \end{multline*}
  Therefore
  \begin{equation}\label{eq:L4}
    \sum_{\substack{o\text{ such that }\\|o|\text{ divides }n}}
    J_o^{2n/|o|}|o|
    =\sum_{\substack{i\in(0,e),\ j\in(0,s)\\ \<i/e\>+\<j/s\>\not\in\Z}}
    J_{\F_{q^n}}\left(\tau^{i(q^n-1)/e},\tau^{j(q^n-1)/s}\right)^{2}.
  \end{equation}

  Comparing \eqref{eq:L4} and \eqref{eq:L3} with
  \eqref{eq:L2} and \eqref{eq:L1} gives the desired equality.
  \qed

  \section{Ranks}
  We give a combinatorial formula for the rank of $J(K)$ where
  $K=\Fq(t^{1/d})$ for general $d$ when $q$ is sufficiently large.  We
  also consider special values of $d$ where we have better control on
  the variation of the rank with $q$.
  Recall that $K=\Fq(u)$ and $K_d=\Fp(u, \mu_d)$ where $u=t^{1/d}$.

  \subsection{The case when $r$ divides $d$ and $d=p^\nu+1$}

  \begin{cor}\label{cor:exact-rank}
    If $r$ divides $d$, $d=p^\nu+1$, and $d$ divides $q-1$, then
    $$\rk_\Z V=\rk_\Z J(\Fq(u))=\ord_{s=1}L(J/\Fq(u),s)=(r-1)(d-2).$$
    In particular, the index of $V$ in $J(K_d)$ is finite.  
    Moreover, the
    leading term of the $L$-function satisfies
    $$L^*(J/\Fq(u),1)=(\log q)^{(r-1)(d-2)}.$$
  \end{cor}

  \begin{proof}
    When $r \mid d$, then 
    $\ord_{s=1}L(J/\Fq(u),s)\le(r-1)(d-2)$ by the calculation of the
    degree of the $L$-function in \eqref{eq:L-degree}.
    Note that $K_d \subset K$ since $d \mid (q-1)$.  
    Thus we have \emph{a priori} inequalities 
    $$\rk_\Z V\le \rk_\Z J(K_d) \le \rk_\Z J(K)\le\ord_{s=1}L(J/K,s),$$
    where the right hand inequality relies on the
    known part of the BSD conjecture for abelian varieties over function
    fields, see \cite[Proposition~6.7]{CRM} for example.  We saw in
    Corollary~\ref{cor:rank-lower-bound} that $V$ has rank $(r-1)(d-2)$,
    so the inequalities are all equalities.

    For the assertion on the leading coefficient, we simply note that the
    equalities in the preceding paragraph show that 
    $$L(J/\Fq(u)),s)=\left(1-q^{1-s}\right)^{(r-1)(d-2)}.$$
    One then computes the leading term by taking the $(r-1)(d-2)$-th
    derivative. 
  \end{proof}

  \subsection{The case when $r$ and $d$ divide $p^\nu+1$}
  We have seen that the rank of $J(K_d)$ is large when $r$ divides $d$
  and $d$ has the form $p^\nu+1$.  In this subsection, we show that the
  rank is also large over various subfields of $K_d$, along the lines of
  \cite[Corollary~4.4]{Legendre}.  The case of $\Fp(t^{1/d})$ is of
  particular interest.

  We write $\varphi(e)$ for Euler's $\varphi$ function, i.e., for the
  cardinality of $(\Z/e\Z)^\times$.  If $q$ and $e$ are relatively prime
  positive integers, let $o_q(e)$ denote the order of $q$ in
  $(\Z/e\Z)^\times$.

  \begin{cor}\label{cor:smallerfldthm}
    Suppose that $r$ and $d$ divide $p^\nu+1$ for some $\nu$.  Then the
    rank of $J$ over $\Fq(t^{1/d})$ is equal to
    \[\sum_{\substack{e|d\\1<s|r}}\frac{\varphi(e)\varphi(s)}{o_q(\lcm(e,s))}
      -2\sum_{1<s|r}\frac{\varphi(s)}{o_q(s)}.\] 
    In particular, for every $p$, and every genus $g=r-1$ with $r$
    dividing $p^\nu+1$, the rank over $\Fp(u)$ of Jacobians of curves of
    genus $g$ is unbounded.
  \end{cor}

  The conclusion in the last sentence is known for every $p$ and every
  genus $g$ by \cite{Ulmer07}, but the ideas of this paper give
  a new, constructive, and relatively elementary proof.

  \begin{proof}
    Choose an integer $\nu$ such that $d$ and $r$ divide $d'=p^\nu+1$.
    Let $u^d=(u^{\prime})^{d'}=t$.  We have field containments
    $\Fq(u)\subset \Fq(\mu_{d'},u')$ and $K_{d'}=\Fp(\mu_{d'},u')\subset
    \Fq(\mu_{d'},u')$, and an equality
    $$J(\Fq(u))\tensor\Q\cong \left(J(\Fq(\mu_{d'},u'))\tensor\Q\right)^{G}$$
    where $G=\Gal(\Fq(\mu_{d'},u')/\Fq(u))$.  To
    bound $\rk J(\Fq(u))=\dim_\Q J(\Fq(u))\tensor\Q$ we just need to
    compute the dimension of a space of invariants.
    Moreover, by Corollary~\ref{cor:exact-rank}, 
    $$J(\Fq(\mu_{d'},u'))\tensor\Q=J(K_{d'})\tensor\Q.$$
    Thus, without loss we may replace $q$ with $\gcd(q,|\Fp(\mu_{d'})|)$,
    so that $\Fq(u)$ is a subfield of $K_{d'}$.  

    Our task then is to compute
    $$\dim_\Q \left(J(K_{d'})\tensor\Q\right)^{G}
    =\dim_\Q \left(V_{d'}\tensor\Q\right)^{G}$$ where
    $G=\Gal(K_{d'}/\Fq(u))$ and $V_{d'}\subset J(K_{d'})$ is the explicit
    subgroup.  We have that $V_{d'}\tensor\Q\cong R_{d'}^0/I_{d'}^0$ where
    $R^0_{d'}$ and $I^0_{d' }$ are as in Sections~\ref{s:relations} and
    \ref{sec:rational-splitting}, with $d$ replaced by $d'$.

    Now $G$ is the semi-direct product of the normal
    subgroup $d\Z/d'\Z$ by $\<q\>$, the cyclic subgroup of
    $(\Z/d'\Z)^\times$ generated by $q$.  The action of $d$ sends $P_{ij}$
    to $P_{i+d,j}$ and the action of $q$ sends $P_{ij}$ to $P_{qi,qj}$.
    Transfering this action to $R_{d'}^0/I_{d'}^0$, and noting that 
    \[\left(R_{d'}^0\right)^{d\Z/d'\Z}\cong R^0_d,\]
    we see that the dimension of $\left(R_{d'}^0/I_{d'}^0\right)^G$ is
    equal to the dimension of the $\Fr_q$-invariants on
    $$\Q[\mu_d\times\mu_r]/I_d^0$$
    where $I_d^0$ is the $\Q$-subspace of the group ring
    $\Q[\mu_d\times\mu_r]\cong\Q[\sigma,\tau]/(\sigma^d-1,\tau^r-1)$
    generated by the elements
    \begin{multline*}
      (\tau^j-1)\sum\sigma^i\ (j=1,\dots,r-1),\qquad
      (\tau^j-1)\sum\sigma^i\tau^{d-i}\ (j=1,\dots,r-1),\\
      \text{and}\qquad\sigma^i\sum\tau^j\ (i=0,\dots,d-1),
    \end{multline*}
    as in Section~\ref{s:relations}.

    Now both $\Q[\mu_d\times\mu_r]$ and $I_d$ have bases that are
    permuted by $\Fr_q$, so to compute the dimension of the space of
    invariants, we just need to count the number of orbits of $\Fr_q$ on
    the basis.  One sees easily that the space of $\Fr_q$ invariants on
    $\Q[\mu_d\times\mu_r]$ has dimension
    \[\sum_{\substack{e|d\\s|r}}\frac{\varphi(e)\varphi(s)}{o_q(\lcm(e,s))},\]
    and the space of $\Fr_q$ invariants on $I_d^0$ has dimension
    \[\sum_{e|d}\frac{\varphi(e)}{o_q(e)}+2\sum_{1<s|r}\frac{\varphi(s)}{o_q(s)}.\]
    Subtracting the last displayed quantity from the previous gives the desired
    dimension as stated in the Corollary.

    To establish the last sentence of the statement, it suffices to note
    that for a fixed $q=p$ and $r$, the dimension computed above is
    unbounded as $d$ varies through numbers of the form $p^\nu+1$
    divisible by $r$.
    Indeed, the negative terms depend only on $p$ and $r$ and the ``main''
    term in the first sum is 
    \[\phi(p^\nu+1)\phi(r)/o_p(p^\nu+1)\ge\phi(p^\nu+1)\phi(r)/(2\nu),\]
    and this is clearly unbounded as $\nu$ varies.
  \end{proof}

  \subsection{General $r$, $d$, $q$} \label{SdefAB}
  Now we treat the most general case, but with slightly less control on
  the rank as a function of $q$.  

  Recall the set $S \subset \Z/d\Z \times \Z/r\Z$ from \eqref{DdefS} in Section~\ref{sec:orbits-jacobi-sums}. We decompose
  $S$ into two disjoint pieces, $S=A\cup B$ where
  \begin{align*}
    A&=\{(i,j)\in S \mid \<i/d\>+\<j/r\>>1\},\\
    B&=\{(i,j)\in S \mid \<i/d\>+\<j/r\><1\}.
  \end{align*}
  Consider $\langle p \rangle \subset (\Z/\lcm(d,r)\Z)^\times$.
  We say that an element $(i,j) \in S$ is \emph{balanced} 
  if, for every $t \in (\Z/\lcm(d,r)\Z)^\times$, the 
  set $\langle p \rangle t (i,j)$ is evenly divided between $A$ and $B$, i.e.,
  \[|\langle p \rangle t (i,j) \cap A|= |\langle p \rangle t (i,j) \cap B|.\]
  
   %For $t\in(\Z/\lcm(d,r)\Z)^\times$, write $t\cdot o$ for
  %$\{(ti,tj)|(i,j)\in o\}$.  
  %We say that an orbit $o\in O$ is
  %\emph{balanced} if for all $t\in(\Z/\lcm(d,r)\Z)^\times$ we have
  %$$|(t\cdot o)\cap A|=|(t\cdot o)\cap B|.$$
  
  Recall that $O$ is the set of orbits of $S$ under $\langle q \rangle$.
  Note that $(i,j)$ is balanced if and only if $(qi,qj)$ is balanced.
  We say that $o \in O$ is balanced if each $(i,j) \in o$ is balanced and not balanced otherwise. 
  
  \begin{prop}\label{prop:rank}
    Let $K=\Fq(t^{1/d})$.  The order of vanishing $\ord_{s=1}L(J/K,s)$
    \textup{(}and therefore the rank of $J(K)$\textup{)} is at most the
    number of orbits $o\in O$ that are balanced in the sense above.  If
    $\Fq$ is a sufficiently large extension of $\Fp$ \textup{(}depending
    only on $d$ and $r$\textup{)}, then the rank is equal to the number
    of balanced orbits.
  \end{prop}

  This generalizes \cite[Theorem~2.2]{Legendre2} except that we have less
  control on how large $q$ should be to have equality.

  \begin{proof}
    We use the notations of the earlier parts of this chapter, in
    particular the Jacobi sums $J_o$ from \eqref{DdefJ} in Section \ref{s:L-calc}.  
    By Theorem~\ref{thm:L}, the order
    of vanishing of $L(J/K,s)$ at $s=1$ is equal to the number of orbits
    $o\in O$ such that $J_o^2=q^{|o|}$.  The proposition follows from
    the claim that $J_o$ is a root of unity times $q^{|o|/2}$ if and
    only if the orbit $o$ is balanced. Indeed, the number of $o$ such
    that $J_o^2=q^{|o|}$ is certainly at most the number of $o$ where
    $J_o$ is a root of unity times $q^{|o|/2}$, and this gives the
    asserted inequality.  Moreover, if we replace $q$ with $q^n$, each
    $J_o$ is replaced with $J_o^n$, so if $q$ is a sufficiently large
    power of $p$, any $J_o$ that is a root of unity times $q^{|o|/2}$
    satisfies $J_o^2=q^{|o|}$.  Here ``sufficiently large'' is certainly
    bounded by the degree of the $L$-function as a polynomial in $T$,
    and this is a function only of $r$ and $d$.
     
     We finish by proving the claim that $J_o$ is a root of unity times $q^{|o|/2}$ if and
    only if $o$ is balanced.  If $|o|$ is odd, then the orbit cannot be balanced and the proof below
    will show it is also impossible for $J_o$ to be a root of unity times $q^{|o|/2}$.  For this reason, we focus on the 
    case that $|o|$ is even.
    
    The argument generalizes that of
    \cite[Proposition~4.1]{Legendre2}, which is the special case $r=2$.
    Recall the set $S\subset \Z/d\Z \times \Z/r\Z$ from \eqref{DdefS} 
    and its decomposition $S=A \cup B$ as in Section \ref{SdefAB},
   where \[A=\{(i,j)\in S \mid \<i/d\>+\<j/r\>>1\}, \ B=\{(i,j)\in S \mid \<i/d\>+\<j/r\><1\}.\] 
    
 Given $(i,j) \in S$, let $i' =i/{\rm gcd}(d,i)$ and $j'=j/{\rm gcd}(r,j)$. 
 Let $d'=d/{\rm gcd}(d,i)$ and $r'=r/{\rm gcd}(r, j)$.  Let $e={\rm lcm}(d', r')$.
 Recall that 
 \[J_o=J(\chi_i,\rho_j) = J(\tau^{i(q^{|o|}-1)/d}, \tau^{j(q^{|o|}-1)/r}).\]
Thus $J_o \in \Q(\mu_e)$.  
 For $a \in (\Z/e\Z)^\times$, let $\sigma_a \in {\rm Gal}(\Q(\mu_e)/\Q)$ be the automorphism with 
 $\sigma_a(\zeta_e) = \zeta_e^a$.
 
 Let $\nu$ be such that $q^{|o|} = p^\nu$.
 Write $\mathfrak{p}$ for the prime of $\Q(\mu_e)$ 
 induced by the fixed prime $\mathfrak{p}$ of $\overline{\Q}$.
 By Stickelberger's Theorem (e.g., \cite[Thm.\ 3.6.6 and Prop.\ 2.5.14]{CohenI}), if the valuation of $p$ is $1$, 
 then the valuation of $J_o$ at the prime $\sigma_a(\mathfrak{p})$ is
\begin{equation}\label{Estickelberger}
-\nu + \sum_{\ell=0}^{\nu-1} \< \frac{a j' p^\ell}{r'} \> + \<\frac{a i' p^\ell}{d'} \> + 
\<\frac{a(-i' r' - j' d')p^\ell}{r'd'} \>.
\end{equation}
Since $J_o^2/q^{|o|}$ is a unit away from primes over $p$, it is a root of unity if and only if its valuation
at every prime over $p$ is $0$.  This is equivalent to the property that the quantity in \eqref{Estickelberger} equals
$\nu/2$ for each $a \in (\Z/e\Z)^\times$.  

Fix $a$ and $\ell$.
The sum of the three fractional parts in \eqref{Estickelberger} is either $1$ or $2$ since the three 
fractions add up to $0$.  The sum is $1$ if and only if 
$\< \frac{a j' p^\ell}{r'} \> + \<\frac{a i' p^\ell}{d'} \> < 1$.
Choose a representative $t \in \Z/{\rm gcd}(d,r)\Z$ for $a$ and note
that the fractional terms in the last inequality do not depend on this choice.
Thus the sum is $1$ if and only if $t p^\ell \cdot (i,j)$ is in $B$.

For fixed $a$, it follows that the quantity in \eqref{Estickelberger} equals $\nu/2$ if and only if
exactly half of the values of $\ell \in \{0, \ldots, \nu-1\}$ have the property that
$t p^\ell \cdot (i,j)$ is in $B$,
which is the definition of $(i,j)$ being balanced.
Thus the quantity in \eqref{Estickelberger} equals $\nu/2$ for all $a \in (\Z/e\Z)^\times$ if and only if 
$o$ is balanced.
  \end{proof}

\begin{remarks}\mbox{}
\begin{enumerate}
\item When $r=2$, it is proved in \cite[Proposition~4.1]{Legendre}
  that the order of vanishing is always the number of balanced orbits,
  i.e., there is no need to enlarge $q$.  Numerical experiments show
  that this is no longer the case for $r>2$.  It would be interesting
  to have a sharp bound on the value of $q$ needed to obtain the
  maximal rank for a given $r$ and $d$.
\item If $r$ divides $d$ and $d$ divides $p^\nu+1$, then it is easy to
  see that every orbit $o$ is balanced, and the argument in
  \cite[Section~8]{Ulmer02} shows that $J_o^2=q$ for all $o$ and any
  $q$.  Thus in this case we get an exact calculation of the rank, 
  which the reader may check agrees with
  Corollary~\ref{cor:smallerfldthm}.
\end{enumerate}
\end{remarks}

  %%% Local Variables: 
  %%% mode: latex
  %%% TeX-master: "EHR"
  %%% End: 

% !TEX root = EHR.tex
%This is chapter 6

%%%%%%%%%%%%%%%%%%%%%%%%%%%%%%%%%%%%%%%%

%%%%%%%%%%%%%%%%%%%%%%%%%%%%%%%%%%%%%%%%

\chapter{Analysis of $J[p]$ and $\NS(\XX_d)_{\tor}$} \label{cp:NS}

In this chapter, we investigate more deeply the arithmetic and
geometry of the smooth projective curve $C: y^r= x^{r-1}(x+1)(x+t)$ of
genus $g=r-1$ and its Jacobian $J$.  We prove several technical
results about the minimal regular model $\XX$ and the \Neron{} model
$\JJ\to\P^1$.  Specifically, in Section~\ref{s:KS}, we analyze the
Kodaira-Spencer map to show that the Jacobian $J$ of $C$ has no
$p$-torsion over any separable extension of $K=\Fq(t)$ (see
Corollary~\ref{cor:ptors}).  In Section~\ref{S:NStf}, we prove that the
\Neron-Severi group of $\XX_d$ is torsion-free (see
Theorem~\ref{thm:NS-tor}).  These results will be used in
Chapter~\ref{ch:Sha} to understand the index of the visible subgroup
$V$ in $J(K_d)$.% in certain situations.

%%%%%%%%%%%%%%%%%%%%%%%%%%%%%%%%%%%%%%%%

\section{Kodaira-Spencer and $p$-torsion}\label{s:KS}

Our goal in this section is to show that the Jacobian $J$ of $C$ has
no $p$-torsion over any separable extension of $K=\Fq(t)$, a result
stated more formally as follows:

\begin{cor}\label{cor:ptors}
The $p$-torsion of $J$ satisfies
$J(K)[p]=J(K^{sep})[p]=0$.
\end{cor}

To prove Corollary~\ref{cor:ptors}, we apply a result of Voloch after
showing that the Kodaira-Spencer map of the \Neron{} model $\JJ\to\Pt$
is generically an isomorphism and that $J$ is ordinary.

%%%%%%%%%%%%%%%%%%%%%%%%%%%%%%%%%%%%%%%%

\subsection{Background on the Kodaira-Spencer map 
and $p$-torsion}\label{ss:KS-background}
Before launching into the technicalities of the proof of
Corollary~\ref{cor:ptors}, we provide some background on the
Kodaira-Spencer map and its connection to $p$-torsion of abelian
varieties.  This subsection is purely motivational and nothing in it
will be used later in the paper.  Thus the expert or impatient reader
may skip directly to Subsection~\ref{ss:KS}.

Consider a non-isotrivial elliptic curve $E$ over a function field
$K=\Fq(\CC)$ of characteristic $p$.  It is known
\cite[Prop.~I.7.3]{UlmerPCMI} that if $E(K)[p]$ is non-trivial, then
the $j$-invariant of $E$ is a $p$-th power, i.e., $j(E)\in K^p$.
Since $E$ is non-isotrivial, the $j$-invariant is non-constant, and it
is a $p$-th power if and only if the morphism $j:\CC\to\P^1$ is
inseparable, if and only if its derivative vanishes identically.
Passing to the contrapositive, we see that if the derivative of
$j:\CC\to\P^1$ does not vanish identically, then $E(K)$ has no
non-trivial $p$-torsion.  In \cite{Voloch}, Voloch extends this result
to higher-dimensional abelian varieties, where ``derivative of $j$''
is replaced with a suitable Kodaira-Spencer map.

Next, we give a brief overview of the Kodaira-Spencer map for a family
of abelian varieties.  More precisely, let $B$ be a smooth (possibly
non-projective) curve over an algebraically closed field $k$ and let
$\pi:\AC\to B$ be an abelian scheme over $B$.  Fix a closed point
$b\in B$ and let $A$ be the fiber $\pi^{-1}(b)$.  There is an exact
sequence of tangent sheaves on $A$:
$$0\to T_A\to \left.(T_{\AC})\right|_A\to \left.(\pi^*(T_{B}))\right|_A \to 0.$$
Taking cohomology, we see that
$H^0(A,\left.(\pi^*(T_{B}))\right|_A)\cong T_{B,b}$ as $k$-vector
spaces, and the coboundary map is a homomorphism
\begin{equation}\label{eq:KS-first-version}
T_{B,b}\to H^1(A,T_A).
\end{equation}
This should be thought of as the derivative at the point $b$ 
of the map from $B$ to the
moduli space of abelian varieties.  Indeed,
$H^1(A,T_A)$ is the space of first-order deformations of $A$, i.e.,
the tangent space to the moduli space at $A$, and the map
\eqref{eq:KS-first-version} measures the variation of the family
$\pi:\AC\to B$ at the point $b$.  (For more details, we suggest the
following references: In \cite[Ch.~4]{Kodaira}, one of the originators
of the theory explains the interpretation of the $H^1$ and the map
above in the context of complex varieties; \cite[III.9.1]{Voisin}
gives a compact but clear presentation of the same material; and
\cite[III.9.13.2]{hartshorne} gives the interpretation of the $H^1$ in
the context of schemes.)

We now reformulate \eqref{eq:KS-first-version}.  The
tangent bundle to an abelian variety is trivial, so
$$H^1(A,T_A)\cong H^1(A,T_{A,0}\tensor_k\OO_A)\cong 
T_{A,0}\tensor_k H^1(A,\OO_A).$$
Thus the map \eqref{eq:KS-first-version} can be rewritten as an
element of
$$\Hom_k\left(T_{B,b},T_{A,0}\tensor_k H^1(A,\OO_A)\right)
\cong
\Hom_k\left(\Omega^1_{A,0},
\Omega^1_{B,b}\tensor_k H^1(A,\OO_A)\right).$$

Next, we note that $\Omega^1_{A,0}$ is canonically isomorphic to
the stalk of $\pi_*\Omega^1_{\AC/B}$ at $b$, and $H^1(A,\OO_A)$
is the stalk at $b$ of $R^1\pi_*\OO_{\AC}$.
If we now let the point $b$ vary, the last description of the
derivative \eqref{eq:KS-first-version} globalizes to a morphism
$$KS:\pi_*\Omega^1_{\AC/B}\to
\Omega^1_B\tensor_{\OO_B} R^1\pi_*\OO_{\AC}$$
of $\OO_B$-modules.

Voloch's theorem then states that if the Kodaira-Spencer map $KS$ is
generically an isomorphism (i.e., an isomorphism over a dense open
subset of $B$), then the generic fiber of $\AC$ over $B$ has no
$p$-torsion points over the field $k(B)$.

In the next subsection, we restart from the beginning, defining the
Kodaira-Spencer map for our context, and in the following subsections
we prove that it is generically an isomorphism.

%%%%%%%%%%%%%%%%%%%%%%%%%%%%%%%%%%%%%%%%

\subsection{The Kodaira-Spencer map for $\JJ$ and $\YY$}\label{ss:KS}

In this subsection we work over $\Fq(u)$ where $u^d=t$ and $r$ and $d$
are relatively prime to $p$.  We make no further assumptions on $r$,
$d$, or $q$.

Let $U\subseteq\Pu$ be the open subset where
$u^d\not\in\{0,1,\infty\}$.  In Section~\ref{ss:YY} we constructed a
proper smooth model $\pi:\YY\to U$ of $C/\F_q(u)$, i.e., a scheme with
a proper smooth morphism to $U$ whose generic fiber is $C$.  The
N\'eron model $\sigma:\JJ\to U$ is an abelian scheme whose fiber over
a point of $U$ is just the Jacobian of the fiber of $\pi$ over that
point.

%%%%%%%%%%%%
Consider the sheaves of relative differentials (of 1-forms)
$$
	\Omega^1_U,\ \Omega^1_\JJ,\ \Omega^1_{\JJ/U}
$$
on the schemes $U/\Fq$, $\JJ/\Fq$, and $\JJ/U$ respectively (see
\cite[\S 6.1.2]{Liu}).  The following lemma, applied with
$S=\spec(\Fq)$ and $f=\sigma$, implies there is an exact sequence of
locally free $\OO_\JJ$-modules
\begin{equation}\label{eqn:differential-exact-sequence}
  0 \to \sigma^*\Omega^1_U \to \Omega^1_{\JJ} \to \Omega^1_{\JJ/U} \to 0
\end{equation}
since $\JJ/\Fq$, $U/\Fq$, and $\sigma$ are smooth and of finite type.

\begin{lemma}\label{lem:differential-exact-sequence}
  Let $X$, $Y$, and $S$ be locally Noetherian schemes and let
  $f\colon X\to Y$ and $g\colon Y\to S$ be smooth morphisms of finite
  type.  
%If the fibers of $f$ and $g$ are equidimensional of respective dimensions $m$ and $n$,
  Then there is an exact sequence
$$
	0\to f^*\Omega^1_{Y/S} \to \Omega^1_{X/S} \to \Omega^1_{X/Y} \to 0
$$
of locally free $\OO_X$-modules.
\end{lemma}

\begin{proof}
First, there is an exact sequence
\begin{equation}\label{eqn:basic-differential-exact-sequence}
	f^*\Omega^1_{Y/S} \to \Omega^1_{X/S} \to \Omega^1_{X/Y} \to 0
\end{equation}
of $\OO_X$-modules since $X,Y,S$ are all schemes (see
\cite[6.1.24]{Liu}).  We must show that the terms of this sequence are
all locally free as $\OO_X$-modules and that the first map is
injective.

Let $x\in X$ be a geometric point, and let $y=f(x)$ and $s=g(y)$.
Then the fibers 
$$
	\Omega^1_{X/S,x},\ \Omega^1_{Y/S,y},\ \Omega^1_{X/Y,x}
$$
are smooth of ranks
$$
	\dim_x X_s,\ \dim_y Y_s,\ \dim_x X_y
$$
respectively, since $gf$, $g$, $f$ are smooth (see \cite[6.2.5]{Liu}).
Hence $\Omega^1_{X/S}$ and $\Omega^1_{X/Y}$ are locally free
$\OO_X$-modules and $f^*\Omega^1_{Y/S}$ is a locally free
$f^*\OO_Y$-module (and thus a locally free $\OO_X$-module).  Finally, 
$$
	\dim_x X_s = \dim_x X_y + \dim_y Y_s
$$
so the first map of \eqref{eqn:basic-differential-exact-sequence} is
injective as claimed. 
\end{proof}
%%%%%%%%%%%%%

Taking the direct image of \eqref{eqn:differential-exact-sequence}
under $\sigma$ and applying the projection formula (see
\cite[5.2.32]{Liu}) leads to a morphism
$$
  \KSJ : \sigma_*\Omega^1_{\JJ/U} \to \RoneOJtensorOmegaU
$$
which is the ``Kodaira-Spencer map'' of the family $\sigma:\JJ\to U$.
Similarly, Lemma~\ref{lem:differential-exact-sequence} implies there
is an exact sequence of $\OO_\YY$-modules 
\begin{equation}\label{eq:Omega_Y}
  0 \to \pi^*\Omega^1_U \to \Omega^1_\YY \to \Omega^1_{\YY/U} \to 0
\end{equation}
and a morphism
$$
  \KSY : \pi_*\Omega^1_{\YY/U} \to \RoneOYtensorOmegaU.
$$
The main technical point of this section is the following.

\begin{theorem}\label{thm:KS}
  The maps $\KSJ$ and $\KSY$ are isomorphisms of locally free
  $\OO_U$-modules of rank $r-1$.
\end{theorem}

\noindent
The proof is given in the remaining part of this section.  The
key point is to explicitly calculate the ``Kodaira-Spencer
pairing'' on
$$
	H^0(U,\Omega^1_{\YY/U})\times H^0(U,\Omega^1_{\YY/U})
$$
and to show that it is non-degenerate.

Our motivation for considering the Kodaira-Spencer map $\KSJ$ is that
we use Theorem~\ref{thm:KS} to prove Corollary~\ref{cor:ptors}.

\begin{proof}[Proof that Theorem~\ref{thm:KS} implies
  Corollary~\ref{cor:ptors}] 
  The statement over $K$ follows from that over $K^{sep}$.  The latter
  follows from \cite[Page~1093, Proposition]{Voloch}, which says that
  if an abelian variety over a global function field is ordinary and
  its Kodaira-Spencer map is generically an isomorphism, then it has
  no $p$-torsion over any separable extension.
  Proposition~\ref{prop:ordinary} (see Section~\ref{ss:ordinary})
  states that $J$ is ordinary, and Theorem~\ref{thm:KS} states that
  the Kodaira-Spencer map is generically an isomorphism.
\end{proof}

\begin{remark} 
  The proof that $J$ has no $p$-torsion over $K^{sep}$ via
  Kodaira-Spencer is not so simple. The ideas of
  \cite[9.4]{Legendre3}, also not so simple, yield a proof that $J$
  has no $p$-torsion over $\Fq(u)$ where $u^d=t$ and $d=p^\nu+1$.  The
  more straightforward idea of using $p$-descent (i.e., calculating
  the $p$-Selmer group and comparing with the rank) is simpler, but
  yields a much weaker result, namely that $J$ has no torsion over
  $\Fq(u)$ where $u^d=t$ and $d=p^\nu+1$ with $\nu\le2$.  As soon as
  $\nu>2$, the $p$-part of the Tate-Shafarevich group is non-trivial
  and the $p$-descent strategy fails.
\end{remark}

%%%%%%%%%%%%%%%%%%%%%%%%%%%%%%%%%%%%%%%%

\subsection{Reductions to $\YY$}

The following statement is probably
well-known, but we have not found a suitable reference.

\begin{prop}\label{prop:KS-reduction-to-YY}
  There are isomorphisms
  $$
    \sigma_*\Omega^1_{\JJ/U}\cong\pi_*\Omega^1_{\YY/U}
    \quad\text{and}\quad
    R^1\sigma_*\OO_{\JJ}\cong R^1\pi_*\OO_{\YY}
  $$ 
   of locally free $\OO_U$-modules of rank $g$ 
  such that the following diagram commutes:
  $$\xymatrix{
    \sigma_*\Omega^1_{\JJ/U}\ar[r]^{\KSJ\qquad}\ar[d]
  	& \RoneOJtensorOmegaU\ar[d] \\
    \pi_*\Omega^1_{\YY/U}\ar[r]^{\KSY\qquad}
  	& \RoneOYtensorOmegaU.
  }$$
  In particular, $\KSJ$ is an isomorphism of $\OO_U$-modules if and
  only if $\KSY$ is.
\end{prop}

\begin{proof}
\newcommand\AJ{\mathrm{AJ}}
  The map $\pi:\YY\to U$ admits a section $U\to\YY$ since it is proper and
  its generic fiber $C$ has a rational point.  The section can be used to
  construct a map $\AJ:\YY\to\JJ$, the so-called Abel-Jacobi map.  It is a
  closed immersion (cf.~\cite[Proposition~2.3]{MilneJV}), and thus $\AJ_*$ is
  exact.  Therefore there are isomorphisms of $\OO_U$-modules
  $$ R^i\sigma_*(\AJ_*\OO_{\YY})\cong R^i\pi_*\OO_{\YY},\quad
     \sigma_*(\AJ_*\Omega^1_{\YY/U})\cong\pi_*\Omega^1_{\YY/U},
  $$
  isomorphisms of $\OO_{\JJ}$-modules
  $$ \AJ_*(\pi^*\Omega^1_U)\cong \AJ_*(\AJ^*(\sigma^*\Omega^1_U))
     \cong\sigma^*\Omega^1_U\otimes_{\OO_\JJ}\AJ_*\OO_{\YY},
  $$
  and an exact sequence of $\OO_{\JJ}$-modules
  $$ 0\to \sigma^*\Omega^1_U\otimes_{\OO_\JJ}\AJ_*\OO_{\YY}
      \to \AJ_*\Omega^1_\YY\to \AJ_*\Omega^1_{\YY/U}\to0.
  $$
  
  The structure map $\OO_\JJ\to \AJ_*\OO_{\YY}$ associated to
  $\AJ:\YY\to\JJ$ induces a morphism
  $R^1\sigma_*\OO_{\JJ}\to R^1\sigma_*\left(\AJ_*\OO_\YY\right)$ and
  thus a morphism
  $$ R^1\sigma_*\OO_{\JJ}\to R^1\pi_*\OO_\YY. $$
  Similarly, the pull-back map on 1-forms
  $\Omega^1_{\JJ/U}\to \AJ_*\Omega^1_{\YY/U}$ induces a morphism
  $$ \sigma_*\Omega^1_{\JJ/U}\to\pi_*\Omega^1_{\YY/U}. $$ 
  Both displayed morphisms are isomorphisms of locally free
  $\OO_U$-modules of rank $g$ since the respective fibers at each
  $x\in U$ are isomorphisms of $g$-dimensional vector spaces
  (cf.~\cite[Proposition~2.1 and Proposition~2.2]{MilneJV}).
  
  The displayed exact sequence lies in a commutative diagram
  $$\xymatrix{
	0\ar[r] & \sigma^*\Omega^1_U\ar[r]\ar[d]
		& \Omega^1_\JJ\ar[r]\ar[d] & \Omega^1_{\JJ/U}\ar[r]\ar[d] & 0 \\
	0\ar[r] & \sigma^*\Omega^1_U\otimes_{\OO_U}\AJ_*\OO_{\YY}\ar[r]
		& \AJ_*\Omega^1_\YY\ar[r]&\AJ_*\Omega^1_{\YY/U}\ar[r] & 0
  }$$
  of $\OO_{\JJ}$-modules whose first row is exact and where the right two
  vertical maps are pull-back maps on 1-forms.  Applying $\sigma_*$, 
  the projection formula, and the isomorphisms displayed above yields 
  a commutative diagram whose rows are long exact sequences of
  $\OO_U$-modules and a portion of which is the desired diagram
  $$\xymatrix{
  \sigma_*\Omega^1_{\JJ/U}\ar[r]\ar[d]
  	& \RoneOJtensorOmegaU\ar[d] \\
  \pi_*\Omega^1_{\YY/U}\ar[r] & \RoneOYtensorOmegaU.
  }$$
\end{proof}

%%%%%%%%%%%%%%%%%%%%%%%%%%%%%%%%%%%%%%%%

\subsection{Reduction to the Kodaira-Spencer pairing}

For the rest of this section, we suppose that $d=1$.  This suffices to
prove Theorem~\ref{thm:KS} since $U$ is an \'etale cover of
$\Pt\smallsetminus\{0,1,\infty\}$.

Rather than showing that $\KSY$ is an isomorphism directly, it is
more convenient for us to consider the ``Kodaira-Spencer pairing'' on
global 1-forms 
$$\begin{array}{ccc}
  H^0(U,\pi_*\Omega^1_{\YY/U}) \times H^0(U,\pi_*\Omega^1_{\YY/U})
  	& \longrightarrow &  H^0(U,\Omega^1_U)\cong R\,dt \\
  \omega_i \times \omega_j
    & \longmapsto & (\omega_i,\omega_j)\qquad
\end{array}$$
where $R=H^0(U,\OO_U)$.  The pairing is defined by taking the cup
product 
$$
  \KSY(\omega_i)\cup\omega_j \in H^0(U,\RoneOmegaYtensorOmegaU)
$$
followed by the map
$$
  H^0(U,\RoneOmegaYtensorOmegaU) \quad \isoto \quad
  H^0(U,\Omega^1_U\tensor_{\OO_U}\OO_U)\cong H^0(U,\Omega^1_U)
$$
induced by the relative trace
$$
	R^1\pi_*\Omega^1_{\YY/U}\quad\isoto\quad \OO_U.
$$
In particular, to show that $\KSY$ is an isomorphism is the same as to
show that the Kodaira-Spencer pairing is a perfect pairing of free
$R$-modules.  Proposition~\ref{prop:KS-reduction-to-YY} then implies
that $\KSJ$ is an isomorphism, completing the proof of
Theorem~\ref{thm:KS}.

After some preparatory material, a proof that the pairing is perfect is
given in Section~\ref{ss:ks-pairing}.

%%%%%%%%%%%%%%%%%%%%%%%%%%%%%%%%%%%%%%%%

\subsection{Relative 1-forms}\label{ss:1-forms-on-C}

Recall that $d=1$ and thus $R=H^0(U,\OO_U)=\F_q[t][1/(t(t-1))]$.  Recall also
that $C$ is the smooth proper curve over $K$ associated to the affine curve
$y^r=x^{r-1}(x+1)(x+t)$, that $\YY\to U$ is a proper smooth map with generic
fiber $C$, and that $\YY$ is covered by the sets
\begin{align*}
    \YY_1&:=
\spec\left(R[x_{11},y_{11}]/(y_{11}-x_{11}^{r-1}(x_{11}y_{11}+1)(x_{11}y_{11}+t))\right),\\
    \YY_2&:=\spec\left(R[x_2,z_2]/(z_2-x_2^{r-1}(x_2+z_2)(x_2+tz_2))\right), \\
    \YY_3&:=\spec\left(R[y_3,z_3]/(y_3^rz_3-(1+z_3)(1+tz_3))\right)
\end{align*}
(cf.~Section~\ref{ss:YY}).  These coordinates are related by the identities
$$
	(x,y) = (x_{11}y_{11},y_{11}) = (x_2/z_2,1/z_2) = (1/z_3,y_3/z_3).
$$

For $1\le i\le r-1$, the expression $x^{i-1}dx/y^i$ corresponds to a unique
meromorphic 1-form $\omega_i$ on $\YY$.  The 
restrictions of $\omega_i$ to the open sets $\YY_1$, $\YY_2$, $\YY_3$ are
$$ \frac{x_{11}^idy_{11}}{y_{11}} + x_{11}^{i-1}dx_{11},\quad
   x_2^i\left(\frac{dx_2}{x_2}-\frac{dz_2}{z_2}\right),\quad
   -\frac{dz_3}{y_3^iz_3}
$$
respectively.  These are clearly non-zero forms.

\begin{lemma}\label{lem:regular-omega_i}
  $\omega_i$ is everywhere regular, i.e., is an element of
  $H^0(\YY,\Omega^1_{\YY/U})$.
\end{lemma}

\begin{proof}
  On the one hand, $\omega_i$ is clearly regular on $\YY_1$ away from
  $y_{11}= 0$.  On the other hand, if $y_{11}=0$, then
  $dy_{11}/y_{11}$ has a pole of order one and $x_{11}^i$ vanishes to
  order $i$, so $\omega_i$ is regular on $\YY_1$.  For use just below,
  we record that $\omega_i$ vanishes to order exactly $i-1$ along the
  divisor $x_{11}=0$.

  Similarly, if $x_2$ or $z_2$ vanish, then both vanish and $\omega_i$
  is regular on $\YY_2$.

  Finally, $z_3$ never vanishes, and if $y_3=0$, then the identity
  \[
  	  ry_3^{r-1}z_3dy_3 + (y_3^r - (1+tz_3) - t(1+z_3))dz_3
	  = 0
  \]
  on $\YY_3$ shows that $dz_3$ has a zero of order at least $r-1$.
  More precisely, the coefficient of $dz_3$ is a unit in a
  neighborhood of $y_3=0$ while the coefficient of $dy_3$ has a zero
  of order at least $r-1$.  Hence $\omega_i$ is regular on $\YY_3$.
\end{proof}

\begin{lemma}\label{lemma:R-basis-of-1-forms}
  The relative 1-forms $\omega_i$ form an $R$-basis of
  $H^0(\YY,\Omega^1_{\YY/U})$.
\end{lemma}

\begin{proof}
  There is an isomorphism
  $H^0(\YY,\Omega^1_{\YY/U})\cong H^0(U,\pi_*\Omega^1_{\YY/U})$.
  Since $\pi$ is a family of smooth projective curves of genus
  $g=r-1$, the sheaf $\pi_*\Omega^1_{\YY/U}$ is a locally free sheaf
  of $\OO_U$-modules of rank $r-1$ whose fiber at a closed point
  $u\in U$ is $H^0(\pi^{-1}(u),\Omega^1_{\pi^{-1}(u)/\kappa(u)})$.
  This last is a vector space of dimension $r-1$ over the residue
  field $\kappa(u)$, and to prove the lemma it suffices to show that
  the images of the $\omega_i$ in
  $H^0(\pi^{-1}(u),\Omega^1_{\pi^{-1}(u)/\kappa(u)})$ form a
  $\kappa(u)$ basis for all $u\in U$.  But, as mentioned above,
  $\omega_i$ has a zero of order $i-1$ at the point $x_{11}=y_{11}=0$
  in each fiber, so the restrictions of the $\omega_i$ to the fibers
  are linearly independent.  Since there are $r-1$ of them, they form
  a basis.
\end{proof}

Let $\Cbar=C\times_K\Kbar$.

\begin{cor}\label{cor:Kbar-basis-of-1-forms}
  The relative 1-forms $\omega_1,\ldots,\omega_{r-1}$ form a
  $\K$-basis of $H^0(C,\Omega^1_{C/K})$ and a $\Kbar$-basis of
  $H^0(\Cbar,\Omega^1_{\Cbar/\Kbar})$.
\end{cor}

\begin{proof}
This is immediate from Lemma~\ref{lemma:R-basis-of-1-forms} since 
$$H^0(C,\Omega^1_{C/K})\cong H^0(\YY,\Omega^1_{\YY/U})\tensor_R K$$
and
$$H^0(C,\Omega^1_{\Cbar/\Kbar})\cong 
H^0(\YY,\Omega^1_{\YY/U})\tensor_R \Kbar.$$
\end{proof}

%%%%%%%%%%%%%%%%%%%%%%%%%%%%%%%%%%%%%%%%

\subsection{Lifting 1-forms}

Recall that there is an exact sequence of $\OO_\YY$-modules
$$
  0 \to \pi^*\Omega^1_U \to \Omega^1_\YY \to \Omega^1_{\YY/U} \to 0
$$
and that $\YY_1\cup\YY_2\cup\YY_3$ is an open affine cover of
$\YY\to U$.  In this subsection, we regard $\omega_i$ as a section in
$H^0(\YY,\Omega^1_{\YY/U})$ and find, for each $j=1,2,3$, a lift of
$\omega_i$ to a section in $H^0(\YY_j,\Omega^1_{\YY})$ so that we can
calculate $\KSY(\omega_i)$.

\begin{prop}\label{prop:lifts}
  The 1-forms 
  $$
	\frac{x_{11}^idy_{11}}{y_{11}}+x_{11}^{i-1}dx_{11},\quad
	x_2^i\left(\frac{dx_2}{x_2}-\frac{dz_2}{z_2}\right),\quad
	-\frac{dz_3}{y_3^iz_3}-\frac{1+z_3}{y_3^i}\frac{dt}{t-1}.
  $$
are sections in $H^0(\YY_j,\Omega^1_{\YY})$ for $j=1,2,3$
respectively, and each lifts $\omega_i$.
\end{prop}

\noindent
The proof occupies the remainder of this subsection.

First consider $\YY_1$, where (dropping subscripts) there is an equality
\begin{equation}\label{eq:X1}
  0 = y - x^{r-1}(xy+1)(xy+t),
\end{equation}
the differential of which leads to the relation
\begin{multline}\label{eq:dX1}
 0 = \left(1-x^r(xy+t)-x^r(xy+1)\right)dy
   - x^{r-1}(xy+1)\,dt \\
   - x^{r-2}\left((r-1)(xy+1)(xy+t)+xy(xy+t)+(xy+1)xy\right)dx.
\end{multline}
Now consider the naive lift of $\omega_i$ to a 1-form on $\YY_1$
$$
  \frac{x^{i-1}d(xy)}{y}=\frac{x^idy}{y}+x^{i-1}dx.
$$
This is obviously regular away from $y=0$.  The equality \eqref{eq:X1}
shows that, in an open neighborhood of $y=0$, the function $y$ is a
unit times $x^{r-1}$.  Also, the coefficients of $dx$ and $dt$ in
\eqref{eq:dX1} are divisible by $x^{r-2}$ and, near $y=0$, the
coefficient of $dy$ is a unit.  Therefore, we may rewrite $x^idy$
(with $i\ge1$) as a regular 1-form times $x^{r-1}$, and thus
$x^{i}dy/y$ is everywhere regular on $\YY_1$.  This shows that the
naive lift of $\omega_i$ is a section in $H^0(\YY_1,\Omega^1_{\YY})$.

Next we turn to $\YY_2$, where (dropping subscripts) there is an
equality
\begin{equation}\label{eq:X2}
  0 = z - x^{r-1}(x+z)(x+zt),
\end{equation}
the differential of which leads to the relation
\begin{multline}\label{eq:dX2}
  0 = \left(1 - x^{r-1}(x+zt) - x^{r-1}(x+z)t\right)dz 
    - x^{r-1}(x+z)z\,dt \\
    - x^{r-2}\left((r-1)(x+z)(x+zt) + x(x+zt) + x(x+z)\right)dx.
\end{multline}
Now consider the naive lift of $\omega_i$ to a 1-form on $\YY_2$:
$$
 \frac{x^{i-1}d(x/z)}{1/z} = x^i\left(\frac{dx}{x}-\frac{dz}{z}\right).
$$
This is obviously regular away from $z=0$.  Near $z=0$, the equality
\eqref{eq:X2} shows that $z$ is a unit times $x^{r+1}$.  Also, near
$z=0$, the coefficient of $dz$ in \eqref{eq:dX2} is a unit and the
coefficients of $dx$ and $dt$ are divisible by $x^{r}$.  Therefore, we
may rewrite $x^idz$ (with $i\ge1$) as a regular 1-form times
$x^{r+1}$, and thus $x^idz/z$ is everywhere regular on $\YY_2$.  This
shows that the naive lift of $\omega_i$ is a section in
$H^0(\YY_2,\Omega^1_{\YY})$.

Finally, we turn to $\YY_3$, where (dropping subscripts) there is an
equality
\begin{equation}\label{eq:X3}
  0 = y^rz-(1+z)(1+tz),
\end{equation}
the differential of which leads to the relation
\begin{equation}\label{eq:dX3}
  0 = \left(ry^{r-1}z\right)dy
    + \left(y^r-(1+tz)-(1+z)t\right) dz
    - \left((1+z)z\right)dt.
\end{equation}
This time it is necessary to work harder since the naive lift of
$\omega_i$ turns out not to be regular on all of $\YY_3$.  Instead of
it, we add a term involving $dt$ and consider
$$
  \frac{-dz}{y^iz}-\frac{1+z}{y^i}\frac{dt}{t-1}.
$$
This is regular where $y\neq 0$ since $t-1$ and $z$ are units on
all of $\YY_3$, so it remains to show it is regular in a neighborhood
of $y=0$.  The equations \eqref{eq:X3} and \eqref{eq:dX3}
and some algebra allow us to rewrite this lift as
$$
  \frac{ry^{r-i-1}}{f}\,dy
    - \frac{1+z}{y^i}\left(\frac{1}{f}+\frac{1}{t-1}\right)\,dt
  = \frac{y^{r-1-i}}{f}\left(r\,dy + \frac{y(z-1)}{t-1}\,dt\right)
$$
where $f=y^r-(1+tz)-(1+z)t$.  The right side is regular in a
neighborhood of $y=0$ since then $t-1$ and $f$ are units.  Therefore
this lift of $\omega_i$ is a section in $H^0(\YY_3,\Omega^1_\YY)$.

%%%%%%%%%%%%%%%%%%%%%%%%%%%%%%%%%%%%%%%%

\subsection{Computing the Kodaira-Spencer pairing}
\label{ss:ks-pairing}

In this section we calculate the pairing
$$
  \begin{array}{ccc}
    H^0(U,\pi_*\Omega^1_{\YY/U}) \times H^0(U,\pi_*\Omega^1_{\YY/U})
      & \to & H^0(U,\Omega^1_U) \\
  \omega_i \times\omega_j
      & \mapsto & (\omega_i,\omega_j).
  \end{array}
$$
The proof of the following proposition occupies the remainder of this
subsection:

\begin{prop} \label{P:ww}
  $(\omega_i,\omega_j) = \frac{r\,dt}{t(t-1)}$ if $i+j=r$,
  and otherwise $(\omega_i,\omega_j)=0$.
\end{prop}

\noindent
In particular, Proposition~\ref{P:ww} and
Corollary~\ref{lemma:R-basis-of-1-forms} together imply that the
pairing is a perfect pairing of free $R$-modules since $r/t(t-1)$ is a
unit in $R$.

Recall that $H^i(\YY,\mathcal{F})\cong H^0(U,R^i\pi_*\mathcal{F})$ for
any coherent sheaf $\mathcal{F}$ on $\YY$ since $U$ is affine.  Recall
also that there is a long exact sequence of $\OO_U$-modules
$$
	\cdots\to R^i\pi_*\pi^*\Omega^1_U
	      \to R^i\pi_*\Omega^1_\YY
	      \to R^i\pi_*\Omega^1_{\YY/U}
	      \to\cdots
$$
obtained by applying $\pi_*$ and its right derived functors to
\eqref{eq:Omega_Y}.
Therefore the corresponding sequence of global sections
$$
	\cdots\to H^0(U,R^i\pi_*\pi^*\Omega^1_U)
	      \to H^0(U,R^i\pi_*\Omega^1_\YY)
	      \to H^0(U,R^i\pi_*\Omega^1_{\YY/U})
	      \to\cdots
$$
is equal to the long exact cohomology sequence
$$
	\cdots\to H^i(\YY,\pi^*\Omega^1_U)
	      \to H^i(\YY,\Omega^1_\YY)
	      \to H^i(\YY,\Omega^1_{\YY/U})
	      \to\cdots.
$$
In particular, $\KSY$ induces a map
$$
	H^0(U,\pi_*\Omega^1_{\YY/U})\to H^1(U,R^1\pi_*\pi^*\Omega^1_U),
$$
which is the boundary map of cohomology
$$
	H^0(\YY,\Omega^1_{\YY/U})\to H^1(\YY,\pi^*\Omega^1_U).
$$

Fixing $i$ and taking differences, on $\YY_j\cap\YY_k$, for
$j,k\in\{1,2,3\}$, of the lifts in Proposition~\ref{prop:lifts} yields the
following \czech{} cocycle in $H^1(\YY,\pi^*\Omega^1_U$)
representing $\KSY(\omega_i)$:
$$ g_{12} = g_{21} = 0, \qquad
   g_{23} = -g_{32} = g_{13} = -g_{31} = \frac{1+z_3}{y_3^i}\frac{dt}{t-1}
$$
where $g_{jk}$ is a section in $H^0(\YY_j\cap\YY_k,\pi^*\Omega^1_U)$.

Taking the cup product of $\KSY(\omega_i)$ with $\omega_j$ yields a
class in
\begin{align*}
  H^1(\YY,\pi^*\Omega^1_U\otimes_{\OO_Y}\Omega^1_{\YY/U})
  & \cong H^0(U,\RoneOmegaYtensorOmegaU) \\
  & \cong H^0(U,\Omega^1_U)\otimes_R H^0(U,R^1\pi_*\Omega^1_{\YY/U})
\end{align*}
given by the product of $\frac{dt}{t-1}$ and the class $h$ in
$H^0(U,R^1\pi_*\Omega^1_{\YY/U})$ represented by the \czech{} cocycle
$$ h_{12} = h_{21} = 0, \qquad
   h_{23} = -h_{32} = h_{13} = -h_{31}
   = \frac{1+z_3}{y_3^{i+j}}\frac{dz_3}{z_3}.
$$
It remains to calculate the image of $h$ via the relative trace map
$$
  H^0(U,R^1\pi_*\Omega^1_{\YY/U}) \to H^0(U,\OO_U).
$$

Consider, for $j=1,2,3$, the \emph{meromorphic} relative 1-forms
$\sigma_j$ on $\YY_j$ given by
$$
  \sigma_1 = \sigma_2 = 0, \qquad
  \sigma_3 = -\frac{1+z_3}{y_3^{i+j}}\frac{dz_3}{z_3}.
$$
On $\YY_j\cap\YY_k$, they satisfy $h_{jk}=\sigma_j-\sigma_k$.
Therefore, for $z\in U$ and $P\in\YY_{j,z}$, the residue
$r_P=\res_P(\sigma_j)$ satisfies $r_P=\res_P(\sigma_k)$ if
$P\in\YY_{k,z}$.  In particular, the relative trace of $h$ is the
global section of $\OO_U$ whose restriction to $\OO_{U,z}$ is
$ \sum_{P\in\YY_z}r_P.  $

It is clear that $r_P=0$ except possibly at the points
$(y_3,z_3)=(0,-1)$ and $(0,-1/t)$ in $\YY_{3,z}$.  The identities
$$
	y_3^rz_3 - (1+z_3)(1+tz_3) = 0 \quad \text{and} \quad
	ry_3^{r-1}z_3\,dy_3 = (y_3^r - (1+tz_3) - t(1+z_3))\,dz_3
$$
allow us to rewrite $\sigma_3$ as
$$
	-\frac{rz_3(1+z_3)}{1-z_3^2t}\frac{dy_3}{y_3^{1+i+j-r}}
	= -\frac{rz_3^2}{(1+tz_3)(1-z_3^2t)}\frac{dy_3}{y_3^{1+i+j-2r}}.
$$
The specializations of the left and right at $z_3=-1/t$
and $z_3=-1$ are
$$
	-\frac{r(-1/t)(1+(-1/t))}{1-(1/t^2)t}\frac{dy_3}{y_3^{1+i+j-r}}
	= \frac{r}{t}\frac{dy_3}{y_3^{1+i+j-r}}$$
	and
$$	-\frac{r}{(1-t)(1-t)}\frac{dy_3}{y_3^{1+i+j-2r}}
$$
respectively.  In particular, if $P\in\YY_3$, then
$$
	r_P = \begin{cases}
      \frac{r}{t} & \text{if $P=(0,-1/t)$ and $i+j=r$,} \\
      0           & \text{otherwise,}
    \end{cases}
$$
since $1+i+j-2r\leq -2$.  Therefore
$$
	(\omega_i,\omega_j) = \begin{cases}
	  \frac{dt}{t-1}\frac{r}{t} & \mbox{if $i+j=r$,} \\
	  0                         & \mbox{otherwise,}
	\end{cases}
$$
as claimed.

\begin{remark}
  Variants of this calculation for $r=2$, i.e., for the Legendre
  curve, go back to the origins of hypergeometric functions and appear
  in many places in the literature, sometimes lifted to the level of
  the Gauss-Manin connection, and with varying conventions, bases, and
  signs.
% It would be nice to have a specific thing to refer to here
\end{remark}

%%%%%%%%%%%%%%%%%%%%%%%%%%%%%%%%%%%%%%%%

\subsection{$J$ is ordinary}\label{ss:ordinary}

We recall that ``$J$ is ordinary'' means that $J(\Kbar)[p]$
has cardinality $p^g$ where $g=r-1$ is the dimension of $J$.  This
property is obviously preserved under change of ground field.

Recall (e.g., \cite[Section~2]{Serre58}) that the Cartier operator 
$$
  \car : H^0(\Cbar,\Omega^1_{\Cbar/\Kbar})
     \to H^0(\Cbar,\Omega^1_{\Cbar/\Kbar})
$$ 
is a semi-linear operator satisfying
$$
  \car(\omega + \omega') = \car(\omega) + \car(\omega')
  \ {\rm and} \
  \car(f^p \omega) = f \car(\omega)
$$
for all $f$ in the function field $\Kbar(\Cbar)$.  Also, for
$x\in\Kbar(\Cbar)$, 
$$
  \car\bigg(\frac{x^i dx}{x}\bigg) = 0\ \text{if $p \nodiv i$}
  \ \text{and} \ 
  \car\bigg(\frac{dx}{x}\bigg) = \frac{dx}{x}.
$$ 
It is known that $J$ being ordinary is equivalent to the Cartier
operator of $C$ being an isomorphism.  (This can be deduced from
\cite[Proposition~10, page~41]{Serre58}.)

\begin{proposition}\label{prop:ordinary}
  The operator
  $\car:H^0(\Cbar,\Omega^1_{\Cbar/\Kbar}) 
   \to H^0(\Cbar,\Omega^1_{\Cbar/\Kbar})$
  is an isomorphism.  In particular, the Jacobian $J$ of $C$ is ordinary.
\end{proposition}

\begin{proof}
  Corollary~\ref{cor:Kbar-basis-of-1-forms} says that
  $\omega_1,\ldots,\omega_{r-1}$ form a $\K$-basis of
  $H^0(\Cbar,\Omega^1_{\Cbar/\Kbar})$.
  We show, for all $1 \le i \le r-1$, that $\car(\omega_i)$ is a non-zero
  multiple of $\omega_a$ where $pa\equiv i\mod r$.  This implies that
  $\car$ is an isomorphism, as required.

  Since $p\nodiv r$, given $i$ with $1 \le i \le r-1$, we may solve
  $ap-br=i$ in integers $a,b$.  Moreover, adjusting $a,b$ by $mr,mp$,
  for some $m$, we may assume that $0\le b<p$, and having done this,
  it follows that $0<a<r$.  We then have
  $$
	\omega_i
		= \left(\frac xy\right)^i\frac{dx}x
		= \frac{x^iy^{br}}{y^{ap}}\frac{dx}x
		= \frac{h(x)}{y^{ap}}\frac{dx}x
  $$
  where $h(x)=x^{i+(r-1)b}(x+1)^b(x+t)^b$.  Thus
  $$
    \car(\omega_i)=y^{-a}\car\left(h(x)\frac{dx}x\right).
  $$
  Now the exponents of $x$ appearing in $h$ are in the range
  $$
    [i+(r-1)b, i+(r+1)b]=[ap-b, ap+b],
  $$
  and the only multiple of $p$ in this range is $ap$.  Letting $c$ be the
  coefficient of $x^{ap}$ in $h(x)$, then
  $\car(h(x)dx/x)=c^{1/p}x^a\,dx/x$ and 
  $$
    \car(\omega_i) = c^{1/p}(x/y)^a dx/x = c^{1/p}\omega_a.
  $$
  Thus it remains to prove that $c\neq0$.

  It is clear that $c$ is the coefficient of $x^b$ in $(x+1)^b(x+t)^b$,
  and so
  $$
    c = \sum_{j=0}^b {b \choose j}^2 t^j
      = 1 + b^2 t + \cdots + b^2t^{b-1} + t^b.
  $$
  Since $t$ is transcendental over $\Fp$, this expression is not zero in
  $\Kbar$, and this completes the proof.
\end{proof}

%%%%%%%%%%%%%%%%%%%%%%%%%%%%%%%%%%%%%%%%%%%%%%%%%%%%%%%%%%%%%%%%

\section{N\'eron-Severi of $\XX_d$ is torsion-free} \label{S:NStf}

In this section, we assume that $k$ is a perfect field of
characteristic $p\geq 0$ not dividing $d$ and containing $\mu_d$ and
that $r$ divides $d$.  Let $\XX_d\to\Pu$ be the minimal regular model
of $C/K_d$ constructed in Section~\ref{s:models}.  We regard $\XX_d$
as a surface over $k$.  Our aim in this section is to prove the
following result.

\begin{theorem}\label{thm:NS-tor}
  The \Neron-Severi group of $\XX_d$ is torsion-free.
\end{theorem}

\noindent
Along the way we work in more generality so that the same result
may be deduced for most surfaces related to the Berger construction.
It would be possible to remove the restrictions that $r$ divides $d$
and that $k$ contains $\mu_d$, but we leave this as an exercise for
the reader.

We occasionally refer to the \Neron-Severi and Picard groups
of certain singular surfaces.  We recall here three familiar facts
that continue to hold for singular but normal surfaces.  Namely, if
$\SS$ is a projective, normal, geometrically irreducible surface over
a perfect field $k$, then the Picard functor $\Pic_{\SS/k}$ is
represented by a scheme locally of finite type over $k$, the identity
component is represented by a projective algebraic group, and the
tangent space at the identity is canonically isomorphic to
$H^1(\SS,\OO_\SS)$.  See 9.4.8, 9.5.4, and 9.5.11 in \cite{KleimanPicard}
and recall that $\SS$ is integral and normal over $\kbar$ since it is
integral and normal over the perfect field $k$.

By definition, the \Neron-Severi group of a projective, normal,
irreducible surface $\SS$ over a field $k$ is the image $\NS(\SS)$ of
$\Pic(\SS)$ in 
$$
  \NS(\SS\times_k\kbar) :=
  \Pic(\SS\times_k\kbar)/\Pic^0(\SS\times_k\kbar).
$$ 
Thus $\NS(\SS)$ is a subgroup of $\NS(\SS\times_k\kbar)$.  Therefore,
to prove Theorem~\ref{thm:NS-tor} it suffices to treat the case where
$k$ is algebraically closed; in this section, we assume $k=\kbar$ when
convenient, but in some places we consider more general fields k.

%%%%%%%%%%%%%%%%%%%%

\subsection{Shioda-Tate isomorphism}\label{ss:shioda-tate}

Let $k$ be a perfect field, let $\BB$ be a smooth, projective,
geometrically irreducible curve over $k$, and let $\SS$ be a smooth,
projective, geometrically irreducible surface over $k$ equipped with a
generically smooth and surjective morphism $\pi:\SS\to\BB$.  Let
$K=k(\BB)$ be the function field of $\BB$, let $J/K$ be the Jacobian of
the generic fiber of $\pi$, and let $(A,\tau)$ be the $K/k$-trace of
$J$.

Recall that $L^1\Pic(\SS)$ is the subgroup of $\Pic(\SS)$ consisting
of classes of divisors orthogonal to a fiber of $\pi$, that
$L^2\Pic(\SS)$ is the subgroup of $L^1\Pic(\SS)$ consisting of classes
of divisors supported in the fibers of $\pi$, and that $L^i\NS(\SS)$,
for $i=1,2$ is the corresponding subgroup of $\NS(\SS)$.

\begin{prop}\label{prop:shioda-tate}
  There is a homomorphism
  $$
   \frac{L^1\NS(\SS)}{L^2\NS(\SS)}
   \to
   \MW(J) = \frac{J(K)}{\tau A(k)}
  $$
  with finite kernel and cokernel.  It is an isomorphism if $\pi$ has a
  section and if $k$ is either finite or algebraically closed.
\end{prop}

\noindent
See \cite[Proposition~4.1]{CRM}.

%%%%%%%%%%%%%%%%%%%%%%%%%%%%%%%%%%%%%%%%

\subsection{$\NS_\tor$ and $J_\tor$}

We continue the notation of the previous section.  We further assume
that $k$ is finite or algebraically closed.

If $\Pic^0(\SS)=0$, then $\NS(\SS)\cong\Pic(\SS)$ and the $K/k$-trace
of $J$ vanishes.  Therefore Proposition~\ref{prop:shioda-tate} implies
that there is an isomorphism
$$
  \frac{L^1\Pic(\SS)}{L^2\Pic(\SS)}\cong J(K).
$$
If, moreover, $\pi$ admits a section, then $L^2\Pic(\SS)$ is
torsion-free.  (See, for example, Section~4.1 in \cite{CRM}.)
Finally, $\NS(\SS)_\tor$ is contained in $L^1\Pic(\SS)$ since its
elements are numerically equivalent to zero.  Therefore we conclude
the following:

\begin{prop}\label{prop:NS_tor-into-J_tor}
  If $\Pic^0(\SS)=0$ and if $\pi$ admits a section, then the
  Shioda-Tate isomorphism induces an injection
  $\NS(\SS)_\tor\to J(K)_\tor$.
\end{prop}

Later in the paper, we use Proposition~\ref{prop:NS_tor-into-J_tor}
and bounds on $\NS(\XX)_\tor$ to bound $J(K)_\tor$.  The reverse is
also possible: good control on $J(K)_\tor$ suffices to bound
$\NS(\XX)_\tor$.

%%%%%%%%%%%%%%%%%%%%%%%%%%%%%%%%%%%%%%%%

\subsection{Birational invariance}

\begin{prop}\label{prop:birational-invariance}
  Suppose $\SS_1$ and $\SS_2$ are projective, normal surfaces over $k$
  and $f:\SS_1\to\SS_2$ is a birational map.  Then
  $\NS(\SS_1)_\tor\cong\NS(\SS_2)_\tor$ and
  $\Pic^0(\SS_1)\cong\Pic^0(\SS_2)$.
\end{prop}

\begin{proof}
  By resolution of singularities and \cite[V.5.5]{hartshorne}, there
  is a smooth projective surface $\SS$ with birational maps
  $f_1:\SS\to\SS_1$ and $f_2:\SS\to\SS_2$ satisfying $f=f_2\circ
  f_1^{-1}$.  It suffices to prove the proposition with
  $(\SS,\SS_1,f_1)$ and $(\SS,\SS_2,f_2)$ in lieu of
  $(\SS_1,\SS_2,f)$.  Therefore we may suppose, without loss of
  generality, that $\SS_1$ is smooth and projective and that $f$ is a
  birational morphism, in other words, that $f$ is a morphism and
  induces a birational isomorphism.

  If $s\in\SS_2$ is a point over which $f$ is not an isomorphism, then
  since $\SS_1$ is smooth, it is known (e.g., Corollary~2.7 in
  \cite{Badescu}) that
  $$
    f^{-1}(s)=\sum_{i=1}^nr_iE_i
  $$
  where the $E_i$ are pairwise distinct integral curves on $\SS_1$ and
  the $r_i$ are positive integers.  Moreover, the restriction of the
  intersection pairing on $\SS_1$ to the subgroup of $\NS(\SS_1)$
  generated by the classes of the $E_i$ is negative definite.

  Let $\{s_1,\dots,s_m\}$ be the set of points over which $f$ is not
  an isomorphism, let $n_i$ be the number of components of
  $f^{-1}(s_i)$, let $E_{i1},\ldots,E_{in_i}$ denote the components of
  $f^{-1}(s_i)$, and let $N=\sum_{i=1}^m n_i$ so that $N$ is the total
  number of exceptional curves introduced in passing from $\SS_2$ to
  $\SS_1$.

  There is a homomorphism $\Z^N\to\Pic(\SS_1)$ given by sending
  $(a_{11},\ldots,a_{mn_m})$ to the class of
  $\sum_{i=1}^m\sum_{j=1}^{n_i}a_{ij}E_{ij}$.  We also have
  $f^*:\Pic(\SS_2)\to\Pic(\SS_1)$.  Trivial modifications of the proof
  of V.3.2 in \cite{hartshorne}, show that these maps induce an
  isomorphism
  $$
    \Pic(\SS_1) \cong \Pic(\SS_2)\oplus\Z^N.
  $$
  It follows that $\Pic^0(\SS_1)\cong\Pic^0(\SS_2)$ as claimed.  It
  also follows that $\NS(\SS_1)\cong\NS(\SS_2)\oplus\Z^N$ and thus
  that $\NS(\SS_1)_\tor\cong\NS(\SS_2)_\tor$ as claimed.
\end{proof}

%%%%%%%%%%%%%%%%%%%%%%%%%%%%%%%%%%%%%%%%

\subsection{Geometric method}

In this subsection, we use a geometric method to kill torsion in
\Neron-Severi under suitable hypotheses.  In the application to
$\XX_d$, this method suffices to kill torsion of order coprime to $r$
and not divisible by $p=\ch(k)$; however, by itself, it does not seem
to handle primes dividing $r$.

Let $\SS$ be a smooth, irreducible, projective surface over $k$, and
let $G\subseteq\Aut_k(\SS)$ be a finite subgroup.

\begin{lemma}
  The quotient $\SS/G$ is normal, irreducible, and projective.
\end{lemma}

\begin{proof}
  It is clear that $\SS/G$ is irreducible.  It follows from
  \cite[Chapter III, Section 12, Corollary]{Serre:GACC} that it is
  normal and from \cite[Lecture 10]{Harris} that it is projective.
\end{proof}

Therefore $\Pic(\SS/G)$ has the properties detailed in the second
paragraph after the statement of Theorem~\ref{thm:NS-tor}.

\begin{prop}\label{prop:geometric-method}
  Suppose some fiber of $\SS\to\SS/G$ contains exactly one point.  If
  $\ell\neq p$ is a prime number such that $\Pic(\SS)[\ell]^G=0$, then
  $\NS(\SS/G)[\ell]=\Pic(\SS/G)[\ell]=0$.
\end{prop}

\begin{proof}
  Every element of $\NS(\SS/G)[\ell]$ lifts to $\Pic(\SS/G)[\ell]$ since
  $\Pic^0(\SS/G)$ is divisible, and thus it suffices to show that
  $\Pic(\SS/G)[\ell]$=0.  Suppose that $\LL$ is a line bundle on $\SS/G$
  whose class in $\Pic(\SS/G)$ is $\ell$-torsion.  We must show that
  it is trivial in $\Pic(\SS/G)$.

  If we choose an isomorphism $\LL^\ell\cong\OO_{\SS/G}$, then the
  $\OO_{\SS/G}$-module
  $$
    \mathcal{A}
    = \OO_{\SS/G}\oplus\LL\oplus\LL^2\oplus\cdots\oplus\LL^{\ell-1}
  $$
  inherits the structure of a sheaf of $\OO_{\SS/G}$-algebras.  Let 
  $$
    \TT=\underline{\spec}_{\OO_{\SS/G}}\mathcal{A}
  $$
  (global Spec) so that there is a finite \'etale morphism $\TT\to \SS/G$
  of degree $\ell$.  This morphism has a section if and only if $\LL$ is
  trivial, i.e., $\LL\cong\OO_{\SS/G}$. 

  The pull back of $\LL$ to $\SS$ is trivial since $\Pic(\SS)[\ell]^G=0$.
  Therefore, the fiber product $\SS\times_{\SS/G}\TT$ is trivial as an
  \'etale cover of $\SS$; in other words, there is a section of
  the projection
  $$
    \SS\times_{\SS/G}\TT\to\SS.
  $$
  This yields a commutative diagram
  $$
    \xymatrix{
      \SS\ar[rr]\ar[rd] & & \TT\ar[ld] \\
      & \SS/G. &
    }
  $$ 

  On one hand, by hypothesis some fiber of the quotient map
  $\SS\to\SS/G$ contains exactly one point.  On the other hand,
  $\TT\to\SS/G$ is finite \'etale of degree $\ell$.  It follows that
  the image of $\SS$ in $\TT$ has degree 1 over $\SS/G$ and that
  $\TT\to\SS/G$ is not connected.  Hence the covering $\TT\to\SS/G$ is
  trivial, i.e., $\TT\cong(\SS/G)\times(\Z/\ell\Z)$.  Therefore $\LL$
  is trivial in $\Pic(\SS/G)$ as claimed.
\end{proof}

%%%%%%%%%%%%%%%%%%%%

\subsection{Some group cohomology}

In this subsection, we collect some facts about the group cohomology
of $G=\mu_r$.

Let $g$ be a generator of $G$.  Recall that for an $\FlG$-module $M$,
the elements $D=1-g$ and $N=1+g+\cdots+g^{r-1}$ of $\FlG$ act on $M$
and there are isomorphisms
$$ H^i(G,M) \cong
  \begin{cases}
    \ker(D)        & \text{if $i=0$,} \\
    \ker(N)/\im(D) & \text{if $i\ge1$ is odd,} \\
    \ker(D)/\im(N) & \text{if $i\ge2$ is even.}
  \end{cases}
$$

Let $R=\FlG$ be the regular representation of $G$, and let $W$ be the
quotient of $R$ by the subspace of $G$-invariants $R^G$.

\begin{lemma}\label{lem:W-cohomology}\ 
  \begin{enumerate}
  \setlength{\itemsep}{0.05in}
  \item
	  $H^i(G,R)\, \cong \begin{cases}
		\Fl & \mbox{if }i=0, \\
		0   & \mbox{if }i>0;
	  \end{cases}$
  \item
	  $H^i(G,\Fl)\cong\begin{cases}
		\Fl & \mbox{if }i=0\mbox{ or }\ell\mid r, \\
		0   & \mbox{if }i>0\mbox{ and }\ell\nodiv r;
	  \end{cases}$
  \item $H^i(G,W)\cong H^{i+1}(G,\Fl)$ for $i\geq 0$;%;
  \item $H^i(G,W\tensor W)\cong H^{i+1}(G,W)$ for $i>0$.
  \end{enumerate}
\end{lemma}

\begin{proof}
  We may identity $R^G$ with the trivial $\FlG$-module $\Fl$, which is
  the $i=0$ part of (1).  The rest of part (1) follows from \cite[page
  112, Proposition~1]{SerreLF} since $\FlG$ is co-induced.  Part (2)
  is a simple exercise using the isomorphisms displayed just before
  the lemma.  For part (3), by the definition of $W$, there is an
  exact sequence
  \begin{equation}\label{eq:W-def}
    0 \to \Fl \to R \to W \to 0.
  \end{equation}
  Taking cohomology yields an exact sequence
  $$
	0 \to H^0(G,\Fl) \to H^0(G,R) \to H^0(G,W) \to H^1(G,\Fl) \to 0
  $$
  and identities $H^i(G,W)\cong H^{i+1}(G,\Fl)$, for $i\geq 0$.

  Since $R\tensor W$ is co-induced, applying \cite[p.~112,
  Proposition~1]{SerreLF} implies that $H^i(G,R\tensor W)=0$ for
  $i>0$.  Tensoring \eqref{eq:W-def} with $W$ and taking cohomology
  produces an exact sequence
  $$
    0 \to H^0(G,W) \to H^0(G,R\tensor W) \to H^0(G,W\tensor W) \to
    H^1(G,W)\to 0
  $$
  and identities $H^i(G,W\tensor W)\cong H^{i+1}(G,W)$ for $i>0$.
\end{proof}

Consider an exact sequence of $\FlG$-modules
$$
  0\to \Fl\to \tilde W\to W\to 0.
$$

\begin{lemma}\label{lem:W-extension}
  $\tilde W\cong\Fl\oplus W$ or $\tilde W\cong R$ as $\FlG$-modules.
\end{lemma}

\begin{proof}
  There is an element $w\in W$ that generates $W$ as an $\FlG$-module,
  that is, $W$ is cyclic as an $\FlG$-module.  Let
  $\tilde w\in\tilde W$ be a lift of $w$.  The $\FlG$-submodule of
  $\tilde W$ generated by $\tilde w$ maps surjectively to $\tilde W$.
  If this map is an isomorphism, the above sequence is split and
  $\tilde W\cong\Fl\oplus W$.  Otherwise the submodule must be all of
  $\tilde W$, in which case $\tilde W$ is cyclic and has
  $\Fl$-dimension $r$, so is isomorphic to $R$.
\end{proof}

Note that $R\cong\Fl\oplus W$ if $\ell\nodiv r$.

%%%%%%%%%%%%%%%%%%%%%%%%%%%%%%%%%%%%%%%%

\subsection{Cohomological method}

Recall that $r>1$ is an integer not divisible by $p$.  In this
subsection, we develop a more elaborate, cohomological method to
kill torsion in \Neron-Severi.  We need it to kill $\ell$-torsion
in $\NS(\XX_d)$ when $\ell$ is a prime dividing $r$.

To state the result, let $\CC\to\P^1$ (resp., $\DD\to\P^1$) be a
Galois branched cover with group $G=\mu_r$ that is totally ramified
over $\dOne>0$ (resp., $\dTwo>0$) points of $\P^1$ and unramified
elsewhere.  Let $G$ act diagonally on $\SS=\CC\times_k\DD$.

\begin{prop}\label{prop:cohom-method}
  $\NS(\SS/G)[\ell]=\Pic(\SS/G)[\ell]=0$ for any prime number $\ell\neq p$.
\end{prop}

\noindent
The proof occupies the remainder of this subsection.  It suffices to
treat the case where $k$ is algebraically closed, so we make this
assumption for the rest of this section.

To lighten notation, for a scheme $\YY$ over $k$ we write $H^i(\YY)$
for the \'etale cohomology group $H^i(\YY_{et},\Fl)$ and $H^i_c(\YY)$
for the \'etale cohomology group with compact supports.

\begin{lemma}\label{lem:curve-cohomology}
The following $G$-modules are isomorphic:
\begin{enumerate}
\setlength{\itemsep}{0.05in}
\item $H^0(\CC)\cong H^2(\CC)\cong H^0(\DD)\cong H^2(\DD)\cong \Fl$;
\item $H^1(\CC)\cong W^{\dOne-2}$ and $H^1(\DD)\cong W^{\dTwo-2}$.
\end{enumerate}
\end{lemma}

\begin{proof}
  The first part is well known since $\CC,\DD$ are irreducible
  projective curves, and it suffices to prove the second part for
  $\CC$ since the argument for $\DD$ is identical.

  Let $\CC^o$ denote the maximal open subset of $\CC$ where
  $\CC\to\P^1$ is unramified and let $\mathcal{P}^o$ be the image of
  $\CC^o$ in $\P^1$.  Then $\mathcal{P}^o$ is $\P^1$ minus $\dOne$
  points and $f:\CC^o\to \mathcal{P}^o$ is an \'etale Galois cover
  with group $G$.  We first check that there is an isomorphism
  $$
    H^1(\CC^o)\cong\Fl\oplus R^{\dOne-2}.
  $$
  Indeed, $H^1(\CC^o)\cong H^1(\mathcal{P}^o,f_*\Fl)$ since $f$ is
  finite.  The stalk of $f_*\Fl$ at the generic point of
  $\mathcal{P}^o$ may be identified as a $G$-module with $ R$, and it
  has an action of $\pi_1^{(p)}$, the prime-to-$p$ fundamental group
  of $\mathcal{P}^o$.  Moreover \cite[V.2.17]{Milne},
  $H^1(\mathcal{P}^o,f_*\Fl)$ is isomorphic to the Galois cohomology
  group $H^1(\pi_1^{(p)}, R)$. Using the fact that $\pi_1^{(p)}$ is
  the free pro-prime-to-$p$ group on $\dOne-1$ generators
  $\sigma_1,\dots,\sigma_{\dOne-1}$, it is an easy exercise to check
  that $H^1(\pi_1^{(p)}, R)$ is isomorphic to the cokernel of the map
  $$
    R \to R^{\dOne-1},\qquad
    \lambda \mapsto (\sigma_1\lambda-\lambda,\dots,
    \sigma_{\dOne-1}\lambda-\lambda).
  $$
  Since each generator $\sigma_i$ acts on $ R$ as multiplication by
  some generator of $G$ (by our hypothesis that $\CC\to\P^1$ is
  totally ramified at each branch point), the cokernel is isomorphic
  to $\Fl\oplus R^{\dOne-2}$, as desired.

  To finish, we note that $H^1(\CC^o)\cong H^1_c(\CC^o)^*$ by
  Poincar\'e duality and that, by excision, there is an exact sequence
  $$
    0 \to H^0(\CC)\to H^0(\CC\setminus \CC^o)\to H^1_c(\CC^o)\to H^1(\CC)\to 0.
  $$
  Since $H^0(\CC\setminus \CC^o)\cong\F_\ell^{d}$ as a $G$-module, the result
  follows easily.
\end{proof}

\begin{lemma}\label{lemma:G-invariant-cohomology}
  Let $X$ be a variety over $k$. Let $\mathcal{G}$ be a
  finite group of order prime to $p$ which acts on $X$, and let
  $Y = X/\mathcal{G}$. Then for $i \geq 1$,
  \[
    H^i(Y, \mathcal{O}) = H^i(X, \mathcal{O})^\mathcal{G}.
  \]
\end{lemma}

\begin{proof}
  Recall that by a variety over $k$, we mean a separated scheme of
  finite type over $k$.  Let $(V_i)$ be a cover of $Y$ by open
  affines. Let $U_i \subset X$ be the preimage of $V_i$; the $V_i$ are
  $\mathcal{G}$-invariant and, since $X \to Y$ is finite, are also
  affine. Separability implies that intersections of the $V_i$
  (resp.~$U_i$) are also affine. If $\check{H}$ denotes \v{C}ech
  cohomology, by~\cite[III.4.5]{hartshorne}, then
  \[
    \check{H}^i((V_i), \mathcal{O}) = H^i(Y, \mathcal{O})
  \]
  and similarly for $X$. For $i \geq 1$, consider the \v{C}ech complex
  \[
    \cdots \longrightarrow C^{i-1}((U_i), \mathcal{O}) \longrightarrow 
C^{i}((U_i), \mathcal{O}) \longrightarrow C^{i+1}((U_i), \mathcal{O}) 
\longrightarrow \cdots.
  \]
  If we take $\mathcal{G}$-invariants, we obtain the corresponding
  \v{C}ech complex for $Y$. All of the groups above are $k$-vector
  spaces and the order of $\mathcal{G}$ is prime to $p$, so taking
  $\mathcal{G}$-invariants commutes with taking homology. The claim
  follows.
\end{proof}

Recall $G=\mu_r$, and let $\TT=\SS/G$.

\begin{lemma}\label{lem:vanishing-Pic^0}
$\Pic^0(\TT)=0$.
\end{lemma}

% Proof of Lemma 6.22:  \TT needs to be defined.

\begin{proof}
  Lemma~\ref{lemma:G-invariant-cohomology} and the K\"unneth formula
  imply
  $$
    H^1(\TT,\OO)
    = H^1(\SS,\OO)^{G}
    = \left(H^1(\CC,\OO)\oplus H^1(\DD,\OO)\right)^{G}
    = 0.
  $$
  %via Lemma~\ref{lemma:G-invariant-cohomology} and the K\"unneth formula.  
  In particular, $\Pic^0(\TT)=0$ since its tangent space space is
  $H^1(\TT,\OO)$ and thus is trivial.
\end{proof}

Therefore, $\NS(\TT)=\Pic(\TT)$ and
$$
  \NS(\TT)[\ell] = \Pic(\TT)[\ell] = H^1(\TT,\mu_\ell) = H^1(\TT),
$$
since $k$ is algebraically closed.  The rest of the proof of
Proposition~\ref{prop:cohom-method} is a somewhat elaborate
calculation of $H^1(\TT)$; in particular, we show it vanishes.
%  Before proceeding to the calculation, we review the cohomology of $\CC$ and $\DD$.

Let $Z\subset \SS$ be the reduced subscheme of fixed points, which by
our hypotheses consists simply of $\dOne\dTwo$ distinct points.  We
identify $Z$ with its image in $\TT$ as well.  Let
$\SS^o=\SS\setminus Z$ and $\TT^o=\TT\setminus Z$ and note that
$\SS^o \to \TT^o$ is an \'etale Galois cover with group $G$ and that
$\TT^o$ is smooth.

\begin{lemma}\label{lem:S-cohomology}
We have the following isomorphisms of $\FlG$-modules:
$$
	H^i(\SS)\cong
	\begin{cases}
		\Fl               & \mbox{if }i=0\mbox{ or }4, \\
		W^{\dOne+\dTwo-4} & \mbox{if }i=1\mbox{ or }3, \\
		\F_\ell^2\oplus (W\tensor W)^{(\dOne-2)(\dTwo-2)} & \mbox{if }i=2.
	\end{cases}
$$
\end{lemma}

\begin{proof}
This follows from the K\"unneth formula
$$
    H^i(\SS)\cong \oplus_{j=0}^i\left(H^j(\CC)\otimes H^{i-j}(\DD)\right)
$$
and Lemma~\ref{lem:curve-cohomology}.
\end{proof}

By excision, there is an exact sequence
$$
  0\to H^0(\TT)\to H^0(Z)\to H^1_c(\TT^o)\to H^1(\TT)\to 0
$$
and an isomorphism
$$
  H^j_c(\TT^o)\cong H^j(\TT),
$$
for $j\geq 2$, since $Z$ is zero dimensional.  We also have the
Poincar\'e duality isomorphism
$$
  H^j_c(\TT^o)\cong H^{4-j}(\TT^o)^*
$$
for $0\leq j\leq 4$.  Therefore, to show that $H^1(\TT)=0$ we must
show that $H^3(\TT^o)$ has dimension $\dOne\dTwo-1$ as an $\Fl$-vector
space; its dimension is at least this.  To show equality we use the
Hochschild-Serre spectral sequence
\begin{equation}\label{eq:hochschild-serre}
  E_2^{i,j} = H^i(G,H^j(\SS^o))\Longrightarrow H^{i+j}(\TT^o).
\end{equation}
To compute the $E_2$ term, we begin by computing $H^j(\SS^o)$.

\begin{lemma}\label{lem:U-compact-cohomology}
We have the following isomorphisms of $\FlG$-modules:
  $$
	H^j_c(\SS^o)
%	\cong H^{4-j}(\SS^o)
	\cong  %H^j(\SS^o)^*\cong
	\begin{cases}
		\F_\ell^{\eOne}\oplus R^{\eTwo}\oplus W^{\eThree} & \mbox{if }j=1, \\
		\F_\ell^2\oplus (W\tensor W)^{(\dOne-2)(\dTwo-2)} & \mbox{if }j=2, \\
		W^{\dOne+\dTwo-4} & \mbox{if }j=3, \\
		\Fl & \mbox{if }j=4, \\
		0 & \mbox{otherwise}%\mbox{if }j\le 0\mbox{ or }j\ge 4,
	\end{cases}
  $$
  where $\eOne+\eTwo=\dOne\dTwo-1$ and $\eTwo+\eThree=\dOne+\dTwo-4$.
\end{lemma}

\begin{proof}
  By excision, there is an exact sequence
  $$
    0\to H^0(\SS)\to H^0(Z)\to H^1_c(\SS^o)\to H^1(\SS)\to 0
  $$
  and isomorphisms
  $$
    H^j_c(\SS^o)\cong H^j(\SS)%\cong H^{4-j}(\SS)^*
  $$
  for $j\geq 2$, since $Z$ is zero dimensional and non-empty.
  Therefore, Lemmas~\ref{lem:W-extension} and \ref{lem:S-cohomology}
  imply that $H^{4-j}_c(\SS^o)$ has the desired form.  (Roughly
  speaking, $\eThree$ is the number of copies of $W$ in $H^1(\CC)$
  over which the extension $H^1_c(\SS^o)$ is split.)
\end{proof}  

\begin{cor}\label{cor:U-cohomology}
  We have the following isomorphisms of $\FlG$-modules:
  $$
	H^j(\SS^o)
	\cong
	\begin{cases}
		\Fl & \mbox{if }j=0, \\
		W^{\dOne+\dTwo-4} & \mbox{if }j=1, \\
		\F_\ell^2\oplus (W\tensor W)^{(\dOne-2)(\dTwo-2)} & \mbox{if }j=2, \\
		\F_\ell^{\eOne}\oplus R^{\eTwo}\oplus W^{\eThree} & \mbox{if }j=3, \\
		0 & \mbox{otherwise},
	\end{cases}
  $$
  where $\eOne+\eTwo=\dOne\dTwo-1$ and $\eTwo+\eThree=\dOne+\dTwo-4$.
\end{cor}

\begin{proof}
  The Poincar\'e duality isomorphism states that
  $$
    H^j(\SS^o)^*\cong H^{4-j}_c(\SS^o)
%    H^{4-j}(\SS^o)^*\cong H^j_c(\SS^o)
  $$
  for $0\leq j\leq 4$.  In particular, $\Fl,R,W$ are self-dual as
  $\FlG$-modules, so $H^j(\SS^o)^*$ is also self-dual, and thus
  $H^j(\SS^o)$ has the desired form.
\end{proof}

Applying Lemma~\ref{lem:W-cohomology} and
Corollary~\ref{cor:U-cohomology}, we find that if $\ell\nodiv r$, then
\begin{equation}\label{eq:E2-dims-ell-coprime-r}
  \dim E_2^{i,j} = \dim(H^i(G,H^j(\SS^o))) = \begin{cases}
    \dOne\dTwo-1 & \text{if }i=0, j=3, \\
    0            & \text{if }i\ge 1, j\ge 0.
  \end{cases}
\end{equation}
One can deduce more, but when $\ell\nodiv r$ this already suffices to
show that \eqref{eq:hochschild-serre} degenerates and
$$
   \dim H^3(\TT^o) = \dim E_2^{0,3} = \dOne\dTwo-1
$$
as claimed.  Therefore, we suppose for the remainder of this
subsection that $\ell\mid r$ and apply Lemma~\ref{lem:W-cohomology}
and Corollary~\ref{cor:U-cohomology} to deduce
\begin{equation}\label{eq:E2-dims}
  \dim E_2^{i,j} = \begin{cases}
    1                                 & \text{if $i\ge0$, $j=0$,}\\
%    \eTwo+\eThree                     & \text{if $i\ge0$, $j=1$,}\\
    \dOne+\dTwo-4                     & \text{if $i\ge0$, $j=1$,}\\
%    2+(\dOne-2)(\dTwo-2)(r-1)         & \text{if $i=0$, $j=2$,}\\
%    2+(\dOne-2)(\dTwo-2)              & \text{if $i>0$, $j=2$,}\\
    \dOne\dTwo-2\dOne-2\dTwo+6        & \text{if $i>0$, $j=2$,}\\
%    e_1+e_2+e_3                       & \text{if $i=0$, $j=3$,}\\
    \dOne\dTwo-1+\eThree              & \text{if $i=0$, $j=3$,}\\
%    e_1+e_3                           & \text{if $i>0$, $j=3$.}
    \dOne\dTwo-\dOne-\dTwo+3+2\eThree & \text{if $i>0$, $j=3$.}
  \end{cases}
\end{equation}
One can also deduce more in this case, but these suffice for our
purposes.  More precisely, to show $\dim H^3(\TT^o)=\dOne\dTwo-1$, it
suffices to show that $\eThree=0$ and
$\dim H^3(\TT^o)\leq \dim E_2^{0,3}$.

The spectral sequence \eqref{eq:hochschild-serre} has non-zero groups
only in the first quadrant and has only four non-trivial rows, i.e.,
$E_h^{i,j}=0$ unless $i\ge 0$ and $0\le j\le3$.  It follows
immediately that $d_{h}^{i,j}=0$ if $h>4$ or if $j>3$ or if $h>j+1$,
and also that $E_\infty^{i,j}=E_5^{i,j}$.

\begin{lemma}\label{lem:vanishing}
  The differentials $d_h^{i,j}$ in the spectral sequence
  \eqref{eq:hochschild-serre} satisfy the following: For $h\ge2$ and
  $i\ge1$, $\rk d_h^{i,3}=\dim E_2^{i+h,4-h}$.  Moreover, with
  notation as in Corollary~\ref{cor:U-cohomology}, $\eThree=0$.
\end{lemma}

\begin{proof}
  Since $\TT^o$ is not complete and $\dim(\TT^o)=2$, it follows that
  $H^{i+j}(\TT^o)$ vanishes for $i\ge1$ and $j=3$.  Thus $E_5^{i,3}$
  vanishes, and the definitions of $E_h^{i,3}$ for $h=1,2,3$ imply
  that
$$\dim E_2^{i,3} = \rk d_2^{i,3}+\rk d_3^{i,3}+\rk d_4^{i,3}$$
for $i\geq 1$.  Moreover,
$$\rk d_2^{i,3}+\rk d_3^{i,3}+\rk d_4^{i,3}
  \le \dim E_2^{i+2,2}+\dim E_2^{i+3,1}+\dim E_2^{i+4,0}$$
and so
$$\dim E_2^{i,3}  \le \dim E_2^{i+2,2}+\dim E_2^{i+3,1}+\dim E_2^{i+4,0}$$
for $i\geq 1$.  Comparing dimensions using \eqref{eq:E2-dims}, we see
that $\eThree=0$ and that
  \begin{equation}\label{eq:rk-vs-dim}
    \rk d_h^{i,3} = \dim E_2^{i+h,4-h}
  \end{equation}
  for $i\ge 1$ and $h=2,3,4$.  Trivially, $\rk d_h^{i,3}=0$ for
  $h\ge5$ since $E_2^{i+h,4-h}=0$ for $h\ge5$.  This establishes the
  claims of the lemma.
\end{proof}

\begin{lemma}\label{lem:zero-diffs-in-degree-4}
  The differentials $d_h^{i,j}$ in the spectral sequence
  \eqref{eq:hochschild-serre} satisfy
$$d_h^{2,2}=d_h^{3,1}=d_h^{4,0}=0$$
for $h\ge2$.
\end{lemma}

\begin{proof}
  By Lemma~\ref{lem:vanishing}, $\rk d_3^{1,3}=\dim E_2^{4,1}$.
  Therefore, $\dim E_3^{4,1}=\dim E_2^{4,1}$ and $d_2^{2,2}=0$.
  Similarly, Lemma~\ref{lem:vanishing} says that
  $\rk d_4^{1,3}=\dim E_2^{5,0}$ which implies that
  $d_2^{3,1}=d_3^{2,2}=0$.  Trivially, $d_h^{2,2}$ vanishes for $h>3$,
  $d_h^{3,1}$ vanishes for $h>2$, and $d_h^{4,0}$ vanishes for
  $h>1$.  This completes the proof of the lemma.
\end{proof}

\begin{lemma}\label{lem:bound-H3}
With notation as above, $\dim H^3(\TT^o)\le\dim E_2^{0,3}=b_1b_2-1$.
\end{lemma}

\begin{proof}
  On one hand, the previous lemma says that all the differentials
  $d_h^{i,j}$ with domain $E_h^{i,j}$ vanish when $i+j=4$, $i\ge2$,
  and $h\ge2$.  On the other hand, $H^4(\TT^o)=0$, so
  $E_\infty^{i,j}=0$ for $i+j=4$.  It follows that
$$\dim H^3(\TT^o)=\sum_{i+j=3}\dim E_\infty^{i,j}\le
\sum_{i+j=3}\dim E_2^{i,j} -\sum_{\substack{i+j=4\\i\ge2}}\dim
E_2^{i,j}.$$
Applying Equation~\eqref{eq:E2-dims}, this last difference is
$\dim E_2^{0,3}$ and \eqref{eq:E2-dims} together with
Lemma~\ref{lem:vanishing} shows that $\dim E_2^{0,3}=b_1b_2-1$.  This
completes the proof of the lemma.
\end{proof}

As noted after Lemma~\ref{lem:S-cohomology}, the inequality
$\dim H^3(\TT^o)\le b_1b_2-1$ completes the proof that $H^1(\TT)=0$
and that of Proposition~\ref{prop:cohom-method}.

%%%%%%%%%%%%%%%%%%%%

\subsection{Proof of Theorem~\ref{thm:NS-tor}}

The statement of Theorem~\ref{thm:NS-tor} is that
$\NS(\XX_d)_{\tor}=0$.  By Proposition~\ref{prop:NS_tor-into-J_tor},
there is an injection $\NS(\XX_d)_\tor\to J(K)_\tor$.  Moreover, we
proved in Corollary~\ref{cor:ptors} that $J(K)$ has no $p$-torsion,
thus neither does $\NS(\XX_d)$.  It thus suffices to prove that
$\NS(\XX_d)[\ell]=0$ for every prime $\ell \neq p$.

For the rest of the proof, suppose $\ell \neq p$.  By
Proposition~\ref{prop:birational-invariance}, it suffices to prove
$\NS(\TT_1)[\ell]=0$ for some $\TT_1$ that is birational to $\XX_d$.
Recall from Section~\ref{s:DPC} that $\XX_d$ is birational to the
quotient $\SS/(\mu_r\times\mu_d)$ constructed as follows:

Let $\CC_d$ and $\DD_d$ be the smooth, projective curves over $k$ with
affine models
$$
  \CC_d : z^d=x^r-1\quad\text{and}\quad\DD_d : w^d=y^r-1
$$
respectively, and let $\SS=\CC_d\times_k\DD_d$.  The action of
$\mu_r\times\mu_d$ on $\A^2\times_k\A^2$ given by
$$
  (x,y,z,w) \mapsto (\eta x,\eta^{-1}y,\zeta z,\zeta^{-1}w)
$$
induces an action on $\SS$.  

Let $\TT=\SS/\mu_r$.  Observe that Proposition~\ref{prop:cohom-method}
implies $\NS(\TT)[\ell]=0$ and that Lemma~\ref{lem:vanishing-Pic^0}
implies $\Pic^0(\TT)=0$.  Now let $\SS_1\to\TT$ be a resolution of
singularities of $\TT$ that is an isomorphism away from the singular
points.  The action of $\mu_d$ on $\SS$ has isolated fixed points that
are disjoint from the fixed points of the action of $\mu_r$.  It also
descends to an action on $\TT$ and then lifts (uniquely) to an action
on $\SS_1$ with isolated fixed points
(cf.~\cite[II.7.15]{hartshorne}).
Proposition~\ref{prop:birational-invariance} implies that
$\Pic^0(\SS_1)=\Pic^0(\TT)=0$, and so
$\Pic(\SS_1)[\ell]=\NS(\SS_1)[\ell]=0$.  \emph{A fortiori\/},
$\Pic(\SS_1)[\ell]^G=0$, and thus we may apply
Proposition~\ref{prop:geometric-method} to deduce that
%with $\SS_1$ playing the role of
%$\SS$ and $\mu_d$ playing the role of $G$ to conclude that
$\NS(\TT_1)[\ell]=0$ for $\TT_1=\SS_1/\mu_d$ and $\ell\neq p$.  This
completes the proof since $\TT_1$ is birational to
$\SS/(\mu_r\times\mu_d)$ and thus to $\XX_d$.  \qed

For future use, we record one other byproduct of our analysis.

\begin{prop}\label{prop:K/k-trace}
  $\Pic^0(\XX_d)=0$ and thus the $K/k$-trace of $J$ is trivial.
\end{prop}

\begin{proof}
  As observed in the proof of Theorem~\ref{thm:NS-tor},
  Lemma~\ref{lem:vanishing-Pic^0} implies that
  $\Pic^0(\SS/\mu_r)=0$. Using the fact that
  $H^1(\SS/\mu_r, \mathcal{O})$ is the tangent space of
  $\Pic^0(\SS/\mu_r)$ (and similarly for $\SS/(\mu_r \times \mu_d)$)
  along with Lemma~\ref{lemma:G-invariant-cohomology}, we see that
 $$\Pic^0(\SS/(\mu_r\times\mu_d))=0.$$  
 Therefore $\Pic^0(\XX_d)=0$ since $\XX_d$ and
 $\SS/(\mu_r\times\mu_d)$ are birational.  Finally, the $K/k$-trace of
 $J$ vanishes since it is inseparably isogenous to
 $\Pic^0(\XX_d)/\Pic^0(\P^1)$---see \cite{Conrad} or
 \cite{Shioda99}---and since $\Pic^0(\XX_d)$ vanishes.
\end{proof}

%%%%%%%%%%%%%%%%%%%%%%%%%%%%%%%%%%%%%%%%%%%%%%%%%%%%%%%%%%%%%%%%

%%% Local Variables: 
%%% mode: latex
%%% TeX-master: "EHR"
%%% End: 

% This is Chapter 7

\chapter{Index of the visible subgroup and the 
Tate-Shafarevich group}\label{ch:Sha}

In this chapter, we work under the hypotheses that $r$ divides $d$,
$d=p^\nu+1$, and $d$ divides $q-1$.  The first goal is to understand
the index of the visible subgroup $V$ in $J(K_d)$.  Ultimately, we
find that the index is a power of $p$ and equal to the square root of
the order of the Tate-Shafarevich group $\sha(J/K_d)$.  Specifically,
in Section~\ref{S:VvMW}, we determine the torsion subgroup
$J(\Fq(u))_\tor$ and prove that the index of $V$ in $J_r(\Fq(u))$ is a
power of $p$, Theorem~\ref{thm:V-vs-MW}.  In Section~\ref{s:tamagawa},
we find the Tamagawa number $\tau(J/\F_q(u))$ of the Jacobian $J$ of
$C$ over $\Fq(u)$, Proposition~\ref{cor:tau-value}.  Finally, in
Section~\ref{S:appBSD}, we prove an analytic class number formula
relating the Tate-Shafarevich group $\sha(J/\Fq(u))$ and the index
$[J(\Fq(u)):V]$, Theorem~\ref{thm:sha}.

\section{Visible versus Mordell-Weil} \label{S:VvMW}

Let $V$ be the visible subgroup of
$J(\Fq(u))$, that is, the subgroup generated by 
$$
  P = (u,u(u+1)^{d/r})\in C(\Fq(u))\into J(\Fq(u))
$$ 
and its Galois conjugates.  By Corollary~\ref{cor:exact-rank}, we know
that
$$
  \rk V = \rk J(\Fq(u)) = (r-1)(d-2).
$$
In particular, $V$ has finite index in $J(\Fq(u))$.  In this section, we
show that this index is a power of $p$ thus completing our knowledge of
$J(\Fq(u))$.  More precisely:

\begin{theorem}\label{thm:V-vs-MW}
Suppose that $r$ divides $d$, that $d=p^\nu+1$,
and that $d$ divides $q-1$.  
  The torsion subgroup $J(\Fq(u))_\tor$ equals $V_\tor$ and has
  order $r^3$.  The index of $V$ in $J_r(\Fq(u))$ is a power of $p$.
\end{theorem}

\noindent
The proof is given later in this section.  Before giving it,
we prove a general integrality result for regulators of Jacobians over
function fields.

%%%%%%%%%%%%%%%%%%%%

\subsection{Integrality}

Let $\BB$ (resp.~$\SS$) be a curve (resp.~surface) over $k=\F_q$.
We assume that $\BB$ and $\SS$ are smooth, projective, and geometrically
irreducible, that $\SS$ is equipped with a surjective and
generically smooth morphism $\pi:\SS\to\BB$, and that $\pi$ has a section
whose image we denote $O$.  Let $\NS(\SS)$ be the \Neron-Severi
group of $\SS$, and let $L^1\NS(\SS)$ and $L^2\NS(\SS)$ be the subgroups
of $\NS(\SS)$ defined in Section~\ref{ss:shioda-tate}.

Let $K=k(\BB)$ be the function field of $\BB$, let $J/K$ be the Jacobian
of the generic fiber of $\pi$, let $(A,\tau)$ be the $K/k$-trace of
$J$, and let $\MWJ$ be the Mordell-Weil group $J(K)/\tau A(k)$.  By
Proposition~\ref{prop:shioda-tate} there is an isomorphism
$$
  \frac{L^1\NS(\SS)}{L^2\NS(\SS)}\to\MWJ,
$$
since $\pi$ has a section and since $k$ is finite.

We write $\det(\MWJ/\tor)$ for the discriminant of the height pairing
on $\MWJ$ modulo torsion.  Also, for each place $v$ of $K$, we write
$N_v$ for the subgroup of $\NS(\SS)$ generated by non-identity
components of $\pi^{-1}(v)$ and $d_v$ for the discriminant of the
intersection pairing of $\SS$ restricted to $N_v$; by convention, we
set $d_v=1$ if $N_v=0$.  If $\JJ\to\BB$ is the \Neron{} model of
$J/K$, then it follows from \cite[Section 9.6, Theorem~1]{blr} that
$d_v$ is also the order of the group of connected components of the
fiber of $\JJ\to\BB$ over $v$.
%We computed these numbers in Proposition~\ref{prop:component-groups}.

With these notations, our integrality result is as follows.

\begin{prop}\label{prop:integrality}
The rational number
$$
  |\NS(\SS)_\tor|^2\left(\prod_vd_v\right)\frac{\det(\MWJ/\tor)}{|\MWJ_\tor|^2}
$$
is an integer.
\end{prop}

This generalizes \cite[Proposition~9.1]{Legendre}, which is the case
where $\dim(J)=1$ and $A=0$.  (In that case, $\NSS_\tor$ is
known to be trivial.)  The general case is closely related to, but
apparently not contained in, the discussion in \cite[Section 5]{Gordon}. 

\begin{proof}
  Let $F$ be the class in $\NSS$ of a fiber of $\pi$.  Then
  $\Lone$ is the subgroup of $\NSS$ consisting of classes orthogonal
  to $F$.

  The intersection form on $\SS$ is degenerate when restricted to
  $\Lone$; indeed its radical is $\Z F$.  We write $\Lonebar$ and
  $\Ltwobar$ for the respective quotients of $\Lone$ and $\Ltwo$ by
  $\Z F$ so that the intersection pairing on $\SS$ then defines a
  non-degenerate pairing on $\Lonebar$.  For any torsion-free subgroup
  $L\subset\Lonebar$, we write $\disc(L)$ for the discriminant of the
  intersection form restricted to $L$ (i.e., the absolute value of the
  determinant of the matrix of pairings on a basis); by convention, we
  set $\disc(0)=1$.

  We identify $N_v$ with its image in $\Ltwobar$ so that
  $\disc(N_v)=d_v$ and so that there is an orthogonal direct sum
  decomposition
  $$
    \Ltwobar = \bigoplus_v N_v.
  $$
  We also let $d=\prod_v d_v$ so that $d=\disc(\Ltwobar)$.

\newcommand\Qtilde{\tilde{Q}}
\newcommand\Rtilde{\tilde{R}}

  Choose elements $Q_1,\dots,Q_m\in\MWJ$ that map to a basis
  of $\MWJ/\tor$ and thus a basis of
  $$ %\MWJ/\tor\subset
    \MWJ\tensor\Q
    = \frac{\Lonebar\tensor\Q}{\Ltwobar\tensor\Q}.
  $$
  Each $Q_i$ has a naive lift $\Qtilde_i$ to
  $\Lonebar\tensor\Q$ represented by a $\Z$-linear combination of
  ``horizontal'' divisors.  The projection of $\Qtilde_i$ onto the
  orthogonal complement of $\Ltwobar\tensor\Q$ is represented by
  a $\Q$-linear combination of divisors.  It follows from Cramer's rule
  that the denominator appearing in the coefficient of a component of
  $\pi^{-1}(v)$ divides $d_v$.  In particular, the multiple $R_i=dQ_i$
  has a lift $\Rtilde_i$ to $\Lonebar$ (i.e., with \emph{integral}
  coefficients) that is orthogonal to $\Ltwobar$.

  By the definition of the height pairing, 
  $$
    \langle R_i,R_j\rangle = -(\Rtilde_i)\cdot(\Rtilde_j)
  $$
  where the dot on the right hand side signifies the intersection
  pairing on $\Lonebar$.  It follows that 
  $$
    \disc\left(\Z\Rtilde_1+\cdots+\Z\Rtilde_m\right)
    = d^{2m}\det(\MWJ/\tor),
  $$
  since the $\Rtilde_i$ map to a basis of $d\cdot(\MWJ/\tor)$.
  Now let $N$ be the subgroup of $\Lonebar$
  generated by the $\Rtilde_i$ and by $\Ltwobar$.
  
  On the one hand, there is an orthogonal direct sum decomposition
  $$
    N = \left(\Z\Rtilde_1+\cdots+\Z\Rtilde_m\right) \oplus \Ltwobar
  $$
  and so 
  $$
    \disc(N)=d^{2m+1}\det(\MWJ/\tor).
  $$
  Moreover, the assumption that $\pi$ has a section implies that
  $\Ltwo$ is torsion-free and that $F$ is indivisible in $\Ltwo$,
  and thus $\Ltwobar$ and $N$ are also torsion-free.

  On the other hand, the index of $N$ in $\Lonebar$ is
  $d^m|\MWJ_\tor|$.  It follows that
  $$
    \frac{\disc(\Lonebar/\tor)}{|\Lonebar_\tor|^2}
    = \frac{\disc(N/\tor)}{d^{2m}|\MWJ_\tor|^2|N_\tor|^2}
    = d\frac{\det(\MWJ/\tor)}{|\MWJ_\tor|^2},
  $$
  since $N$ is torsion-free.  Finally, if we note that
  $$
    \Lonebar_\tor = \Lone_\tor = \NSS_\tor,
  $$
  then we find that
  $$
    \disc(\Lonebar/\tor)
    = |\NSS_\tor|^2\left(\prod_vd_v\right)
        \frac{\det(\MWJ/\tor)}{|\MWJ_\tor|^2}.
  $$ 
  In particular, the left side is an integer since the intersection
  pairing on $\SS$ is integer valued, and thus the right side
  is an integer as claimed.
\end{proof}

%%%%%%%%%%%%%%%%%%%%

\subsection{Proof of Theorem~\ref{thm:V-vs-MW}}

We apply the integrality result Proposition~\ref{prop:integrality}
to $\XX_d$ and $J$.

On the one hand, $\NS(\XX_d)$ is torsion-free by
Theorem~\ref{thm:NS-tor}.  Moreover,
$$
  \prod_v d_v = d^{2r-2}r^{d+2}
$$
by Proposition~\ref{prop:component-groups}.

On the other hand, the $\Fq(u)/\Fq$-trace of $J$ vanishes by
Proposition~\ref{prop:K/k-trace}, and thus $\MWJ=J(\Fq(u))$.
Moreover,
$$
  \frac{\det(J(\Fq(u))/\tor)}{|J(\Fq(u))_\tor|^2}
  = [J(\Fq(u)):V]^{-2} \frac{\det(V/\tor)}{|V_\tor|^2}.
$$
We also have
$$
  \frac{\det(V/\tor)}{|V_\tor|^2}
  = (d-1)^{(r-1)(d-2)}r^{-d-2}d^{2-2r}.
$$
by Corollaries~\ref{cor:visible} and \ref{cor:Vdet}.

Applying Proposition~\ref{prop:integrality} gives that
$$\frac{(d-1)^{(r-1)(d-2)}}{[J(\Fq(u)):V]^2}$$
is an integer.  Since $d=p^\nu+1$, this shows that the index is a
power of $p$.  By Corollary~\ref{cor:ptors}, $J(\Fq(u))$ has no
$p$-torsion, so $J(\Fq(u))_\tor=V_\tor$.  This completes the proof of
the theorem.  \qed

%%%%%%%%%%%%%%%%%%%%%%%%%%%%%%%%%%%%%%%%%%%%%%%%%%%%%%%%%%%%%%%%

\section{Tamagawa number}\label{s:tamagawa}

In this section we compute the Tamagawa number $\tau(J/\F_q(u))$ of
the Jacobian $J$ of $C$ over $\Fq(u)$.  First, we review the
definition for a general abelian variety over a function field and
show how to calculate $\tau$ in terms of more familiar invariants.
Next, we specialize to the case of a Jacobian and relate $\tau$ to
invariants of the curve.  Finally, we specialize to the Jacobian $J$
over $\Fq(u)$ studied in the rest of this paper.

\subsection{Tamagawa numbers of abelian varieties over function
  fields}
\newcommand{\Ad}{\mathbb{A}}
\newcommand{\AAA}{\mathcal{A}}

Let $\BB$ be a curve of genus $g_\BB$ over $k=\Fq$.  Let $F=\Fq(\BB)$
be the function field of $\BB$, and let $\Ad_F$ be the ad\`eles of
$F$.  There is a natural measure $\mu=\prod\mu_v$ on $\Ad_F$ where
$\mu_v$ is the Haar measure that assigns measure 1 to the ring of
integers $\OO_v$ in the completion $F_v$ for each place $v$ of $F$.
The quotient $\Ad_F/F$ is compact and we set $D_F=\mu(\Ad_F/F)$.
By \cite[2.1.3]{WeilAAG},
\begin{equation}\label{eq:disc}
D_F=q^{g_\BB-1}.  
\end{equation}

Let $A$ be an abelian variety of dimension $g$ over $F$ and $\omega$ a
top-degree differential on $A$.  For each $v$, the differential
$\omega$ induces a differential $\omega_v$ on the base change $A_v$ of
$A$ to $F_v$.  Using $\mu_v$, this produces a measure
$|\omega_v|\mu_v^g$ on $A_v(F_v)$.  When the differential $\omega_v$
is a N\'eron differential, then Tate has shown
(cf.~\cite[1.4]{TateBourbaki}) that
\begin{equation}
  \label{eq:integral}
  \int_{A_v(\OO_v)}|\omega_v|\mu_v^g=\frac{\#\AAA(\F_v)}{q_v^g}
\end{equation}
where $\F_v$ is the residue field at $v$, $q_v=q^{\deg(v)}$ is its
cardinality, and $\#\AAA(\F_v)$ is the number of points on the special
fiber of the N\'eron model of $A$.  Thus if $\#\AAA(\F_v)^\circ$ is
the number of points on the identity component of the special fiber of
the N\'eron model of $A$ and we set
\begin{equation}
  \label{eq:lambda}
\lambda_v=\frac{\#\AAA(\F_v)^\circ}{q_v^g},  
\end{equation}
then $\{\lambda_v\}$ is a set of convergence factors in the sense of
Weil \cite[2.3]{WeilAAG}.  In this situation, we may form the product
measure
$$\Omega=\Omega(F,\omega,(\lambda_v))=
D_F^{-g}\prod_v\lambda_v^{-1}|\omega_v|\mu_v^g.$$ 
By the product formula, $\Omega$ is independent of the choice of
$\omega$.  Finally, we define the {\it Tamagawa number\/} $\tau(A/F)$
to be the measure of the set of $\Ad_F$ points of $A$ with respect to
$\Omega$.

Since $A$ is a projective variety, $A(\Ad_F)=\prod_vA(\OO_v)$ and the
measure can be computed as a product of local factors:
$$\tau(A/F)=D_F^{-g}
\prod_v\lambda_v^{-1}\int_{A_v(\OO_v)}|\omega_v|\mu_v^g.$$
Using \eqref{eq:integral} and \eqref{eq:lambda}, the local factor
$\lambda_v^{-1}\int_{A_v(\OO_v)}|\omega_v|\mu_v^g$ is equal to
$q_v^{f_v}d_v$ where $d_v$ is the order of the group of components on
the N\'eron model at $v$ and $f_v$ is the integer such that
$\pi_v^{f_v}\omega_v$ is a N\'eron differential at $v$.  (Here $\pi_v$
is a uniformiser at $v$.)  Thus the product of local terms is
$$\prod_v q_v^{f_v}d_v=q^{\sum_v\deg(v)f_v}\prod_v d_v.$$

We want to write $\sum_v\deg(v)f_v$ as a global invariant.  Let
$\sigma:\AAA\to\BB$ be the \Neron{} model of $A/L$, let $z:\BB\to\AAA$ 
be the zero section, and let
$$
  \omega
  = z^*\left(\bigwedge^{g_X}\Omega^1_{\JJ/\BB}\right)
  = \bigwedge^{g_X}\left(z^*\Omega^1_{\JJ/\BB}\right).
$$
This is an invertible sheaf on $\BB$ whose degree we denote $\deg\omega$.
It is then clear from the definition above of $f_v$ that $\sum_v
\deg(v)f_v=-\deg\omega$.  Combined with the local calculation in the
preceding paragraph, this yields
\begin{equation}
  \label{eq:tau-boiled-down}
  \tau(A/F) = q^{g(1-g_\BB)-\deg(\omega)} \prod_v d_v.
\end{equation}

\subsection{Tamagawa numbers of Jacobians}
Let $\BB$ be a smooth, projective, geometrically irreducible curve of
genus $g_\BB$ over $k=\Fq$.  Let $F=k(\BB)$ be its function field, and
let $X/F$ be a smooth, projective, and geometrically irreducible curve
of genus $g_X$.

We give ourselves two sorts of models of $X$.  First, let $\SS'$ be a
normal, projective, geometrically irreducible surface over $k=\Fq$
equipped with a surjective morphism $\pi':\SS'\to\BB$ whose generic
fiber is $X/F$.  We assume that $\pi'$ is smooth away from a finite
set of points.  We also assume that $\pi'$ admits a section
$s:\BB\to\SS'$ whose image lies in the smooth locus of $\SS'$.
Furthermore, assume that $\SS'$ has at worst rational double point
singularities (cf.~\cite[Ch.~3]{Badescu}).  Note that the
singularities of $\SS'$ lie in the singularities of $\pi'$, since
$\BB$ is smooth.

Second, let $\sigma:\SS\to\SS'$ be a minimal resolution of
singularities, so that the composition
$\pi=\pi'\compose\sigma:\SS\to\SS'\to\BB$ is a minimal regular model
of $X/F$.  In the applications, $\SS'$ is the model $\YY$ constructed
in Chapter~\ref{ch:models} and $\SS$ is $\XX$.

Now let $A/F$ be the Jacobian of $X/F$, and let $\tau:\AAA\to\BB$ be
the \Neron{} model of $A/F$, with zero section $z:\BB\to\AAA$.  Our
goal in this subsection is to describe the invariants entering into the
Tamagawa number of $A$ in terms of the surfaces $\SS$ or $\SS'$.

We first consider the local term $d_v$, the order of the group of
connected components of the fiber of the N\'eron model at $v$.  The
next proposition is not strictly necessary for our purposes (because
we were able to determine the $d_v$ from examples treated in
\cite[Section 9.6]{blr}), but we include it for completeness.

Let $X_0,\dots,X_n$ be the reduced irreducible components of
$\pi^{-1}(v)$.  We number them so that the section $s$ passes through
$X_0$.  Let $M$ be the $n\times n$ matrix of intersection numbers:
$$M_{i,j}=(X_i\cdot X_j), \qquad 1\le i,j\le n.$$
(Note that we do not include intersections with $X_0$.)

\begin{prop}\label{prop:comp=det}
  $d_v=\det M$.
\end{prop}

\begin{proof}
  This follows from \cite[Theorem~9.6.1]{blr}.  Indeed, let
  $I=\{0,\dots,n\}$, and for $i\in I$, let $\delta_i$ be the
  multiplicity of $X_i$ in $\pi^{-1}(v)$ and $e_i$ the geometric
  multiplicity of $X_i$.  (These integers are defined more precisely
  in \cite[Definition~9.1.3]{blr}.)  Since the section $s$ passes
  through $X_0$, we have $\delta_0=1$.  Since the residue field $\F_v$
  is a finite extension of $\Fq$, and is therefore perfect, we have
  $e_i=1$ for all $i$.  (A reduced scheme over $\F_v$ remains reduced
  after base change to the algebraic closure of $\F_v$.)

  Let $\Z^I$ be the free abelian group on $I$.  Let $\beta:\Z^I\to\Z$
  be given by $\beta(a_0,\dots,a_n)=\sum a_i\delta_i$, and let
  $\alpha:\Z^I\to\Z^I$ be given by the intersection matrix
$$\left(e_i^{-1}(X_i\cdot X_j)\right)_{i,j\in I}=
\left(X_i\cdot X_j\right)_{i,j\in I}.$$
Then \cite[Theorem~9.6.1]{blr} says that the group of connected
components of $\AAA$ at $v$ is canonically isomorphic to
$\ker\beta/\im\,\alpha$.  Because $\delta_0=1$, we may identify
$\ker\beta$ with the free abelian group on $X_1,\dots,X_n$, and the
result then follows immediately from the definition of $M$.
\end{proof}

Next we turn to the global invariant $\deg(\omega)$.

\begin{prop}\label{prop:omega-via-YY}
  Let $\SS^o\subset\SS'$ be the smooth locus, and let
  $\pi^o:\SS^o\to\BB$ the restriction of $\pi'$ to $\SS^o$.  Then
  there is an isomorphism
$$z^*\Omega^1_{\AAA/\BB}\cong\pi^o_*\Omega^1_{\SS^o/\BB}.$$
In particular,
$$\omega\cong 
\bigwedge^{g_X}\left(\pi^o_*\Omega^1_{\SS^o/\BB}\right).$$
\end{prop}

\begin{proof}
  Let $\Lie(G)\to\BB$ be the Lie algebra of a group scheme $G\to\BB$
  (see \cite[II.2]{SGA3}).  By \cite[9.7/1]{blr}, the \Neron{} model
  $\AAA\to\BB$ represents the relative Picard functor
  $\Pic^0_{\SS/\BB}$ since $\SS$ is smooth and $\pi$ admits a section.
  Therefore
  $$ (z^*\Omega^1_{\AAA/\BB})^\vee
     \cong \Lie(\AAA/\BB)
     \cong \Lie(\Pic^0(\SS/\BB))
     \cong \Lie(\Pic(\SS/\BB))
     \cong R^1\pi_*\OO_{\SS}
  $$
  by \cite[1.1 and 1.3]{LiuLorenziniRaynaud04}.
  
  By \cite[Corollary~24]{KleimanDuality} (with $X=\SS$, $Y=\BB$, and
  $S=\spec k$), the relative dualizing sheaf $\omega_{\SS/\BB}$ exists
  and satisfies
  $$ (R^1\pi_*\OO_{\SS})^\vee\cong \pi_*\omega_{\SS/\BB}$$
and
$$\omega_{\SS/\BB}\cong 
\Omega^2_{\SS/k}\tensor\pi^*(\Omega^1_{\BB/k})^\vee.$$
Combining these facts and using the projection formula, we have
$$z^*\Omega^1_{\AAA/\BB}\cong
\pi_*\Omega^2_{\SS/k}\tensor(\Omega^1_{\BB/k})^\vee.$$

To finish the proof, we show that 
$$\pi_*\Omega^2_{\SS/k}\cong
\pi^o_*\Omega^1_{\SS^o/\BB}\tensor\Omega^1_{\BB/k}.$$
To that end, let $\omega_{\SS/k}$, $\omega_{\SS'/k}$, and
$\omega_{\SS^o/k}$ be the dualizing sheaves of $\SS$, $\SS'$, and
$\SS^o$ respectively.  Since these surfaces have at worst rational
double points, their dualizing sheaves are invertible \cite[\S3.11,
Corollary~4.19]{Badescu}, and since $\SS$ and $\SS^o$ are smooth,
$\omega_{\SS/k}\cong\Omega^2_{\SS/k}$ and
$\omega_{\SS^o/k}\cong\Omega^2_{\SS^o/k}\cong\omega_{\SS'/k}|_{\SS^o}$.
Moreover, by \cite[Corollary~4.19]{Badescu},
$\sigma_*\omega_{\SS}\cong\omega_{\SS'}$.  Thus
$$\pi_*\Omega^2_{\SS/k}\cong\pi'_*\omega_{\SS'}
\cong\pi^o_*\Omega^2_{\SS^o/k}$$
where the second isomorphism holds because the complement of $\SS^o$
in $\SS'$ has codimension 2.  Finally, since $\pi^o:\SS^o\to\BB$ is smooth,
there is an exact sequence of locally free sheaves
$$0\to\pi^{o*}\Omega^1_{\BB/k}\to\Omega^1_{\SS^o/k}
\to\Omega^1_{\SS^o/\BB}\to0,$$
and so, taking the second exterior power,
$$\Omega^2_{\SS^o/k}\cong
\pi^{o*}\Omega^1_{\BB/k}\tensor\Omega^1_{\SS^o/\BB}.$$
\end{proof}

\subsection{The Tamagawa number of $J$}
We now specialize to the Jacobian $J$ that is the subject of
this paper.  The main result of this section is the following
calculation of the Tamagawa number of $J$.

\begin{prop}\label{cor:tau-value}
If $r$ divides $d$ and if $d$ divides $q-1$, then
$$
  \tau(J/\Fq(u))=q^{-(d-2)(r-1)/2}d^{2r-2}r^{d+2}.
$$
\end{prop}

The proof occupies the rest of this subsection.

Suppose that $r$ divides $d$ and that $d$ divides $q-1$.  Recall that
in Section~\ref{ss:YY} we constructed proper models $\pi':\YY\to\BB$
and $\pi:\XX\to\BB$ of $C/\F_q(u)$ over $\BB=\Pu$, i.e., schemes with
proper morphisms to $\BB$ whose generic fibers are $C$.  The models
$\XX\to\BB$ and $\YY\to\BB$ have the properties required of
$\SS$ and $\SS'$ in the preceding section, so
Propositions~\ref{prop:comp=det} and \ref{prop:omega-via-YY} apply.

However, rather than applying Proposition~\ref{prop:comp=det}, we
simply refer to Proposition~\ref{prop:component-groups} to obtain:
$$\prod_vd_v=(rd^{r-1})^2r^d=d^{2r-2}r^{d+2}.$$
To finish the proof, we must compute 
$q^{g_C(1-g_\BB)-\deg(\omega)}=q^{r-1-\deg(\omega)}$.

Recall from Lemma~\ref{lemma:R-basis-of-1-forms} the relative 1-forms
$\omega_i$ which form a basis of the $R$-module
$H^0(U,\pi'_*\Omega^1_{\YY/\BB})$.

\begin{lemma}
  Each $\omega_i$ extends to a section in
  $H^0(\BB,\pi_*\Omega^1_{\YY/\BB})$ that has order of vanishing
  $di/r$ at $u=\infty$ and is non-vanishing everywhere else.
\end{lemma}

\begin{proof}
The proof of Lemma~\ref{lem:regular-omega_i} shows that $\omega_i$
extends to a nowhere vanishing section of $\pi_*\Omega^1_{\YY/\BB}$
over $\BB\smallsetminus\{\infty\}$.  There is an involution
$$
	(x,y,u)\mapsto (x/u^d,y/u^{d(r+1)/r},1/u),
$$
since $r$ divides $d$.  The pullback of $x^{i-1}dx/y^i$ via this
involution is $u^{di/r}x^{i-1}dx/y^i$ and thus it takes the non-zero
regular 1-form of $\YY_0$ to a regular 1-form on $\YY_\infty$ with
order of vanishing $di/r$.
\end{proof}

Clearly the sections $\omega_i\in H^0(\BB,\pi_*\Omega^1_{\YY/\BB})$
restrict to elements of $H^0(\BB,\pi^o_*\Omega^1_{\YY^o/\BB})$
where $\YY^o$ is the complement in $\YY$ of the finitely many
singularities of $\YY\to\BB$ (which, since $d>1$, also happen to be the
finitely many singular points of $\YY$, see
Proposition~\ref{prop:YY-props}).

We conclude that 
$$\omega_1\wedge\cdots\wedge\omega_{r-1}$$
yields a global section of $\wedge^{r-1}\pi^o_*\Omega^1_{\YY^o/\BB}$.
Moreover, the proof of Lemma~\ref{lemma:R-basis-of-1-forms} shows that
the $\omega_i$ are linearly independent on each fiber of $\YY\to\BB$
where $u$ is finite, and over $u=\infty$, the sections
$u^{di/r}\omega_i$ are (non-vanishing and) linearly independent.  It
follows that $\omega_1\wedge\cdots\wedge\omega_{r-1}$ is everywhere
regular, non-vanishing away from $u=\infty$, and has a zero of order
$$\sum_{i=1}^{r-1}\frac{di}{r}=\frac{d(r-1)}{2}$$
at $u=\infty$.  We conclude that
$\deg(\omega)=d(r-1)/2$ and
$$\tau(J/K)=q^{(r-1)-d(r-1)/2}\prod_vd_v=q^{-(d-2)(r-1)/2}d^{2r-2}r^{d+2},$$
as desired.
This completes the proof of Proposition~\ref{cor:tau-value}.
\qed

%%%%%%%%%%%%%%%%%%%%%%%%%%%%%%%%%%%%%%%%%%%%%%%%%%%%%%%%%%%%%%%%

\section{Application of the BSD formula} \label{S:appBSD}
We saw in Theorem~\ref{thm:BSD} that the Birch and Swinnerton-Dyer
conjecture holds for $J/\Fq(u)$.  Moreover, under the assumptions that
$r$ divides $d$, that $d=p^\nu+1$, and that $d$ divides $q-1$, we have
calculated most of the terms appearing in the leading coefficient
formula of this conjecture.  Synthesizing this leads to a beautiful
analytic class number formula relating the Tate-Shafarevich group
$\sha(J/\Fq(u))$ and the index $[J(\Fq(u)):V]$.

Before deriving this result, we compare the formulation of the BSD
conjecture in Theorem~\ref{thm:BSD} to that in \cite{KT}.

\subsection{Two variants of the refined BSD conjecture}\label{ss:BSDs}
At the time that Tate stated the BSD conjecture in its most general
form in \cite{TateBourbaki}, there was uncertainty as to the right
local factors of the $L$-function at places of bad reduction.  Tate
therefore used the Tamagawa principle to state the leading coefficient
part of the BSD conjecture.  The correct local factors were defined
later by Serre in \cite{Serre70}, and using them we formulate
the leading coefficient conjecture (as Theorem~\ref{thm:BSD}) in what
we feel is its most natural form.  However, the best reference for the
proof of the leading coefficient conjecture, namely \cite{KT}, uses
Tate's formulation.  In this subsection, we compare the two
formulations and show that they are equivalent for Jacobians of curves
with a rational point.

To that end, let $F$ be the function field of a curve over $\Fq$, let
$Y/F$ be a smooth projective curve of genus $g$ with an $F$-rational
point, and let $A$ be the Jacobian of $Y$.  Define local $L$-factors
for each place $v$ of $F$ by
$$L_v(q_v^{-s}):=\det\left(1-\Fr_v\,q_v^{-s}
  \left|H^1(A\times\overline{F},\Ql)^{I_v}\right.\right).$$ 
Let $\mu=\prod\mu_v$ and $D_L$ be as in Section~\ref{s:tamagawa}.
Choose a top-degree differential $\omega$ on $A$ and form the local
integrals
$$\int_{A(F_v)}|\omega_v|\mu_v^g$$
and the convergence factors
$$\lambda_v:=\frac{\#\AAA_v(\F_v)^0}{q_v^g}$$
as in Section~\ref{s:tamagawa}.  Finally, choose a finite set $S$ of
places of $F$ containing all places where $Y$ has bad reduction.

Tate's formulation of the leading term conjecture is that the leading
term as $s\to1$ of
$$\frac{D_F^g}{\left(\prod_{v\in
      S}\int_{A(F_v)}|\omega_v|\mu_v^g\right) \left(\prod_{v\not\in
      S}L_v(q_v^{-s})\right)}$$ 
is $|\sha(A/F)|R/{|A(F)_{tor}|^2}$.  On the other hand, our
formulation asserts that the latter quantity (i.e.,
$|\sha(A/F)|R/{|A(F)_{tor}|^2}$) is the leading coefficient as $s\to1$
of
$$\frac{L(A/F,s)}{\tau(A/F)}=\left(\prod_v L_v(q_v^{-s})^{-1}\right)
D_F^g\left(\prod_v\frac{\lambda_v}{\int_{A(F_v)}|\omega_v|\mu_v^g}\right)$$
where both products are over all places of $F$.  The factor on the
right is 1 if $v\not\in S$, so to see that the two formulations are
equivalent, it will suffice to show that
$$L_v(q_v^{-1})=\lambda_v$$
for all $v\in S$.  

In fact this equality holds for all $v$.  Indeed, \cite[Lemma, page
182]{Milne72} implies that $L_v(q_v^{-1})$ is equal to
$\#\Pic^0(Y_v)/q_v^g$ where $Y_v$ is the fiber at $v$ of a regular
minimal model of $Y$.  But as we noted in the proof of
Proposition~\ref{prop:invariants}, the assumption that $Y$ has a
rational point implies that $\Pic^0(Y_v)$ is the group of $\F_v$
rational points on $\AAA_v^0$, the identity component of the N\'eron
model of $A$ at $v$.  Thus
$$L_v(q_v^{-1})=\frac{\#\Pic^0(Y_v)}{q_v^g}
=\frac{\#\AAA_v^0(\F_v)}{q_v^g}=\lambda_v.$$
This completes the verification that the two formulations of the BSD
conjecture are equivalent.

\subsection{An analytic class number formula}
Now we turn to the application of the BSD conjecture to a formula for
the order of $\sha(J/\Fq(u))$.

\begin{theorem}\label{thm:sha}
  Assume that $r$ divides $d$, that $d=p^\nu+1$, and that $d$ divides
  $q-1$.  Then the Tate-Shafarevich group $\sha(J/\Fq(u))$ has order
$$|\sha(J/\Fq(u))|=[J(\Fq(u)):V]^2\left(\frac q{p^{2\nu}}\right)^{(r-1)(d-2)/2}.$$
In particular, its order is a power of $p$.  In the special case
$\Fq(u)=K_d$, then
$$|\sha(J/K_d)|=[J(K_d):V]^2.$$
\end{theorem}

\begin{proof}
By Corollary~\ref{cor:exact-rank} the leading coefficient of the
$L$-function is 
$$L^*(J/\Fq(u),1)=(\log q)^{(r-1)(d-2)}.$$  
Taking into account the factor of $\log q$ relating the $\Q$-valued
height pairing of Chapter~\ref{ch:heights} and the \Neron-Tate
canonical height, the BSD formula for the leading coefficient says
$$1=\frac{|\sha(J/\Fq(u))|\,\det(J(\Fq(u))/\tor)\,\tau(J/\Fq(u))}
{|J(\Fq(u)))_\tor|^2}.$$

Using that 
$$\det(J(\Fq(u))/\tor)=\frac{\det(V/\tor)}{[J(\Fq(u)):V]^2}$$ 
and our calculations
$$\det(V/\tor)=(d-1)^{(r-1)(d-2)}r^{4-d}d^{2-2r}$$
(Corollary~\ref{cor:Vdet}),
$$\tau(J/\Fq(u))=q^{-(d-2)(r-1)/2}d^{2r-2}r^{d+2}$$
(Proposition~\ref{cor:tau-value}),
and
$$|J(\Fq(u))_\tor|=r^3$$
(Theorem~\ref{thm:V-vs-MW}),
we find
$$|\sha(J/\Fq(u))|=[J(\Fq(u)):V]^2\left(\frac
  q{p^{2\nu}}\right)^{(r-1)(d-2)/2},$$
as desired.  

We showed in Theorem~\ref{thm:V-vs-MW} that $[J(\Fq(u)):V]$ is a
power of $p$, so the same is true of $|\sha(J/\Fq(u))|$.

The assertion for the special case $\Fq(u)=K_d$ follows from the fact
that the field of constants of $K_d$ is $\Fp(\mu_d)=\F_{p^{2\nu}}$.
\end{proof}

\begin{remark}
  Under the hypotheses of this section, it is possible to describe
  $\sha(J/\Fq(u))$ and $J(\Fq(u))/V$ as modules over the group ring
  $\Zp[\Gal(\Fq(u)/\Fp(t))]$ in terms of the combinatorics of the
  action by multiplication of the cyclic group $\langle
  p\rangle\subset(\Z/d\Z)^\times$ on the set $\Z/d\Z\times\mu_r$.
  See \cite[Section 9.4]{Legendre3} for details.
\end{remark}

%%% Local Variables: 
%%% mode: latex
%%% TeX-master: "EHR"
%%% End: 

%This is chapter 8

\chapter{Monodromy of $\ell$-torsion and 
decomposition of the Jacobian}\label{ch:monodromy}
%%%%%%%%%%%%%%%%%%%%%%%%%%%%%%%%%%%%%%%%%%%%%%%%%%%%%%%%%%%%%%%%%%%%%%%%

%%%%%%%%%%%%%%%%%%%%%%%%%%%%%%%%%%%%

In this chapter, we consider the action of Galois on torsion points of
the Jacobian $J$ and use the results to understand the decomposition
of $J$ up to isogeny into a sum of simple abelian varieties.  Our results
depend heavily on knowledge of the regular proper model $\XX\to\P^1$
constructed in Chapter~\ref{ch:models}.  Interested readers are
referred to \cite{Hall}, where a general technique for computing
monodromy groups of certain superelliptic curves is developed.  The
methods of \cite{Hall} yield results similar to those in this chapter
in a more general context without the need to construct regular
models.

\section{Statement of results}\label{cjh:sec:intro}

Let $k$ be an algebraically closed field of characteristic $p\geq 0$,
let $r\geq 2$ be an integer not divisible by $p$, and let $\ell$ be a
prime satisfying $\ell\neq p$ and $\ell\nmid r$.  As in the rest of
this paper, let $C=C_r$ be the smooth, projective curve over $K=k(t)$
birational to the affine curve given by
\begin{equation}\label{cjh:eqn1}
	y^r = x^{r-1}(x+1)(x+t),
\end{equation}
let $J$ be its Jacobian, and let $J[\ell]$ be the Galois module of
$\ell$-torsion.  

In this chapter, we study the structure of a monodromy group, namely
the Galois group of $K(J[\ell])$ over $K$.  We use the results about
the monodromy group to bound $\ell$-torsion over solvable extensions
of $K$ and to determine how $J$ decomposes up to isogeny into a sum of
simple abelian varieties, both over $K$ and over $\Kbar$.

We first state the consequences of the monodromy result that
motivated its study, then discuss the monodromy result itself.

\begin{theorem}\label{cjh:p1}
  If $L/K$ is an abelian extension, then $J[\ell](L)=\{0\}$.  If
  $\ell>3$ or $r$ is odd, then the same holds for any solvable extension
  $L/K$.
\end{theorem}

% we could say r<>2,4 or ell>3 if we did the new part separately.

In the following section we define the ``new part'' of $J$,
denoted $J^{new}_r$, and we show that there is an isogeny 
\begin{equation}\label{eq:isog}
\bigoplus_{s|r}J^{new}_s\longrightarrow J
\end{equation}
over $K$, where the sum runs over positive divisors $s$ of $r$ and
$J^{new}_s$ is the new part of the Jacobian of $C_s$.  It turns out
that $J^{new}_1=J_1=0$ and that $J^{new}_s$ has dimension $\phi(s)$
where $\phi(s)$ is the cardinality of $(\Z/s\Z)^\times$ when $s >1$.
Moreover, the action of $\mu_r$ on $C_r$ induces an action of the
ring of integers $\Z[\zeta_r]\subset\Q(\zeta_r)$ on $J^{new}_r$.  Our
second main result says that $J^{new}_r$ does not decompose further
over certain extensions of $K$:

\begin{theorem}\label{cjh:p2}
  The new part $J^{new}_r$ is simple over $K$, and
  $\End_K(J^{new}_r)\cong\Z[\zeta_r]$.  The same conclusions hold over
  $K(u)$ where $u^d=t$ for any positive integer $d$ not divisible by $p$.
\end{theorem}

If $r=2$, $J^{new}_r$ is an elliptic curve, so is obviously absolutely
simple.  Moreover, it is non-isotrivial, so $\End_{\overline
  K}(J^{new}_r)=\Z$.  For $r>2$, although $J^{new}_r$ is simple over
many extensions of $K$, we see below it is not absolutely
simple.

Write $\Z[\zeta_r]^+=\Z[\zeta_r+\zeta_r^{-1}]$ for the ring of
integers in the real cyclotomic field $\Q(\zeta_r)^+$.

\begin{theorem}\label{thm:decomp}
Suppose that $r>2$, and let $K'=K((1-t)^{1/r})$.  Then there is an
abelian variety $B$ defined over $K$ such that:
\begin{enumerate}
\item There is an isogeny $J^{new}_r\to B^2$ defined over
  $K'$ whose kernel is killed by multiplication by $2r$.
\item $\End_K(B)=\End_{\overline K}(B)=\Z[\zeta_r]^+$, and $B$ is
  absolutely simple.
\end{enumerate}
In particular, $\End_{\overline K}(J^{new}_r)$ is isomorphic to an
order in $M_2(\Z[\zeta_r]^+)$.
\end{theorem}

In Section~\ref{ss:twist} below, we introduce a twist $C_\chi$ of $C$
(closely related to the extension $K'/K$) with Jacobian $J_\chi$ and
new part $A_\chi:=J_\chi^{new}$.  The curve $C_\chi$ has an involution
$\sigma$ that allows us to show that $A_\chi$ is isogenous to
$B\times B$ over $K$.  Since $A_\chi$ becomes isomorphic to
$J_r^{new}$ over $K'$, this explains the factorization in
Theorem~\ref{thm:decomp}.

The theorems above are applications of results on the monodromy groups
of $J[\ell]$ and $J_\chi[\ell]$, in other words on the image of the
natural homomorphisms from $\gal(K^{sep}/K)$ to
$\Aut_{\F_\ell}(J[\ell])$ and $\Aut_{\F_\ell}(J_\chi[\ell])$.  Our
detailed knowledge of the regular proper model $\XX\to\P^1$ of $C$ and
of the N\'eron model of $J$ (in Chapter~\ref{ch:models}) together with
some group theory allow us to determine the monodromy groups.

To define the group-theoretic structure of the monodromy group,
consider $\Lambda=\Fl[z]/(z^{r-1}+\cdots+1)$, which is a quotient of
the group ring of $\mu_r$ over $\Fl$.  The torsion points $J[\ell]$
and $J_\chi[\ell]$ have natural structures of free, rank 2 modules
over $\Lambda$, and $J_\chi[\ell]$ admits an action of $\sigma$ that
``anti-commutes'' with the $\mu_r$ action.  We ultimately find that
for $\ell>3$, the monodromy group of $J_\chi[\ell]$ is
$$\SL_2(\Lambda^+)\subset\GL_2(\Lambda)$$
where $\Lambda^+$ is the subring of $\Lambda$ generated by
$\zeta+\zeta^{-1}$ and $\zeta$ is the class of $z$ in $\Lambda$.  This
is very natural, because $\SL_2(\Lambda^+)$ is the commutator subgroup
of the centralizer in $\GL_2(\Lambda)$ of the semi-direct product
$\mu_r\sdp\langle\sigma\rangle$.  The results of \cite{Hall} extend
this conclusion to a broad class of superelliptic Jacobians.

%%%%%%%%%%%%%%%%%%%%%%%%%%%%%%%%%%%%

\section{New and old}
In this section, we establish the decomposition of $J$ into new and old
parts, leading to the isogeny~\eqref{eq:isog}.

It is convenient to work with coordinates on $C$ different than
those in \eqref{cjh:eqn1}.  Namely, for $s$ a positive divisor of $r$,
let $C_s$ be the smooth, projective curve birational to the affine
curve
\begin{equation}\label{eq:Cs}
x_sy_s^s=(x_s+1)(x_s+t).
\end{equation}
(For $s=r$, the coordinates here are related to those in
\eqref{cjh:eqn1} by $(x,y)=(x_r,x_ry_r)$.)  For positive divisors $s$
and $s'$ of $r$ with $s'$ dividing $s$, there is a morphism
$\pi_{s,s'}:C_s\to C_{s'}$ defined by
$\pi_{s,s'}:(x_s,y_s)\mapsto(x_{s'},y_{s'})=(x_s,y_s^{s/s'})$.

Let $J_s$ be the Jacobian of $C_s$; it is a principally polarized
abelian variety of dimension $s-1$.  By Albanese functoriality (push
forward of divisors), $\pi_{s,s'}$ induces a map $J_s\to J_{s'}$, which
we denote again by $\pi_{s,s'}$.  Picard functoriality (pull back of
divisors) induces another map $\pi^*_{s,s'}:J_{s'}\to J_s$.
Considering $\pi_{s,s'}$ and $\pi^*_{s,s'}$ at the level of divisors
shows that the endomorphism $\pi_{s,s'}\compose\pi^*_{s,s'}$ of
$J_{s'}$ is multiplication by $s/s'$.

The group $\mu_r\subset k^\times$ acts on $C_r$ by
$\zeta_r(x_r,y_r)=(x_r,\zeta_ry_r)$.  We let $\mu_r$ act on $C_s$ via
the quotient map $\mu_r\to\mu_s$, so that
$\zeta_r(x_s,y_s)=(x_s,\zeta_r^{r/s}y_s)$.  With these definitions,
the induced maps $\pi_{s,s'}:J_s\to J_{s'}$ and
$\pi^*_{s,s'}:J_{s'}\to J_s$ are equivariant for the $\mu_r$ actions.

Let $R$ be the group ring $\Z[\mu_r]$.  (This agrees with the
definition of $R$ in Section~\ref{ss:R} since $d=1$.)  Then
$\pi_{s,s'}$ and $\pi^*_{s,s'}$ are homomorphisms of $R$-modules.

Now we define $J^{new}_s$ as the identity component of the
intersection of the kernels of $\pi_{s,s'}$ where $s'$ runs through
positive divisors of $s$ strictly less than $s$:
$$J^{new}_s:=\left(\bigcap_{s'<s}\ker
    \left(\pi_{s,s'}:J_s\to J_{s'}\right)\right)^0.$$
% connected component not needed since C_r\to C_s is totally ramified 
Note that $J^{new}_s$ is preserved by the action of $\mu_r$ on $J_s$.

The main result of this section is a decomposition of $J_r$ up to
isogeny. 

\begin{proposition}\label{prop:isog}
For $s>1$, the dimension of $J^{new}_s$ is $\phi(s)$ and
$J^{new}_1=J_1=0$.
The homomorphism
\begin{align*}
\bigoplus_{s|r}J^{new}_s&\to\quad J\\
(z_s)\qquad&\mapsto\quad\sum_{s|r}\pi^*_{r,s}(z_s)
\end{align*}
is an isogeny whose kernel is killed by multiplication by $r$.
\end{proposition}

\begin{proof}
  The cotangent space at the origin of $J_s$ is canonically isomorphic
  to the space of 1-forms $H^0(C_s,\Omega^1_{C_s/k})$, so we may
  compute the differential of $\pi^*_{s,s'}:J_{s'}\to J_s$ by
  examining its effect on 1-forms.

  We computed the space of 1-forms on $C_s$ in the proof of
  Lemma~\ref{lemma:R-basis-of-1-forms}.  In terms of the coordinates
  used here, $H^0(C_s,\Omega^1_{C_s/k})$ has a basis consisting of
  eigenforms for the action of $\mu_r$, namely
  $\omega_{s,i}=y_s^{-i}dx_s/x_s$ for $i=1,\dots,s-1$.  It is then
  evident that
%don't need this since it is obvious from the formulas:
% $\pi_{s,s'}\compose\pi^*_{s,s'}$ is multiplication
% by $s/s'$, so 
$\pi^*_{s,s'}$ induces the inclusion on $1$-forms
$$H^0(C_{s'},\Omega^1_{C_{s'}/k})\hookrightarrow 
H^0(C_{s},\Omega^1_{C_{s}/k})$$
that sends $\omega_{s',i}$ to $\omega_{s,(s/s')i}$.

It follows that the cotangent space of $J^{new}_s$ is spanned
by the 1-forms $\omega_{s,i}$ where $i$ is relatively prime to $s$.
In particular, for $s>1$, the dimension of $J^{new}_s$ is $\phi(s)$.
For $s=1$, $C_s$ has genus 0, so $J^{new}_1=J_1=0$.

It is also clear that the map displayed in the statement of the
proposition induces an isomorphism on the cotangent spaces, so it is a
separable isogeny.  It remains to prove that the kernel is killed by
$r$.

Write $r$ as a product of primes $r=\ell_1\cdots\ell_m$.  We proceed
by induction on $m$.   If $r=\ell_1$ is prime, the result is obvious,
since $J_1=0$ and $J^{new}_{\ell_1}=J_{\ell_1}$. 

Before giving the proof for general $r$, we note that if
$\ell_1$ divides $r$, considering the action of the maps $\pi_{s,s'}$
on divisors yields the formula:  
\begin{equation}\label{eq:pi's}
\pi_{r,r/\ell_1}\compose\pi_{r,s}^*=
\begin{cases}
\ell\pi_{r/\ell_1,s}&\text{if $s$ divides $r/\ell_1$,}\\
\pi_{r/\ell_1,s/\ell_1}^*\compose\pi_{s,s/\ell_1}&\text{otherwise.}
\end{cases}
\end{equation}

Now suppose that $(z_s)_{s|r}$ is in the kernel, i.e.,
$$0=\sum_{s|r}\pi_{r,s}^*(z_s)$$
in $J_r$.  Applying $\pi_{r,r/\ell_1}$ and using the
formula~\eqref{eq:pi's}, we have
\begin{align*}
0&=\sum_{s|r}\pi_{r,r/\ell_1}\pi_{r,s}^*(z_s)\\
&=\ell_1\sum_{s|(r/\ell_1)}\pi_{r/\ell_1,s}(z_s)+
\sum_{s\nmid(r/\ell_1)}\pi_{r/\ell_1,s/\ell_1}^*\pi_{s,s/\ell_1}(z_s)\\
&=\ell_1\sum_{s|(r/\ell_1)}\pi_{r/\ell_1,s}(z_s)
\end{align*}
where the last equality holds because $z_s$ is in $J^{new}_s$, so is
killed by $\pi_{s,s/\ell_1}$.  By induction, each $\ell_1z_s$ is
killed by $r/\ell_1$, so each $z_s$ with $s|(r/\ell_1)$ is
$r$-torsion.  Repeating the argument with $\ell_1$ replaced by the
other $\ell_i$ implies that all the $z_s$ with $s<r$ are $r$-torsion.
Finally, the equality $0=\sum_{s|r}\pi_{r,s}^*(z_s)$ in $J_r$ implies
that $z_r$ is $r$-torsion as well.
\end{proof}

\begin{remarks}\mbox{}
\begin{enumerate}
\item We used that $J_1=0$, but this is not necessary.  A slight
  variant of the argument works for the new part of any cyclic cover
  $C_r\to C_1$ even when $C_1$ is not assumed to be rational.  
\item Temporarily write $J^{new,sub}_r$ for $J^{new}_r$ as defined
  above.  We could also consider a new quotient:
$$J^{new,quot}_r=\frac{J_r}{\sum_{s<r}\pi_{r,s}^*J_s}.$$
Arguments similar to those in the proof above show that the natural
map $J^{new,sub}_r\to J_r\to J^{new,quot}_r$ is an isogeny whose
kernel is killed by $r$.
\end{enumerate}
\end{remarks}

\begin{cor}\label{cor:torsion-decomp}
Suppose that $\ell$ is a prime not dividing $r$.
Then there is an isomorphism of $\Fl$-vector spaces
$$\bigoplus_{s|r}J^{new}_s[\ell]\cong J_r[\ell]$$
compatible with the action of $\mu_r$ and the action of the Galois
group $\Gal(K^{sep}/K)$.
\end{cor}

\begin{proof}
  The isomorphism is immediate from Proposition~\ref{prop:isog},
since $\ell$ does not divide $r$.
\end{proof}

\section{Endomorphism rings}\label{s:rings}
In this section we define a ring $\Lambda$ that acts naturally on
$J[\ell]$ and record some auxiliary results about it.  As always, $r>1$
is an integer and $\ell$ is a prime not dividing $r$.

\subsection{Definition of $\Lambda$}
For each positive divisor $s$ of $r$, let $\Phi_s(z)$ be the $s$-th
cyclotomic polynomial, and let $\Psi_s(z)=z^{s-1}+\cdots+1$.  Then
$$\prod_{s|r}\Phi_s(z)=z^r-1\quad\text{and}\quad
\prod_{1<s|r}\Phi_s(z)=\Psi_r(z).$$

Consider the group ring of $\mu_r$ over $\Fl$:
$$\Fl[\mu_r]\cong\frac{\Fl[z]}{(z^r-1)}$$
and its quotient
$$\Lambda:=\frac{\Fl[z]}{(\Psi_r(z))}=\frac{\Fl[z]}{(z^{r-1}+\cdots+1)}.$$
We often write $\zeta$ for the class of $z$ in $\Fl[\mu_r]$ or
$\Lambda$.

Since $\ell$ does not divide $r$, the $r$-th roots of unity are
distinct in $\Flbar$, so the polynomials $\Phi_s$ are pairwise
relatively prime in $\Fl[z]$.  By the Chinese Remainder Theorem,
$$\Lambda=\frac{\Fl[z]}{(\Psi_r(z))}\cong
\prod_{1<s|r}\frac{\Fl[z]}{(\Phi_s(z))}$$
and 
$$\Fl[\mu_r]\cong\Fl\oplus\Lambda$$
where $(1+\zeta+\cdots+\zeta^{r-1})/r$
on the left corresponds to $(1,0)$ on the right.

Note that $\OO_s:=\Z[z]/(\Phi_s(z))$ is isomorphic to the ring of
integers $\Z[\zeta_s]$ in the cyclotomic field $\Q(\zeta_s)$ and that
$\OO_s/\ell\cong\Fl[z]/(\Phi_s(z))$.  Therefore
$$\Lambda\cong\prod_{1<s|r}\OO_s/\ell,$$
and $\zeta$ on the left maps to an $s$-th root of unity $\zeta_s$ in
the factor $\OO_s/\ell$ on the right, justifying the notational use of
$\zeta$ on the left.  This isomorphism is convenient as it allows us
to use certain well-known results from the theory of cyclotomic
fields.

\subsection{The subring $\Lambda^+$}
Consider the involution of $\Fl[\mu_r]$ that sends $\zeta$ to
$\zeta^{-1}$.  We write $\Fl[\mu_r]^+$ for the subring of invariant
elements.  The factors in the decomposition
$\Fl[\mu_r]\cong\Fl\oplus\Lambda$ are preserved by the involution,
and we write $\Lambda^+$ for the invariant subring
$\Fl[\mu_r]^+\cap\Lambda$.

\begin{lemma}\mbox{}
  \begin{enumerate}
  \item $\Lambda^+$ is the subring of $\Lambda$ generated by
    $\zeta+\zeta^{-1}$.
  \item Let $\OO_s^+$ be the ring of integers in the real cyclotomic
    field $\Q(\zeta_s+\zeta_s^{-1})$.  Then
$$\Lambda^+\cong\prod_{1<s|r}\OO_s^+/\ell.$$
  \end{enumerate}
\end{lemma}

\begin{proof}
  (1) The group ring $\Fl[\mu_r]$ has $\Fl$-basis
  $1,\zeta,\dots,\zeta^{r-1}$, and $\Lambda$ is the quotient by the
  line generated by $1+\cdots+\zeta^{r-1}$.  Let
  $\tau_i=\zeta^i+\zeta^{-i}$.  If $r$ is odd, it is clear that
  $\Fl[\mu_r]^+$ has basis $1,\tau_1,\dots,\tau_{(r-1)/2}$.  If $r$
  is even, a basis of $\Fl[\mu_r]^+$ is given by
  $1,\tau_1,\dots,\tau_{(r-2)/2},\zeta^{r/2}$.  Since $\ell\neq2$ when
  $r$ is even, $\tau_{r/2}=2\zeta^{r/2}$ and another basis is
  $1,\tau_1,\dots,\tau_{r/2}$.  Projecting to $\Lambda$, we see that
  $1,\dots,\tau_u$ is a basis of $\Lambda^+$, where $u$ is $(r-3)/2$
  or $(r-2)/2$ as $r$ is odd or even.  Since $\tau_1^i=\tau_i$ plus a
  linear combination of $1$ and the $\tau_j$ with $j<i$, it follows
  that $\Lambda^+$ is generated as a ring by $\tau_1$.

  (2) Under the isomorphism $\Lambda\cong\prod_{1<s|r}\OO_s/\ell$, the
  involution on the left coresponds to complex conjugation on the
  right.  Taking invariants yields
$$\Lambda^+\cong\prod_{1<s|r}(\OO_s/\ell)^+.$$
By part (1), $(\OO_s/\ell)^+$ is generated as a ring by the image of
$\zeta+\zeta^{-1}$.  Since $\OO_s^+$ is generated as a ring by
$\zeta_s+\zeta_s^{-1}$ \cite[Proposition~2.16]{WashingtonCF}, 
the reduction map $\OO_s^+\to(\OO_s/\ell)^+$ is surjective, so
$(\OO_s/\ell)^+\cong\OO_s^+/\ell$, and this completes the proof.
\end{proof}

\subsection{Primes of $\Lambda$ and $\Lambda^+$}
Since $\ell$ does not divide $r$, the roots of $\Psi_r(z)$ are
distinct modulo $\ell$, and so $\Lambda$ and $\Lambda^+$ are
semi-simple algebras over $\Fl$.  

We write $\lambda$ for a prime ideal of $\Lambda$ and $\F_\lambda$ for
the quotient $\Lambda/\lambda$.  This is a finite extension field of
$\Fl$.  We say that $\lambda$ \emph{has level $s$} if the quotient map
$\Lambda\to\Lambda/\lambda$ factors through $\Lambda\to\OO_s/\ell$, or
equivalently, if $\Phi_s(z)\in\lambda$. Clearly each $\lambda$ has a
well-defined level $s>1$ that is a divisor of $r$, and we may
identify the primes of $\Lambda$ of level $s$ with the primes of
$\OO_s$ over $\ell$.

Similarly, for a prime $\lambda^+\subset\Lambda^+$, we define
$\F_{\lambda^+}:=\Lambda^+/\lambda^+$, and we define the level of
$\lambda^+$ to be the divisor $s$ of $r$ such that the quotient
$\Lambda^+\to\Lambda^+/\lambda^+$ factors through
$\Lambda^+\to\OO_s^+/\ell$.  Thus the primes of $\Lambda^+$ of level
$s$ are naturally identified with the primes of $\OO^+_s$ over $\ell$.

We say that $\lambda\subset\Lambda$ lies over
$\lambda^+\subset\Lambda^+$ if $\lambda\cap\Lambda^+=\lambda^+$.
In this case, if $\lambda$ has level $s$ then so does $\lambda^+$, and
the prime of $\OO_s$ corresponding to $\lambda$ lies over the prime of
$\OO_s^+$ corresponding to $\lambda^+$.

\subsection{Splitting of primes}
In this subsection, we focus on the ``new'' quotients
$\OO_r/\ell$ and $\OO_r^+/\ell$ of $\Lambda$ and $\Lambda^+$.  For
typographical convenience, we omit the subscript and write $\OO$ and
$\OO^+$ for $\OO_r$ and $\OO_r^+$.

We review the structure of $\OO^+/\ell$ and $\OO/\ell$, dividing
into three cases:
First, if $r=2$, then $\OO=\OO^+=\Z$ and
$\OO^+/\ell=\OO/\ell=\F_\ell$.

Before defining the second and third cases, we introduce some notation.
Let $o_r(\ell)$ be the order of $\ell$ in $(\Z/r\Z)^\times$.  Let
$o_r^+(\ell)$ be the order of $\ell$ in
$(\Z/r\Z)^\times/\langle\pm1\rangle$.  Standard results in cyclotomic
fields (see, e.g., \cite{WashingtonCF}, Chapter 2) indicate that
$\ell$ splits into $h=\phi(r)/(2o^+_r(\ell))$ primes in $\OO^+$. Write
$\lambda_1^+\dots,\lambda_h^+$ for the primes of $\OO^+$ over $\ell$.
Let $\F_{\lambda_i}:=\OO/\lambda_i$ and
$\F_{\lambda_i^+}:=\OO^+/\lambda_i^+$ be the
residue fields.

The second case, which we call the inert case, is when $r>2$ and $-1$
is congruent to a power of $\ell$ modulo $r$.  In this case,
$o_r(\ell)=2o^+_r(\ell)$.  Each $\lambda_i^+$ remains prime in $\OO$,
i.e., $\lambda_i=\lambda^+_i\OO$ is a prime ideal of $\OO$.  The
residue field $\F_{\lambda_i}$ is a quadratic extension of
$\F_{\lambda^+_i}$.

The third case, which we call the split case, is when $r >2$ and $-1$
is not congruent to a power of $\ell$ modulo $r$.  In this case,
$o^+_r(\ell)=o_r(\ell)$ and the $h$ primes $\lambda_i^+$ of $\OO^+$
over $\ell$ each split into two primes, call them $\lambda_i$ and
$\lambda_{g-i}$, in $\OO$, where $g=2h$.  The residue fields satisfy
$\F_{\lambda_i}\cong\F_{\lambda_{g-i}}\cong\F_{\lambda^+_i}$ and
$\OO/\lambda^+_i$ is a semi-simple quadratic algebra over
$\F_{\lambda^+_i}$, namely $\F_{\lambda_i}\oplus\F_{\lambda_{g-i}}$.

Via the identification of primes of $\OO$ and $\OO^+$ over $\ell$ with
primes of $\Lambda$ and $\Lambda^+$ of level $r$, the discussion in
the second and third cases applies to the splitting behavior of primes
$\lambda^+\subset\Lambda^+$ in $\Lambda$.

One of the reasons it is convenient to focus on the new part
$\OO=\OO_r$ is the possibility that the primes of $\OO_r^+$ over
$\ell$ may be inert in $\OO_r$ while the primes of $\OO_s^+$ over
$\ell$ may be split in $\OO_s$ for a divisor $s$ of $r$.

\subsection{Auxiliary results}
We record two lemmas to be used later.

Note that $\Lambda$ is a direct sum of fields $\F_\lambda$ and that
$\F_\lambda\cong\F_{\ell^{o_r(\ell)}}$ for
all $\lambda$ of level $r$.
However, the various $\F_\lambda$ are non-isomorphic as
$\Lambda$-modules.  Similarly, the various $\F_{\lambda^+}$ are
non-isomorphic as $\Lambda^+$-modules.  We state this
more formally for later use:

\begin{lemma}\label{lemma:non-iso}
  Suppose that $\lambda_1^+$ and $\lambda_2^+$ are distinct primes of
  $\Lambda^+$.  Then there does not exist an isomorphism of fields
  $\F_{\lambda_1^+}\cong\F_{\lambda_2^+}$ carrying the class of
  $\zeta+\zeta^{-1}$ in $\F_{\lambda_1^+}$ to its class in
  $\F_{\lambda_2^+}$.
\end{lemma}

\begin{proof}
  Since $\Lambda^+$ is generated over $\Fl$ by $\zeta+\zeta^{-1}$,
  a field isomorphism $\F_{\lambda_1^+}\cong\F_{\lambda_2^+}$ as in
  the statement would induce an isomorphism of $\Lambda^+$-modules.
  But the $\Lambda^+$-modules $\F_{\lambda_1^+}$ and
  $\F_{\lambda_2^+}$ are not isomorphic since they have distinct
  annihilators.
\end{proof}

\begin{lemma}\label{lemma:Flambda=F3}
Suppose that $\ell=3$.  Then the number of primes
$\lambda^+\subset\Lambda^+$ such that $\F_{\lambda^+}\cong\F_3$ is
$$\begin{cases}
0&\text{if $r$ is odd,}\\
1&\text{if $r\equiv2\pmod4$,}\\
2&\text{if $r\equiv0\pmod4$}.
\end{cases}$$
If $r\equiv2\pmod4$, the prime has level $2$, and if $r\equiv0\pmod4$
one of the primes has level 2 and the other has level 4.
\end{lemma}

\begin{proof}
  Suppose there is a prime $\lambda^+\subset\Lambda^+$ with
  $\F_{\lambda^+}\cong\F_3$ and choose a prime $\lambda\subset\Lambda$
  over it.  Then $\F_\lambda$ is a subfield of $\F_9$, so the
  multiplicative order of $\zeta$ in $\F_\lambda$ must divide 8 and
  the level of $\lambda$ must divide $8$.  (In particular, $\lambda^+$
  does not exist if $r$ is odd.)  To finish, we note that the unique
  prime of $\OO^+_8$ over $\ell=3$ has residue field $\F_9$, while
  $\OO_4^+$ and $\OO_2^+$, both being isomorphic to $\Z$, have unique
  primes over $3$, each with residue field $\F_3$.
\end{proof}

\section{The $\Lambda$-module structure of $J[\ell]$}
Recall that $\Lambda$ is $\Fl[z]/(z^{r-1}+\cdots+1)$ and that
$J[\ell]$ denotes the $\ell$-torsion in $J$.

\begin{proposition}\label{prop:rank2}
  The action of $\mu_r$ on $J$ gives $J[\ell]$ the structure of a free
  $\Lambda$-module of rank 2.  For every prime $\lambda$ of
  $\Lambda$, the submodule $J[\lambda]\subset
  J[\ell]$ of $\lambda$-torsion has the structure of a free
  $\F_\lambda=\Lambda/\lambda$-module of rank 2.
\end{proposition}

\begin{proof}
The action of $\mu_r$ on $C=C_r$ and $J=J_r$ gives the Tate module
$$V_\ell J\cong H^1(C,\Ql)$$
the structure of a module over 
$$\Ql[\mu_r]\cong\prod_{s|r}\Ql[z]/\Phi_s(z).$$
  
The map $C_r\to C_1=\P^1$ presents $C_r$ as a Galois branched cover of
$\P^1$ with Galois group $\mu_r$.  In this context, a formula of Artin
gives the character of $H^1(C_r,\Ql)$ as a representation of $\mu_r$
in terms of the ramification data of $C_r\to C_1$.  See
\cite[Corollary~2.8]{Milne} for the precise statement.  One finds that
the character is $2(\chi_{reg}-\chi_{triv})$ where $\chi_{reg}$ and
$\chi_{triv}$ are the characters of regular and trivial
representations respectively.  Thus $V_\ell J$ is isomorphic to the
direct sum of two copies of the regular representation modulo the
trivial representation.  Equivalently,
$$V_\ell J\cong\left(\prod_{1<s|r}\Ql[z]/\Phi_s(z)\right)^2$$
where the product is over divisors of $r$ that are $>1$.

Since $T_\ell J\subset V_\ell J$ is preserved by the action of
$\mu_r$ and $\ell$ is prime to $r$, we have that
$$T_\ell J\cong\left(\prod_{1<s|r}\Zl[z]/\Phi_s(z)\right)^2$$
and
\begin{align*}
J[\ell]&\cong\left(\prod_{1<s|r}\Fl[z]/\Phi_s(z)\right)^2\\
&\cong\Lambda^2.
\end{align*}
This is the first assertion of the proposition.  The second follows
immediately from the equality 
$$\Lambda[\lambda]\cong\Lambda/\lambda.$$
\end{proof}

A slight elaboration of this argument shows that 
$J_s^{new}[\ell]$ is a free module of rank 2 over
$\OO_s/\ell$ for each divisor $s$ of $r$.

\section{Monodromy of $J[\lambda]$}

Our next task is to study the action of $\Gal(K^{sep}/K)$ on
$J[\lambda]$ where $\lambda$ is a prime of $\Lambda$.

\subsection{Fundamental groups}
Let $\P^1_k$ be the projective line over $k$ with coordinate $t$, so
that the function field of $\P^1_k$ is $K=k(t)$.  Let $U$ be the
Zariski open subset $\P^1_k\setminus\{0,1,\infty\}$.  We saw in
Chapter~\ref{ch:models} that $J$ has good reduction at every place of
$U$. Proposition~\ref{prop:reduction} and the discussion in
Section~\ref{ss:general-d} show that the action of
$\Gal(K^{sep}/K)$ on $H^1(C,\Ql)$ is at worst tamely ramified at
places in $\P^1_k\setminus U$.  It follows that the actions of
$\Gal(K^{sep}/K)$ on $J[\ell]$ and on $J[\lambda]\subset J[\ell]$
factor through the quotient $\Gal(K^{sep}/K)\to\pi_1^t(U)$ where
$\pi_1^t(U)$ is the tame fundamental group (with base point the
geometric generic point given by the choice of $K^{sep}$, which we
omit from the notation).

It is known (\cite[Corollary~to~Theorem~14]{GrothendieckGFGA} or
\cite[XIII.2.12]{SGA1}) that $\pi_1^t(U)$ is topologically generated
by elements $\gamma_0,\gamma_1,\gamma_\infty$ with
$\gamma_0\gamma_1\gamma_\infty=1$ and with $\gamma_x$ topologically
generating the inertia group at $x$.

Choose a basis of the free, rank 2 $\Lambda$-module
$J[\ell]$, and fix the corresponding isomorphism 
$$\Aut_{\Lambda}(J[\ell])\cong \GL_2(\Lambda).$$
Let $\rho:\pi_1^t(U)\to\GL_2(\Lambda)$ be the
representation giving the action of $\pi^1_t(U)$ on $J[\ell]$.
Also, let $\rho_\lambda:\pi_1^t(U)\to\GL_2(\F_\lambda)$ be the
composition
$$\pi_1^t(U)\to\GL_2(\Lambda)\to\GL_2(\F_\lambda),$$
giving the action of $\pi_1^t(U)$ on $J[\lambda]$.  Later in this
section, we determine the image of $\rho_\lambda$.

\subsection{Twisting}\label{ss:twist}
It is convenient to consider a twist of $C$ and of its Jacobian.  
Let $C_\chi$ be the smooth projective
curve over $K$ associated to the affine curve
$$(1-t)xy^r=(x+1)(x+t).$$
It is evident that $C_\chi$ becomes isomorphic to $C$ over the Kummer
extension $K(v)$ where $v^r=1-t$.  

The extension $K(v)/K$ is unramified over $U$, so the action of
$\gal(K^{sep}/K)$ on $K(v)$ factors through $\pi_1^t(U)$, and Kummer
theory shows that the character $\chi:\pi_1^t(U)\to\mu_r$ with
$\chi(g):=g(v)/v$ satisfies $\chi(\gamma_0)=1$,
$\chi(\gamma_1)=\zeta^{-1}$ and $\chi(\gamma_\infty)=\zeta$ for some
primitive $r$-th root of unity $\zeta\in k$.

Now consider the Jacobian $J_\chi$ of $C_\chi$.  It admits an action
of $\Z[\mu_r]$ and we may define $A_\chi:=J_\chi^{new}$ and
$J_\chi[\lambda]$ in the same manner we defined $A=J^{new}$ and
$J[\lambda]$.  Over $K(v)$, since $J_\chi$ and $J$ are isomorphic, it
follows that $J_\chi[\ell]\cong J[\ell]\cong\Lambda^2$ and
$J_\chi[\lambda]\cong J[\lambda]\cong\F_\lambda^2$.

Since the action of $\mu_r$ on $C$ and $C_\chi$ is via the $y$
coordinate, we may identify $\zeta$ above with an element of
$\mu_r\subset \Lambda\to\F_\lambda$.  Let
$$\rho_{\chi}:\pi_1^t(U)\to\Aut(J_\chi[\ell])
\cong\GL_2(\Lambda)$$
be the representation giving the action of $\pi_1^t(U)$ on
$J_\chi[\ell]$, and let
$\rho_{\chi,\lambda}:\pi_1^t(U)\to\GL_2(\F_\lambda)$ be the quotient
giving the action on $J_\chi[\lambda]$.  Then the discussion above
shows that there are isomorphisms $\rho_{\chi}\cong\rho\tensor\chi$
and $\rho_{\chi,\lambda}\cong\rho_\lambda\tensor\chi$.  We use this
``twisting'' to deduce information about $\rho_\lambda$ and $\rho$.

\subsection{Local monodromy}
Our next goal is to record the Jordan forms of the matrices
$\rho_\lambda(\gamma_x)$ and $\rho_{\chi,\lambda}(\gamma_x)$.

\begin{proposition}\label{prop:local-monodromy}
  Suppose that $\lambda\subset\Lambda$ is a prime of level $r>2$.  For
  $x\in\{0,1,\infty\}$, let $g_x=\rho_\lambda(\gamma_x)$ and
  $g_{\chi,x}=\rho_{\chi,\lambda}(\gamma_x)$.  Let
  $\zeta\in\F_\lambda$ be the primitive $r$-th root of unity
  $\zeta=\chi(\gamma_\infty)$.  Then:
\begin{enumerate}
\item $g_0$ is unipotent and non-trivial, $g_1$ is semi-simple with
  eigenvalues $1$ and ${\zeta}^2$, and $g_\infty$ is non-semi-simple
  with eigenvalue ${\zeta}^{-1}$ repeated twice.  Equivalently,
  writing $\sim$ for conjugacy in $\GL_2(\F_\lambda)$,
$$g_0\sim\left(\begin{matrix}1&1\\0&1\end{matrix}\right),\qquad
g_1\sim\left(\begin{matrix}1&0\\0&{\zeta}^2\end{matrix}\right),
\qquad\text{and}\qquad
g_\infty\sim\left(\begin{matrix}{\zeta}^{-1}&1\\
0&{\zeta}^{-1}\end{matrix}\right).$$
\item $g_{\chi,0}$ and $g_{\chi_,\infty}$ are unipotent and
  non-trivial, and $g_{\chi,1}$ is semi-simple with
  eigenvalues $\zeta^{-1}$ and $\zeta$.
  Equivalently,
$$g_{\chi,0}\sim\left(\begin{matrix}1&1\\0&1\end{matrix}\right),\qquad
g_{\chi,1}\sim\left(\begin{matrix}\zeta^{-1}&0\\0&\zeta\end{matrix}\right),
\qquad\text{and}\qquad
g_{\chi,\infty}\sim\left(\begin{matrix}1&1\\0&1\end{matrix}\right).
$$
\end{enumerate}
\end{proposition}

Note that parts (1) and (2) are equivalent via the isomorphism
$\rho_{\chi,\lambda}\cong\rho_\lambda\tensor\chi$, but we use
both $\rho_\lambda$ and $\rho_{\chi,\lambda}$ in the proof.

\begin{proof}[Proof of Proposition~\ref{prop:local-monodromy}]
  By Proposition~\ref{prop:connected-component}, the
  minimal regular model $\XX$ of $C$ has semi-stable reduction at
  $t=0$.  Indeed, the fiber at 0 of $\XX$, call it $\XX_0$, is a pair
  of smooth rational curves crossing transversally at $r$ points.  It
  follows that the action of $\gamma_0$ on $J[\ell]$ is unipotent (see
  \cite[Theorem~1.4]{Abbes00} for a modern account) and therefore that
  the action of $\gamma_0$ on $J[\lambda]$ is unipotent.  It remains
  to see that it is non-trivial.  To that end, let $\JJ_0$ be the
  fiber at 0 of the N\'eron model of $J$.  If
  $I_0\subset\pi_1^t(U)$ denotes the inertia subgroup at $0$, we have
$$J[\ell]^{I_0}\cong \JJ_0[\ell].$$
By \cite[9.5, Corollary~11]{blr}, the group of connected components of
$\JJ_0$ has order $r$, so is prime to $\ell$.  By
Proposition~\ref{prop:connected-component}, the identity
component of $\JJ_0$ is a torus of dimension $r-1$.  It follows that
$$\JJ_0[\ell]\cong(\Z/\ell\Z)^{r-1}.$$
We need to understand the action of $\mu_r$ on this group.  Since
$\XX\to\P^1$ admits a section, \cite[9.5, Theorem~4]{blr} shows that
$\JJ_0\cong\Pic^0(\XX_0)$.  Noting that $\mu_r$ acts on $\XX_0$ by
cyclically permuting the points where the two components cross, we see
that there is an isomorphism of $\Lambda$-modules
$$\JJ_0[\ell]\cong \Lambda.$$
It follows that
$$J[\lambda]^{I_0}\cong \F_\lambda.$$
Since this has dimension 1 over $\F_\lambda$, we deduce that $g_0$ is
not the identity.  This proves our claim for $g_0$.

Our claim for $g_{\chi,0}$ follows from the isomorphism
$\rho_{\chi,\lambda}\cong\rho_\lambda\tensor\chi$.  Alternatively, it
also follows from the fact that $1-t$ is an $r$-th power in the
completed local ring $k[[t]]$, so the regular minimal models of $C$
and $C_\chi$ are isomorphic over $k[[t]]$ and the action of inertia
is the same.

Now we turn to $C_\chi$ in a neighborhood of $t=\infty$.  Changing
coordinates $(x,y)\mapsto(tx,y)$, the defining equation of $C_\chi$
becomes
$$\frac{1-t}txy^r=(x+1)(x+t^{-1}).$$
But $(1-t)/t$ is a unit, and thus an $r$-th power, in $k[[t^{-1}]]$
so we may change coordinates $(x,y)\mapsto(x,(t/(1-t))^{1/r}y)$,
yielding
$$xy^r=(x+1)(x+t^{-1}).$$
Up to substituting $t^{-1}$ for $t$, this is exactly the defining
equation of $C$.  We conclude that the action of $\gamma_\infty$ on
$J_\chi[\ell]$ is the same as the action of $\gamma_0$ on $J[\ell]$
and similarly for the submodules $J_\chi[\lambda]$ and $J[\lambda]$.
In particular, $g_{\chi,\infty}$ is unipotent and non-trivial, as
claimed.

The claim for $g_\infty$ follows from that for $g_{\chi,\infty}$ and
the isomorphism $\rho_{\chi,\lambda}\cong\rho_\lambda\tensor\chi$.  

Now we turn to a consideration of $g_1$.  Letting
$I_1\subset\pi_1^t(U)$ be the inertia group at $t=1$, our first claim
is that $J[\lambda]^{I_1}$ is a one-dimensional $\F_\lambda$-vector
space.  The proof is very similar to the proof above that
$J[\lambda]^{I_0}$ is 1-dimensional.  First we note that
$$J[\ell]^{I_1}\cong\JJ_1[\ell]$$ 
where $\JJ_1$ is the fiber of the N\'eron model of $J$ at $t=1$.  By
Proposition~\ref{prop:component-groups}, the component group of
$\JJ_1$ has order $r$ (the hypothesis that $r$ divides $d$ is not
needed at $t=1$), and by Proposition~\ref{prop:connected-component},
the identity component is an extension of a 1-dimensional torus by an
abelian variety of dimension $(r-2)/2$ if $r$ is even, and is an
abelian variety of dimension $(r-1)/2$ if $r$ is odd.  In both cases,
this abelian variety is the Jacobian of the smooth model of the curve
$zy^r=(1+z)^2$.  Viewing this curve as a $\mu_r$-Galois cover of
the line allows us to compute the structure of the $\ell$-torsion of
its Jacobian as a $\Lambda$-module, and we find that
$$\JJ_1[\ell]\cong \Lambda.$$
It follows that
$$J[\ell]^{I_1}\cong \Lambda$$ 
and that 
$$J[\lambda]^{I_1}\cong \F_\lambda.$$
Since this has dimension 1 over $\F_\lambda$, we deduce that $g_1$ has
1 as an eigenvalue.  Our second claim is that $\det(g_1)=\zeta^2$,
which follows from the equality $g_1=g_0^{-1}g_{\infty}^{-1}$ and from
previous computations for $g_0$ and $g_\infty$.  Thus the eigenvalues
of $g_1$ are $1$ and $\zeta^2$, and since $r>2$, these are distinct
and $g_1$ is semi-simple as claimed.

Finally, our claim about $g_{\chi,1}$ follows from the isomorphism
$\rho_{\chi,\lambda}\cong\rho_\lambda\tensor\chi$. 
\end{proof}

\begin{proposition}\label{prop:local-monodromy-r=2}
  Suppose $r=2$.  For $x\in\{0,1,\infty\}$, let
  $g_x=\rho_\lambda(\gamma_x)$ and let
  $g_{\chi,x}=\rho_{\chi,\lambda}(\gamma_x)$.  Then $g_0$ and $g_1$
  are unipotent and non-trivial, and $g_\infty$ is non-semi-simple
  with eigenvalue $-1$ repeated twice.  Equivalently, writing $\sim$
  for conjugacy in $\GL_2(\F_\lambda)$,
$$g_0\sim g_1\sim\left(\begin{matrix}1&1\\0&1\end{matrix}\right),
\qquad\text{and}\qquad
g_\infty\sim\left(\begin{matrix}-1&1\\0&-1\end{matrix}\right),$$
and
$$g_{\chi,0}\sim g_{\chi,\infty}\sim
\left(\begin{matrix}1&1\\0&1\end{matrix}\right),
\qquad\text{and}\qquad
g_{\chi,1}\sim\left(\begin{matrix}-1&1\\0&-1\end{matrix}\right).$$

\end{proposition}

\begin{proof}
  The same proof as in the case $r>2$ works up until the penultimate
  paragraph, where $g_1$ has eigenvalues 1 and $\zeta^2=1$, and thus
  we can no longer deduce that $g_1$ is semi-simple.  If it were
  semi-simple, $g_1$ would be the identity, contradicting the fact
  that $C=J$ has bad reduction at $t=1$.  Thus $g_1$ is unipotent and
  non-semi-simple in this case.
\end{proof}

\subsection{Group theory}
We write $G_\lambda$ for $\rho_\lambda(\pi_1^t(U))$ and
$G_{\chi,\lambda}$ for $\rho_{\chi,\lambda}(\pi_1^t(U))$.  The main
result of this section is a calculation of these groups.

%%% dlu:  The following is correct, but I think it would be confusing
%%% to include it as our tilde A_5 is in characteristic 3.
%(3) We have the following isomorphisms:
%$PSL_2(F_5) \simeq A_5$ and
%$SL_2(F_5) \simeq \tilde{A}_5$

\begin{proposition}\label{prop:lambda-monodromy}
Let $\lambda\subset\Lambda$ be a prime of level $r$.
\begin{enumerate}
\item If $\ell=2$, then $G_{\lambda,\chi}$ is isomorphic to the
  dihedral group $D_{2r}$ of order $2r$.
\item If $\ell=3$ and $r=10$,
%then there is a unique prime $\;lambda$ of $\OO$ over $\ell$, and 
then $G_{\chi,\lambda}\subsetneq\SL_2(\F_{\lambda^+})=\SL_2(\F_{9})$
and $G_{\chi, \lambda}$ is isomorphic to $\tilde A_5$, a double
cover of the alternating group $A_5$.  
\item If $\ell>3$ or $\ell=3$ and $r\neq10$, then
$$G_{\chi,\lambda}\cong\SL_2(\F_{\lambda^+})\subset\GL_2(\F_\lambda)$$
where $\lambda^+$ is the prime of $\Lambda^+$ under
$\lambda$.
\item For all $\ell$ and $r$,
$$G_\lambda\cong\mu_r\cdot G_{\chi,\lambda}.$$
\end{enumerate}
\end{proposition}

\begin{proof}
  We first prove part (4):  To see that
  $G_\lambda\cong\mu_r\cdot G_{\chi,\lambda}$, note that
  $G_\lambda\subset\mu_r\cdot G_{\chi,\lambda}$, since the values of
  $\chi$ lie in $\mu_r$.  For the opposite containment, we observe
  that if $m$ is an integer such that $\ell m\equiv1\pmod r$, then
  $g_{\infty}^{\ell m}$ is the scalar matrix $\zeta^{-1}$ and it
  follows that $\mu_r$ and $G_{\chi,\lambda}$ are contained in
  $G_\lambda$.

  Next we claim that the lines fixed by $g_{\chi,0}$ and
  $g_{\chi,\infty}$ are distinct.  Indeed, if they were not, then
  $g_{\chi,1}=g_{\chi,0}^{-1}g_{\chi,\infty}^{-1}$ would fix the same
  line, but by Proposition~\ref{prop:local-monodromy}, $1$ is not an
  eigenvalue of $g_{\chi,1}$.  Thus there is a basis $e_1,e_2$ of
  $J[\lambda]$ such that $g_{\chi,0}$ fixes $e_1$ and
  $g_{\chi,\infty}$ fixes $e_2$.  Scaling $e_2$ if necessary, the
  matrices of $g_{\chi,0}$ and $g_{\chi,\infty}$ in the new basis have
  the form
$$g_{\chi,0}=\left(\begin{matrix}1&1\\0&1\end{matrix}\right)
\quad\text{and}\quad
g_{\chi,\infty}=\left(\begin{matrix}1&0\\c&1\end{matrix}\right)$$ for
some uniquely determined $c\in\F_\lambda$ with $c\neq0$.
Proposition~\ref{prop:local-monodromy} implies that $g_{\chi,1}$ has
trace $\zeta+\zeta^{-1}$.  Since
$g_{\chi,1}=g_{\chi,0}^{-1}g_{\chi,\infty}^{-1}$, we calculate that
$c=\zeta+\zeta^{-1}-2$.

If $\ell=2$, setting
$$h=\left(\begin{matrix}1&1\\1+\zeta&1+\zeta^{-1}\end{matrix}\right),$$
the reader may check that 
$$h^{-1}g_{\chi,\infty}h=\left(\begin{matrix}0&1\\1&0\end{matrix}\right)
\quad\text{and}\quad
h^{-1}g_{\chi,0}^{-1}g_{\chi,\infty}^{-1}h
  =\left(\begin{matrix}\zeta^{-1}&0\\0&\zeta\end{matrix}\right).$$
It follows that $G_{\chi,\lambda}$ is dihedral of order $2r$, and this
proves part (1).

To prove parts (2) and (3), we assume that $\ell>2$, and we apply
Dickson's theorem \cite[page~44]{GorensteinFG}.  It says that if
$\ell>2$, then the subgroup of $\SL_2(\overline{\F}_\ell)$ generated
by
$$\left(\begin{matrix}1&1\\0&1\end{matrix}\right)\quad\text{and}\quad
\left(\begin{matrix}1&0\\c&1\end{matrix}\right)$$
%$\left(\begin{matrix}1&1\\0&1\end{matrix}\right)$ and
%$\left(\begin{matrix}1&0\\c&1\end{matrix}\right)$ 
is $\SL_2(\F_\ell(c))$ except for one exceptional case, namely where
$\ell=3$ and $c^2=-1$, in which case the group is a double cover of
$A_5$.\footnote{Gorenstein does not state explicitly which $c$ give
  rise to the exceptional case, but the paragraph containing the first
  display on page 45 of \cite{GorensteinFG} shows that we are in the
  exceptional case exactly when $\ell=3$ and $c^2=-1$.}  For our $c$,
$\F_\ell(c)=\F_{\lambda^+}$ so, apart from the possible exceptional
case, we have $G_{\chi,\lambda}\cong\SL_2(\F_{\lambda^+})$.  Equality
holds here in particular when $\ell>3$.

Note that in the exceptional case $\F_{\lambda^+}=\F_3(c)=\F_9$
since $[\F_3(c):\F_3] \leq 2$ and since $-1$ is not a square in $\F_3$.
If $\F_{\lambda^+}=\F_9$, then $\zeta\in\F_\lambda\subset\F_{81}$, so $r$
divides $80=16\cdot5$.  We cannot be in the exceptional case if $20|r$
or $8|r$, because the order of $g_{\chi,1}$ in $\PSL_2$ is $r$ or
$r/2$ as $r$ is odd or even, and $A_5$ has no elements of order $10$
or $4$.  Also, $c$ does not generate $\F_9$ if $r=4$ or $r=2$, so the
only possible exceptional cases are when $r=5$ and $r=10$.

Recalling that $c=\zeta_r+\zeta_r^{-1}-2$ and $\ell=3$, we have
$$c^2=\zeta_r^2+\zeta_r^{-2}-\zeta_r-\zeta_r^{-1}.$$
When $r=5$, we have $(c^2)^3=-c^2$, so $c^2\not\in\F_3$ and we are not
in the exceptional case.  When $r=10$, $-\zeta_{10}=\zeta_5$ and we
see that
\begin{align*}
c^2&=\zeta_{10}^2+\zeta_{10}^{-2}-\zeta_{10}-\zeta_{10}^{-1}\\
&=\zeta_{5}^2+\zeta_{5}^{-2}+\zeta_{5}+\zeta_{5}^{-1}\\
&=-1,
\end{align*}
so we are in the exceptional case, i.e., $G_{\chi,\lambda}$ is a double
cover of $A_5$.
\end{proof}

\begin{remark}\label{rem:reduction-to-new}
  Note that if $\lambda\subset\Lambda$ is a prime of level $s>2$, then
  $J[\lambda]\cong J_s[\lambda]$ as a module over $\pi_1^t(U)$, so
  Proposition \ref{prop:lambda-monodromy} determines the monodromy of
  $J[\lambda]$ for all primes $\lambda$.
\end{remark}

\section{Independence}
\subsection{Statement}
In the previous section, we determined $G_{\lambda}$ and
$G_{\chi,\lambda}$, the images of $\pi_1^t(U)$ in
$\Aut_{\F_\lambda}(J[\lambda])$ and
$\Aut_{\F_\lambda}(J_\chi[\lambda])$.  Our goal in this section is to
determine $G$ and $G_\chi$, the images of $\pi_1^t(U)$ in
$\Aut_{\Lambda}(J[\ell])$ and $\Aut_{\Lambda}(J_\chi[\ell])$, i.e.,
the image of the representations
$$\rho_\ell:\pi_1^t(U)\to\Aut_{\Lambda}(J[\ell])\cong\GL_2(\Lambda)$$
and 
$$\rho_{\chi, \ell}:\pi_1^t(U)\to\Aut_{\Lambda}(J_\chi[\ell])\cong\GL_2(\Lambda)$$
where $\Lambda$ is the ring of endomorphisms discussed in
Section~\ref{s:rings}.  Since $\rho_{\chi, \ell}\cong\rho_\ell\tensor\chi$, it
suffices to determine $G_\chi$.  It turns out that $G_\chi$ is 
the product over a suitable set of $\lambda$ of the $G_{\chi,\lambda}$;
the set in question is not all $\lambda$, because there is one obvious
dependency among the $G_{\chi,\lambda}$.

To motivate the main result, consider a prime $\lambda^+$ of
$\Lambda^+$ that splits in $\Lambda$ into primes $\lambda_1$ and
$\lambda_2$.  The proof of Proposition~\ref{prop:lambda-monodromy}
shows that after choosing suitable bases, the image of
$$\pi_1^t(U)\to
\Aut_{\F_{\lambda_1}}(A[\lambda_1])\times\Aut_{\F_{\lambda_2}}(A[\lambda_2])
\cong\GL_2(\F_{\lambda_1})\times\GL_2(\F_{\lambda_2})$$
is generated by the elements
$$\left(\pmat{1&1\\0&1},\pmat{1&1\\0&1}\right)
\quad\text{and}\quad
\left(\pmat{1&0\\c_1&1},\pmat{1&0\\c_2&1}\right)$$ where $c_1$ and
$c_2$ are the images of $\zeta+\zeta^{-1}-2$ in $\F_{\lambda_1}$ and
$\F_{\lambda_2}$.  Since $\lambda_1$ and $\lambda_2$ lie over the same
prime $\lambda^+$ of $\Lambda^+$, and since $c_1$ and $c_2$ lie in
$\F_{\lambda^+}$, there is a field isomorphism $\F_{\lambda_1}\cong
\F_{\lambda_2}$ that carries $c_1$ to $c_2$.  This shows that the
image of the map under consideration is ``small'': it is the graph of
an isomorphism $G_{\chi,\lambda_1}\cong G_{\chi,\lambda_2}$.  The main
result of this section shows that when $\ell>2$ this is the only
relation among the $G_{\chi,\lambda}$.

\begin{theorem}\label{thm:monodromy}
Let $S$ be a set of primes of $\Lambda$ such that for
  every prime $\lambda^+$ of $\Lambda^+$ there is a unique prime in
  $S$ over $\lambda^+$.  Let $G_\chi$ be the image of
$$\rho_{\chi, \ell}:\pi_1^t(U)\to\Aut_{\Lambda}(J_\chi[\ell])\cong\GL_2(\Lambda)$$
and let $G$ be the image of
$$\rho_\ell:\pi_1^t(U)\to\Aut_{\Lambda}(J[\ell])\cong\GL_2(\Lambda).$$
\begin{enumerate}
\item If $\ell>2$, then there is an isomorphism
$$G_\chi\cong\prod_{\lambda\in S} G_{\chi,\lambda}.$$
In particular, if $\ell>3$ or $\ell=3$ and $10\nodiv r$, then
$$G_\chi\cong\SL_2(\Lambda^+)\subset\GL_2(\Lambda).$$
\item If $\ell=2$, then 
$$G_\chi\cong D_{2r}.$$
\item $G\cong\mu_r\cdot G_\chi$.
\end{enumerate}
\end{theorem}

The proof of the theorem occupies the rest of this section.  In the
next subsection, we dispose of the easy parts of the proof.  The
remaining sections deal with the main issue, namely the isomorphism
$G_\chi\cong\prod_{\lambda\in S}G_{\chi,\lambda}$ for $\ell>2$.

\subsection{First part of the proof of 
Theorem~\ref{thm:monodromy}}
The proof of part (3) is essentially identical to that of part (4) of
Proposition~\ref{prop:lambda-monodromy} and is left to the
reader.  

Now consider part (2), the case $\ell=2$.  In view of part (1) of
Proposition~\ref{prop:lambda-monodromy}, the conclusion here is
exactly the opposite of that in part (1): the
$G_{\chi,\lambda}$ are highly dependent.  To prove it, we note that
$C_\chi$ is hyperelliptic, as we see from the defining equation
$(1-t)xy^r=(x+1)(x+t)$ via projection to the $x$-line. Rewriting the
equation as
$$x^2+\left(t+1+(t-1)y^r\right)x+t=0$$ 
and completing the square (as we may do since $p\neq\ell=2$), the
equation takes the form
$$z^2=
y^{2r}+2\left(\frac{t+1}{t-1}\right)y^r+1.$$ 

The 2-torsion points on the Jacobian of a hyperelliptic curve
$z^2=f(y)$ are represented by divisors of degree zero supported on the
points $(y,0)$ where $y$ is a zero of $f$.  It follows that the
monodromy group of the 2-torsion is equal to the Galois group of $f$.
In our case, the Galois group is $D_{2r}$.  Indeed, the roots of $f$
are the solutions of $y^r=w_1$ and $y^r=w_2$ where $w_1$ and
$w_2=1/w_1$ are the roots of $w^2+(t+1)/(t-1)w+1$.  The discriminant
of this quadratic polynomial is $16t/(t-1)^2$, so its roots lie in
$K(t^{1/2})$.  The splitting field $K_0$ of $f$ is thus a degree $r$
Kummer extension of $K(t^{1/2})$, and $\Gal(K(t^{1/2})/K)$ acts on
$\Gal(K_0/K(t^{1/2}))$ by inversion, so $G_\chi\cong\Gal(K_0/K)\cong
D_{2r}$.  This proves part (2).  

For use in the next section, we note that the fixed field of
the cyclic group $C_r\subset D_{2r}$ is the quadratic extension
$K(t^{1/2})$ of $K=k(t)$.

To end this subsection, we prove the ``in particular'' part of (1).
Recall that we have shown that if $\ell>3$ or $\ell=3$ and the level
of $\lambda$ is not 10, then
$G_{\chi,\lambda}\cong\SL_2(\F_{\lambda^+})$.  
Let $S^+$ be the set of all primes of $\Lambda^+$ and let $S$ be
as in the statement of the theorem, so that there is a bijection
$S\to S^+$ that sends a prime $\lambda$ to the prime
$\lambda^+$ under it.  Then the image of
$\SL_2(\Lambda^+)\subset\SL_2(\Lambda)$ 
under the projection
$$\SL_2(\Lambda)=\prod_{\lambda}\SL_2(\F_\lambda)
\to\prod_{\lambda\in S}\SL_2(\F_{\lambda})$$
is the product $\prod_{\lambda^+\in S^+}\SL_2(\F_{\lambda^+})$.  
Since 
$$\prod_{\lambda\in S}G_{\chi,\lambda}=
\prod_{\lambda^+\in S^+}\SL_2(\F_{\lambda^+}),$$
this establishes the desired isomorphism
$G_{\chi}\cong\SL_2(\Lambda^+)$.

To finish the proof of the theorem, it remains to establish the first
sentence of part (1).  We do this in
Section~\ref{ss:end-of-monodromy} below. 

\subsection{Several lemmas}
We collect together several group-theoretic lemmas to be used below.
Recall that a group is said to be \emph{perfect} if it is its own
commutator subgroup, or equivalently, if it has no non-trivial abelian
quotients, and it is said to be \emph{solvable} if its Jordan-Holder
factors are all abelian.

\begin{lemma}\label{lemma:groups}
\mbox{}
\begin{enumerate}
\item $\SL_2(\Fq)$ is perfect unless $q=2,3$, in which case it is
  solvable.
\item The group $\tilde A_5$ of
  Proposition~\ref{prop:lambda-monodromy}(2) is perfect.
\item If $q>3$, the non-trivial quotients of $\SL_2(\Fq)$ are
  $\SL_2(\Fq)$ and $\PSL_2(\Fq)$.  The non-trivial quotients of
  $\SL_2(\F_3)$ are $\SL_2(\F_3)$, $\PSL_2(\F_3)$, and $\Z/3\Z$.  The
  non-trivial quotients of $\tilde A_5$ are $\tilde A_5$ and $A_5$.
\item Suppose $\ell\ge3$ and let $H_a=\SL_2(\F_{\ell^a})$ for $a\ge1$.
  If $H$ is a non-trivial quotient of both $H_a$ and $H_b$, then
  $a=b$.  If $\ell=3$, then for all $a$, $\tilde A_5$ and $H_a$ have
  no common non-trivial quotients.
\end{enumerate}
\end{lemma}

\begin{proof}
The assertions in (1) and (3) related to $\SL_2(\Fq)$ are well known,
see \cite[Section~3.3.2]{Wilson}

The group $\tilde A_5\subset\SL_2(\F_9)$ is generated by
$h_0=\psmat{1&1\\0&1}$ and $h_\infty=\psmat{1&0\\i&1}$ where $i^2=-1$.
If $\tilde A_5\to H$ is a non-trivial quotient with kernel $N$, then
$N$ projects to a normal subgroup of $A_5$, i.e., to the trivial group
or all of $A_5$ since $A_5$ is simple \cite[Section 2.3.3]{Wilson}.
In the former case, $N$ is either $\tilde A_5$ or $A_5$.  In the
latter case, since $H$ is non-trivial, $N\neq \tilde A_5$, so projects
isomorphically to $A_5$.  We claim no such $N$ exists.  Indeed, if it
did, $\tilde A_5$ would be the product of $A_5$ and $\pm1$.  On the
other hand, the reader may check that $(h_\infty h_0 h_\infty^{-1}
h_0)^2=-1$, which shows that $\tilde A_5$ is not the product
$A_5\times\{\pm1\}$.  This shows that the quotients of $\tilde A_5$
are as stated in part (3).

Since $A_5$ is non-abelian and simple, and thus perfect, the
commutator subgroup of $\tilde A_5$ projects onto $A_5$.  The analysis
of the preceding paragraph shows it is all of $\tilde A_5$, i.e.,
$\tilde A_5$ is perfect. This establishes part (2).

Part (3) gives us a list of quotients of $\SL_2(\Fq)$ and $\tilde
A_5$, and part (4) is then reduced to an easy exercise by considering
the orders of the quotients.  Indeed, if $\ell>3$, the non-trivial
quotients of $\SL_2(\F_{\ell^a})$ have order $\ell^a(\ell^{2a}-1)$ or
$\ell^a(\ell^{2a}-1)/2$ and these numbers are all distinct for
distinct values of $a$.  If $\ell=3$, the non-trivial quotients have
order $\ell^a(\ell^{2a}-1)$ or $\ell^a(\ell^{2a}-1)/2$ or 3, with 3
occuring only if $a=1$.  Again, there are no coincidences, and this
establishes the part of (4) related to $H_a$ and $H_b$.  To establish
the last sentence, note that the non-trivial quotients of $\tilde A_5$
have order 120 or 60.  These numbers are divisble by 3 and not by 9,
and they are not $3(3^2-1)=24$ nor $3(3^2-1)/2=12$, so $\tilde A_5$
and $\SL_2(\F_{3^a})$ have no common non-trivial quotients.
\end{proof}

Given a field automorphism $\phi:\Fq\to\Fq$, we define an
automorphism $\SL_2(\Fq)\to\SL_2(\Fq)$  by applying $\phi$ to the
matrix entries.  Similarly, $\phi$ gives a well-defined automorphism
of $\PSL_2(\Fq)$.

\begin{lemma}\label{lemma:autos}
Assume that $q$ is odd.
\begin{enumerate}
\item Every automorphism of $\PSL_2(\Fq)$ is given by congugation by an
element of $\GL_2(\Fq)$ composed with a field automorphism as above.
\item Every automorphism of $\PSL_2(\Fq)$ lifts
  \textup{(}uniquely\textup{)} to $\SL_2(\Fq)$.
%\item Every automorphism of $\SL_2(\Fq)$ sends unipotent matrices to
%unipotent matrices.
\end{enumerate}
\end{lemma}

\begin{proof}
  For (1), see \cite[p.~795]{Ribet}.  It follows immediately that an
  automorphism of $\PSL_2(\Fq)$ lifts to $\SL_2(\Fq)$ since
  conjugation and field automorphisms both preserve the kernel
  $\{\pm1\}$ of $\SL_2(\Fq)\to\PSL_2(\Fq)$.  Since the kernel is
  central, any two lifts would differ by a homomorphism
  $\PSL_2(\Fq)\to\{\pm1\}$, and there are no non-trivial such
  homomorphisms by Lemma~\ref{lemma:groups} part (3).  This
  establishes part (2).
%Part (3) holds because unipotent elements
%  $g\in\SL_2(\Fq)$ are characterized by the equation $(g-1)^2=0$, an
%  equation clearly preserved by all automorphisms.
\end{proof}

\begin{lemma}\label{lemma:goursat}
\mbox{}
\begin{enumerate}
\item \textup{(}``Goursat's lemma''\textup{)} Let $H_1$ and $H_2$ be
  groups, and let $H\subset H_1\times H_2$ be a subgroup that
  projects surjectively onto $H_1$ and $H_2$.  Identify the kernel
  $N_i$ of $H\to H_{3-i}$ with a subgroup of $H_i$.  Then the image of
  $H$ in $H/N_1\times H/N_2$ is the graph of an isomorphism $H/N_1\to
  H/N_2$.
\item With assumptions as in part \textup{(1)}, assume that
  $H_1$ and $H_2$ have no common non-trivial quotients.  Then
  $H=H_1\times H_2$.
\item Suppose that $H_1,\dots,H_n$ are groups with each $H_i$ perfect,
  and suppose that $H\subset H_1\times\cdots\times H_n$ is a subgroup
  such that for all $1\le
  i<j\le n$, the projection $H\to H_i\times H_j$ is surjective.  Then
  $H=H_1\times\cdots\times H_n$.
\end{enumerate}
\end{lemma}

\begin{proof}
  Part~(1) is proved in \cite[Lemma 5.2.1]{Ribet}.  Part~(2) is
  immediate from part~(1).  Part~(3) is \cite[Lemma 5.2.2]{Ribet}.
\end{proof}

\subsection{Pairwise independence}
Our aim in this section is to prove the following pairwise independence
result.

\begin{proposition}\label{prop:pairwise}
If $\lambda_1\neq\lambda_2$
are distinct primes in $S$, then
$$\pi_1^t(U)\to G_{\chi,\lambda_1}\times G_{\chi,\lambda_2}$$
is surjective.  
\end{proposition}

\begin{proof}
  Note that if $r=2$ then $S$ is a single prime, so the proposition is
  vacuous.  Thus we assume $r>2$.

  We write $G_{12}$ for the image in the proposition, and we note that
  by the definition of the $G_{\chi,\lambda}$, $G_{12}$ projects
  surjectively onto each factor $G_{\chi,\lambda_i}$.

  We first treat the case $\ell>3$.  Fix isomorphisms
  $G_{\chi,\lambda_i}\cong\SL_2(\F_{\lambda_i^+})$ for $i=1,2$.  Here
  and below, we write $\lambda_i^+$ for the prime of $\Lambda^+$ under
  $\lambda_i$ .  Let $g_{i,0}$ and $g_{i,\infty}$ be the images of
  $\gamma_0$ and $\gamma_\infty\in\pi_1^t(U)$ in
  $\SL_2(\F_{\lambda_i^+})$.  By
  Proposition~\ref{prop:local-monodromy}, these are unipotent
  matrices.

  By Lemma~\ref{lemma:goursat}(2) we may assume that
  $G_{\chi,\lambda_1}$ and $G_{\chi,\lambda_2}$ have common
  non-trivial quotients.  By Lemma~\ref{lemma:groups}(3) this occurs
  if and only if $\F_{\lambda_1^+}$ and $\F_{\lambda_2^+}$ have the
  same cardinality.

  If $G_{12}$ is not all of the product, the Lemma~\ref{lemma:goursat}(1)
  yields either an isomorphism
  $\SL_2(\F_{\lambda_1^+})\to\SL_2(\F_{\lambda_2^+})$ or an isomorphism
  $\PSL_2(\F_{\lambda_1^+})\to\PSL_2(\F_{\lambda_2^+})$.  In the former
  case, since this isomorphism is induced by the image of $\pi_1^t(U)$
  in $G_{12}$, it sends $g_{1,0}$ to $g_{2,0}$ and $g_{1,\infty}$ to
  $g_{2,\infty}$.  In the latter case, the isomorphism lifts to
  $\SL_2$ by Lemma~\ref{lemma:autos}(2).  Moreover, the lifted
  isomorphism sends $g_{1,0}$ to $\pm g_{2,0}$.  In fact, by
  Lemma~\ref{lemma:autos}(3) the image must be $+g_{2,0}$ because
  $g_{1,0}$ is unipotent and $-g_{2,0}$ is not.  Similarly, the lifted
  automorphism must send $g_{1,\infty}$ to $g_{2,\infty}$.

  Summarizing, if $G_{12}$ is not all of the product, we have an
  isomorphism
$$\psi:\SL_2(\F_{\lambda_1^+})\to\SL_2(\F_{\lambda_2^+})$$
such that $\psi(g_{1,0})=g_{2,0}$ and
$\psi(g_{1,\infty})=g_{2,\infty}$.  But such an isomorphism is
impossible.  Indeed, by Lemma~\ref{lemma:autos}(1), $\psi$ is the
composition of conjugation and a field automorphism
$\phi:\F_{\lambda_1^+}\to\F_{\lambda_2^+}$.  Since
$\psi(g_{1,0}^{-1}g_{1,\infty}^{-1})=g_{2,0}^{-1}g_{2,\infty}^{-1}$,
$\phi$ must send the trace of $g_{1,0}^{-1}g_{1,\infty}^{-1}$ to the
trace of $g_{2,0}^{-1}g_{2,\infty}^{-1}$.  By
Proposition~\ref{prop:local-monodromy}(2), these traces are the images
of $\zeta+\zeta^{-1}\in\Lambda^+$ in $\F_{\lambda^+_1}$ and
$\F_{\lambda_2^+}$.  But Lemma~\ref{lemma:non-iso} shows that no such
$\phi$ exists, so no such $\psi$ exists either.  We conclude that
$G_{12}$ is all of the product, as desired.

Now assume $\ell=3$.  If $G_{\chi,\lambda_i}$ are both
  $\SL_2(\F_{\ell^a})$ with $a>1$, then the argument above applies
  verbatim.  Thus it remains to treat the possibilities that
  $G_{\chi,\lambda_i}\cong\tilde A_5$ or $\SL_2(\F_3)$.  The
  $\tilde A_5$ case does not in fact occur.  Indeed,
  $G_{\chi,\lambda}\cong\tilde A_5$ if and only of $r=10$, and
  $\OO_{10}^+$ has a unique prime over $\ell=3$, so there do not exist
  two distinct primes $\lambda_i\in S$ with
  $G_{\chi,\lambda_i}\cong\tilde A_5$.

The last case to discuss is when
$G_{\chi,\lambda_i}\cong\SL_2(\F_3)$, and by
Lemma~\ref{lemma:Flambda=F3} this does indeed occur exactly when
$4|r$, the two primes being the unique primes over $\ell=3$ of levels
$2$ and $4$.  The argument above is not sufficient in this case,
because $\SL_2(\F_3)$ has an additional quotient, namely $\Z/3\Z$.
But we may argue directly as follows:  By
Proposition~\ref{prop:local-monodromy}, in this case $G_{12}$ is
generated by
$$h_0=\left(\pmat{1&1\\0&1},\pmat{1&1\\0&1}\right)
\quad\text{and}\quad
h_\infty=\left(\pmat{1&0\\-1&1},\pmat{1&0\\1&1}\right).$$
Then we compute directly that
$$\left(g_0g_\infty g_0^{-1}g_\infty g_0g_\infty^{-1}\right)^2=
\left(\pmat{1&0\\-1&1},\pmat{1&0\\0&1}\right)$$
and
$$\left(g_\infty g_0 g_\infty^{-1}g_0g_\infty g_0^{-1}\right)^2=
\left(\pmat{1&1\\0&1},\pmat{1&0\\0&1}\right).$$
It follows immediately that $G_{12}=\SL_2(\F_3)\times\SL_2(\F_3)$.
This completes the proof of the proposition.
\end{proof}

\subsection{End of the proof of 
Theorem~\ref{thm:monodromy}}\label{ss:end-of-monodromy}
We divide $S$ into the disjoint union of 
$$S_1=\left\{\lambda\in S
\left|G_{\chi,\lambda}\not\cong\SL_2(\F_3)\right.\right\}$$
and
$$S_2=\left\{\lambda\in S
\left|G_{\chi,\lambda}\cong\SL_2(\F_3)\right.\right\}.$$

If $\lambda\in S_1$, then by Lemma~\ref{lemma:groups},
$G_{\chi,\lambda}$ is perfect.  Applying
Proposition~\ref{prop:pairwise} and Lemma~\ref{lemma:goursat}(3), we
conclude that 
$$\pi_1^t(U)\to\prod_{\lambda\in S_1}G_{\chi,\lambda}=:H_1$$
is surjective.

Note that by Lemma~\ref{lemma:Flambda=F3}, $S_2$ has at most two
elements, so Proposition~\ref{prop:pairwise} shows that 
$$\pi_1^t(U)\to\prod_{\lambda\in S_2}G_{\chi,\lambda}=: H_2$$
is surjective.

Now $H_1$ is a product of perfect groups, so is perfect, whereas $H_2$
is a product of solvable groups, so is solvable.  Therefore $H_1$ and
$H_2$ have no common non-trivial quotients.  It follows from
Lemma~\ref{lemma:goursat}(2) that 
$$\pi_1^t(U)\to H_1\times H_2$$
is surjective.  Since
$$H_1\times H_2=\prod_{\lambda\in S} G_{\chi,\lambda},$$
this completes the proof of the theorem.  \qed

\section{Conclusion}
We are now in position to prove the results stated in
Section~\ref{cjh:sec:intro}. 

\subsection{Torsion}
In view of Corollary~\ref{cor:torsion-decomp}, the following is a
slight strengthening of Theorem~\ref{cjh:p1}.

\begin{theorem}\label{thm:new-torsion}
If $L/K$ is an abelian extension, then $J^{new}_r[\ell](L)=0$.  If
$r\neq2,4$ or $\ell>3$, then the same conclusion holds for any
solvable extension $L/K$.
\end{theorem}

\begin{proof}
  Let $L/K$ be a finite extension and write $A$ for $J^{new}_r$.
  Noting that $A[\ell](L)=A[\ell](L\cap K(A[\ell]))$ and that the
  intersection of a Galois extension with a solvable or abelian
  extension is again solvable or abelian, we may replace $L$ with
  $L\cap K(A[\ell])$.

  If $L/K$ is abelian, we have $\gal(K(A[\ell])/L)\supset[G,G]$ where
  $[G,G]$ is the commutator subgroup of $G=\gal(K(A[\ell]/K))$.  Thus
  $A[\ell](L)\subset A[\ell](F)$ where $F$ is the subfield of
  $K(A[\ell])$ fixed by $[G,G]$, and it suffices to show that
  $A[\ell](F)=0$.

  If $r\neq2,4$ or $\ell>3$, then by Theorem~\ref{thm:monodromy},
  $G=\gal(K(A[\ell]/K))$ is isomorphic to
  $\mu_r\cdot\SL_2(\OO^+/\ell)$ and $\SL_2(\OO^+/\ell)$ is a product
  of groups $\SL_2(\F_\lambda)$ with $|\F_\lambda|>3$.  It follows
  that the commutator subgroup $[G,G]$ satisfies
$$[G,G]\cong\prod_\lambda\SL_2(\F_\lambda).$$
The invariants of this group acting on $A[\ell]\cong\prod_\lambda
\F_\lambda^2$ are trivial, so $A[\ell](F)=0$ as desired.

If $\ell=3$ and $r=2$ or $4$, then $G\cong\mu_r\cdot\SL_2(\F_3)$ and
$[G,G]$ is the subgroup of $\SL_2(\F_3)$ generated by
$\psmat{0&1\\-1&0}$ and $\psmat{-1&-1\\-1&1}$.  (This is the 2-Sylow
subgroup of $\SL_2(\F_3)$.)  Since the eigenvalues of
$\psmat{0&1\\-1&0}$ are $\pm\sqrt{-1}$, already this matrix has no
invariants on $\F_3^2$, so \emph{a fortiori} $[G,G]$ has no
invariants, and again $A[\ell](F)=0$ as desired.

If $\ell=2$, then $G\cong D_{2r}$ and $[G,G]\cong C_r$, the cyclic
group of order $r$.  This groups acts on $A[\ell]$ by characters of
order $r$, so has no non-zero invariants, and we again have
$A[\ell](F)=0$. 

If $L$ is only assumed to be solvable, the same argument works
provided that $\ell>3$ or $r\neq2,4$, because in these cases the
derived series of $G$ stabilizes at $\prod_\lambda\SL_2(\F_\lambda)$.
\end{proof}

\subsection{Decomposition of $A_\chi$}

In this section, we prove a slight refinement of
Theorem~\ref{thm:decomp}.  Throughout, we assume $r>2$.

Recall that $C_\chi$ was defined by
$$(1-t)xy^r=(x+1)(x+t).$$
We observe that there is an involution $\sigma:C_\chi\to C_\chi$
defined by
$$\sigma(x,y)=\left(\frac{-x-t}{x+1},\frac1y\right)$$
and that we have the equality $\sigma\zeta_r=\zeta_r^{-1}\sigma$ of
automorphisms of $C_\chi$.  There is an induced action of $\sigma$ on
$J_\chi$ that preserves $A_\chi=J_\chi^{new}$, and the equality
$\sigma\zeta_r=\zeta_r^{-1}\sigma$ holds in the endomorphism ring of
$A_\chi$ as well.

Let $K'=K((1-t)^{1/r})$, so that $A_\chi$ and $A$ become isomorphic
over $K'$.  In view of this isomorphism, Theorem~\ref{thm:decomp} is
implied by the following.

\begin{theorem}
Let $B$ be the abelian subvariety $(1+\sigma)A_\chi\subset A_\chi$.  
\begin{enumerate}
\item There is an isogeny $A_\chi\to B^2$ over $K$ whose kernel is
  killed by multiplication by $2r$.
\item $\End_K(B)=\End_{\Kbar}(B)=\Z[\zeta_r]^+$, and $B$ is absolutely
  simple.
\end{enumerate}
\end{theorem}

\begin{proof}
  Define morphisms
\begin{align*}
A_\chi&\to (1+\sigma)A_\chi\times(1-\sigma)A_\chi\\
P&\mapsto\left((1+\sigma)P,(1-\sigma)P\right)
\end{align*}
and
\begin{align*}
 (1+\sigma)A_\chi\times(1-\sigma)A_\chi&\to A_\chi\\
\left(P_1,P_2\right)\qquad\quad&\mapsto P_1+P_2.
\end{align*}
Using that $\sigma^2=1$, we find that the compositions are both
multiplication by 2.  This proves that $A_\chi$ is isogenous to
$(1+\sigma)A_\chi\times(1-\sigma)A_\chi$ by an isogeny whose kernel is
killed by 2. 

Now consider the element $\delta=\zeta_r-\zeta_r^{-1}$ of
$\End(A_\chi)$.  Using that $r>2$ and considering the action on
differentials we see that $\delta$ is an isogeny, and since the norm
of $\delta$ as an element of $\Z[\zeta_r]$ divides $r$, the kernel of
$\delta$ is killed by $r$.

We compute that $(1+\sigma)\delta=\delta(1-\sigma)$ and
$(1-\sigma)\delta=\delta(1+\sigma)$, so the isogeny $\delta:A_\chi\to
A_\chi$ exchanges the subvarieties $(1+\sigma)A_\chi$ and
$(1-\sigma)A_\chi$.  In particular, $(1-\sigma)A_\chi$ is isogenous to
$B=(1+\sigma)A_\chi$ by an isogeny whose kernel is killed by $r$.

Combining this with the isogeny
$A_\chi\to(1+\sigma)A_\chi\times(1-\sigma)A_\chi$, we see that
$A_\chi$ is isogenous to $B\times B$ by an isogeny with kernel killed
by $2r$.  This proves the first part of the theorem.

For the second part, since 
$$(1+\sigma)(\zeta_r+\zeta_r^{-1})=(\zeta_r+\zeta_r^{-1})(1+\sigma),$$
we have that $\OO^+=\Z[\zeta_r]^+\subset\End_K(B)$.  Thus it
suffices to prove that $\End_{\Kbar}(B)=\OO^+$.

Let $F$ be a finite extension of $K$ such that all elements of
$\End_{\Kbar}(B)$ are defined over $F$.  Let $\ell$ be a prime $\neq
p$ and not dividing $2r$ such that $\ell>[F:K]$.  We claim that
restriction induces an isomorphism
$\gal(F(A_\chi[\ell])/F)\cong\gal(K(A_\chi[\ell])/K)$.  Clearly it is
injective, so it suffice to show it is onto.  Let $H$ be the image, a
subgroup of $G_\chi=\gal(K(A_\chi[\ell])/K)$ and note that the index
of $H$ in $G_\chi$ is at most $[F:K]$.  If $g\in G_\chi$ has order
$\ell$, then the orbits of $g$ on the coset space $G_\chi/H$ have size
1 or $\ell$.  Since $|G_\chi/H|\le[F:K]<\ell$, they must have order 1,
so $g\in H$.  But Theorem~\ref{thm:monodromy} and the proof of
Proposition~\ref{prop:lambda-monodromy} show that $G_\chi$ is
generated by elements of order $\ell$, so $H=G_\chi$, establishing our
claim.

Next we note that the existence of the isogeny $A_\chi\to B\times B$
and the isogeny $\delta:A_\chi\to A_\chi$ switching the two factors
shows that $F(B[\ell])=F(A_\chi[\ell])$.  Thus we have
$$\gal(F(B[\ell])/F)\cong\gal(F(A_\chi[\ell])/F)\cong\gal(K(A_\chi[\ell])/K)
\cong\SL_2(\OO^+/\ell)$$ where the last isomorphism is
Theorem~\ref{thm:monodromy}.

Now we assume for convenience that $\ell$ splits completely in
$\Q(\zeta_r)^+$, i.e., that $\ell\equiv\pm1\in(\Z/r\Z)^\times$.  In
this case $\OO^+/\ell$ is the product of $\phi(r)/2$ copies of $\Fl$
and $\SL_2(\OO^+/\ell)$ is the product of $\phi(r)/2$ copies of
$\SL_2(\Fl)$.  The $\Fl$-subalgebra of $\Aut_{\Fl}(B[\ell])\cong
M_{\phi(r)}(\Fl)$ generated by $\SL_2(\OO^+/\ell)$ is then isomorphic
to the product of $\phi(r)/2$ copies of $M_2(\Fl)$ and thus has
dimension $2\phi(r)$.  By the double centralizer theorem
\cite[Theorem.~2.43]{KnappAA}, the centralizer of $\SL_2(\OO^+)$ in
$\Aut_{\Fl}(B[\ell])$ has dimension $\phi(r)/2$ over $\Fl$.  Since
$\End_F(B)/\ell$ lies in this centralizer, it has dimension at most
$\phi(r)/2$, and thus $\End_F(B)$ has $\Z$-rank at most $\phi(r)/2$.
But $\OO^+\subset\End_F(B)$ has $\Z$-rank $\phi(r)/2$ and is a maximal
order in its fraction field, so we have
$\OO^+=\End_F(B)=\End_{\Kbar}(B)$, as desired.

Finally, we note that since $\End_{\Kbar}(B)$ is a domain, $B$ is
absolutely simple.

This completes the proof.
\end{proof}

\subsection{Simplicity of $A$}
Note that $k(t^{1/d})$ is linearly disjoint from $k((1-t)^{1/r})$ for
any value of $d$.  Thus the following implies Theorem~\ref{cjh:p2}.

\begin{theorem}
Suppose that $F$ is a finite extension of $K$ that is linearly
disjoint from $k((1-t)^{1/r})$.  Then $A=J^{new}_r$ is simple over $F$,
and we have $\End_F(A)\cong\Z[\zeta_r]$.  
\end{theorem}

\begin{proof}
  $\OO=\Z[\zeta_r]$ is a domain, so if $\End_F(A)\cong\OO$, then $A$
  is simple over $F$.  It thus suffices to show that
  $\End_F(A)\cong\OO$.

  Noting that $\End_\Kbar(A)\tensor\Q=M_2(\Q(\zeta_r)^+)$ is a central
  simple algebra of dimension 4 over $\Q(\zeta_r)^+$, the double
  centralizer theorem implies that
  $$\dim_{\Q(\zeta_r)}\left(\End_F(A)\tensor\Q\right)\le2.$$
But $\End_\Kbar(A)\tensor\Q$ is generated over $\Q(\zeta_r)$ by 1 and
$\sigma$.  Our hypothesis on $F$ and Proposition~\ref{thm:monodromy}
imply that there is an element of $\Gal(F(A[\ell])/F)$ acting on
$A[\ell]$ as $\zeta_r$.  Since $\sigma$ does not commute with
$\zeta_r$, we conclude that $\sigma\not\in\End_F(A)$ and therefore
$\End_\Kbar(A)\tensor\Q=\Q(\zeta_r)$.  Since $\OO$ is the maximal order
in $\Q(\zeta_r)$, we have $\End_F(A)\cong\OO$, as desired.
\end{proof}

%%%%%%%%%%%%%%%%%%%%%%%%%%%%%%%%%%%%

%%% Local Variables: 
%%% mode: latex
%%% TeX-master: "EHR"
%%% End: 

\appendix
%    Include appendix "chapters" here.
% !TEX root =  EHR.tex

\chapter{An additional hyperelliptic family} \label{Sappendix}

\section{Introduction}
In Section~\ref{s:other-curve}, we write that the methods used in this
paper may be applied to the study of the arithmetic of Jacobians of
other generalizations of the Legendre curve.  In this appendix, we
give more details on a family of curves mentioned in that section.

Let $g$ denote an odd positive integer, $p$ an odd prime, and $k$ a
finite field of characteristic $p$ and cardinality $q$.  Let
$a_1,\ldots,a_g\in k$ be distinct and non-zero, and let $X$ be the
smooth, projective, hyperelliptic curve over $K=k(t)$ with affine model
\begin{equation}\label{eq:X-eq}
	y^2 = x \prod_{i = 1}^g (x+a_i)(a_ix+t).
\end{equation}
Note that $X$ has genus $g$.  When $g=1$, $X$ is essentially the
Legendre curve, and there are differences between the cases $g=1$ and
$g>1$, so from now on we assume that $g>1$.

Write $J_X$ for the Jacobian of $X$.  Let $\nu$ be a nonnegative
integer, $d = q^{\nu}+1$, and set $K_d=k(\mu_d,u)$ where $u^d=t$.
Our main object of study is the Mordell-Weil group $J_X(K_d)$ and a
certain subgroup of it generated by explicit divisors on $X$.  

The principal results of this appendix are:

\begin{theorem}\label{thm:app}
\leavevmode
\begin{enumerate}
\item $J_X$ satisfies the conjecture of Birch and Swinnerton-Dyer over
  each of the fields $K_d$. 
\item The 2-power torsion subgroup of $J(K_d)$ has the form
$$(\Z/2^a\Z)\times(\Z/2\Z)^{2g-1}$$
for some integer $a>1$.
\item The rank of  $J_X(K_d)$ is at least $d-1$. 
\item The rank of $J_X(K_d)$ is at most $g^2d$.
\end{enumerate}
\end{theorem} 

The proof of the theorem will occupy the rest of the appendix.  For
point (3), we will exhibit explicit divisors generating a subgroup of
$J_X(K_d)$ of rank at least $d-1$.

\section{The BSD conjecture}\label{s:app-DPCT}
In this section, we prove part (1) of Theorem~\ref{thm:app}.

Consider the curve $X$ over $K_d$ and let $\XX_d$ be a smooth,
projective surface over $k(\mu_d)$ equipped with a morphism
$\XX_d\to\P^1$ whose generic fiber is $X/K_d$.  One may construct
$\XX_d$ starting from the affine surface over $k(\mu_d)$ defined by
$$y^2 = x \prod_{i = 1}^g (x+a_i)(a_ix+u^d)$$
together with its projection to the affine line with coordinate $u$.

As reviewed in \cite{CRM} (and already used in Theorem~\ref{thm:BSD}),
in order to show that $\XX_d$ satisfies the Tate conjecture and $J_X$
satisfies the BSD conjecture, it suffices to show that $\XX_d$ is
dominated by a product of curves.

To show that the surface $\XX_d$ admits a dominant rational map from a
product of curves, we consider the affine surface $\YY_d$ defined by
the equation
$$x^gy^2 = \prod_{i=1}^g (x+a_i)(a_ix+u^d).$$
Observe that since $g$ is odd, $\YY_d$ is birational to $\XX_d$.
  
Let $\CC_d$ denote the smooth, affine curve defined by the
equation $$w^2 = \prod _{i=1}^g (z^d + a_i).$$
Define a morphism
\[
\varphi: \CC_d \times \CC_d \to \YY_d
\]
by $(w_1, z_1, w_2, z_2) \mapsto (x,y,u) = (z_1^d, w_1w_2,z_1z_2)$.
It is easy to see that $\varphi$ is generically finite of degree
$2d$.  This proves that $\YY_d$, and thus $\XX_d$, is dominated by a
product of curves, and it completes the proof of part (1) of
Theorem~\ref{thm:app}.   

Since it is easy to do so, we add some further details on $\varphi$.
First, note that $\CC_d$ admits an action of the group
$G = \mu_2 \times \mu_d$.  Let $G$ act on $\CC_d \times \CC_d$
``anti-diagonally''; that is, for $g \in G$ and
$P,Q \in \CC_d(\overline{K})$, we define 
$$(P,Q)^g = (P^g, Q^{g^{-1}}).$$
We claim that $\varphi$ induces a birational isomorphism from the
quotient $(\CC_d \times \CC_d)/G$ to the surface $\XX_d$.  Indeed, it
is clear from the expression defining $\varphi$ that $\varphi$ factors
through the quotient, and therefore induces a dominant rational map of
degree 1, in other words, a birational isomorphism.  

We note that the arguments of this section prove more generally that
the BSD conjecture holds for $X$ over the fields $\F_{q^n}(t^{1/d})$
for any $n$ and any $d$ prime to $p$.

\section{Descent}
In this section, we prove parts (2) and (3) of Theorem~\ref{thm:app}.

Let $Q_\infty$ be the unique point at infinity with respect to the
model \eqref{eq:X-eq}.  We embed $X$ in $J_X$ using $Q_\infty$ as a
base point:
\begin{align*}
X&\to J_X\\
P&\mapsto P-Q_\infty,
\end{align*}
and we identify points of $X$ with their images in $J_X$.

Let $Q_0$ denote the point $(0,0)$.  For $1 \leq j \leq g$, let $Q_j$
denote the point $(-a _j, 0)$, and let $Q'_j$ denote the point
$(-t/a_j , 0)$.  These (together with $Q_\infty$) are the Weierstrass
points of $X$, they are $K$-rational, and it is well known that their
images in $J_X$ generate its 2-torsion subgroup.  The divisor of the
function $y$ is
$$Q_0+\sum_{j=1}^g \left(Q_j + Q_j'\right)-(2g+1)Q_\infty,$$
so 
\begin{equation}\label{eq;Q-rels}
Q_0=-\sum_{j=1}^g\left(Q_j+Q_j'\right)=\sum_{j=1}^g\left(Q_j+Q_j'\right)
\end{equation}
in $J_X$.  Thus the points $Q_j$ and $Q_j'$ for $1\le j\le g$ form a
basis of the $\F_2$-vector space $J_X(K)[2]$.

Fix a primitive $d$-th root of unity $\zeta_d$ in $K_d$.  For
$0 \leq j \leq d-1$, let
\[
	P_j = \left(\zeta_d^ju , 
   (\zeta_d^ju)^{(g+1)/2}\prod_{i=1}^g(\zeta_d^ju + a_i)^{d/2}\right).
\]
Recall that $g$ is odd, and that $q$ is odd, so $d=q^\nu+1$ is even.
Observe that the $P_j$ are in $X(K_d)$.
Indeed, substituting $\zeta_d^ju$ for $x$ in the right hand side of
\eqref{eq:X-eq}, we find
\begin{eqnarray*}
	\zeta_d^ju \prod_{i = 1}^g (\zeta_d^ju+a_i)(a_i\zeta_d^ju+u^d)
	& = &
	\zeta_d^ju \prod_{i = 1}^g(\zeta_d^ju+a_i)\zeta_d^ju(\zeta_d^{-j}u^{q^\nu}+a_i) \\
	& = &
	(\zeta_d^ju)^{g+1} \prod_{i = 1}^g(\zeta_d^ju+a_i)^{q^\nu+1}
\end{eqnarray*}
since $a_i^{q^\nu}=a_i$ and $\zeta_d^{q^\nu}=\zeta_d^{-1}$.  

Let $T$ be the subgroup of $J_X(K_d)$ generated by the $Q_j$ and
$Q_j'$, and let $V$ be the subgroup of $J_X(K_d)$ generated by $T$ and
the $P_j$.  Using a map related to 2-descent, we are going to prove
that the image of $T$ in $J_X(K_d)/2J_X(K_d)$ has dimension $2g-1$ and
that the image of $V$ in $J_X(K_d)/2J_X(K_d)$ has dimension $d+
2g-1$.  

These assertions imply parts (2) and (3) of Theorem~\ref{thm:app} by a
standard descent argument which we now review.  We have already seen
that $T$, the subgroup generated by the $Q_j$ and $Q_j'$, is equal to
$J_X(K_d)[2]$.  By the structure theorem for finitely generated
abelian groups, the 2-power torsion subgroup $J_X(K_d)[2^\infty]$
satisfies
$$J_X(K_d)[2^\infty]\cong\bigoplus_{\ell=1}^t(\Z/2^{e_\ell}\Z)$$
for uniquely determined integers $t$ and $e_\ell$ with
$e_1\ge e_2\ge\cdots\ge e_t>0$.  Since $J_X(K_d)[2]$ has dimension
$2g$ over $\F_2$, we have that $t=2g$.  Once we know that the image of
$J_X(K_d)[2]$ in $J_X(K_D)/2J_X(K_d)$ has dimension $2g-1$, we find
that exactly one of the $e_\ell$ is $>1$.  This reduces part (2) of
Theorem~\ref{thm:app} to our claim that the image of $T$ in
$J_X(K_d)/2J_X(K_d)$ has dimension $2g-1$.

For part (3), we note that the structure theorem for finitely
generated abelian groups plus the calculation that $J_X(K_d)[2]$ has
dimension $2g$ over $\F_2$ implies that
$$\dim_{\F_2}\left(J_X(K_d)/2J_X(K_d)\right)=\rho+2g$$
where $\rho$ is the rank of $J_X(K_d)$.  Once we know that the
dimension of the image of $V$ in $J_X(K_d)/2J_X(K_d)$ is $d+2g-1$, we
may conclude that $\rho\ge d-1$.  This reduces part (3) of
Theorem~\ref{thm:app} to our claim that the image of $V$ in
$J_X(K_d)/2J_X(K_d)$ has dimension $d+2g-1$.

We now turn to calculating the dimensions of the images of $T$ and $V$
in $J_X(K_d)/2J_X(K_d)$.  In parallel with the discussion in
Section~2.2, we define a 2-descent map
\[
	(x-T)\colon \Div X(K_d) \longrightarrow 
\big({K_d^\times}/{K_d^{\times 2}}\big)^{2g+1}
\]
to serve as a substitute for the coboundary map from
$J_X(K_d)/2J_X(K_d)$ to the cohomology group $H^1(K_d,J_X[2])$.  We
  start by defining a map 
$$
	(x-T)\colon
	X(K_d) \longrightarrow \big({K_d^\times}/{K_d^{\times 2}}\big)^{2g+1}
$$
and then extend by $\Z$-linearity to $\Div X(K_d)$.

To lighten notation, write $W$ for
$({K_d^\times}/{K_d^{\times 2}})^{2g+1}$ and 
$$\underline{w}=(w_0,w_1,\dots,w_g,w_1',\dots,w_g')$$ 
for an element of $W$.  If $P \in X(K_d)$ and
$P \neq Q_j, Q_j',Q_\infty$, then the map is defined by
$$(x-T)(P) =(w_0,w_1,\dots,w_g,w_1',\dots,w_g')$$ 
where
\begin{align*}
w_0&=x(P),\\
w_i&=x(P) + a_i\qquad\ \text{for}\ 1\le i\le g,\\
w_i'&=a_ix(P) + t\qquad \text{for}\ 1\le i\le g.
\end{align*}
When $P = Q_j$ or $Q_j'$, this expression needs further clarification,
since it gives zero for one coordinate.  Instead, we set the value at
that coordinate to be the product of the other coordinates
(cf.~Prop~\ref{computation}).  Finally, we define
$(x-T)(Q_\infty) = (1,1,\dots,1)$. 

An analysis parallel to that in Chapter~2 and \cite{bps} shows that the
composition
\begin{multline*}
$$\Div^0 X(K_d)\to\J_X(K_d)\to J_X(K_d)/2J_X(K_d)\\
\into
H^1(K_d,J_X[2])\subset \big({K_d^\times}/{K_d^{\times 2}}\big)^{2g+1}
\end{multline*}
is equal to the restriction of $(x-T)$
to $\Div^0 X(K_d)$.  In particular, to compute the images of $T$ and
$V$ in $J_X(K_d)/2J_X(K_d)$, it will suffice to compute their images
in $W$, i.e., their images under $(x-T)$.

Note that $t$ and the elements of $k=\F_q$ are squares in $K_d^\times$,
i.e., trivial in $K_d^\times/K_d^{\times 2}$.  From this it follows
that $(x-T)(Q_0)$ is trivial, and in view of \eqref{eq;Q-rels}, we have
$$(x-T)\left(\sum_{j=1}^g\left(Q_j+Q_j'\right)\right)=(1,\dots,1).$$
(This can also of course be checked directly.)  It follows that the
dimension of the image of $T$
in $W$ is at most $2g-1$ and the dimension of $V$ in $W$ is at most
$d+2g-1$. To complete the proof, we must show that these inequalities
are in fact equalities. 

Observe that the field of constants in $K_d$ is isomorphic to
$\F_{q^{2\nu}}$. The norm map 
\[
  \F_{q^{2\nu}}^\times \longrightarrow \F_{q^{\nu}}^\times
\]
is given by $\alpha \mapsto \alpha^d$ and is surjective.  It follows
that any $a\in k=\Fq$ has a $d$-th root in $K_d$, and the place of
$K=k(t)$ where $t-a$ vanishes splits into $d$ places in $K_d$. These
are the places where $u-\zeta_d^j\alpha$ vanishes with $0\le j\le d-1$
and $\alpha$ a fixed $d$-th root of $a$.

Now fix a $d$-th root $\alpha_i$ of $a_1a_i$ for $1\le i\le g$.  It
will be convenient later to assume that $\alpha_1=-a_1$.  (This is
legitimate, since $a_1\in\Fq$ so $(-a_1)^d=(-a_1)^{q+1}=a_1^2$.)  Let
$\pi_i$ be the place of $K_d$ where $u-\alpha_i$ vanishes, and let
$\ord_{\pi_i}$ be the corresponding valuation.

In parallel with the proof of Prop.~\ref{dimensiond}, define a linear
map $pr_1:W\to\F_2^{2g}$ by
$$pr_1(\underline{w})
=\left(\ord_{\pi_1}(w_1),\dots,\ord_{\pi_g}(w_1),
\ord_{\pi_1}(w_1'),\dots,\ord_{\pi_g}(w_1')\right).$$
Let $I$ be the $g\times g$ identity matrix over $\F_2$ and let $B$ be
the $g\times g$ matrix over $\F_2$ whose first row entries are all 1 and
whose other entries are 0.  Then a straightforward application of the
definitions shows that the matrix whose rows are 
$\pr_1\circ(x-T)(Q_j)$ (for $1\le j\le g$) followed by
$\pr_1\circ(x-T)(Q_j')$ (for $1\le j\le g$) has the form
\begin{equation}\label{eq:Q-image}
\begin{pmatrix}
B&I\\I&B
\end{pmatrix}.
\end{equation}
This matrix has rank $2g-1$ which implies that the dimension of the
image of $T$ in $W$ is $2g-1$.  This completes the proof of part (2)
of Theorem~\ref{thm:app}.

Working toward part (3) of the theorem, we next consider the images of
the points $P_j$ under $pr_1\circ(x-T)$.  Keeping in mind our choice
of $\alpha_1$ above, we find that
\begin{equation}\label{eq:P-image}
pr_1\circ(x-T)(P_j)=\begin{cases}
(1,0,\dots,0,1,0,\dots,0)&\text{if $j=0$}\\
(0,\dots,0)&\text{if $j\neq 0$}
\end{cases}
\end{equation}
where the entries 1 appear in columns $1$ and $g+1$.  In particular,
using equations~\eqref{eq:Q-image} and \eqref{eq:P-image}, we have
\begin{align}\label{eq:PQ-relation1}
pr_1\circ(x-T)(P_0)&=pr_1\circ(x-T)\left(Q_1+\sum_{j=2}^gQ_j'\right)\\
\label{eq:PQ-relation2}
&=pr_1\circ(x-T)\left(Q_1'+\sum_{j=2}^gQ_j\right).
\end{align}
It follows that the image of $V$ in $\F_2^{2g}$ is the same as the
image of $T$ in $\F_2^{2g}$, and this image has dimension $2g-1$.
Let $V_1$ denote the kernel of the map
$$pr_1\circ(x-T):V\to\F_2^{2g}.$$
We have that $V_1$ contains $2V$, $P_j$ for $1\le j\le d-1$, and (by
equations~\eqref{eq:PQ-relation1} and \eqref{eq:PQ-relation2}) the elements
$$P_0+Q_1+\sum_{j=2}^gQ_j'\quad\text{and}\quad
P_0+Q_1'+\sum_{j=2}^gQ_j.$$
A dimension count then shows that these elements generate $V_1$.

Now we introduce a second projection $pr_2:W\to\F_2^d$.  Namely, let
$\rho_j$ be the place of $K_d$ where $u+\zeta_d^{-j}a_1$ vanishes, and
let $\ord_{\rho_j}$ be the corresponding valuation.  Then define
$$pr_2(\underline{w})=
\left(\ord_{\rho_0}(w_1),\dots,\ord_{\rho_{d-1}}(w_1)\right).$$
Let $e_0,\dots,e_{d-1}$ be the standard basis of $\F_2^d$ with a shift
of one in the indexing (so $e_0=(1,0,\dots,0)$ and $e_{d-1}=(0,\dots,0,1)$).
Then a straightforward calculation reveals that
\begin{align*}
pr_2\circ(x-T)(P_j)&=e_j,\\
pr_2\circ(x-T)(Q_1)&=\sum_{\ell=0}^{d-1}e_\ell,\\
pr_2\circ(x-T)(Q_j)&=0\quad\text{for $2\le j\le g$},\\
pr_2\circ(x-T)(Q_1')&=\sum_{\ell=0}^{d-1}e_\ell,\\
\noalign{\text{and}}
pr_2\circ(x-T)(Q_j')&=0\quad\text{for $2\le j\le g$.}
\end{align*}
It follows easily that $pr_2\circ(x-T)$ sends $V_1$ surjectively onto
the codimension 1 subspace of $\F_2^d$ where the first entry vanishes.
Denoting by $V_2$ the kernel of $pr_2\circ(x-T)$ on $V_1$, we also see
that $V_2$ is generated by $2V$ and the element
$$\sum_{j=0}^{d-1}P_j+Q_1+\sum_{j=2}^gQ_j'.$$

To finish the proof, we note that 
$$(x-T)\left(\sum_{j=0}^{d-1}P_j+Q_1+\sum_{j=2}^gQ_j'\right)\neq0.$$
For example, its component $w_2$ is
$$\prod_{\ell=3}^g(t-a_2a_\ell),$$
and this is not a square in $K_d$ since the $a_i$ are distinct.

In summary, we have shown that $V/V_1$ has dimension $2g-1$, $V_1/V_2$
has dimension $d-1$ and $V_2$ has a 1-dimensional image in $W$.  This
shows that the image of $V$ in $W$ has dimension $d+2g-1$, and this
completes the proof of part (3) of Theorem~\ref{thm:app}.

\section{Degree of the $L$-function}
In this section, we sketch a proof of part (4) of
Theorem~\ref{thm:app}.
 
Since the BSD conjecture holds for $J_X$, the rank of $J_X(K_d)$ is
equal to the order of vanishing of $L(J_X/K_d,s)$ at $s=1$.  We will
show that the $L$-function is a polynomial in $q^{-s}$ and estimate
its degree, thus giving an upper bound on the order of vanishing and
the rank.

It is known that $L(J_X/K_d,s)$ is a rational function in $q^{-s}$ and
that it is a polynomial in $q^{-s}$ if and only if the $K/k$-trace of
$J_X$ (or more precisely, the $K_d/k(\mu_d)$-trace) vanishes; see
\cite[Chap.~5, Lemma~6.5]{CRM}.  Arguing as in
Proposition~\ref{prop:K/k-trace} and using the explicit domination of
$\XX_d$ by a product of curves given in Section~\ref{s:app-DPCT}, we
see that the trace vanishes and the $L$-function is a polynomial.

To complete this sketch we estimate the degree of the $L$-function of
the Jacobian, and thus determine the upper bound on the rank.  By the
Grothendieck-Ogg-Shafarevich formula, the degree of the $L$-function
is
$$-4g_X+\deg(\n)$$
where $\n$ denotes conductor of $J_X$ over
$K_d$ and  $\deg(\n)$ denotes its degree.

We start by considering the case $d = 1$.

For $1 \leq i\leq j \leq g$, let $S$ be the set of places
corresponding to the polynomials $t - a_ia_j$.  Letting $c_v$ be as in
Section~\ref{ss:GOS}, one checks that $\XX_1$ has semistable reduction
at each such place and that
\[
	\sum_{v \in S} c_v = g + 2\binom{g}{2} = g^2.
\]
More precisely, for each pair $i\leq j$, the reduction of $X$ at
$t-a_ia_j$ has ordinary double points at $(x,y)=(a_i,0)$ and
$(x,y)=(a_j,0)$, and their contribution to the conductor is 1 when
$i=j$ and 2 when $i<j$.  Moreover, the reduction of $X$ is smooth away
from such double points, so it suffices to count the number of pairs
$(i,j)$ with $1\leq i, j\leq g$.

The only other places of (possibly) bad reduction are at
$t = 0,\infty$.  We claim the Tate module of $J_X$ has tame reduction
there and thus the corresponding contribution to the conductor equals
the drop of the degree of the corresponding Euler factor (which is
between 0 and $2g$).  Indeed, the extension $K(J_X[4])$ is
Galois over $K_1$ of degree a power of two, so it is a tamely ramified
extension of $K$.  Moreover, $J_X$ acquires semiabelian reduction
over it.  In particular, this implies the Tate module of $J_X$ is
everywhere tamely ramified over the extension $K(J_X[4])$ and hence
over $K$.

Now we consider the case $d>1$.

For each of the $d$ places of $K_d$ over $t-a_ia_j$, the contribution
to the conductor remains unchanged, so is 1 when $i=j$ and 2 when
$i<j$.  Moreover, the contribution to the conductor is between $0$ and
$2g$ for $u=0,\infty$.  Therefore we have
$$\deg(\n)\le d\cdot\sum_{v\in S}c_v + 2\cdot 2g=dg^2+4g$$
which implies that the degree of the $L$-function is $\le dg^2$.  It
follows that the rank of the Mordell-Weil group of $J_X(K_d)$ is also
at most $g^2d$.  This completes our sketch of the proof of part (4) of
Theorem~\ref{thm:app}.

%The precise rank over $K_d$ is an open question.

\section{Additional remarks}

\begin{remark}
  It is interesting to note the differences between the Jacobian
  studied in this appendix and the Legendre curve $E$ studied in
  \cite{Legendre}.  In particular, the rank of $J_X(K_d)$ is either
  $d-1$ or $d$, whereas the rank of $E(K_d)$ is $d-2$.  There are two
  relations among the analogues of the $P_j$ on $E$ which seem not to
  generalize readily to $X$, although it is possible that there is one
  relation.
\end{remark}

\begin{remark}
  There is an interesting involution of $X$, given in the coordinates
  above by 
$$\iota(x,y)=\left(t/x,yt^{(g+1)/2}/x^{g+1}\right).$$
(An analogous involution exists for $E$, where it is translation by
$Q_0$.)  On $X$, we have $\iota(Q_0)=Q_\infty$, $\iota(Q_i)=Q_i'$ for
$1\le i\le g$, and $\iota(P_j)=P_j'$ where the $P_j'$ are new points
with coordinates
$$P_j'=\left(\zeta_d^{-j}u^{q^\nu},
(\zeta_d^{-j}u^{q^\nu})^{(g+1)/2}\prod_{i=1}^g(\zeta_d^ju+a_i)^{d/2}\right).$$
We do not know whether these points are independent of the $P_j$.

It would be interesting to investigate the consequences for the
monodromy of $J_X[\ell]$ of the existence of $\iota$, along the lines
of Chapter~\ref{ch:monodromy}.
\end{remark}

%%% Local Variables: 
%%% mode: latex
%%% TeX-master: "EHR"
%%% End: 

\backmatter

\bibliographystyle{amsplain.bst}
\bibliography{EHR}

\printindex

\end{document}